\documentclass[pdflatex,sn-basic]{sn-jnl}% Math and Physical Sciences Numbered Reference Style
\usepackage[all]{xy}

\usepackage{graphicx}%
\usepackage{multirow}%
\usepackage{amsmath,amssymb,amsfonts}%
\usepackage{amsthm}%
\usepackage{mathrsfs}%
\usepackage[title]{appendix}%
\usepackage{xcolor}%
\usepackage{textcomp}%
\usepackage{manyfoot}%
\usepackage{booktabs}%
\usepackage{algorithmicx}%
\usepackage{algpseudocode}%
\usepackage{listings}%
\theoremstyle{thmstyleone}%
\newtheorem{theorem}{Theorem}%  meant for continuous numbers
\newtheorem{proposition}[theorem]{Proposition}% 

\theoremstyle{thmstyletwo}%

\theoremstyle{thmstylethree}%

\newtheorem{corollary}[theorem]{Corollary}
\newtheorem{lemma}[theorem]{Lemma}
\newtheorem{fact}[theorem]{Fact}
\newtheorem*{fact*}{Fact}
\theoremstyle{definition}
\newtheorem{definition}[theorem]{Definition}
\newtheorem{problem}[theorem]{Problem}
\newtheorem*{claim}{Claim}
\newtheorem{remark}[theorem]{Remark}
\numberwithin{equation}{section}
\newtheorem{algorithm}[theorem]{Algorithm}

\newcommand{\bfs}{\boldsymbol}

\newcommand{\N}{\mathbb N}
\newcommand{\Z}{\mathbb Z}

\newcommand{\A}{\mathbb A}
\newcommand{\F}{\mathbb F}

\newcommand{\Q}{\mathbb Q}
\newcommand{\SO}{\mathcal{O}^{\sim}}

\newcommand{\fp}{\F_{\hskip-0.7mm p}}
\newcommand{\fpe}{\F_{\hskip-0.7mm p^e}}
\newcommand{\fq}{\F_{\hskip-0.7mm q}}
\newcommand{\fqe}{\F_{\hskip-0.7mm q^e}}

\newcommand{\cfq}{\overline{\F}_{\hskip-0.7mm q}}
\newcommand{\cfp}{\overline{\F}_{\hskip-0.7mm p}}
\newcommand{\ck}{\overline{k}}
\def\ifm#1#2{\relax \ifmmode#1\else#2\fi}

\newcommand{\klk}    {\ifm {,\ldots,} {$,\ldots,$}}

\newcommand{\plp}    {\ifm {+\cdots+} {$+\ldots+$}}

\newcommand{\om}[2]   {{#1}_1 \klk {#1}_{#2}}

\newcommand{\xon}    {\ifm {\om X n} {$\om X n$}}

\begin{document}

\title[A Kronecker algorithm]{A Kronecker algorithm for locally closed sets
over a perfect field}


\author[1,2]{\fnm{Nardo} \sur{Gim\'enez}}\email{nardo.gimenez@unahur.edu.ar}

\author[2,3]{\fnm{Joos} \sur{Heintz}}\email{joos@dc.uba.ar}

\author*[2,4]{\fnm{Guillermo} \sur{Matera}}\email{gmatera@ungs.edu.ar}

\author[5]{\fnm{Luis Miguel} \sur{Pardo}}\email{luis.m.pardo@gmail.com}

\author[1,2]{\fnm{Mariana} \sur{P\'erez}}\email{mariana.perez@unahur.edu.ar}

\author[1,2]{\fnm{Melina} \sur{Privitelli}}\email{melina.privitelli@unahur.edu.ar}

\affil[1]{\orgdiv{Instituto de Tecnolog\'ia e
Ingenier\'ia}, \orgname{Universidad Nacional de Hurlingham}, \orgaddress{\street{Av. Gdor. Vergara 2222}, 
\city{Villa Tesei}, \postcode{B1688GEZ}, \state{Buenos Aires}, \country{Argentina}}}

\affil[2]{\orgname{National Council of Science and Technology (CONICET)}, \orgaddress{\country{Argentina}}}

\affil[3]{\orgdiv{Departamento de Computaci\'on}, \orgname{Facultad de Ciencias Exactas y
Naturales, Universidad de Buenos Aires}, \orgaddress{\street{Ciudad Univ., Pab. I}, \postcode{1428}, \state{Buenos Aires}, 
\country{Argentina}}}

\affil*[4]{\orgdiv{Instituto del Desarrollo Humano}, \orgname{Universidad Nacional de Gene\-ral Sarmiento}, 
\orgaddress{\street{J.M. Guti\'errez 1150}, \city{Los Polvorines}, \postcode{B1613GSX}, 
\state{Buenos Aires}, \country{Argentina}}}

\affil[5]{\orgname{Professor of Algebra (ret.)}, \orgaddress{\street{Paseo de Altamira, n. 83, 1-A}, \postcode{39006}, 
\state{Santander}, \country{Spain}}}

\abstract{We develop a probabilistic algorithm of Kronecker type for computing
a Kronecker representation of a zero-dimensional linear section of
an algebraic variety $V$ defined over a perfect field $k$. The
variety $V$ is the Zariski closure of the set of common zeros
$\{F_1=0,\ldots,F_r=0,G\not=0\}$ of multivariate polynomials
$F_1,\ldots,F_r\in k[X_1,\ldots,X_n]$ outside a prescribed
hypersurface $\{G=0\}$. We assume that $F_1,\ldots,F_r$ satisfy
natural geometric conditions, such as regularity and radicality, in
the local ring $k[X_1,\ldots,X_n]_G$. Our approach combines
homotopic deformation techniques with symbolic Newton-Hensel lifting
and elimination. We discuss the concept of lifting curves as
intermediate geometric objects that enable efficient computation.

The complexity of the algorithm is expressed in terms of the degrees and
arithmetic size of the input and achieves soft-quadratic complexity in these
parameters. We provide detailed complexity analyses for arbitrary perfect
fields, as well as for two important cases in computer algebra: finite fields
and the field of rational numbers. For each case, we obtain sharp bounds on
the size of the base field or required primes.
}




\maketitle

\section{Introduction}This research was promoted, supported, and occasionally conducted in
collaboration with Joos Heintz during his final years. We consider
it his posthumous research work, which is why we have included his
name as one of the co-authors.

Our main objective is to present a version of an algorithm for
solving multivariate polynomial systems that works over a general
perfect field $k$. This algorithm belongs to a family of algorithms
known as ``Kronecker'' algorithms. Therefore, we begin by discussing
the general concept of a Kronecker algorithm. More precisely, we
will see that a Kronecker algorithm can be viewed as a procedure
consisting of a sequence of {\em iterative deformations within a
solution variety}.

We believe it is useful to contextualize the main contributions of
this work by providing an overview of the TERA-Kronecker algorithm
suite and illustrating where our contributions fit within it. We do
not present the algorithm in the customary way. On the contrary, we
have attempted to offer an unconventional presentation (probably not
in accordance with Heintz's taste), intended not only to serve the
aforementioned purpose but also provides a more didactic view ---at
least as much as possible--- of how and why these algorithms differ
from others in the standard mathematical literature and why they are
more efficient.

We aim to present the algorithm in a high-level language,
emphasizing the geometric and algebraic elements involved in its
iterations, while relegating the more technical details to the
remainder of the manuscript. This approach should also offer the
reader a clearer view of the specific contributions of the
TERA-Kronecker suite to the problem of solving systems of polynomial
equations.
%
% ----------------------------------------------------------------------
% ----------------------------------------------------------------------
% ----------------------------------------------------------------------
% ----------------------------------------------------------------------
%
\subsection{Outline of the TERA-Kronecker suite}
To simplify the discussion, throughout this section we consider an
algebraically closed field $\overline{k}$. We denote by $\A^n$ the
$n$-dimensional affine space $\overline{k}{}^n$. We begin with a
generalization of the notion of {\em solution variety} of
\cite{ShSm93} to the field $\overline{k}$.
\subsubsection{The solution variety}
\label{subsec: solution variety} For any pair $(n,d)$ of positive
integers, let
$$\mathcal{P}_d(X_1,\ldots,X_n):=\overline{k}[X_1,\ldots,X_n]_d$$
be the $\overline{k}$-vector space of polynomials in the variables
$X_1,\ldots,X_n$ with coefficients in $\overline{k}$ and degree at
most $d$. This vector space has dimension
$$\dim_{\overline{k}}(\mathcal{P}_d(X_1,\ldots,X_n))=
\binom{d+n}{n}.$$
We denote by $\mathcal{P}_{d,\bfs 0}(X_1,\ldots,X_n)$ the vector
subspace of polynomials of $\mathcal{P}_d(X_1,\ldots,X_n)$ that
vanish at the origin $\bfs 0:=(0,\ldots,0)\in\A^n$, namely:
$$\mathcal{P}_{d,\bfs 0}(X_1,\ldots,X_n):=
\{F\in \mathcal{P}_d(X_1,\ldots,X_n):F(\bfs 0)=0\}.$$
Clearly, $\mathcal{P}_{d,\bfs 0}(X_1,\ldots,X_n)$ is a subspace of
codimension 1 of $\mathcal{P}_{d}(X_1,\ldots,X_n)$.

For every $r$-tuple of degrees $\bfs d_r:=(d_1,\ldots,d_r)$, we
consider the following $\overline{k}$-vector spaces:
\begin{align*}
\mathcal{P}_{\bfs d_r}(X_1,\ldots,X_n) &:=
\prod_{i=1}^r\mathcal{P}_{d_i}(X_1,\ldots,X_n), \\
\mathcal{P}_{\bfs d_r,\bfs 0}(X_1,\ldots,X_n) &:=
\prod_{i=1}^r\mathcal{P}_{d_i,\bfs 0}(X_1,\ldots,X_n).
\end{align*}
For simplicity, we fix $n$ and drop the variables from the notation,
writing $\mathcal{P}_{\bfs d_r}$ and $\mathcal{P}_{\bfs d_r,\bfs 0}$
instead. It is clear that $\mathcal{P}_{\bfs d_r,\bfs 0}$ is
$\overline{k}$-vector subspace of $\mathcal{P}_{\bfs d_r}$ of
codimension $r$. The dimension of $\mathcal{P}_{\bfs d_r}$ is
$$N_{\bfs d_r}:=\dim_{\overline{k}}\mathcal{P}_{\bfs d_r}=
\sum_{i=1}^{r}\binom{d_i+n}{n},$$
so that
$$\dim_{\overline{k}}\mathcal{P}_{\bfs d_r,\bfs 0}=N_{\bfs d_r}-r.$$

With these notations, we define the solution variety over
$\overline{k}$ associated with the degree sequence $\bfs d_r$ in $n$
variables as follows:
$$V_{\bfs d_r}:=V_{\bfs d_r}^{\overline{k}}(n):=
\{(\bfs x,\bfs F_r)\in\A^n\times \mathcal{P}_{\bfs d_r}: \bfs
F_r(\bfs x)=\bfs 0\}.$$
Next, we examine some elementary properties of $V_{\bfs d_r}$.
Consider the {\em evaluation map}
\begin{align*}
  ev_{\bfs d_r}:\A^n\times\mathcal{P}_{\bfs d_r} &\to \A^r, \\
  (\bfs x, \bfs F_r) & \mapsto \bfs F_r(\bfs x).
\end{align*}
Observe that $V_{\bfs d_r}$ is the fiber $ev_{\bfs d_r}^{-1}(\{\bfs
0\})$. Moreover, it is easy to see that $ev_{\bfs d_r}$ is
surjective. In fact, this follows from the alternative expression
$$V_{\bfs d_r}=\{(\bfs x,\bfs w,\bfs G)\in\A^n\times\A^r\times\mathcal{P}_{\bfs d_r,\bfs 0}:
\bfs G(\bfs x)+\bfs w=\bfs 0\}.$$
This identity shows that $V_{\bfs d_r}$ is, up to a permutation of
coordinates, the graph of the polynomial map
\begin{align*}
  \varphi_{\bfs d_r}:\A^n\times\mathcal{P}_{\bfs d_r,\bfs 0} & \to\A^r, \\
  (\bfs x,\bfs G) & \mapsto \bfs G(\bfs x).
\end{align*}
Thus, there is a biregular isomorphism between $V_{\bfs d_r}$ and
the affine space $\A^n\times \mathcal{P}_{\bfs d_r,\bfs 0}$. In
particular, we obtain the following:
\begin{remark}
$V_{\bfs d_r}$ is a smooth, irreducible affine variety of Krull dimension
$$\dim V_{\bfs d_r}=N_{\bfs d_r}+n-r.$$
\end{remark}

Next we consider the canonical projections
$$\xymatrix{
 & V_{\bfs d_r}\ar[ld]_{\Pi_1} \ar[rd]^{\Pi_2} & \\
\mathcal{P}_{\bfs d_r}  & & \A^n} \quad\xymatrix{
 & (\bfs x,\bfs F_r)\ar[ld]_{\Pi_1} \ar[rd]^{\Pi_2} & \\
\bfs F_r  & & \bfs x}$$
The projection $\Pi_2$ is an onto mapping whose fiber at a given point $\bfs
x\in\A^n$ is the set
of $r$-tuples of polynomials vanishing on $\bfs
x$. More precisely, given $W\subset\A^n$, let $I(W)\subset
\overline{k}[X_1,\ldots,X_n]$ be the vanishing ideal of
$W$, that is, the ideal of all polynomials in
$\overline{k}[X_1,\ldots,X_n]$ vanishing at all the points of $W$.
Then the fiber $\Pi_2^{-1}(W)$ is
$$\Pi_2^{-1}(W)=\mathcal{P}_{\bfs d_r}\cap (I(W)^r),$$
where $I(W)^r:=\prod_{i=1}^rI(W)$. In the case that $W$ is the
singleton $W=\{\bfs x\}$, we denote its vanishing ideal as
$\mathfrak{m}_{\bfs x}:=I(\{\bfs x\})$ and the corresponding fiber
is
$$\Pi_2^{-1}(\bfs x)=\mathcal{P}_{\bfs d_r}\cap (\mathfrak{m}_{\bfs x}^r)=
\prod_{i=1}^r(\mathcal{P}_{d_i}\cap \mathfrak{m}_{\bfs x}).$$
Observe that $\mathcal{P}_{d_i}\cap \mathfrak{m}_{\bfs x}$ is a
vector subspace of codimension 1 of $\mathcal{P}_{d_i}$. Hence, the
fiber at any point $\bfs x\in\A^n$ can be identified with a vector
subspace of $\mathcal{P}_{\bfs d_r}$ of codimension $r$, i.e.,
$$\dim\big(\Pi_2^{-1}(\{\bfs x\})\big)=N_{\bfs d_r}-r.$$
%
%It is easy to see that the following mapping is onto:
%%
%\begin{align*}
%  \mathcal{T}_{(\bfs x,\bfs F_r)}\Pi_2:\mathcal{T}_{(\bfs x,\bfs F_r)}V_{\bfs d_r} & \to
%  \mathcal{T}_{\bfs x}\A^n=\A^n, \\
%  (\dot{\bfs x},\dot{\bfs F_r}) & \mapsto \dot{\bfs x},
%\end{align*}
%%
%where $(\dot{\bfs x},\dot{\bfs F_r})$ satisfies $\dot{\bfs F_r}(\bfs x)^t
%+D{\bfs F_r}(\bfs x)\cdot\dot{\bfs x}^t=\bfs 0$. Further, the tangent space
%at any fiber may be identified with the kernel
%%
%$$\mathcal{T}_{(\bfs x,\bfs F_r)}\Pi_2^{-1}(\{\bfs x\})=
%\mathrm{ker}\big(\mathcal{T}_{(\bfs x,\bfs F_r)}\Pi_2\big).$$
%%
%This can be proved by an argument of dimensions, as done with
%$\mathcal{T}_{(\bfs x,\bfs F_r)}V_{\bfs d_r}$.

Finally, we may define the problem of solving a system of $r$
polynomial equations of given degrees in terms of the projections $\Pi_1$
and $\Pi_2$.
\begin{problem}[Solving $r$ Polynomial Equations in $n$ Variables]
Given an algebraically closed field $\overline{k}$ and an $r$-tuple
$\bfs F_r\in\mathcal{P}_{\bfs d_r}$ of polynomials in $n$ variables
of degrees bounded by the list $\bfs d_r:=(d_1,\ldots,d_r)$, compute
$\Pi_2\big(\Pi_1^{-1}(\{\bfs F_r\})\big)$.
\end{problem}

A problem with solving in this sense is that the mapping $\Pi_1$ is
not onto for $r\ge 2$ as it can be easily verified. Nevertheless,
there is a Zariski open subset
$\mathcal{U}_0\subset\mathcal{P}_{\bfs d_r}$ such that for any $\bfs
F_r\in \mathcal{U}_0$, the fiber $\Pi_1^{-1}(\{\bfs F_r\})$ is not
empty. This may be proved considering multivariate resultants or the
effective Nullstellensatz, which we omit. In particular, the mapping
$\Pi_1$ is dominant and we always assume that $\bfs F_r\in
\mathcal{U}_0\subset\Pi_1(V_{\bfs d_r})$.

In \cite{GiMaPePr23}, the set of ``defective'' $r$-tuples $\bfs
F_r\in \mathcal{P}_{\bfs d_r}$ with $r\le n$ are studied. In
particular, \cite[Theorem 3.6]{GiMaPePr23} shows that there exists a
nonempty Zariski open subset $\mathcal{U}_1\subset \mathcal{U}_0$
with the following property: for any $\bfs F_r:=(F_1,\ldots,F_r)\in
\mathcal{U}_1$,
\begin{itemize}
  \item[$({\sf H}_1)$] the polynomials $F_1,\ldots,F_r$ form a regular sequence,
  \item[$({\sf H}_2)$] the ideal $(F_1,\ldots,F_r)$ is radical.
\end{itemize}
In other words, any $\bfs F_r\in\mathcal{U}_1$ defines an
ideal-theoretic complete intersection. Due to the fact that these
conditions hold for a generic $\bfs F_r\in\mathcal{P}_{\bfs d_r}$,
we now restrict our discussion about solving to this case.
\begin{problem}[Solving a System Defining an Ideal-Theoretic Complete Intersection]
Given an algebraically closed field $\overline{k}$ and an $r$-tuple
$\bfs F_r\in\mathcal{P}_{\bfs d_r}$ of polynomials in $n$ variables
of degrees bounded by the sequence $\bfs d_r:=(d_1,\ldots,d_r)$,
defining an ideal-theoretic complete-intersection (i.e., satisfying
$({\sf H}_1)$-$({\sf H}_2)$), compute $\Pi_2\big(\Pi_1^{-1}(\{\bfs
F_r\})\big)$.
\end{problem}
%
% ----------------------------------------------------------------------
% ----------------------------------------------------------------------
% ----------------------------------------------------------------------
% ----------------------------------------------------------------------
%
\subsubsection{The left action of $GL(n,\overline{k})$ on the solution
variety}
\label{subsec: left action GL}
Let $r,n$ be positive integers with $1\le r\le n$, and let $\bfs
d_r:=(d_1,\ldots,d_r)$ be a degree sequence with $d_i\ge 1$ for
$1\le i\le r$. Let $V_{\bfs d_r}$ be the solution variety associated
with $\bfs d_r$ over $\overline{k}$. Let $\mathcal{U}_{\bfs d_r}$ be
a nonempty Zariski open subset of $\mathcal{P}_{\bfs d_r}$ such that
any $\bfs F_r\in\mathcal{U}_{\bfs d_r}$ defines an ideal-theoretic
complete intersection.

Now consider the general linear group $GL(n,\overline{k})$ and its
left action, given by
\begin{align*}
  \mathcal{L}:GL(n,\overline{k})\times V_{\bfs d_r}&\to V_{\bfs d_r},\\
  \big(\bfs\lambda,(\bfs x,\bfs F_r)\big)&\mapsto (\bfs\lambda\cdot\bfs x, \bfs F_r\circ \bfs\lambda^{-1}).
\end{align*}
%
%This action defines linear isomorphisms between fibers
%$\big\{\Pi_2^{-1}(\{\bfs x\}):\bfs x\in\A^n\big\}$.
%
Since any $\bfs F_r\in\mathcal{U}_{\bfs d_r}$ defines an
ideal-theoretic complete intersection, the Noether normalization
theorem guarantees the existence of a nonempty Zariski open subset
$\mathcal{U}(\bfs F_r)\subset GL(n,\overline{k})$ such that for any
$\bfs\lambda\in\mathcal{U}(\bfs F_r)$, the following map is a finite
morphism:
\begin{align}
  \pi_{\bfs\lambda}:V(\bfs F_r)&\to \A^{n-r},\label{eq: Noether normalization}\\
  \bfs x&\mapsto (\bfs\lambda\cdot\bfs x)_{n-r},\notag
\end{align}
where $V(\bfs F_r)$ denotes the zero set of $F_1,\ldots,F_r$ in $\A^n$
and $(\bfs\lambda\cdot\bfs x)_{n-r}$ denotes the first $n-r$
coordinates of $\bfs\lambda\cdot\bfs x$. In algebraic terms, let
$Y_1,\ldots,Y_n$ be new variables defined via
$$\left(
    \begin{array}{c}
      Y_1 \\
      \vdots \\
      Y_n \\
    \end{array}
  \right):=\bfs\lambda
\left(
  \begin{array}{c}
    X_1 \\
    \vdots \\
    X_n \\
  \end{array}
\right),
$$
and let $\overline{k}[V(\bfs F_r)]$ be the $\overline{k}$-algebra of
polynomial functions defined on $V(\bfs F_r)$. Then we have an
integral ring extension
\begin{equation}\label{eq: integral ring extension}
R_r:=\overline{k}[Y_1,\ldots,Y_{n-r}]\hookrightarrow \overline{k}[V(\bfs
F_r)].
\end{equation}
We will show that, for a given $\bfs d_r$, there exists a finite,
effectively computable subset
$\mathcal{S}:=\mathcal{S}(\bfs d_r)\subset \overline{k}$
such that, for every
$\bfs F_r\in \mathcal{U}_{\bfs d_r}$ and for most choices of
$\bfs\lambda\in\mathcal{S}^{n\times n}$, the inclusion above
is an integral ring extension.

Now consider the $r\times r$ Jacobian matrix
$$J_r\bfs F_r:=\left(
               \begin{array}{ccc}
                 \frac{\partial F_1}{\partial Y_{n-r+1}} & \cdots & \frac{\partial F_1}{\partial Y_n} \\
                 \vdots &  & \vdots \\
                 \frac{\partial F_r}{\partial Y_{n-r+1}} & \cdots & \frac{\partial F_r}{\partial Y_n} \\
               \end{array}
             \right)\in(\overline{k}[V(\bfs F_r)])^{r\times r}.
$$
If $\bfs F_r\in\mathcal{U}_{\bfs d_r}$ and
$\bfs\lambda\in\mathcal{U}(\bfs F_r)$, then the determinant
$\det(J_r\bfs F_r)$ is not a zero-divisor in $\overline{k}[V(\bfs
F_r)]$. In particular, when $r=n$, the zero-dimensional variety
$V(\bfs F_n)\subset \A^n$ is smooth.

Moreover, $\overline{k}[V(\bfs F_r)]$ is a Cohen-Macaulay ring, and
a free $R_r$-module of finite rank (see \cite{Iversen73}, or the
inductive argument using the Quillen-Suslin theorem of
\cite{GiHeSa93}). Therefore, by localizing at the multiplicative
set $S:=R_r\setminus\{0\}$, we obtain an integral
ring extension of function fields
$$K_r:=\overline{k}(Y_1,\ldots,Y_{n-r})\hookrightarrow K_r
\otimes_{R_r}\overline{k}[V(\bfs
F_r)]=S^{-1}\overline{k}[V(\bfs F_r)],$$
where $S^{-1}\overline{k}[V(\bfs F_r)]$ is a reduced, smooth Artin
ring (i.e., it has Krull dimension 0). Additionally, under these
conditions, we have
$$\mathrm{rank}_{R_r}\big(\overline{k}
[V(\bfs F_r)]\big)=\dim_{K_r}\big(
S^{-1}\overline{k}[V(\bfs F_r)]\big).$$
According to \cite{Heintz83}, these quantities are bounded by the
degree of $V(\bfs F_r)$, that is,
$$\mathrm{rank}_{R_r}\big(\overline{k}
[V(\bfs F_r)]\big)\le \deg\big(V(\bfs F_r)\big)\le \mathcal{B}_{\bfs
d_r}:=\prod_{i=1}^rd_i.$$
Here $\mathcal{B}_{\bfs d_r}$ is the {\em B\'ezout number associated
with the degree sequence $\bfs d_r$}. Furthermore,
\cite{Heintz83} shows that for generic $\lambda\in\mathcal{U}(\bfs F_r)$, the
first inequality is an equality. More precisely, there is a nonempty
Zariski open subset $\mathcal{U}'(\bfs F_r)\subset \mathcal{U}(\bfs
F_r)$ such that, for any $\bfs\lambda\in \mathcal{U}'(\bfs F_r)$,
$$\mathrm{rank}_{R_r}\big(\overline{k}
[V(\bfs F_r)]\big)=\deg\big(V(\bfs F_r)\big).$$

The projection $\pi_{\bfs\lambda}:V(\bfs F_r)\to\A^{n-r}$ of
\eqref{eq: Noether normalization} is generically unramified, and for
a generic point $\bfs y\in\A^{n-r}$, the fiber
$\pi_{\bfs\lambda}^{-1}(\{\bfs y\})$  is a smooth, zero-dimensional
variety. Moreover, by \cite{Heintz83},
$$\#\big(\pi_{\bfs\lambda}^{-1}(\{\bfs y\})\big)=
\deg\big(\pi_{\bfs\lambda}^{-1}(\{\bfs y\})\big)=\deg (V(\bfs F_r)).$$

Later, we will exhibit bounds (and probabilistic algorithms) to find
a ${\bfs\lambda}\in GL(n,\overline{k})$ and $\bfs y\in\A^{n-r}$
such that $\bfs\lambda\in\mathcal{U}(\bfs F_r)$ and the fiber
$\pi_{\bfs\lambda}^{-1}(\{\bfs y\})$ is generically unramified.
Specifically, we will show that there exists a finite set
$\mathcal{S}:=\mathcal{S}(\bfs d_r)\subset \overline{k}$ such that,
for any $\bfs
F_r\in\mathcal{U}_{\bfs d_r}\subset\mathcal{P}_{\bfs d_r}$,
most choices of $\bfs\lambda\in\mathcal{S}^{n\times n}$ and $\bfs y\in
\mathcal{S}^{n-r}$ satisfy:
\begin{itemize}
  \item The extension $R_r\subset\overline{k}
[V(\bfs F_r)]$ is integral, $\overline{k} [V(\bfs F_r)]$ is a free
$R_r$-module, $\det J_r\bfs F_r$ is not a zero-divisor in
$\overline{k} [V(\bfs F_r)]$ and the following inequality holds:
$$\mathrm{rank}_{R_r}\big(\overline{k}
[V(\bfs F_r)]\big)\le \deg\big(V(\bfs F_r)\big).$$
\item The fiber $\pi_{\bfs\lambda}^{-1}(\{\bfs y\})$ is a smooth,
zero-dimensional variety, whose ideal is defined by the polynomials
$F_1(\bfs y, Y_{n-r+1},\ldots,Y_n),\ldots,F_r(\bfs y,
Y_{n-r+1},\ldots,Y_n)$, and is unramified, with
$$\#\big(\pi_{\bfs\lambda}^{-1}(\{\bfs y\})\big)=\deg\big(\pi_{\bfs\lambda}^{-1}(\{\bfs y\})\big)
=\mathrm{rank}_{R_r}\big(\overline{k}
[V(\bfs F_r)]\big).$$
\end{itemize}
%
% ----------------------------------------------------------------------
% ----------------------------------------------------------------------
% ----------------------------------------------------------------------
% ----------------------------------------------------------------------
%
\subsubsection{First approach: a procedure by jumping through solution
varieties}
\label{subsec: first approach}
With notations as before, given a degree sequence $\bfs
d_r:=(d_1,\ldots,d_r)$, we consider the solution variety $V_{\bfs
d_r}$ and the canonical projections
$$\xymatrix{
 & V_{\bfs d_r}\ar[ld]_{\Pi_1} \ar[rd]^{\Pi_2} & \\
\mathcal{P}_{\bfs d_r}  & & \A^n}$$
Let $\mathcal{U}_{\bfs d_r}\subset\mathcal{P}_{\bfs d_r}$ be the
nonempty Zariski open set such that any $\bfs
F_r\in\mathcal{U}_{\bfs d_r}$ defines an ideal-theoretic complete
intersection (i.e., satisfies $({\sf H}_1)$ and $({\sf H}_2)$). We
now discuss a first approach to solving the system $V(\bfs
F_r)=\Pi_2\big(\Pi_1^{-1}(\{\bfs F_r\})\big)$ by means of an
iterative algorithm.

To introduce the algorithm, we further assume that the list of
polynomials $\bfs F_r:=(F_1,\ldots,F_r)$ satisfies:
\begin{itemize}
  \item[$({\sf H}_3)$]  For $1\le s\le r$, the ideal $(F_1,\ldots,F_s)$ satisfies
 $({\sf H}_1)$ and $({\sf H}_2)$.
\end{itemize}
Note that $({\sf H}_3)$ implies both $({\sf H}_1)$ and $({\sf
H}_2)$; that is, we assume that the list $\bfs
F_r:=(F_1,\ldots,F_r)$ forms a {\em smooth regular sequence}, as in
\cite{GiHeMoPa95} or \cite{Pardo95}. This assumption is not overly
restrictive. As shown in \cite{KrPa96} (see also
\cite{Jouanolou83}), given $\bfs F_r\in\mathcal{U}_{\bfs d_r}$, one
can iteratively  obtain linear combinations $G_1,\ldots,G_{r-1}$ of
$F_1,\ldots,F_r$ such that the sequence $G_1,\ldots,G_{r-1},F_r$
satisfies $({\sf H}_3)$. Nevertheless, we avoid such a construction
and simply assume from the outset that  $\bfs F_r$ satisfies $({\sf
H}_3)$.

We introduce the following notations to express the iterative
structure of the algorithm.
\begin{itemize}
  \item For $1\le s\le r$, let $\bfs d_s:=(d_1,\ldots,d_s)$
  be the truncated degree sequence.
  \item Let $\mathcal{P}_{\bfs d_s}:=\prod_{i=1}^s
  \mathcal{P}_{d_i}$ be
  the $\overline{k}$-vector space of polynomial lists  of length
  $s$.
  \item For $\bfs F_r:=(F_1,\ldots,F_r)\in\mathcal{P}_{\bfs d_r}$,
  denote $\bfs F_s:=(F_1,\ldots,F_s)$.
  \item Let $V_{\bfs d_s}\subset\mathcal{P}_{\bfs d_s}$ be the
  corresponding solution variety.
  \item Let $\mathcal{U}_{\bfs d_s}\subset\mathcal{P}_{\bfs d_s}$ be the nonempty
  Zariski open set of lists satisfying hypothesis
  $({\sf H}_3)$.
\end{itemize}

We can now formulate our goal in this context as follows:
\begin{problem}[Solving a System Defining a Smooth Regular Sequence]
Given an algebraically closed field $\overline{k}$ and an $r$-tuple
$\bfs F_r\in\mathcal{P}_{\bfs d_r}$ of polynomials in $n$ variables
of degrees bounded by $\bfs d_r:=(d_1,\ldots,d_r)$, defining a
smooth regular sequence (i.e., satisfying hypothesis $({\sf H}_3)$),
compute $$\Pi_2\big(\Pi_1^{-1}(\bfs F_r)\big).$$
\end{problem}

The approach proposed in \cite{GiHeMoPa95} and \cite{Pardo95}
rediscovered a forgotten idea originally introduced by L. Kronecker
in \cite{Kronecker82}. This idea can be illustrated using the
following diagram, which represents a sequence of jumps between
solution varieties:

$$\xymatrix{
V_{\bfs d_1}\ar@{->}@/^{5mm}/[r]^{+ F_2} & V_{\bfs d_2}\ar@{->}@/^{5mm}/[r]^{+ F_3} &\quad \cdots\quad
\ar@{->}@/^{5mm}/[r]^{+ F_r} & V_{\bfs d_r}\\
\Pi_1^{-1}(\bfs F_1)\ar@{->}@/^{5mm}/[r] \ar@{}[u]|-*[@]{\subseteq}
& \Pi_1^{-1}(\bfs F_2)\ar@{->}@/^{5mm}/[r]
\ar@{}[u]|-*[@]{\subseteq}& \quad \cdots\quad \ar@{->}@/^{5mm}/[r] &
\Pi_1^{-1}(\bfs F_r)\cong V(\bfs F_r)\ar@{}[u]|-*[@]{\subseteq} }$$

\bigskip
\noindent Each jump $\curvearrowright$ represents an ``elimination''
step, where a new equation is added to the system. That is, at the
$s$th step, we use the data computed for the variety $V(\bfs
F_s)\cong \Pi_1^{-1}(\bfs F_s)$, and add the equation defined by
$F_{s+1}$ to solve $\Pi_1^{-1}(\bfs F_{s+1})$. This process is what
we called an ``elimination step'' of the algorithm.

To perform such a step, we must first define what we mean by a
``solution'' of $V(\bfs F_s)\cong \Pi_1^{-1}(\bfs F_s)$. Our
presentation differs slightly from those of \cite{Kronecker82},
\cite{GiHeMoPa95} or \cite{Pardo95}, though the underlying concepts
remain consistent.

Since each sequence $\bfs F_s$ $(1\le s\le r)$ satisfies hypothesis
$({\sf H}_3)$, the previous arguments show the existence of a
nonempty Zariski open set $\widetilde{\mathcal{U}}(\bfs F_r)\subset
GL(n,\overline{k})$ such that for all $1\le s\le r$, and any
$\bfs\lambda\in \widetilde{\mathcal{U}}(\bfs F_r)$, we have an
integral ring extension
\begin{equation}\label{eq: integral ring extension recursive}
R_s:=\overline{k}[Y_1,\ldots,Y_{n-s}]\hookrightarrow \overline{k}[V(\bfs
F_s)]=\overline{k}[Y_1,\ldots,Y_{n}]/(F_1,\ldots,F_s).
\end{equation}
We say that the matrix $\bfs\lambda$ puts the variables into {\em
simultaneous Noether position} for all the sequences $\bfs F_s$
$(1\le s\le r)$. A probabilistic method for computing such a
$\bfs\lambda$ was introduced in \cite{KrPa96}. In Section
\ref{subsec: Noether normalization}, we present a refined version
with improved bounds.

Assuming such a $\bfs\lambda$ is available, we present the classical
notion of solving of \cite{Kronecker82}, and also
of \cite{GiHeMoPa95} and \cite{Pardo95}. {\em Solving a smooth
complete intersection} like $V(\bfs F_s)$ $(1\le s\le r)$ involves
computing:
\begin{itemize}
  \item[$({\sf S}_1)$] A linear change
  of variables $\bfs\lambda\in\widetilde{\mathcal{U}}(\bfs F_r)$ such that:
  \begin{itemize}
\item The variables are in
Noether position with respect to $V(\bfs F_s)$ for all $1\le s\le
r$.
\item The mapping $\Pi_{\bfs\lambda}^s:\Pi_1^{-1}(\bfs F_s) \to
H_s$, defined by $\Pi_{\bfs\lambda}^s(\bfs x, \bfs F_s):=
(y_1,\ldots,y_{n-s+1})$, is birational onto a hypersurface
$H_s\subset\A^{n-s+1}$, for all $1\le s\le r$:
\end{itemize}
  \item[$({\sf S}_2)$] The minimal polynomial $M_s\in \overline{k}[Y_1,\ldots,Y_{n-s+1}]$
  defining $H_s$.
  \item[$({\sf S}_3)$] The rational inverse
  $\Phi_s:\mathfrak{O}_s\to \Pi_1^{-1}(\bfs F_s)$
  of $\Pi_{\bfs\lambda}^s$, defined over a
  nonempty Zariski open subset $\mathfrak{O}_s\subset H_s$.
\end{itemize}

We show that there exists a nonempty Zariski open subset
$\mathcal{U}^*(\bfs F_r)\subset \widetilde{\mathcal{U}}(\bfs F_r)$
such that any $\bfs\lambda\in\mathcal{U}^*(\bfs F_r)$ satisfies
$({\sf S}_1)$. Furthermore, for any
$\bfs\lambda\in\mathcal{U}^*(\bfs F_r)$ there exists:
\begin{itemize}
  \item A nonzero polynomial
  $\rho_s\in\overline{k}[Y_1,\ldots,Y_{n-s}]$,
  \item Polynomials
$V_{n-s+2}^s,\ldots,V_n^s\in\overline{k}[Y_1,\ldots,Y_{n-s},T]$,
\end{itemize}
such that
\begin{itemize}
\item $\mathfrak{O}_s:=\mathfrak{O}(\rho_s):=\{(y_1,\ldots,y_{n-s+1})
\in H_s:\rho_s(y_1,\ldots,y_{n-s})\not=0\}$ is a distinguished
non-empty Zariski open subset, dense in $H_s$.
\item The inverse mapping $\Phi_s:\mathfrak{O}_s\to \Pi_1^{-1}(\bfs F_s)$
is given by:
$$\Phi_s(\bfs y)=\big(\bfs y,(\rho_s^{-1}\cdot V_{n-s+2}^s)(\bfs y),\ldots,
(\rho_s^{-1}\cdot V_n^s)(\bfs y),\bfs F_s\big).$$
\end{itemize}
The nonzero polynomial $\rho_s$ is the {\em discriminant} of $M_s$
with respect to $Y_{n-s+1}$.

This triple of data ---$M_s$, $\rho_s$, and
$V_{n-s+2}^s,\ldots,V_n^s$--- constitutes a {\em Kronecker
representation} (also called a {\em geometric solution}) of the
variety $V(\bfs F_s)$. Given such a representation for $V(\bfs F_s)$
and the polynomial $F_{s+1}$, via ``elementary'' linear algebra
computations in dimension $\deg(V(\bfs F_s))\cdot d_{s+1}$, one can
compute a Kronecker representation of $V(\bfs F_{s+1})$, completing
an elimination step, that jumps from fibers in $V_{\bfs d_s}$ to
fibers in $V_{\bfs d_{s+1}}$.

This process forms the basis of the following iterative algorithm.
\begin{algorithm}[Kronecker algorithm - first approach]\label{algo: Kronecker - first approach} ${}$
\begin{itemize}
\item[]{\bf Input:} A list $\bfs F_r:=(F_1,\ldots,F_r)\in \mathcal{U}_{\bfs
d_r}$, such that $\bfs F_s:=(F_1,\ldots,F_s)\in \mathcal{U}_{\bfs
d_s}$ for all $1\le s\le r$.
\item[]{\bf Output:} A Kronecker representation
of $V(\bfs F_r)$. \smallskip

\item[]{\bf 1} Set $s:=1$.
\item[]{\bf 2} Compute a Kronecker representation
of $V(\bfs F_1)$.
\item[]{\bf 3} For $s=2,\ldots,r$, do
  \begin{enumerate}
    \item[]{\bf 3.1} Use the representation of
$V(\bfs F_s)$ and the polynomial $F_{s+1}$,
    to compute a Kronecker
representation of $V(\bfs F_{s+1})$.
  \end{enumerate}
  \item[]{\bf 4} Return the Kronecker representation of $V(\bfs F_r)$.
\end{itemize}
\end{algorithm}

This was the approach in \cite{Kronecker82}, but it has a number of
drawbacks that we now discuss.
\begin{remark}
Let $K_s:=\overline{k}(Y_1,\ldots,Y_{n-s})$. The ``elementary''
linear algebra mentioned above involves a sequence of {\em tensor
matrices} $M_{X_{n-s+1}},\ldots,M_{X_n}\in M_{D_s}(K_s)$, associated
with the homothecies
\begin{align*}
\eta_{X_k}:=K_s \otimes_{R_s} \overline{k}[V(\bfs F_s)] & \to K_s
\otimes_{R_s} \overline{k}[V(\bfs
F_s)]\\
\overline{H}&\mapsto \overline{X_kH},\end{align*}
where $\delta_s:=\deg (V(\bfs F_s))$, for $n-s+1\le k\le n$. The
``elimination step'' consist in computing the minimal polynomial of
the matrix
$$M_{F_{s+1}}:=F_{s+1}(Y_1,\ldots,Y_{n-s},M_{X_{n-s+1}},\ldots,M_{X_n})
\in M_{\delta_s}(K_s).$$
\end{remark}

\begin{remark}\label{rem: cost kronecker elimination}
To carry out this ``elementary'' linear algebra computations, we
need the polynomials
\begin{itemize}
  \item $\rho_s\in K_s$,
  \item The minimal defining polynomial of $H_s$, and
  \item $V_{n-s+2}^s,\ldots,V_n^s\in K_s[T]$,
\end{itemize}
all represented by their coefficients in $K_s$. This corresponds to
\begin{itemize}
  \item $\mathcal{O}((s+1)\delta_s)$ polynomials of degree bounded by
  $\delta_s$;
  \item One polynomial of degree $\delta_s^2$.
\end{itemize}
If these polynomials in $K_s$ are written using dense encoding (i.e.,
as lists of all their coefficients in $\overline{k}$), the amount of
information becomes roughly
$$\mathcal{O}\left(s\cdot \binom{\delta_s^2+n-s}{n-s}\right)\cong
\mathcal{O}\left(s\cdot \delta_s^{2(n-s)}\right),$$
which is impracticable, and leads to algorithms with arithmetic
complexity at least
$$\Omega(s\cdot d^{2s(n-s)}),$$
where $d:=\max\{d_1,\ldots,d_s\}$. In the case $s=\theta(n)$, this yields
complexity greater than $d^{2n^2}$, which is unfeasible in practice.
This is the main reason why Kronecker's original approach was
ultimately abandoned---it is computationally intractable, even on
modern hardware.
\end{remark}

\begin{remark}
To address the issue discussed in Remark \ref{rem: cost kronecker
elimination}, the proposal in \cite{GiHeMoPa95} and \cite{Pardo95}
was to avoid dense representations, encoding instead polynomials in
$K_s:=\overline{k}[Y_1,\ldots,Y_{n-s}]$ using {\em division-free nonscalar straight-line
programs}. In modern terms, a division-free nonscalar straight-line program can be seen as a
{\em neural network} over $\overline{k}$ with activation function
$\varphi(T):=T^2$ (see \cite{PaSe26}). This representation was not new---it had already
been used in \cite{FiGiSm95} or \cite{KrPa96} for instance. In
\cite{KrPa96} a detailed discussion of such a {\em data structure} was
provided.
\end{remark}

\begin{remark}
While the use of straight-line programs allows a more efficient
implementation of the elimination step, it is still not sufficient
to make the Kronecker algorithm efficient overall. Assume that the
coefficients of the Kronecker representation of $V(\bfs F_s)$ in
$R_s$ are given by a straight-line program of size $L_s$. Then, in a
number of arithmetic operations polynomial in
$$L_s,n,d_{s+1},\delta_s,$$
we can compute the $\mathcal{O}((s+1)\delta_{s+1})$ coefficients of the
polynomials that form a Kronecker representation of $V(\bfs
F_{s+1})$, encoded as a straight-line program of size $L_{s+1}$.
Unfortunately, the size of the resulting straight-line programs
grows multiplicatively, i.e.,
$$L_{s+1}\cong (d_{s+1}\delta_s)^{\mathcal{O}(1)}L_s.$$
This phenomenon was called {\em nesting of straight-line programs}
in \cite{Pardo95}. In the zero-dimensional case (i.e., $s=n$), such
an approach leads to an algorithm with complexity
$d^{\mathcal{O}(n^2)}$, where $d:=\max\{d_1,\ldots,d_n\}$, which is
worse (in terms of the number of arithmetic operations) than those
in \cite{FiGiSm95} or \cite{KrPa96}, which work in time
$(nd)^{\mathcal{O}(n)}$.
\end{remark}

In conclusion, the approach based on jumping between solution
varieties has critical shortcomings:
\begin{itemize}
  \item Kronecker's naive method is too costly computationally,
  \item Using straight-line
programs simplifies individual steps but leads to exponential
blow-up due to nesting,
  \item As a result, the total complexity remains
  $d^{\mathcal{O}(n^2)}$,
  worse than competing algorithms with complexity $(nd)^{\mathcal{O}(n)}$.
\end{itemize}

In \cite{Pardo95} or \cite{GiHeMoPa95}, a first idea on how to break
this nesting was suggested: use a {\em Newton-Hensel lifting} in the
non-Archimedean case. This was made precise in
\cite{GiHaHeMoMoPa97}, \cite{GiHeMoPa97} and \cite{GiHeMoMoPa98}. We
discuss this idea in the next section.
%
% ----------------------------------------------------------------------
% ----------------------------------------------------------------------
% ----------------------------------------------------------------------
% ----------------------------------------------------------------------
%
\subsubsection{Second approach: deformations through zero-dimensional
smooth varieties}
\label{subsec: second approach}
The problem with the previous approach---jumping between solution
varieties by successively adding equations---is not only its high
complexity but its {\em somewhat unnatural} nature. In
\cite{GiHaHeMoMoPa97}, \cite{GiHeMoMoPa98} and \cite{GiHeMoPa97}, a
new technique was introduced that not only breaks the nesting
phenomenon but also transforms Algorithm \ref{algo: Kronecker -
first approach} into one that {\em deforms the input polynomial
system within the solution variety of zero-dimensional varieties}.

To explain how it works, we continue with notations and assumptions
from the previous sections. Let $r,n\in\N$ with $r\le n$, and let
$\bfs d_r:=(d_1,\ldots,d_r)$ be a degree sequence. Let $\bfs
F_r:=(F_1,\ldots,F_r)\in \mathcal{P}_{\bfs d_r}$ be a system that
defines a smooth complete intersection, namely, it satisfies
hypothesis $({\sf H}_3)$.

Let $\bfs\lambda\in\mathcal{\widetilde{U}}(\bfs F_r)$ be a linear
change of coordinates such that the new variables $\bfs
Y:=\bfs\lambda\cdot \bfs X$ ensure that the ring extension
\eqref{eq: integral ring extension recursive} is integral for all
$1\le s\le r$. We now search for a point $\bfs
p:=(p_1,\ldots,p_{n-1})\in\A^{n-1}$ and define the following
sections of the varieties associated with the input systems $\bfs
F_1,\ldots,\bfs F_r$ under the change of coordinates
$\bfs\lambda\in\mathcal{\widetilde{U}}(\bfs F_r)$:
\begin{equation}\label{eq: linear sections}
\begin{array}{l}
\bfs G_r^1:=(Y_1-p_1,\ldots,Y_{n-1}-p_{n-1},F_1)\in\mathcal{P}_{\bfs 1_{n-r},\bfs d_r}, \\[1ex]
\bfs G_r^2:=(Y_1-p_1,\ldots,Y_{n-2}-p_{n-2},F_1,F_2)\in\mathcal{P}_{\bfs 1_{n-r},\bfs d_r} \\
 \qquad \vdots \\
\bfs
G_r^r:=(Y_1-p_1,\ldots,Y_{n-r}-p_{n-r},F_1,\ldots,F_r)\in\mathcal{P}_{\bfs
1_{n-r},\bfs d_r},
\end{array}
\end{equation}
where ${\bfs 1_{n-r},\bfs d_r}:=(1,\ldots,1,d_1,\ldots,d_s)$ is the
combined degree sequence. We have $r$ linear sections of the input
varieties:
$$V(\bfs G_r^s):=V(Y_1-p_1,\ldots,Y_{n-s}-p_{n-s})\cap V(\bfs F_s)\quad (1\le s\le r),$$
which define $r$ jumps to finally compute a Kronecker representation of
the final section $V(\bfs G_r^r)$.

Under the given assumptions, it is easy to see that there is a
nonempty Zariski open set $\mathfrak{O}:=\mathfrak{O}(\bfs
F_r,\bfs\lambda)\subset\A^{n-1}$ such that the following hold for
$1\le s\le r$:
\begin{itemize}
  \item[$({\sf L}_1)$] $V(\bfs G_r^s)$ is a smooth zero-dimensional variety;
  \item[$({\sf L}_2)$] $\# V(\bfs G_r^s)=\deg V(\bfs G_r^s)
  =\mathrm{rank}_{R_s} \overline{k}[V(\bfs F_s)]$.
\end{itemize}
We will show that such a point $\bfs p$ can be found by a
probabilistic algorithm, which selects $\bfs p$ randomly in a
distinguished Zariski open set $\mathfrak{D}(H):=\{\bfs
p\in\A^{n-1}:H(\bfs p)\not=0\}\subset \mathfrak{O}$, with sharp
degree bounds on $H$.

It is worth noting that the solution variety $V_{\bfs 1_{n-r},\bfs
d_r}\subset\A^n\times \mathcal{P}_{\bfs 1_{n-r},\bfs d_r}$ has the
property that a generic fiber $\Pi_1^{-1}(\bfs G)$ is a smooth
zero-dimensional variety (outside the discriminant variety
$\Sigma_{\bfs 1_{n-r},\bfs d_r}$, or equivalently, for systems in
the open set $\mathcal{U}({\bfs 1_{n-r},\bfs d_r})$). This
corresponds to the claim of $({\sf L}_1)$--$({\sf L}_2)$.

The sequence of varieties $V(\bfs G_r^1),\ldots,V(\bfs G_r^r)$ can
be viewed as a homotopic deformation within the solution variety
$V_{\bfs 1_{n-r},\bfs d_r}$. We consider any curve $\{\bfs
G_r^{(t)}\in\mathcal{P}_{\bfs 1_{n-r},\bfs d_r}:t\in[0,1]\}$ such
that the following conditions hold:
$$%\begin{equation}\label{eq: linear sections}
\begin{array}{l}
\bfs G_r^{(t_0)}:=G_r^1:=(Y_1-p_1,\ldots,Y_{n-1}-p_{n-1},F_1), \\
 \qquad \vdots \\
\bfs G_r^{(t_{s-1})}:=G_r^s:=(Y_1-p_1,\ldots,Y_{n-s}-p_{n-s},\bfs F_s)\textrm{ for }1\le s\le r-2, \\
 \qquad \vdots \\
\bfs G_r^{(t_{r-1})}:=\bfs
G_r^r:=(Y_1-p_1,\ldots,Y_{n-r}-p_{n-r},\bfs F_r),
\end{array}
$$%\end{equation}
with $t_0:=0<t_1<\cdots<t_{r-2}<t_{r-1}:=1$. The curve may be chosen
as polygonal chain. We then have a sequence of zero-dimensional
varieties:

$$\xymatrix{
V_{\bfs 1_{n-r},\bfs d_r} & V_{\bfs 1_{n-r},\bfs d_r} &\quad
\cdots\quad
 & V_{\bfs 1_{n-r},\bfs d_r}\\
\Pi_1^{-1}(\bfs G_r^{(t_0)})\ar@{}[u]|-*[@]{\subseteq} &
\Pi_1^{-1}(\bfs G_r^{(t_1)})\ar@{}[u]|-*[@]{\subseteq}& \quad
\cdots\quad  & \Pi_1^{-1}(\bfs
G_r^{(1)})\ar@{}[u]|-*[@]{\subseteq}}$$
Using the parameter $t\in[0,1]$, this becomes a continuous path in
the solution variety $V_{\bfs 1_{n-r},\bfs d_r}$:

$$\xymatrix{
V_{\bfs 1_{n-r},\bfs d_r} \\
\{\Pi_1^{-1}(\bfs G_r^{(t)}):t\in[0,1]\}\ar@{}[u]|-*[@]{\subseteq} }
$$
The question becomes how to follow this path through a sequence of
iterations. In numerical settings, this is typically done using
Newton iterations (see \cite{ShSm93} or \cite{BePa09}, \cite{BePa11}
and the references therein). In our symbolic setting, we replace the
Newton method with a {\em non-Archimedean Newton-Hensel (NH) lifting
method}. In this vein, the sequence of varieties $V(\bfs
G_r^1),\ldots,V(\bfs G_r^r)$ are called {\em lifting fibers} in
\cite{GiHaHeMoMoPa97}.

Under the assumptions $({\sf H}_3)$,$({\sf L}_1)$ and $({\sf L}_2)$,
this lifting method works reliably. The process is visualized in the
following diagram:

$$\xymatrix{
V_{\bfs d_1}\ar@{->}[r] & V_{\bfs d_2}\ar@{->}[r] &\quad \cdots\quad
\ar@{->}[r] & V_{\bfs d_r}\\
\Pi_1^{-1}(\bfs F_1)\ar@{->}@/^{5mm}/[r]^{\textrm{Eliminate }F_2}
\ar@{}[u]|-*[@]{\subseteq} \ar@{->}[dr]^{\textrm{Intersect}}& \Pi_1^{-1}(\bfs
F_2)\ar@{->}@/^{5mm}/[r]^{\textrm{Eliminate }F_3}
%\ar@{->}@/_{5mm}/[d]
\ar@{}[u]|-*[@]{\subseteq}& \quad
\cdots\quad\ar@{->}[dr]^{\textrm{Intersect}}
\ar@{->}@/^{5mm}/[r]^{\textrm{Eliminate }F_r} & \Pi_1^{-1}(\bfs
F_r)\cong V(\bfs F_r)
%\ar@{->}@/_{5mm}/[d]_{\textrm{Intersect}}
\ar@{}[u]|-*[@]{\subseteq}\\
\Pi_1^{-1}(\bfs G_r^1)\ar@{->}@/_{5mm}/[r]_{\textrm{Eliminate }F_2}
\ar@{->}[u]^{\textrm{NH lifting}} & \Pi_1^{-1}(\bfs
G_r^2)\ar@{->}@/_{5mm}/[r]_{\textrm{Eliminate
}F_3}\ar@{->}[u]_{\textrm{NH lifting}} &\quad
\cdots\quad\ar@{->}@/_{5mm}/[r]_{\textrm{Eliminate }F_r}
&\Pi_1^{-1}(\bfs
G_r^r) \ar@{->}[u]_{\textrm{NH lifting}}\\
V_{\bfs 1_{n-r},\bfs d_r}\ar@{->}[r] \ar@{}[u]|-*[@]{\supseteq} &
V_{\bfs 1_{n-r},\bfs d_r}\ar@{->}[r] \ar@{}[u]|-*[@]{\supseteq}&
\quad \cdots\quad\ar@{->}[r]
 & V_{\bfs 1_{n-r},\bfs
d_r}\ar@{}[u]|-*[@]{\supseteq}}$$

The Newton-Hensel lifting is detailed in Section \ref{subsec: NH
lifting}. In this diagram, the elimination step from $V(\bfs
F_s)\cong \Pi_1^{-1}(\bfs F_s)$ to $\Pi_1^{-1}(\bfs G_r^{s+1})$ is
done by evaluating the variables $Y_1,\ldots,Y_{n-s-1}$ in the
polynomials that form the Kronecker representation of
$\Pi_1^{-1}(\bfs F_s)$, and the polynomial $F_{s+1}$, at
$p_1,\ldots,p_{n-s-1}$, and performing the intersection step in
dimension one.

This approach has significant advantages:
\begin{itemize}
  \item It breaks the {\em nesting} of straight-line programs from
  the naive approach: at each step, coefficients become constants
in $\overline{k}$ since the lifting fibers are zero-dimensional.
\item The total complexity of the procedure is polynomial
in:
  $$n,d,\delta,L,$$
  where
\begin{itemize}
  \item $\delta:=\max\{\deg V(\bfs F_s):1\le s\le r\}\le \prod_{s=1}^rd_s$,
  \item $L$ is the complexity of evaluating the input polynomials $\bfs
  F_r$.
  \end{itemize}
\end{itemize}
Hence the overall complexity becomes $L(nd\delta)^{\mathcal{O}(1)}$,
which represents a substantial improvement over the naive Kronecker
approach. Many refinements of the bounds and algorithms appear in
the literature (see, e.g., \cite{GiLeSa01}, \cite{HeMaWa01},
\cite{Lecerf03}, \cite{CaMa06a}, \cite{DuLe08}, \cite{HoLe21}).

We note that in \cite{GiHaHeMoMoPa97}, \cite{GiHeMoMoPa98}, and
\cite{GiHeMoPa97}, the point $\bfs p$ is computed iteratively. A
probabilistic process for choosing a {\em single} point $\bfs p\in\A^{n-1}$
such that all $V(\bfs G_r^i)$ satisfy $({\sf L}_1)$ and $({\sf
L}_2)$ was first provided in \cite{HeMaWa01} (see also
\cite{CaMa06a} and \cite{GiMa19}). We present improved bounds and
algorithms in Section \ref{subsec: lifting fibers}.

There is, however, a drawback: the elimination step still jumps
between solution varieties, which is somewhat unnatural, as in the
naive approach. Additionally, the algorithm encodes multivariate
polynomials by straight-line programs, and not many readers like
this encoding. This was independently solved in \cite{GiLeSa01} and
\cite{HeMaWa01}, introducing intermediate curves and avoiding the
full jump between higher-dimensional varieties.
%
% ----------------------------------------------------------------------
% ----------------------------------------------------------------------
% ----------------------------------------------------------------------
% ----------------------------------------------------------------------
%
\subsubsection{Third approach: homotopic deformations by
smooth curves}\label{subsec: third approach inro}
In \cite{GiLeSa01} and \cite{HeMaWa01}, a new approach yields a
three level solution: it benefits from the existence of lifting
fibers, but these authors observed that performing the elimination
step to compute the next lifting fiber does not require
complete access to the full intermediate variety; {\em an
intermediate curve suffices}. These are the {\em lifting curves}
that we now introduce.

With the notations and assumptions from the previous sections, given
$r,n\in\N$ with $r\le n$ and a degree sequence $\bfs
d_r:=(d_1,\ldots,d_r)$, we fix an $r$-tuple $\bfs
F_r:=(F_1,\ldots,F_r)\in \mathcal{P}_{\bfs d_r}$ which defines a
smooth complete intersection, namely, which satisfies hypothesis
$({\sf H}_3)$. Let $\bfs\lambda\in\mathcal{\widetilde{U}}(\bfs F_r)$
be a linear change of coordinates such that the new variables $\bfs
Y:=\bfs\lambda\cdot \bfs X$ ensure that the ring extension
\eqref{eq: integral ring extension recursive} is integral for $1\le
s\le r$. Let $\bfs p:=(p_1,\ldots,p_{n-1})\in\A^{n-1}$ be a point
such that $V(\bfs G_r^1),\ldots,V(\bfs G_r^r)$ are lifting fibers of
$V(\bfs F_1),\ldots,V(\bfs F_r)$, where $\bfs G_r^1,\ldots,\bfs
G_r^r$ are defined as in \eqref{eq: linear sections}, and thus
satisfy assumptions $({\sf H}_3)$, $({\sf L}_1)$ and $({\sf L}_2)$.

Associated with the degree sequence $\bfs d_r$, we introduce a new
degree sequence of length $n-1$ in the following way:
$$(\bfs 1_{n-r-1},\bfs
d_r):=(\underbrace{1,\ldots,1}_{n-r-1},d_1,\ldots,d_r)\in\N^{n-1}.$$
We also consider the solution variety $V_{\bfs 1_{n-r-1},\bfs d_r}$
associated with this degree sequence. Further, we consider the
following sequence of $n-1$ polynomials of $\mathcal{P}_{\bfs
1_{n-r-1},\bfs d_r}$:
$$
\begin{array}{l}
\bfs H_r^1:=(Y_1-p_1,\ldots,Y_{n-2}-p_{n-2},F_1), \\[1ex]
\bfs H_r^2:=(Y_1-p_1,\ldots,Y_{n-3}-p_{n-3},\bfs F_2), \\
 \qquad \vdots \\
\bfs H_r^r:=(Y_1-p_1,\ldots,Y_{n-r-1}-p_{n-r-1},\bfs F_r).
\end{array}
$$
We denote $C_s(\bfs F_r):=V(\bfs H_r^s)\cong \Pi_1^{-1}(\bfs
H_r^s)$, where $\Pi_1:V_{\bfs 1_{n-r-1},\bfs d_r}\to
\mathcal{P}_{\bfs 1_{n-r-1},\bfs d_r}$ is the canonical projection.

One readily observes that $C_s(\bfs F_r)$ is a section of $V(\bfs
F_s)$ and, hence, its degree is bounded by $\deg V(\bfs F_s)$ for
$1\le s\le r$. On the other hand, $C_s(\bfs F_r)$ contains a lifting
fiber $V(\bfs G_r^s)$, and therefore
$$\deg C_s(\bfs F_r)=\deg V(\bfs G_r^s)=\deg V(\bfs F_s)$$
(see Section \ref{section: lifting curve properties} for details).
Moreover, as $V(\bfs G_r^s)$ contains
enough information to recover $V(\bfs F_s)$ through a Newton-Hensel
lifting, it also contains enough information to obtain a Kronecker
representation of $C_s(\bfs F_r)$, also using a Newton-Hensel
lifting. The variety $C_s(\bfs F_r)$ has pure dimension 1---namely
it is a curve---and
is therefore called a {\em lifting curve}. More importantly, the
lifting curve $C_s(\bfs F_r)$ not only contains information to recover $V(\bfs
F_s)$, but also to perform an elimination step--not for recovering
$V(\bfs F_{s+1})$ but for computing the next lifting fiber $V(\bfs
G_r^{s+1})$.

We also have a homotopic deformation in $\mathcal{P}_{\bfs
1_{n-r-1},\bfs d_r}$ whose intermediate steps are the systems
$H_r^1,H_r^2,\ldots,H_r^r$,
$$\{\bfs
H_r^{(t)}\in\mathcal{P}_{\bfs 1_{n-r-1},\bfs d_r}:t\in[0,1]\}.$$
Thus the algorithm takes the form shown in the following diagram:

$$\xymatrix{
V_{\bfs d_1}\ar@{->}[r] & V_{\bfs d_2}\ar@{->}[r] &\quad \cdots\quad
\ar@{->}[r] & V_{\bfs d_r}\\
\Pi_1^{-1}(\bfs F_1)\ar@{->}@/^{5mm}/[r]^{\textrm{Eliminate }F_2}
\ar@{}[u]|-*[@]{\subseteq} & \Pi_1^{-1}(\bfs
F_2)\ar@{->}@/^{5mm}/[r]^{\textrm{Eliminate }F_3}
%\ar@{->}@/_{5mm}/[d]_{\textrm{Intersect}}
\ar@{}[u]|-*[@]{\subseteq}& \quad \cdots\quad
\ar@{->}@/^{5mm}/[r]^{\textrm{Eliminate }F_r} & \Pi_1^{-1}(\bfs
F_r)\cong V(\bfs F_r)
\ar@{}[u]|-*[@]{\subseteq}\\
V_{\bfs 1_{n-r-1},\bfs d_r}\ar@{->}[r] \ar@{-->}[u] & V_{\bfs
1_{n-r-1},\bfs d_r}\ar@{-->}[u]\ar@{->}[r] &\quad \cdots\quad
\ar@{->}[r] & V_{\bfs 1_{n-r-1},\bfs d_r}\ar@{-->}[u]\\
\Pi_1^{-1}(\bfs H_r^1)\ar@{->}@/^{5mm}/[r]^{\textrm{Eliminate }F_2}
\ar@{}[u]|-*[@]{\subseteq}\ar@{->}[dr]^{\textrm{Intersect}} &
\Pi_1^{-1}(\bfs H_r^2)\ar@{->}@/^{5mm}/[r]^{\textrm{Eliminate }F_3}
\ar@{}[u]|-*[@]{\subseteq}& \quad
\cdots\quad\ar@{->}[dr]^{\textrm{Intersect}}
\ar@{->}@/^{5mm}/[r]^{\textrm{Eliminate }F_r} & \Pi_1^{-1}(\bfs
H_r^r)\ar@{}[u]|-*[@]{\subseteq}\\
\Pi_1^{-1}(\bfs G_r^1)\ar@{->}@/_{5mm}/[r]_{\textrm{Eliminate }F_2}
\ar@{->}[u]^{\textrm{NH lifting}} & \Pi_1^{-1}(\bfs
G_r^2)\ar@{->}@/_{5mm}/[r]_{\textrm{Eliminate
}F_3}\ar@{->}[u]_{\textrm{NH lifting}} &\quad \cdots\quad
\ar@{->}@/_{5mm}/[r]_{\textrm{Eliminate }F_r}
 & \Pi_1^{-1}(\bfs
G_r^r)\ar@{-->}[u]_{\textrm{NH lifting}}\\
V_{\bfs 1_{n-r},\bfs d_r}\ar@{->}[r] \ar@{}[u]|-*[@]{\supseteq} &
V_{\bfs 1_{n-r},\bfs d_r}\ar@{->}[r] \ar@{}[u]|-*[@]{\supseteq}&
\quad \cdots\quad \ar@{->}[r]
 & V_{\bfs 1_{n-r},\bfs d_r}\ar@{}[u]|-*[@]{\supseteq}}$$
The dotted arrows indicate that we could lift the information
contained in $V(\bfs H_r^s)$ to compute a Kronecker representation of the
corresponding variety $V(\bfs F_s)$. The point here is that we do
not do it, because we do not need to do it in order to compute the
next lifting fiber.

The final output of the algorithm is the lifting fiber $V(\bfs
G_r^r)$ given by a Kronecker representation (or by the van der
Waerden $U$-resultant), where coefficients are represented by
straight-line programs. If someone wants to compute a Kronecker
representation of $V(\bfs F_r)$, then a Newton-Hensel lifting from $V(\bfs
G_r^r)$ would allow obtaining such data.

Note that this also means that the Newton-Hensel lifting, plus the
knowledge of the $U$-resultant of $V(\bfs G_r^r)$, yields a sharp
approximation of the $U$-resultant of $V(\bfs F_r)$ with respect to
the standard non-Archimedean metric in the local ring $\overline{k}
[\![Y_1-p_1,\ldots,Y_{n-r}-p_{n-r}]\!]$. This approximation is a
nonscalar straight-line program with divisions of size $\mathcal{O}(\deg V(\bfs F_r))$ and depth
$\log_2\log_2(\deg V(\bfs F_r))$.%, using $\varphi(T):=T^2$ and
%$\psi(T):= 1/T$ as activation functions. It is remarkable that the
%depth of the hidden layers is doubly logarithmic in the degree of
%the variety under consideration.

In conclusion, the final approach consists in following a
polygonal chain of homotopic deformations of smooth complete intersection
varieties of dimension 0 or 1, which is shown in the following
diagram that represents a single step in both homotopy paths:

\begin{equation}\label{eq: sth elimination step}
\xymatrix{
V_{\bfs 1_{n-r-1},\bfs d_r}\ar@{->}[r]  & V_{\bfs
1_{n-r-1},\bfs d_r} \\
\Pi_1^{-1}(\bfs
H_r^s)\ar@{->}[dr]^{\textrm{Intersect}}\ar@{->}@/^{5mm}/[r]^{\textrm{Eliminate
}F_s} \ar@{}[u]|-*[@]{\subseteq} & \Pi_1^{-1}(\bfs
H_r^{s+1}) \ar@{}[u]|-*[@]{\subseteq}\\
\Pi_1^{-1}(\bfs G_r^s)\ar@{->}@/_{5mm}/[r]_{\textrm{Eliminate }F_s}
\ar@{->}[u]^{\textrm{NH lifting}} & \Pi_1^{-1}(\bfs
G_r^{s+1})\ar@{->}[u]_{\textrm{NH lifting}} \\
V_{\bfs 1_{n-r},\bfs d_r}\ar@{->}[r] \ar@{}[u]|-*[@]{\supseteq} &
V_{\bfs 1_{n-r},\bfs d_r} \ar@{}[u]|-*[@]{\supseteq}}\end{equation}

To establish the correctness of all these steps and the selection of
the coordinates of the point $\bfs p\in\A^{n-1}$, and the matrix
$\bfs\lambda\in\A^{n\times n}$ such that the linear change of
variables $\bfs Y:= \bfs\lambda\cdot \bfs X$ defines a simultaneous
Noether normalization of all the varieties $V(\bfs F_s)$ $(1\le s\le
r)$, requires a detailed analysis. We have chosen to perform this
analysis with the most advanced techniques to exhibit the sharpest
bounds in the study with the most up-to-date methods.
%
% ----------------------------------------------------------------------
% ----------------------------------------------------------------------
% ----------------------------------------------------------------------
% ----------------------------------------------------------------------
%
\subsection{Main results and outline of the paper}
In what follows, we will consider an algorithmic problem that
slightly generalizes the one discussed in the previous
sections. We now describe it in detail.

Let $k$ be a perfect field\footnote{%
As a referee suggested, it might be possible to replace the
assumption that the base field~$k$ is perfect by a weaker condition
requiring that the Jacobian determinants of all intermediate systems
intersect their zero sets properly, so that the systems define
radical ideals over~$\overline{k}$. While this formulation could
indeed generalize our setting, it is considerably more technical and
less natural from an algebraic-geometric viewpoint. Moreover, it is
not clear that the degree bounds established in this paper for the
hypersurfaces governing the genericity conditions (both on the
linear change of variables and on the choice of lifting points)
would remain valid under that generalization. The assumption
that~$k$ is perfect provides a unified and simpler framework
encompassing the main cases treated in the literature ($k=\Q$ and
finite fields), and allows the relevant nonvanishing discriminant
conditions to be expressed in a clean and intrinsic way, avoiding
inseparable phenomena. Finally, some subroutines of the algorithm
require extracting $p$-th roots, an operation that is naturally well
defined and computationally manageable
when~$k$ is perfect.%
}, and let $\overline{k}$ be its algebraic closure. Given positive
integers $r,n$ with $r\le n$, let $X_1,\ldots,X_n$ be indeterminates
over $k$, and consider polynomials $F_1,\ldots,F_r,G\in
k[X_1,\ldots,X_n]$ satisfying the following conditions:
\begin{itemize}
  \item[$({\sf H}_1)$] The polynomials $F_1\klk F_r$ form a regular sequence in the
  localization $k[X_1\klk X_n]_G$,
  \item[$({\sf H}_2)$] For all $1\le s\le r$, the ideal $(F_1\klk F_s)_G$ is radical in $k[X_1\klk X_n]_G$.
\end{itemize}
We will be interested in the set of common zeros $V(\bfs
F_r)\subset\A^n$ of $F_1,\ldots,F_r$ {\em outside the hypersurface}
$\{G=0\}$. An example arises when $G$ is the determinant of an $(r\times
r)$-minor of the Jacobian matrix $J_r$ of $F_1,\ldots,F_r$: in such
a case, our goal is to compute the set of common zeros of
$F_1,\ldots,F_r$ that are regular with respect to certain linear projection.

More precisely, if $V$ denotes the Zariski closure of
$$V(\bfs F_r)\setminus\{G=0\},$$
our objective is to design and analyze a Kronecker-type algorithm
for computing a {\em Kronecker representation of a lifting fiber of $V$}.
This algorithm follows the methodology introduced in Section \ref{subsec: third
approach inro}: it proceeds via a sequence of jumps between
one-dimensional and zero-dimensional sections of the intermediate
varieties defined as
\begin{equation}\label{eq: definition V_s}
V_s:=\overline{V(\bfs F_s)\setminus\{G=0\}}:=\overline{\{F_1=0,\ldots,F_s=0,G\not=0\}}
\quad (1\le s\le r),
\end{equation}
where $\overline{W}$ denotes the Zariski closure of a
locally closed set $W\subset\A^n$. Since the input polynomials $F_1,\ldots,F_r,G$ are defined
over $k$, there exist Kronecker representations of
lifting fibers of $V:=V_r$ defined over $k$, except for the case when $k$ is a finite
field of small cardinality. Accordingly, we will compute a
Kronecker representation defined over $k$,
unless $k$ is a finite field of small cardinality---in
which case we will work over a suitable finite extension of $k$ with sufficiently
large cardinality.

We assume that the polynomials $F_1,\ldots,F_r,G$ are given by
a straight-line program in $k[X_1,\ldots,X_n]$. The cost of the
algorithm will be expressed in terms of the number of arithmetic
operations in $k$, along with two additional operations: identity tests and
extractions of $p$-th roots (the latter only relevant when
$\mathrm{char}(k)=p>0$).

We will then specialize to two important cases: when $k=\fq$, a finite field,
and when $k=\Q$, the field of rational
numbers. In both cases, we express the cost of the algorithm
in terms of {\em bit complexity}. However, neither case allows a trivial
adaptation: for $\fq$, one must determine the minimal value of $q$ for
which the algorithm succeeds directly over the base field;
for $k=\Q$, we analyze the growth of intermediate results and
find a prime $p$ such that the algorithm succeeds modulo $p$,
allowing a $p$-adic lifting step to recover a
Kronecker representation defined over $\Q$.

Now we summarize the structure of the paper. Section \ref{section:
locally closed sets} addresses the problem of computing a change of
coordinates $\bfs Y:=\bfs \lambda\cdot \bfs X$ such that the new
variables are in a simultaneous Noether position with respect to all
the varieties $V_s$ defined in \eqref{eq: definition V_s}, for $1\le
s\le r$. We also seek a {\em single point} $\bfs
p:=(p_1,\ldots,p_{n-1}) \in\A^{n-1}$ such that the associated
lifting fibers \eqref{eq: linear sections} satisfy properties $({\sf
L}_1)$--$({\sf L}_2)$ of Section \ref{subsec: second approach}.

An additional subtlety arises when choosing the point $\bfs p$: as
shown in Diagram \ref{eq: sth elimination step}, the algorithm
involves computing a Kronecker representation of the $(s+1)$th
lifting fiber $\Pi_1^{-1}(\bfs G_r^{s+1})$, via the intersection of
the $s$th lifting curve $C_s$ with the hypersurface defined by
$F_{s+1}(\bfs p^{s+1},Y_{n-s},\ldots,Y_n)$, where $\bfs
p^{s+1}:=(p_1,\ldots,p_{n-s-1})$. As discussed in Section
\ref{subsec: first approach}, the Kronecker representation of $C_s$
is only valid in a {\em nonempty Zariski open dense subset} of
$C_s$, namely, where the discriminant $\rho_s(\bfs p^{s+1},Y_{n-s})$
does not vanish.

Thus, it is essential that the hypersurface $\{F_{s+1}(\bfs
p^{s+1},Y_{n-s},\ldots,Y_n)=0\}$ {\em intersects properly the
discriminant locus} $\{\rho_s(\bfs p^{s+1},Y_{n-s})=0\}$ within
$C_s$ for $1\le s\le r-1$. We have the following result (see Theorem
\ref{th: preproc: all conditions} and Corollary \ref{coro: preproc:
all conditions for all s} for precise statements).
\begin{theorem}\label{th: preprocessing intro}
Let $\bfs \Lambda:=(\Lambda_{ij})_{1\le i,j\le n}$ be a matrix of
indeterminates over $\overline{k}$, and let
$\widetilde{\bfs Y}:=\bfs \Lambda \cdot \bfs X$. There exists a
nonzero polynomial $B\in\overline{k}[\bfs \Lambda,
\widetilde{Y}_1\klk\widetilde{Y}_{n-1}]$, of degree at most
$2n^2rd\delta^3$, where $\delta:=\max\{\deg V_1,\ldots,\deg V_r\}$ and
$d:=\max\{\deg F_1,\ldots,\deg F_r, \deg G\}$, such
that for any $(\bfs \lambda,\bfs p)\in\A^{n^2}\times \A^{n-1}$
with $B(\bfs \lambda,\bfs p)\not=0$, the following conditions are
satisfied:
\begin{enumerate}
  \item The variables $\bfs Y:=(Y_1\klk Y_{n}):=\bfs \lambda\cdot \bfs
X$ are in Noether position with
  respect to $V_s$ for $1\le s\le r$;
  \item The lifting fiber of $V_s$ defined by $\bfs p^{s}:=(p_1,\ldots,p_{n-s})$
  satisfies conditions $({\sf L}_1)$--$({\sf L}_2)$ for $1\le i\le s$;
  \item The hypersurface $\{F_{s+1}(\bfs p^{s+1},Y_{n-s},\ldots,Y_n)=0\}$
  intersects properly the discriminant locus $\{\rho_s(\bfs p^{s+1},Y_{n-s})=0\}$
within $C_s$ for $1\le s\le r-1$.
\end{enumerate}
\end{theorem}

Similar results are proved in \cite[Theorem 3]{HeMaWa01}, for a
variety $V$ defined over the rationals $\Q$, and in
\cite[Theorem 3.3]{CaMa06a}, for a variety $V$ defined over a
finite field $\fq$. The new result contained in Theorem \ref{th: preproc: all conditions}
improves both results in two important aspects. The first
one is its generality, as Theorem \ref{th: preproc: all conditions} holds for varieties
defined over a general perfect field $k$. The second aspect
is that we obtain a significant improvement of the degree
estimates of \cite[Theorem 3]{HeMaWa01} and \cite[Theorem 3.3]{CaMa06a}.
This is critical in the case when $k$ is a finite field $\fq$,
because it determines the smallest size $q$ such that the
algorithm does not require working on a finite field extension
of $\fq$ (see Algorithm \ref{algo: main algorithm for fq}).  When one is interested in computing
a point {\em with coordinates in $\fq$} on the variety under
consideration, it is essential that the resulting Kronecker representation
of the algorithm is defined over $\fq$ (see \cite{CaMa06a} and
\cite{GiMaPePr23b}).

In view of this result, given a parameter $\varepsilon>0$, the first step of the algorithm
is to choose the entries
of $\bfs\lambda$ and $\bfs p$ randomly from a set $\mathcal{S}\subset k$ of size
at least $\varepsilon^{-1}2n^2rd\delta$, or from a finite field extension
of sufficiently large cardinality, if the base field $k$ does not contain
enough points. The well-known
Zippel-Schwartz lemma (see, e.g., \cite[Lemma 6.44]{GaGe99}) guarantees
that the probability of success of this random choice is at least $1-\varepsilon$.

Section \ref{sec: lifting fibers} studies the {\em lifting curves} $C_s$
$(1\le s\le r-1)$, which appear as intermediate objects in our algorithm,
and the procedure used to compute a Kronecker representation of a
lifting curve from a lifting fiber---namely, the Newton-Hensel lifting (the
arrow ``NH lifting''
in Diagram \ref{eq: sth elimination step}). This version of the
Newton-Hensel process is based on the approach introduced in \cite{GiLeSa01}. In
Proposition \ref{prop: main loop Newton lifting}, we provide
a detailed analysis of the correctness of the Newton-Hensel lifting,
and Theorem \ref{theo: correctedness: Newton lifting} establishes the
validity of the procedure
for computing a Kronecker representation of a lifting curve $C_s$.

Section \ref{sec: intersection step} is devoted to the other
critical component of the algorithm:
the ``intersection step''. This procedure takes as input a Kronecker representation
of a lifting curve $C_s$ and a straight-line program for $F_{s+1}$,
and outputs a Kronecker representation of the $(s+1)$-th lifting fiber
$\Pi^{-1}(\bfs G_r^{s+1})$ (as shown by the ``Intersect'' arrow
in Diagram \ref{eq: sth elimination step}). A suitable choice
of $\bfs \lambda$ and $\bfs p$, in accordance with Theorem \ref{th: preprocessing intro},
greatly simplifies this
step by ensuring that ``spurious'' points can be easily detected and
removed (see Proposition \ref{prop: C_s cap F_(s+1) and G has dim zero}
and Algorithm \ref{algo: computation projection}).

By combining the Newton-Hensel lifting and the intersection step,
Section \ref{sec: the whole algorithm} presents a full procedure for computing a
Kronecker representation of the final lifting fiber $\Pi^{-1}(\bfs G_r^r)$ of the
input variety $\Pi^{-1}(\bfs F_r)$. More precisely, we obtain
the following result (see Theorem \ref{th: main algo for k} for the precise statement).
\begin{theorem}\label{th: algo intro}
Let $k$ be a perfect field, and let $F_1,\ldots,F_r,G\in
k[X_1,\ldots,X_n]$ be polynomials satisfying $({\sf H}_1)$--$({\sf
H}_2)$, given by a straight-line program of length $L$. For $1\le
s\le r$, let $V_s:=\overline{V(F_1,\ldots,F_s)\setminus V(G)}$ and
define $\delta_s:=\deg V_s$. Let $\delta:=\max\{\delta_1,\ldots,
\delta_r\}$ and $d:=\max\{\deg F_1,\ldots,\deg F_r,\deg G\}$. Assume
that $\delta>d$. Then a Kronecker representation of a lifting fiber
of $V_r$ can be computed by a probabilistic algorithm using
$$\mathcal{O}\,\widetilde{\ }\big(n^4+r(L+n^2+r^4)\,d\,\delta^2\big)
$$
arithmetic operations in $k$,
$\mathcal{O}\,\widetilde{\ }\big((rL+rn^2+r^2)\,d\,\delta^2\big)$
identity tests in $k$, and $\mathcal{O}(r\delta)$
extractions of $p$-th roots in $k$  (if $\mathrm{char}(k)=p>0$),
where the soft-Oh notation $\mathcal{O}\,\widetilde{\ }$ ignores
logarithmic factors.
\end{theorem}

This algorithm generalizes those of \cite{HeMaWa01} (for $k=\Q$),
\cite{CaMa06a} (for $k=\fq$, a finite field) and \cite{GiLeSa01}, \cite{DuLe08}
(for characteristic 0). This comes at the cost of relying
on operations such as identity tests and extractions of $p$-th roots in $k$,
whose cost must be specified for the field $k$ under consideration.

Given
a real number $\varepsilon$ with $0<\varepsilon<1/(16r)$, the algorithm succeeds
with probability at least
$1-8r\varepsilon$ (see Theorem \ref{th: main algo - prob success}). Moreover,
if $k$ has cardinality at least $N:=\varepsilon^{-1}2n^2rd\delta^2$, the
output is guaranteed to lie on $k$; otherwise, it lies in
a finite field extension of $k$ of cardinality at least $N$.

We then examine two classical settings in
symbolic computation: when $k$ is a finite field and when
$k=\Q$. In both, we analyze the bit
complexity of the algorithm. For a finite field $\fq$ with
$q>\varepsilon^{-1}2n^2rd\delta^2$, Theorem \ref{th: preprocessing intro}
ensures that a Kronecker representation of a lifting
fiber of $V_r$ exists over $\fq$. We provide a probabilistic
algorithm that computes such a Kronecker representation (see Theorem \ref{th: main algorithm
for fq} for a precise statement).
\begin{theorem}\label{th: main algorithm for fq intro}
Let $F_1\klk F_r,G\in\fq[X_1\klk X_n]$ be polynomials of degree at
most $d>0$, satisfying hypotheses $({\sf H}_1)$--$({\sf H}_2)$,
given by a straight--line program in $\fq[X_1\klk X_n]$ of length
$L$. For $1\le s\le r$, let
$V_s:=\overline{V(F_1,\ldots,F_s)\setminus V(G)}$ and
$\delta_s:=\deg V_s$. Let
$\delta:=\max\{\delta_1,\ldots,\delta_r\}$, and assume that
$\delta>d$. Given $0<\varepsilon<1/4$, if
$q>2\varepsilon^{-1}n^2rd\delta^3$, then a Kronecker representation
of a lifting fiber of $V_r$ defined over $\fq$ can be computed with
bit complexity in
$$\mathcal{O}\,\widetilde{\ }\big((n^4+r(L+n^2+r^4))\,d\,\delta^2\log_2q\big)$$
and probability of success at least $1-2\varepsilon$.
\end{theorem}

We observe that this is not a straight-forward application of the algorithm in
Theorem \ref{th: algo intro} to the case $k=\fq$. To minimize requirements
on $q$, the algorithm incorporates steps over algebraic extensions of $\fq$,
ensuring high probability of success while keeping
the final output in the base field. This significantly
improves the bounds in \cite[Theorem 4.8]{CaMa06a},
which requires $q>5\,\varepsilon^{-1}n^4d\,\delta^4$, and also offers better
bit complexity.
We finally mention that the Kronecker method can also be adapted to the case
of dense polynomial systems over fields that allow fast multipoint evaluation
of multivariate polynomials. This situation has been analyzed in detail by
van der Hoeven and Lecerf \cite{HoLe21} in the setting of finite fields.
Nevertheless, their results concern a computational model and problem setting
different from ours, and thus do not overlap with the present contribution.

In Section \ref{sec: systems over Q} we address the case $k=\Q$. A direct application of
Theorem \ref{th: algo intro} may result in exponential growth
of the bit size of intermediate results. To avoid
this, the modular strategy of \cite{GiMa19} is adopted: the algorithm
is executed modulo a randomly chosen prime number $p$, followed
by $p$-adic lifting (originally introduced in \cite{GiLeSa01}).
If $p$ is ``lucky'', the $p$-adic lifting yields a Kronecker
representation of a lifting fiber over $\Q$ of the input variety.
The selection of $p$ relies on a combination of the techniques in
\cite{GiMa19} and the refined
bounds of Theorem \ref{th: preprocessing intro}. The result is the following
(see Theorem \ref{th: cost algorithm over Q} for a precise statement).
\begin{theorem}\label{th: algorithm over Q intro}
Let $F_1\klk F_r,G\in\Z[X_1\klk X_n]$ be polynomials of degree at
most $d>0$ and coefficients of bit length at most $h$ satisfying
hypotheses $({\sf H}_1)$-$({\sf H}_2)$, which are given by a
straight--line program in $\Z[X_1\klk X_n]$ of length $L$ and
parameters of bit length at most $h$. For $1\le s\le r$, let
$V_s:=\overline{V(F_1,\ldots,F_s)\setminus V(G)}$ and
$\delta_s:=\deg V_s$. Let
$\delta:=\max\{\delta_1,\ldots,\delta_s\}$, and assume that
$\delta>d$. Let $0<\varepsilon<1/18r$ and $m\in \mathbb{N}$ be
given. Then a Kronecker representation of a lifting fiber of $V_r$
can be computed with bit complexity
$$\mathcal{O}\,\widetilde{\ }\Big(n(L+n^4)\,d\,\delta^2\log_2(\varepsilon^{-1}) +
n^2d^rL\,\delta_r\big(h+\log_2(\varepsilon^{-1})\big)+\log_2^3(\varepsilon^{-1})\Big),$$
and probability of success at least $1-9r\varepsilon$.
\end{theorem}

The sharpness of the probability estimates established for the
algorithms in Theorems \ref{th: algo intro}, \ref{th: main algorithm
for fq intro}, and \ref{th: algorithm over Q intro} remains unclear.
The random parameters chosen in these algorithms are required to
avoid certain algebraic hypersurfaces whose existence is proved and
whose degrees are explicitly bounded. These hypersurfaces are,
however, sufficiently intricate to make it difficult to identify
families of polynomial systems for which the bounds would be tight,
or to derive meaningful lower bounds for the success probabilities.
Regarding complexity, our algorithms have cost quadratic in the
maximum degree $\delta$ of the intermediate varieties. Within the
equation-by-equation framework of Kronecker-type methods,
subquadratic dependence on $\delta$ appears unlikely in the general
case, given the inherently dense nature of the computations
involving lifting curves and the structure of the lifting process.
Achieving algorithms whose complexity depends instead of the degree
of the output variety, rather than on $\delta$, would represent a
substantially more difficult challenge. Finally, we note that,
although checking whether the structural hypotheses $({\sf
H}_1)$-$({\sf H}_2)$ hold for a given input system is in general
computationally prohibitive, the algorithms may still be applied to
arbitrary inputs and are expected to fail gracefully when these
assumptions are not satisfied. Moreover, it remains computationally
feasible to verify a posteriori that the obtained output indeed
corresponds to solutions of the input system, thus ensuring
practical reliability even when the hypotheses are not explicitly
checked.
%
%
%----------------------------------------------------------------------
%----------------------------------------------------------------------
%----------------------------------------------------------------------
%----------------------------------------------------------------------
%----------------------------------------------------------------------
%----------------------------------------------------------------------
%----------------------------------------------------------------------
%----------------------------------------------------------------------
%
\section{Notions, notations and preliminary results}
\label{section: notation, notations}
We use standard notions and notations of commutative algebra and
algebraic geometry, as found, for example, in \cite{Harris92},
\cite{Kunz85}, and \cite{Shafarevich94}.

Let $k$ be a perfect field and $\ck$ its algebraic closure. We
denote by $\A^n$ the $n$--dimensional affine space $\ck{}^{n}$. %Both spaces
%are endowed with their respective Zariski topologies over $\K$, for
%which a closed set is the zero locus of a set of polynomials of
%$\K[X_1,\ldots, X_{n}]$, or of a set of homogeneous polynomials of
%$\K[X_0,\ldots, X_{n}]$.
%
An {\em affine variety of $\A^n$ defined over} $k$ (or an affine
$k$--variety) is the set of common zeros in $\A^n$ of polynomials
$F_1,\ldots, F_{m} \in
k[X_1,\ldots, X_{n}]$. %We think a projective or affine
%$K$--variety to be equipped with the induced Zariski topology.
We will often write $V(F_1\klk F_m)$ or $\{F_1=0\klk
F_m=0\}$ to denote this affine $k$--variety. For a polynomial
$G\in k[X_1\klk X_n]$, we denote by $\{G\not=0\}$ the complement in $\A^n$ of the
hypersurface $\{G=0\}$.

A $k$--variety $V$ is said to be $k$--{\em irreducible} if it cannot be
expressed as a finite union of proper $k$--subvarieties of $V$. Any
$k$--variety $V$ admits a unique (up to reordering) irredundant
decomposition $V=\mathcal{C}_1\cup \cdots\cup\mathcal{C}_s$, where each
$\mathcal{C}_i$ is  $k$--irreducible. These are
called the {\em irreducible} $k$--{\em components} of $V$.

For a $k$--variety $V$ of $\A^n$, we denote by $I(V)$ its {\em
defining ideal}, namely, the set of all polynomials in $k[X_1,\ldots,
X_n]$ that vanish on $V$. The {\em coordinate ring} $k[V]$ of $V$ is
defined as $k[X_1,\ldots,X_n]/I(V)$. When $V$ is irreducible, then
$k[V]$ is a domain, and its field of fractions
$k(V)$ is called the {\em field of rational functions} of $V$. The
{\em dimension} $\dim V$ of $V$ is the maximal length $r$ of
chains $V_0\varsubsetneq V_1 \varsubsetneq\cdots \varsubsetneq V_r$
of nonempty irreducible $k$--varieties contained in $V$. We say that
$V$ has {\em pure dimension} $r$ if all its irreducible
$k$--components are of dimension $r$.

The {\em degree} $\deg V$ of an irreducible $k$--variety $V$ is the
maximum number of points in the intersection $V\cap L$, where
$L$ is a linear space of codimension $\dim V$ such that $V\cap L$ is a
finite set. More generally, following \cite{Heintz83} (see also
\cite{Fulton84}), if $V=\mathcal{C}_1\cup\cdots\cup \mathcal{C}_s$
is the decomposition of $V$ into irreducible $k$--components, we
define the degree of $V$ as
$$\deg V:=\sum_{i=1}^s\deg \mathcal{C}_i.$$
We will frequently use the following {\em B\'ezout inequality} (see
\cite{Heintz83}, \cite{Fulton84}, \cite{Vogel84}): if $V$ and $W$
are $k$--varieties of $\A^n$, then
\begin{equation}\label{eq: Bezout}
\deg (V\cap W)\le \deg V \cdot \deg W.
\end{equation}

%Another result we shall use concerns the behavior of degree under
%linear mappings. Let $V\subset\Pp^m$ and $W\subset\Pp^n$ be
%$K$--varieties and let $\phi:V\to W$ be a regular linear map. Then
%(see, e.g., \cite[Lemma 2.1]{CaMa07})
%%
%  \begin{equation}\label{eq:degree linear projection}
%    \deg \overline{\phi (V)} \leq \deg V,
%  \end{equation}
%%
%where $\overline{\phi (V)}$ is the Zariski closure of $\phi(V)$ in
%$\Pp^n$, $\deg \overline{\phi (V)}$ denotes the degree of
%$\overline{\phi (V)}$ as a $K$--subvariety of $\Pp^n$ and $\deg V$
%denotes the degree of $V$ as a $K$--subvariety of $\Pp^m$.
%
%----------------------------------------------------------------
%----------------------------------------------------------------
%
%\paragraph{Singular locus}

Let $V\subset\A^n$ be a $k$--variety with defining ideal $I(V)\subset k[X_1,\ldots,
X_n]$. For a point $\bfs x\in V$, the {\em
dimension} $\dim_{\bfs x}V$ {\em of} $V$ {\em at} $\bfs x$ is the maximum
dimension of the irreducible $k$--components of $V$ containing
$\bfs x$. If $I(V)=(F_1,\ldots, F_m)$, the {\em tangent space} $T_{\bfs x}V$
{\em to $V$ at $\bfs x$} is the kernel of the Jacobian matrix $(\partial
F_i/\partial X_j)_{1\le i\le m,1\le j\le n}(\bfs x)$ of
$F_1,\ldots, F_m$ with respect to $X_1,\ldots, X_n$ evaluated
at $\bfs x$. %We have (see, e.g., \cite[page 94]{Shafarevich94})
%
%$$\dim T_xV\ge \dim_xV.$$
%
The point $\bfs x$ is {\em regular} if $\dim T_{\bfs x}V=\dim_{\bfs x}V$; otherwise,
it is called {\em singular}. The set of singular points
of $V$ is the {\em singular locus} $\mathrm{Sing}(V)$ of $V$; a
variety is called {\em nonsingular} if its singular locus is empty.
%
%----------------------------------------------------------------
%----------------------------------------------------------------
%
%\paragraph{Mappings}

%Regular maps will be represented by solid arrows $\to$, while
%partial rational maps will be indicated with dashed arrows
%$\dashrightarrow$.
Let $V$ and $W$ be irreducible affine $k$--varieties of the same
dimension, and let $f:V\to W$ be a regular map for which
$\overline{f(V)}=W$, where $\overline{f(V)}$ is the closure of
$f(V)$ with respect to the Zariski topology of $W$. Such a map is
called {\em dominant}. Then $f$ induces a ring extension
$k[W]\hookrightarrow k[V]$ by composition with $f$. We say that the
dominant map $f$ is a {\em finite morphism} if this extension is
integral, i.e., each element $\eta\in k[V]$ satisfies a monic
equation with coefficients in $k[W]$. For $V$ pure--dimensional and
$W$ irreducible, a regular map $f:V\to W$ as above is called {\em
finite} if its restriction to each irreducible component of $V$ is a
finite morphism. A basic fact is that any finite morphism is
necessarily closed. We recall the following well--known
fact.
\begin{fact}[{\cite[\S 4.2, Proposition]{Danilov94}}]
\label{fact: preimage finite mapping} Let $f:V\to W$ be a finite
morphism. Then for every irreducible closed subset $S\subset W$,
the preimage $f^{-1}(S)$ is of pure dimension $\dim S$.
\end{fact}

%$\dashrightarrow$.
Let $V$ and $W$ be irreducible affine $k$--varieties of the same
dimension and $f:V\to W$ a dominant map. Then $f$ induces a finite
field extension $k(W)\hookrightarrow k(V)$. The degree of this field
extension is the {\em degree of the morphism} $f$, denoted $\deg f$.

If the field extension $k(W)\hookrightarrow k(V)$ is
separable, then by \cite[Proposition 1]{Heintz83}, the number of
points in the fiber $f^{-1}(\bfs y)$ is at most $\deg f$ for any $\bfs y\in W$ with a finite
fiber, with equality holding in a Zariski open subset of $W$. For $\bfs x\in
f^{-1}(\bfs y)$, $f$ is {\em unramified} at $\bfs x$ if the
differential mapping $d_{\bfs x}\pi \colon T_{\bfs x}V\to T_{\bfs
y}W$ is injective. The fiber $f^{-1}(\bfs y)$ is unramified
if this holds for all $\bfs x\in f^{-1}(\bfs y)$. The map $f$
is {\em generically unramified} if there exists a dense Zariski
open subset $U$ of $W$ such that the fiber $f^{-1}(\bfs y)$ is
unramified for any $\bfs y\in U$.

An affine $k$--variety $V:=V(F_1 \klk F_r) \subseteq \A^n$
defined by $r\le n$ polynomials is a \emph{set-theoretic complete
intersection} if it has pure
dimension $n-r$. If, in addition, the ideal $(F_1 \klk F_r)$ is
radical, then $V$ is an
\emph{ideal-theoretic complete intersection}. The polynomials $F_1 \klk
F_r\in k[X_1 \klk X_n]$  form a \emph{regular sequence} if the ideal
$(F_1 \klk F_r)$ they generate in $k[X_1 \klk X_n]$ is proper,
$F_1$ is nonzero and, for $2\le i \le r$, $F_i$ is neither zero nor
a zero divisor in $k[X_1 \klk X_n]/(F_1 \klk F_{i-1})$.

%
%--------------------------------------------------------------------
%--------------------------------------------------------------------
%--------------------------------------------------------------------
%--------------------------------------------------------------------
%
\subsection{Kronecker representations}
\label{subsection: geometric solutions}
Let $V\subset\A^n$ be a $k$--variety of pure dimension $n-s$, and
let $I\subset k[X_1\klk X_n]$ be its vanishing ideal. For a change
of variables $(X_1\klk X_n)\to(Y_1\klk Y_n)$, denote
$R:=k[Y_1,\dots, Y_{n-s}]$, $A:=k[V]$ and $K:=k(Y_1\klk Y_{n-s})$.
Consider $B:=K[Y_{n-s+1}\klk Y_n]/I^e$ as a $K$--vector space, where
$I^e$ is the extended ideal $IK[Y_{n-s+1}\klk Y_n]$, and let
$\delta:=\dim_{K}B$.
\begin{definition}\label{def: geometric solution}
A \emph{Kronecker representation} of $I$ (or $V$) consists of the
following data:
\begin{itemize}
  \item a linear change
  of variables $(X_1, \dots,\! X_n)$ $\to$ $(Y_1,\dots, Y_n)$ with
  the following properties:
  \begin{itemize}
\item The linear map $\pi:V\to\A^{n-s}$, defined by
$Y_1,\ldots,Y_{n-s}$, is a finite morphism. In this case,
the change of variables is called a {\em Noether normalization} of
$V$, and we say that the variables $Y_1,\ldots,Y_n$ are in {\em
Noether position} with respect to $V$, with
$Y_1,\ldots,Y_{n-s}$ being {\em free variables}.
\item The linear form $Y_{n-s+1}$ induces a \emph{primitive element} of $I$, i.e., an
element $y_{n-s+1}\in A$ whose monic minimal polynomial $M\in
R[T]$ over $K$ satisfies the condition $\deg_T M=\delta$. Note
that $\deg M=\deg_TM\le\deg V$;
\end{itemize}
  \item The monic minimal polynomial $M\in R[T]$ of $Y_{n-s+1}$ modulo
  $I$;
  \item Polynomials $W_{n-s+2}\klk W_n\in K[T]$, of degree at most
  $\delta-1$, such that the following identity holds in
  $K[Y_{n-s+1}, \dots, Y_n]$:
\begin{equation}\label{ident:kronecker_repres_ideals}
I^e=\bigl(M(Y_{n-s+1}),
M'(Y_{n-s+1})Y_{n-s+2}-W_{n-s+2}(Y_{n-s+1}) \klk
M'(Y_{n-s+1})Y_n-W_n(Y_{n-s+1})\bigr),
\end{equation}
where $M'$ denotes the derivative of $M$ with respect to $T$.
\end{itemize}
Considering instead polynomials $V_{n-s+2}\klk V_n$ of degree at
most $\delta-1$, such that
\[
I^e=\bigl(M(Y_{n-s+1}), Y_{n-s+2}-V_{n-s+2}(Y_{n-s+1}), \dots,
Y_n-V_n(Y_{n-s+1})\bigr),
\]
we have a \emph{univariate representation} of $I$ (or $V$).
\end{definition}

The identity \eqref{ident:kronecker_repres_ideals} allows us a
geometric interpretation of Kronecker representations, as described
in Section \ref{subsec: first approach}. Let $\ell:\mathbb{A}^n\rightarrow
\mathbb{A}^n$ be the linear map defined by $Y_1,\dots,Y_n$, and let
$W:=\ell(V)$. We interpret $Y_1,\dots, Y_n$ as new indeterminates,
and consider the projection $\Pi:W\rightarrow \mathbb{A}^{n-s+1}$
onto the first $n-s+1$ coordinates.
Considering $M$ as a polynomial of $k[Y_1,\dots,Y_{n-s+1}]$, we
observe that $\Pi$ defines a birational map between $W$ and the
hypersurface $\{M=0\}\subset\mathbb{A}^{n-s+1}$. Its inverse is the
rational map $\Phi:\{M=0\}\rightarrow W$, given by
$$\Phi(\bfs y):=\left(\bfs y,
\frac{W_{n-s+2}(\bfs y)}{M'(\bfs y)}, \dots, \frac{W_n(\bfs
y)}{M'(\bfs y)}\right).$$
This map is defined  outside
the zero locus of $M'$, namely, on the nonempty Zariski dense open
subset $\{M=0\}\cap\{M'\not=0\}\subset\A^{n-s+1}$. If $\rho\in k[Y_1,\ldots,X_{n-s}]$
denotes the {\em discriminant} of $M$ with respect to $Y_{n-s+1}$,
it is easy to see that the set where $\Phi$ is undefined is
precisely $\{\rho=0\}\cap\{M=0\}$.
%
%----------------------------------------------------------------------
%----------------------------------------------------------------------
%----------------------------------------------------------------------
%----------------------------------------------------------------------
%
\subsection{Computational model and basic operations}
Throughout this work, $\log_2$ denotes the base-2 logarithm. In addition
to the standard Big--Oh notation $\mathcal{O}$, we will also use the Soft--Oh
notation $\mathcal{O}^\sim$, which does not take into
account logarithmic terms. %****Further, will use the quantity
%$\mathcal{U}(m)=m\log_2^2m\log_2\log_2 m$.****
More precisely, for functions $f=f(n,d,h)$ and $g=g(n,d,h)$ depending on
integer parameters $n$, $d$, $h$, we write
$f\in \mathcal{O}^\sim(g)$ if there exists $s\ge 0$ such that
$f\in \mathcal{O}(g\log_2^sg)$.

In computer algebra, algorithms often assume the standard dense
(or sparse) representation model, where multivariate polynomials are
represented via their full list of coefficients (or only the nonzero
ones). However, a general $n$--variate polynomial of
degree $d$ has $\binom{n + d}{n}=\mathcal{O}(d^n)$ nonzero
coefficients, and thus its dense or sparse representation requires an
exponential size in both $d$ and $n$. Consequently, algorithms
manipulating such representations may require an exponential number of
arithmetic operations with respect to $d$ and $n$.

To overcome this, we adopt an alternative representation based on
straight--line programs (cf. \cite{BuClSh97}). A {\em
(division--free) straight--line program} $\beta$ in $k[X_1\klk X_n]$,
which {\em represents} or {\em evaluates} polynomials $F_1\klk F_s
\in k[X_1\klk X_n]$, is a sequence $(Q_1, \dots, Q_r)$ of elements of
$k[X_1\klk X_n]$ satisfying the following conditions:
\begin{itemize}
  \item $\{F_1, \dots, F_s\} \subseteq \{Q_1, \dots, Q_r\}$;
  \item there exists a finite subset $\mathcal{T}\subset
  k$, called the set of \emph{parameters} of $\beta$,
  such that for every $1 \le \rho \le r$, the polynomial $Q_{\rho}$ either is an
element of $\mathcal{T} \cup \{X_1, \dots, X_n \}$, or there exist
$1 \le \rho_1, \rho_2 < \rho$ such that $Q_{\rho}={\sf
op}_{\rho}(Q_{\rho_1},Q_{\rho_2})$, where ${\sf op}_{\rho}$ is one
of the arithmetic operations $+, -, \times$.
\end{itemize}
The {\em length} of $\beta$ is the total number of
arithmetic operations performed during the evaluation process
defined by $\beta$.

Further, we allow decisions and selections (subject to previous
decisions). For this reason, we consider {\em computation trees},
which are straight--line programs with {\em branchings}. The length of a
computation tree is defined analogously to that for
straight--line programs (see, e.g., \cite{BuClSh97} for more details
on the notion of computation trees).

When $k$ is a finite field
$\fp$, or $k=\Q$ and an upper bound for the bit size is known, each arithmetic operation ${\tt
op}\in\{+,-,\times\}$ can be implemented by a Boolean circuit that
receives the bit representation of ${\tt a}\,,{\tt
b}\in k$ and returns the bit representation of ${\tt op}\,({\tt
a}\,,{\tt b})$. The runtime is measured by the size of the
resulting Boolean circuit, i.e., the total number of bit operations
performed.
%----------------------------------------------------------------------
%----------------------------------------------------------------------
%
\subsubsection{Basic operations}
Given $\delta>0$, we will frequently use the notation
$${\sf M}(\delta):=\delta\,\log_2\delta\,\log_2\log_2\delta.$$

Using the Sch\"onage--Strassen multiplication algorithm, the product of two
integers of bit length $m$ can be computed with $\mathcal{O}({\sf
M}(m))$ bit operations (see, e.g.,
\cite[Section 8.3]{GaGe99}). Similarly, division with remainder
of integers of bit length $m$ can be performed
with $\mathcal{O}({\sf M}(m))$ bit operations
(see, e.g., \cite[Section 9.1]{GaGe99}). It is worth mentioning that
there exist algorithms for integer multiplication and
division with remainder of integers of bit length $m$
in $\mathcal{O}(m\log_2m)$ bit operations \cite{HaHo21};
for simplicity, we will not use these algorithms in this paper.

Let ${\sf R}$ be a commutative ring with unity, $T$ an indeterminate
over ${\sf R}$, and ${\sf R}[T]$ the ring of univariate polynomials
in $T$ with coefficients in ${\sf R}$. Using the
Sch\"onhage--Strassen multiplication algorithm, the product of two
polynomials of degree at most $\delta$ can be computed with
$\mathcal{O}({\sf M}(\delta))$ arithmetic operations in ${\sf R}$
(see, e.g., \cite[Section 8.3]{GaGe99}). Moreover, given a monic
polynomial of ${\sf R}[T]$  of degree at most $\delta$, polynomial
division with remainder in ${\sf R}[T]$ can be performed with
$\mathcal{O}({\sf M}(\delta))$ arithmetic operations in ${\sf R}$
using Newton iteration
 (see, e.g., \cite[Section 9.1]{GaGe99}). It follows that
multiplication in the quotient ring ${\sf R}[T]/(f)$
can also be done with $\mathcal{O}({\sf M}(\delta))$ arithmetic
operations in ${\sf R}$.

For polynomials $f,g\in k[T]$ of degree at most $\delta$, their gcd
(along with a B\'ezout identity) can be computed  with $\mathcal{O}({\sf M}(\delta)\log_2\delta)$
arithmetic operations in $k$ and identity tests between elements of
$k$, using a fast Extended
Euclidean algorithm (see, e.g., \cite[Section 11.1]{GaGe99}). In the same
complexity, the resultant $\mathrm{res}(f,g)\in k$ can be computed.
In particular, modular inversion in the residue class ring
$k[T]/(f)$, where $f\in k[T]$ has degree at most $\delta$, can be
done with $\mathcal{O}({\sf M}(\delta)\log_2\delta)$ arithmetic
operations and identity tests in $k$.

Let $f\in k[T]$ be a monic polynomial of degree at most
$\delta$. It admits a unique square--free factorization
$f=g_1g_2^2\cdots g_m^m$,
where $g_1\klk g_m$ are monic polynomials which are pairwise--coprime. The
square--free part of $f$ is $g_1\cdots g_m$. According to
\cite[Corollary 2]{Lecerf08}, a variant of Yun's square--free
factorization algorithm (see, e.g., \cite[Algorithm 14.21]{GaGe99})
computes the square--free part of $f$ with
$\mathcal{O}({\sf M}(\delta) \log_2\delta)$ arithmetic operations in
$k$, and $\mathcal{O}(\delta)$ extractions of $p$--th roots in $k$
(if $\mathrm{char}(k)=p>0$). In particular, if
$k=\fq$ with  $\mathrm{char}(k)=p$, then a $p$--th root extraction requires $\mathcal{O}(\log (q/p))$ arithmetic
operations in $\fq$.

For a commutative ring ${\sf R}$ with unity, and a
polynomial $f\in {\sf R}[T]$ of degree at most $\delta$, its
(multipoint) evaluation at $\delta$ points of ${\sf R}$ can be done
with $\mathcal{O}({\sf M}(\delta)\log_2\delta)$ arithmetic
operations in ${\sf R}$ (see, e.g., \cite[Corollary 10.8]{GaGe99}).
Likewise, given interpolation points
$\alpha_0\klk\alpha_\delta\in{\sf R}$ with pairwise invertible
differences $\alpha_i-\alpha_j$ for $i\not=j$, and values
$\beta_0\klk\beta_\delta\in{\sf R}$, the unique
interpolating polynomial $f\in{\sf
R}[T]$ of degree at most $\delta$ can
be computed with $\mathcal{O}({\sf M}(\delta)\log_2\delta)$
arithmetic operations in ${\sf R}$ (see, e.g., \cite[Corollary
10.12]{GaGe99}).

The determinant and the adjoint matrix of a matrix in
${\sf R}^{n\times n}$ can be computed (without
divisions) using the Samuelson--Berkowitz algorithm with
$\mathcal{O}(n^{\omega+1})$ arithmetic operations (see \cite{Berkowitz84}),
where $\omega\ge 2$ denotes the exponent of matrix multiplication over
${\sf R}$; that is, the infimum over all real numbers $\beta$ such that two
matrices of ${\sf R}^{n\times n}$ can be multiplied using
$\mathcal{O}(n^\beta)$ arithmetic operations in ${\sf R}$. The best
currently known bound is $\omega<2.371552$ (see \cite{WiXuXuZh24}).

For the sake of simplicity and clarity in our asymptotic
estimates, we will take $\omega=3$ throughout the paper. This assumption
is justified by the fact that the parameter $n$ (the dimension of the
Jacobian matrices involved) remains relatively small compared with other
quantities such as the degrees $\delta_s$ and the height bounds. Adopting
$\omega=3$ therefore makes the complexity expressions more transparent, without
affecting their theoretical validity or the practical efficiency of
the algorithms under consideration.
%
%
%----------------------------------------------------------------------
%----------------------------------------------------------------------
%----------------------------------------------------------------------
%----------------------------------------------------------------------
%
\subsubsection{Probabilistic aspects}
Our algorithms are of the {\em Monte  Carlo} or~\textit{BPP} type
(see, e.g., \cite{Pardo95}, \cite{GaGe99}), meaning that they return the
correct output with probability at least a fixed constant
strictly greater than~$1/2$. Consequently, the error probability
can be made arbitrarily small by independently repeating the algorithm. The
probabilistic nature of our methods arises from certain random
choices of points that are required not annihilating specific polynomials.
This is based on
the following result (see, e.g., \cite[Lemma 6.44]{GaGe99}):
\begin{lemma}
\label{lemma: Zippel_Schwartz} Let $k$ be a field, let
$\mathcal{S}\subset k$ be a finite set, and let $F\in k[\xon]$ be a
nonzero polynomial of degree at most~$d$. Then the number of zeros
of $F$ in $\mathcal{S}^n$ is at most $d(\# \mathcal{S})^{n-1}$.
\end{lemma}
Lemma~\ref{lemma: Zippel_Schwartz} was independently discovered by
DeMillo \& Lipton \cite{DeLi78}, Zippel \cite{Zippel79}, and
Schwartz \cite{Schwartz80}. For the  analysis of our algorithms, we
interpret the statement of Lemma~\ref{lemma: Zippel_Schwartz} in
probabilistic terms. Specifically, given a fixed nonzero polynomial
$F\in k[\xon]$ of degree at most~$d$, and a finite set
$\mathcal{S}\subset k$, the lemma implies that the probability of
randomly selecting a point~$\bfs a\in \mathcal{S}^n$ such
that~${F(\bfs a)=0}$ is at most $d/\#\mathcal{S}$, assuming a
uniform distribution on $\mathcal{S}^n$.

It is worth mentioning that there is a more refined approach, based
on the concept of {\em correct-test sequences} \cite{HeSc82}, which
generalizes the probabilistic bound given by Lemma \ref{lemma:
Zippel_Schwartz}. A correct-test sequence for a class
$\mathcal{P}\subset k[\xon]$ is a finite sequence of points in $k^n$
such that, for every $F\in\mathcal{P}$, the vanishing of $F$ at
these points implies $F=0$. The main result in \cite{HeSc82} asserts
that, for every $L>0$, there exist short correct-test sequences for
the class $\mathcal{P}_L$ of polynomials of $k[X_1,\ldots,X_n]$
computable by straight-line programs of length at most $L$ (see also
\cite{PaSe22} for extensions and generalizations). Applying this
framework to the class of polynomials involved in our algorithms,
and taking into account their straight-line program length, could
lead to algorithms with higher reliability. Nevertheless, pursuing
this alternative would considerably increase the complexity of the
implementation, and for this reason, we have opted not to explore it
any further.

Finally, we observe that our algorithms do not appear to be of the {\em Las
Vegas} or~\textit{ZPP} type, since we do not have a mechanism to verify the
correctness of the output with a reasonable complexity.
%
%----------------------------------------------------------------------
%----------------------------------------------------------------------
%----------------------------------------------------------------------
%----------------------------------------------------------------------
%----------------------------------------------------------------------
%----------------------------------------------------------------------
%----------------------------------------------------------------------
%----------------------------------------------------------------------
%
\section{Noether normalization, primitive elements and lifting fibers}
\label{section: locally closed sets}
Let $F_1\klk F_r,G$ be polynomials in $k[X_1\klk X_n]$, with $r\le n$, of
degree at most $d$. We denote by $k[X_1\klk X_n]_G$ the localization
of $k[X_1\klk X_n]$ at the multiplicative set generated by $G$.
For $1\le s\le r$, let $(F_1\klk F_s)_G$ be
the ideal generated by $F_1,\ldots,F_s$ in $k[X_1\klk X_n]_G$.

We assume that the polynomials $F_1\klk F_r,G$ satisfy the
following conditions:
\begin{itemize}
  \item[(${\sf H}_1$)] $F_1\klk F_r$ form a regular sequence in the
  localization $k[X_1\klk X_n]_G$,
  \item[(${\sf H}_2$)] For $1\le s\le r$, the ideal $(F_1\klk F_s)_G$ is radical in $k[X_1\klk X_n]_G$.
\end{itemize}
When these conditions hold, we say that $F_1\klk F_r$ form
a {\em reduced regular sequence on} $\{G\not=0\}$.

Let $V_s\subset\A^n$ denote the Zariski closure of $\{F_1=0\klk
F_s=0,G\not=0\}$ for $1\le s\le r$. Denote $V:=V_r$. Under
assumptions (${\sf H}_1$) and (${\sf H}_2$), each $V_s$ has pure
dimension $n-s$. We denote by $\delta_s:=\deg V_s$ for $1\le s\le
r$, and define $\delta:=\max\{\delta_1,\ldots,\delta_r\}$. Without
loss of generality, we shall assume that $\delta > d$, since the
most relevant asymptotic regime for our complexity analysis is when
the degrees of the intermediate varieties grow significantly faster
than the degrees of the defining polynomials.

Let $\bfs X:=(X_1,\ldots,X_n)$.
In this section, our goal is to prove that there exists a linear
change of variables $\bfs Y:=\bfs\lambda\cdot\bfs X$, where
$\bfs\lambda$ is an $n\times n$-matrix with coefficients in $\overline{k}$,
and a point $\bfs p:=(p_1,\ldots,p_{n-1})\in\A^{n-1}$,
such that the variables $Y_1,\ldots,Y_n$ are
{\em simultaneously} in Noether position with respect to all
the varieties $V_1,\ldots,V_r$, and $\bfs p^s:=(p_1,\ldots,p_{n-s})$
is a {\em lifting point} of $V_s$ for $1\le s\le r$ (a notion to
be defined precisely in Section \ref{subsec: lifting fibers}).
In addition, we will require the points $\bfs p^s$ to satisfy a subtle
but essential condition: since each $V_s$ will be given
by a Kronecker representation---which only faithfully encodes
$V_s$ {\em outside a certain discriminant locus} (see the remark
after Definition \ref{def: geometric solution})--- the choice
of $\bfs p^s$ must ensure that the points involved in the computations
lie outside such a discriminant locus.

For this purpose, we will prove the existence of a polynomial
$B\in\overline{k}[\bfs\Lambda,\widetilde{\bfs Y}]$, where
$\bfs\Lambda$ is an $(n\times n)$-matrix of indeterminates, and
$\widetilde{\bfs Y}$ is a vector of $n-1$ indeterminates over
$\overline{k}$, with the following property: for any
$(\bfs\lambda,\bfs p)\in\A^{n\times n}\times \A^{n-1}$ satisfying
$B(\bfs\lambda,\bfs p)\not=0$, the change of variables $\bfs
Y:=\bfs\lambda\cdot\bfs X$ and the point $\bfs
p:=(p_1,\ldots,p_{n-1})\in\A^{n-1}$ satisfy all the conditions
above.

This allows the algorithm to begin by selecting a suitable pair
$(\bfs\lambda,\bfs p)$ by {\em picking its coordinates at random}
from a finite set $\mathcal{S}\subset\overline{k}$ with
$\#\mathcal{S}\ge\varepsilon^{-1}\deg B$, for a fixed
$0<\varepsilon<1/2$. By Lemma \ref{lemma: Zippel_Schwartz}, the
probability of success of this choice is at least $1-\varepsilon$.
Moreover, if $k$ has sufficiently large cardinality, we may choose
$\mathcal{S}\subset k$, ensuring that the algorithm proceeds
entirely within $k$, without introducing field extensions. For these
reasons, it is crucial to obtain a polynomial $B$ with minimal
possible degree, which will be the main focus of this section.
%
%------------------------------------------------------------------------
%------------------------------------------------------------------------
%------------------------------------------------------------------------
%
\subsection{Simultaneous Noether normalization and a primitive element}
\label{subsec: Noether normalization}
Throughout this section, for $1\le s\le r$ we interpret the elements
of $\A^{(n-s+1)n}$ as $(n-s+1)\times n$--matrices with entries in
$\ck$. Given $\bfs\lambda\in \A^{(n-s+1)n}$, we denote its $i$th
row by $\bfs\lambda_i:=(\lambda_{i,1}\klk\lambda_{i,n})\in\A^n$. We
frequently use the notation
$$\bfs\lambda_i\cdot \bfs X:=\bfs\lambda_i\cdot (X_1\klk X_n):=
\lambda_{i,1}X_1\plp\lambda_{i,n}X_n$$
for $1\le i\le n-s+1$. Accordingly, we write $\bfs\lambda\cdot \bfs
X:=(\bfs\lambda_1\cdot \bfs X \klk \bfs\lambda_{n-s+1}\cdot \bfs
X)$. We also denote by $\bfs\lambda^*\in\A^{(n-s)n}$
the submatrix consisting of the first $n-s$ rows of $\bfs\lambda$.
%
%We denote such matrices as $(\bfs\lambda^*,\bfs\lambda^0)$, where
%$\bfs\lambda^* \in\A^{(n-s+1)n}$ represents the entries of the
%submatrix formed by the last $n$ columns of $\bfs\lambda$ and
%$\bfs\lambda^0\in\A^{n-s+1}$ denotes the first column of
%$\bfs\lambda$.

In this section, we show that a generic choice of $\bfs \lambda\in
\A^{(n-s+1)n}$ defines linear forms $\bfs Y:=(Y_1,
\dots,Y_{n-s+1}):=\bfs\lambda\cdot \bfs X$ such that $\bfs
Y^*:=(Y_1\klk Y_{n-s})$ are in Noether position with respect to
$V_s$, and $Y_{n-s+1}$ is a primitive element of the $\ck(Y_1\klk
Y_{n-s})$-algebra extension
$$\ck(Y_1\klk Y_{n-s})\to\ck(Y_1\klk Y_{n-s})\otimes_{\ck[Y_1\klk Y_{n-s}]}\ck[V_s].$$
To that end, we let $\bfs \Lambda:=(\Lambda_{ij})_{1\le
i\le n-s+1,1\le j\le n}$ be a matrix of indeterminates over
$\ck$, and denote its $i$th row by $\bfs\Lambda_i:=(\Lambda_{i,1}\klk
\Lambda_{i,n})$ for $1\le i\le
n-s+1$. Let $\bfs\Lambda^*$ be the submatrix of $\bfs \Lambda$
consisting of its first $n-s$ rows. We consider the vector of
generic linear forms $\widetilde{\bfs Y}:=\bfs \Lambda \cdot \bfs X$
and set $\widetilde{\bfs Y}^*:=(\widetilde{Y}_1\klk
\widetilde{Y}_{n-s})$.

Our first result, probably well--known,
provides a condition ensuring that the ring homomorphism
$\overline{k}[Y_1,\ldots,Y_{n-s}]\to\overline{k}[V_s]$ is
integral.
\begin{lemma}
\label{lemma: preproc: integral extension} For $1\le s\le r$, there
exists a nonzero polynomial $A_s^1\in k[\bfs\Lambda^*]$ of degree at
most $(n-s)\delta_s$ with the following property: for any
$\bfs\lambda^*\in\A^{(n-s)n}$ with $A_s^1(\bfs \lambda^*)\not=0$, if
$\bfs Y^*:=\bfs\lambda^*\cdot \bfs X:=(Y_1\klk Y_{n-s})$, then the ring homomorphism
$$R_s:=\overline{k}[Y_1,\ldots,Y_{n-s}]\to\overline{k}[V_s]$$
is integral.
\end{lemma}
\begin{proof}
Without loss of generality, we may assume that $V_s$ is irreducible.
Consider the morphism
\begin{align}\label{eq:morph_chow}
  \Phi: \A^{(n-s+1)n}\times V_s  &\to
  \A^{(n-s+1)n}\times \A^{n-s+1} \\
        (\bfs\lambda,\bfs x)  &\mapsto  (\bfs
       \lambda,\bfs\lambda\cdot
       \bfs x).\nonumber
\end{align}
For $(\bfs\lambda,\bfs y)\in Im(\Phi)$, the fiber
$\Phi^{-1}(\bfs\lambda,\bfs y)$ is isomorphic to
the set
$$V_s\cap\{\bfs \lambda_1\cdot\bfs x=y_1,\ldots,
\bfs\lambda_{n-s+1}\cdot\bfs x=y_{n-s+1}\}.$$
For $\bfs \lambda_1,\ldots,\bfs\lambda_{n-s}$ generic,
the intersection $V_s\cap\{\bfs \lambda_1\cdot\bfs x=y_1,\ldots,
\bfs\lambda_{n-s}\cdot\bfs x=y_{n-s}\}$ has dimension zero. It follows
that a generic fiber $\Phi^{-1}(\bfs\lambda,\bfs y)$ has dimension
zero. Therefore, by the Theorem on the Dimension of Fibers (see, e.g.,
\cite[Chapter I, \S 6.3, Theorem 7]{Shafarevich94}) we conclude
that the Zariski closure $\overline{Im(\Phi)}$ has dimension
$(n-s+1)(n+1)-1$, namely it is a hypersurface of $\A^{(n-s+1)n}\times
\A^{n-s+1}$.

An equation for this hypersurface can be easily obtained
from a suitable specialization of the Chow form of $V_s$ (see
\cite[Proposition 3.1]{CaMa06a} or \cite[Section 3.1]{GiMa19} for details).
More precisely, $\overline{Im(\Phi)}$ is
defined by an irreducible polynomial $P_{V_s}\in
k[\bfs\Lambda,\widetilde{Y}_1\klk \widetilde{Y}_{n-s+1} ]$
satisfying the following conditions:
\begin{itemize}
  \item $\deg_{\widetilde{\bfs Y}}P_{V_s}=
  \deg_{\widetilde{Y}_i} P_{V_s}= \delta_s$ for $1\le i\le n-s+1$,
  \item $P_{V_s}$ is homogeneous of degree $\delta_s$ in each group
  of variables $(\bfs \Lambda_i,\widetilde{Y}_i)$ for $1\le i\le n-s+1$.
\end{itemize}

Let $A_s^1\in k[\bfs\Lambda]$ be the (nonzero) polynomial which
arises as coefficient of the monomial
$\widetilde{Y}_{n-s+1}^{\delta_s}$ in the polynomial $P_{V_s}$,
considering $P_{V_s}$ as an element of $k[\bfs
\Lambda][\widetilde{\bfs Y}]$. The above conditions imply $\deg
A_s^1=(n-s)\delta_s$. Furthermore, the equality $\deg_{\bfs
\Lambda_{n-s+1},\widetilde{Y}_{n-s+1}} P_{V_s}= \delta_s$ implies
that $A_s^1\in k[\bfs \Lambda^*]$.

Let $\bfs\lambda^*\in\A^{(n-s)n}$ be a point such that
$A_s^1(\bfs\lambda^*)\not=0$, and let $\bfs Y:=(Y_1\klk
Y_{n-s}):=\bfs\lambda^*\cdot \bfs X$. We claim that the condition of
the lemma holds. Indeed, let $\bfs w_1\klk \bfs w_n\in\A^n$ be
$\ck$--linearly independent elements and let $\ell_1:=\bfs w_1\cdot
\bfs X\klk \ell_n:=\bfs w_n\cdot \bfs X$. Since
$A_{1,s}^*:=A_s^1(\bfs \lambda^*)$ is a nonzero element of $\ck$,
the polynomial $P_{V_s}(\bfs \lambda^*,\bfs w_j,Y_1\klk
Y_{n-s},\ell_j)$ is an integral dependence equation for the
coordinate function induced by $\ell_j$ in the ring homomorphism
$R_s\to\ck[V_s]$ for $1\le j\le s$. Since
$\ck[\ell_1\klk\ell_n]=\ck[\xon]$,
the ring homomorphism $R_s\to\ck[V_s]$ is integral. %Finally,
%the fact that $V_s$ is an irreducible variety of dimension $n-s$
%implies that the homomorphism $R_s\to\ck[V_s]$ is actually
%an extension.
\end{proof}

Given $\bfs\lambda^*$ as in the statement of Lemma
\ref{lemma: preproc: integral extension} and $(Y_1,\ldots,Y_{n-s})
:=\bfs\lambda^*\cdot\bfs X$, by localization at $R_s\setminus\{0\}$ we have the
$\ck$-algebra homomorphism
$$K_s:=\ck(Y_1\klk Y_{n-s})\to
  \ck(Y_1\klk Y_{n-s})\otimes_{R_s}\ck[V_s].$$
In particular, Lemma \ref{lemma: preproc: integral extension} implies
that $\ck[V_s]$ is a finite $R_s$--module and
hence $K_s\otimes_{R_s}\ck[V_s]$ is a
finite--dimensional $K_s$--vector space.
Furthermore, the dimension of $K_s\otimes_{R_s}\ck[V_s]$ as
$K_s$--vector space equals the rank of $\ck[V_s]$ as $R_s$--module.
Recall that a primitive element of this $\ck$-algebra extension
is an element $y\in K_s\otimes_{R_s}\ck[V_s]$ such that its minimal
polynomial has degree $\dim_{K_s}K_s\otimes_{R_s}\ck[V_s]$.

Suppose that $y\in \ck[V_s]$ is a primitive element of the
$\ck$-algebra extension $K_s\hookrightarrow K_s\otimes_{R_s}\ck[V_s]$.
Since $R_s$ is integrally closed, the minimal dependence equation of
$y\in\ck[V_s]$ over $K_s$ equals the minimal integral dependence equation
of $y$ over $R_s$ (see, e.g., \cite[Lemma II.2.15]{Kunz85}).
%It follows that $y$ is also a primitive element of the ring
%extension $R_s\hookrightarrow\ck[V_s]$, namely its powers generate
%$\ck[V_s]$ as $R_s$-module.

Now we obtain a condition which implies that the
projection $\pi_s:V_s\to\A^{n-s}$ defined by
$Y_1,\ldots,Y_{n-s}$ is a finite morphism and the
linear
form $Y_{n-s+1}$ is a primitive element of the
corresponding $K_s$--algebra extension
\begin{proposition}
\label{prop: preproc: Noether pos + prim elem} For $1\le s\le r$,
there exists a nonzero polynomial $A_s\in k[\bfs\Lambda]$ of degree
at most $2(n-s+1)\delta_s^2$ with the following property: for any
$\bfs\lambda\in\A^{(n-s+1)n}$ with $A_s(\bfs \lambda)\not=0$, if
$\bfs Y:=\bfs\lambda\cdot \bfs X:=(Y_1\klk Y_{n-s+1})$, then
\begin{enumerate}
\item the map $\pi_s:V_s\to \A^{n-s}$ defined by $Y_1\klk Y_{n-s}$
is a finite morphism,
\label{item: preproc: condition finite morphism}
  \item the linear form $Y_{n-s+1}$ induces a primitive element
  of the $\ck$-algebra extension $K_s\hookrightarrow
  K_s\otimes_{R_s}\ck[V_s]$. \label{item: preproc: condition primitive elem}
\end{enumerate}
\end{proposition}
\begin{proof}
Let $P_{V_s}\in k[\bfs\Lambda,\widetilde{\bfs Y}]$ be the polynomial
of the proof of Lemma \ref{lemma: preproc: integral extension}. Let
$(\bfs\lambda,\bfs y^*)\in\A^{(n-s+1)n}\times\A^{n-s}$, let $\bfs
Y:=(Y_1,\ldots,Y_{n-s+1}):=\bfs\lambda\cdot \bfs X$ and $\bfs
p^*:=(p_1,\ldots,p_{n-s})$. By the definition of $P_{V_s}$, the
zeros of $P_{V_s}(\bfs\lambda,\bfs p^*,\widetilde{Y}_{n-s+1})$ are the values
that $Y_{n-s+1}$ takes at the points of the set
\begin{equation}\label{eq: zero-dim fiber of V_s}
V_s\cap\{Y_1=p_1,\ldots,Y_{n-s}=p_{n-s}\}.
\end{equation}
According to \cite{Heintz83}, for a generic $(\bfs\lambda_1,\ldots,
\bfs\lambda_{n-s},\bfs p^*) \in\A^{(n-s)n}\times\A^{n-s}$, the
variety of \eqref{eq: zero-dim fiber of V_s} consists of $\delta_s$
points. Furthermore, for $\bfs\lambda_{n-s+1}$ generic, the linear
form $Y_{n-s+1}$ separates the points of this variety. It follows
that the polynomial $P_{V_s}(\bfs\lambda,\bfs
p^*,\widetilde{Y}_{n-s+1})$ has $\delta_s$ distinct roots in
$\overline{k}$, and therefore it is separable. In particular,
$P_{V_s}(\bfs\lambda,\bfs p^*,\widetilde{Y}_{n-s+1})$ is coprime to
its derivative $\partial P_{V_s}/\partial
\widetilde{Y}_{n-s+1}(\bfs\lambda,\bfs p^*,\widetilde{Y}_{n-s+1})$,
and the discriminant $\rho_s(\bfs\lambda,\bfs p^*)$ of
$P_{V_s}(\bfs\lambda,\bfs p^*,\widetilde{Y}_{n-s+1})$ is nonzero.
Furthermore, we have
$$\delta_s\le \deg_{\widetilde{Y}_{n-s+1}} P_{V_s}(\bfs\lambda,\bfs
p^*,\widetilde{Y}_{n-s+1})\le\deg_{\widetilde{Y}_{n-s+1}}
P_{V_s}=\delta_s.$$
%
%Since $k[\bfs\Lambda,\widetilde{\bfs
%Y}]/(P_{V_s})$ is a reduced $k$--algebra and $k$ is a perfect field,
%\cite[Proposition 27.G]{Matsumura80} shows that the
%(zero--dimensional) $k(\bfs\Lambda,\widetilde{\bfs Y}^*)$--algebra
%$k(\bfs\Lambda,\widetilde{\bfs Y}^*)
%[\widetilde{Y}_{\!n-s+1}]/(P_{V_s})$ is reduced. In particular,
We conclude that %$P_{V_s}$ is a separable element of $k(\bfs\Lambda,\widetilde{\bfs
%Y}^*)[\widetilde{Y}_{n-s+1}]$, and thus
$P_{V_s}$ and $\partial
P_{V_s}/\partial \widetilde{Y}_{n-s+1}$ are relatively prime in
$k(\bfs\Lambda,\widetilde{\bfs Y}^*) [\widetilde{Y}_{n-s+1}]$,
and the discriminant
\begin{equation}\label{eq:def_disc_chow}
\rho_s:=\mbox{Res}_{\widetilde{Y}_{n-s+1}}(P_{V_s},\partial P_{V_s}
/\partial \widetilde{Y}_{n-s+1})\in k[\bfs\Lambda,\widetilde{\bfs
Y}^*]\end{equation}
of $P_{V_s}$ with respect to $\widetilde{Y}_{n-s+1}$ is nonzero. It
satisfies the degree estimates
\begin{itemize}
  \item $\deg_{\widetilde{\bfs Y}^*}\rho_s\le
(2\delta_s-1)\delta_s$,
  \item $\deg_{\bfs\Lambda_i,\widetilde{Y}_i}\rho_s=
(2\delta_s-1)\delta_s$ for $1\le i\le n-s$, \item
$\deg_{\bfs\Lambda_{n-s+1}}\rho_s\le (2\delta_s-1)\delta_s$.
\end{itemize}

Let $A_s^2\in k[\bfs\Lambda]$ be a nonzero coefficient of a monomial
of $\rho_s$, considering $\rho_s$ as an element of
$k[\bfs\Lambda][\widetilde{Y}_1\klk \widetilde{Y}_{n-s}]$, and let
$A_s:=A_s^1A_s^2$, where $A_s^1\in k[\bfs\Lambda^*]$ is the
polynomial of the statement of Lemma \ref{lemma: preproc: integral
extension}. Observe that $\deg A_s\le 2(n-s+1)\delta_s^2$. Let
$\bfs\lambda\in\A^{(n-s+1)n}$ satisfy the condition
$A_s(\bfs\lambda)\not=0$, and let $\bfs \lambda^*\in\A^{(n-s)n}$ be
the matrix consisting on the first $n-s$ rows of $\bfs\lambda$.
Denote $\bfs Y:=\bfs\lambda\cdot\bfs X$ and $\bfs
Y^*:=\bfs\lambda^*\cdot \bfs X$. By Lemma \ref{lemma: preproc:
integral extension} it is clear that the ring homomorphism
$R_s\to\overline{k}[V_s]$ is integral.

First we show that condition \eqref{item: preproc: condition primitive
elem} holds. If $P_{V_s}^*$ and $\rho_s^*$ are the polynomials obtained
from $P_{V_s}$ and $\rho_s$ by evaluating $\bfs\Lambda^*$ at
$\bfs\lambda^*$, then the definition of $A_s^1$ implies that
$\rho_s^*$ is a nonzero element of $k[\bfs\Lambda_{n-s+1},\bfs Y^*]$
which equals the discriminant of $P_{V_s}^*(\bfs\Lambda_{n-s+1},\bfs
Y^*,\widetilde{Y}_{n-s+1})$ with respect to $\widetilde{Y}_{n-s+1}$.

Let $\bfs \xi:=(\xi_1\klk \xi_n)$ be the vector of coordinate
functions of $V_s$ defined by $\bfs X$, let $\zeta_i:=\bfs
\lambda_i\cdot\bfs\xi$ for $1\le i\le n-s$ and
$\widehat{Y}_{n-s+1}:=\bfs \Lambda_{n-s+1}\cdot\bfs\xi$. From the
definition of $P_{V_s}$ we conclude that the identity
\begin{align}
0=&P_{V_s}^*(\bfs\Lambda_{n-s+1},\zeta_1\klk\!\zeta_{n-s},\widehat{Y}_{n-s+1})
%\nonumber\\[1ex]
\label{eq: def identity chow form}
=P_{V_s}^*(\bfs\Lambda_{n-s+1},\zeta_1,\ldots,\zeta_{n-s},\sum_{k=1}^n
\Lambda_{n-s+1,k}\,\xi_k)
\end{align}
holds in $\ck[\bfs\Lambda_{n-s+1}]\otimes_{\ck} \ck[V_s]$.
Following, e.g., \cite{AlBeRoWo96} or \cite{Rouillier97}, taking the
partial derivative with respect to the variable $\Lambda_{n-s+1,k}$
at the first and the last sides of \eqref{eq: def identity chow form} we deduce the
following identity in $\ck[\bfs\Lambda_{n-s+1}]\otimes_{\ck}
\ck[V_s]$ for $1\le k\le n$:
\begin{equation}
\label{eq: Kronecker trick chow form} \frac{\partial
P_{V_s}^*}{\partial
\widetilde{Y}_{n-s+1}}(\bfs\Lambda_{n-s+1},\zeta_1 \klk
\zeta_{n-s},\widehat{Y}_{n-s+1})\xi_k+\frac{\partial
P_{V_s}^*}{\partial \Lambda_{n-s+1,k}}
(\bfs\Lambda_{n-s+1},\zeta_1\klk \zeta_{n-s},\widehat{Y}_{n-s+1})=0.
\end{equation}
Since $\rho_s^*$ is the discriminant of $P_{V_s}^*$ with respect to
$\widetilde{Y}_{n-s+1}$, it can be written as a linear combination
of $P_{V_s}^*$ and $\partial P_{V_s}^*/\partial
\widetilde{Y}_{n-s+1}$. Therefore, \eqref{eq: def identity chow
form} and \eqref{eq: Kronecker trick chow form} imply
\begin{equation}
\label{eq: Kronecker trick chow form with discrim}
\rho_s^*(\bfs\Lambda_{n-s+1},\zeta_1\klk \zeta_{n-s})\,\xi_k+ P_k
(\bfs\Lambda_{n-s+1},\zeta_1\klk
\zeta_{n-s},\widehat{Y}_{n-s+1})=0,\end{equation}
where $P_k$ is an element of $k[\bfs\Lambda_{n-s+1},\om Z
{n-s+1}]$ for $1\le k\le n$. Substituting $\lambda_{n-s+1,k}$ for
$\Lambda_{n-s+1,k}$ for $1\le k\le n$ in \eqref{eq: Kronecker trick
chow form with discrim}, we conclude that the powers of the coordinate
function of $\ck[V_s]$ defined by $Y_{n-s+1}$ generate
$K_s\otimes_{\ck[Y_1\klk Y_{n-s}]}\ck[V_s]$ as a
$K_s$--vector space. In particular, its minimal polynomial
over $K_s$ has degree equal to $\dim_{K_s}(K_s\otimes_{R_s}\ck[V_s])$.
In other words, $Y_{n-s+1}$
induces a primitive element of the
$K_s$--algebra extension $K_s\hookrightarrow
K_s\otimes_{R_s}\ck[V_s]$.
%Condition \eqref{item: preproc: condition finite morphism} implies
%that $\ck[V_s]$ is a finite $R_s:=\ck[Y_1\klk Y_{n-s}]$--module and
%hence $\ck(Y_1\klk Y_{n-s})\otimes_{\ck}\ck[V_s]$ is a
%finite--dimensional $K_s:=\ck(Y_1\klk Y_{n-s})$--vector space.
%Furthermore, the dimension of $K_s\otimes_{\ck}\ck[V_s]$ as
%$K_s$--vector space equals the rank of $\ck[V_s]$ as $R_s$--module.
%On the other hand, since $R_s$ is integrally closed, the minimal
%dependence equation of any element $f\in\ck[V_s]$ over $K_s$ equals
%the minimal integral dependence equation of $f$ over $R_s$ (see,
%e.g., \cite[Lemma II.2.15]{Kunz85}). As $Y_{n-s+1}$ induces a
%primitive element of the $\ck$--algebra extension
%$K_s\hookrightarrow K_s \otimes_{\ck} \ck[V_s]$, we see that
%$Y_{n-s+1}$ also induces a primitive element of the $\ck$--algebra
%extension $R_s\hookrightarrow\ck[V_s]$.
This proves condition \eqref{item: preproc: condition primitive
elem}.

Now we show that condition \eqref{item: preproc: condition finite
morphism} holds. For this purpose, it remains to prove that the
mapping $\pi_s:V_s\to \A^{n-s}$ is dominant. As
$\rho_s^*(\lambda_{n-s+1},Y_1,\ldots,Y_{n-s})$ is nonzero, there
exists $\bfs y^*\in\A^{n-s}$ with $\rho_s^*(\lambda_{n-s+1},\bfs
y)\not=0$. Let $y_{n-s+1}$ be any root in $\overline{k}$ of
$P_{V_s}(\bfs\lambda,\bfs y^*,T)$ and $\bfs y:=(\bfs
y^*,y_{n-s+1})$. Finally, in view of \eqref{eq: Kronecker trick chow
form with discrim}, let $x_k\in\overline{k}$ be defined as
$$x_k:=\frac{P_k
(\bfs\lambda_{n-s+1},\bfs y)}{\rho_s(\bfs\lambda,\bfs y)}$$
for $1\le k\le n$ and $\bfs x:=(x_1,\ldots,x_n)$. It is easy to see
that $\bfs x\in V_s$ and $\pi_s(\bfs x)=\bfs y^*$ (see, e.g.,
\cite[Proposition 6.3]{CaMa06} for details).
This finishes the proof of the proposition.
\end{proof}

From the proof of Proposition \ref{prop: preproc: Noether pos + prim
elem}, one deduces that if $Y_1,\ldots,Y_{n-s}$ are linear forms
as in the statement, then
$$\dim_{K_s}(K_s\otimes_{R_s}\ck[V_s])=\delta_s.$$
Furthermore, it follows readily
that the polynomial $A_s$ has degree at
most $2\delta_s^2$ in $\bfs\Lambda_i$ for $1\le i\le n-s+1$.

Now let $Y_1\klk Y_{n-s+1}$ be linear forms as in the statement
of the proposition. From identity \eqref{eq: Kronecker trick chow form with
discrim}, we conclude that there exist polynomials $Q_1\klk Q_n\in
k[Y_1\klk Y_{n-s+1}]$, each of degree at most $\delta_s-1$, such that
\begin{equation}
\label{eq: Kronecker trick chow form with discrim specialized}
\rho_s(Y_1\klk Y_{n-s})X_k\equiv Q_k (Y_1\klk
Y_{n-s+1})\end{equation}
in $V_s$ for $1\le k\le n$, where $\rho_s:=\mbox{Res}_{
Y_{n-s+1}}(P_{V_s}(\bfs\lambda,Y_1\klk Y_{n-s+1}),\frac{\partial
P_{V_s}}{\partial Y_{n-s+1}}(\bfs\lambda,Y_1\klk Y_{n-s+1}))$, and
$P_{V_s}(\bfs\lambda,Y_1\klk Y_{n-s+1})\in k[Y_1\klk Y_{n-s+1}]$ is
the minimal polynomial of $Y_{n-s+1}$ in the ring extension
$R_s\hookrightarrow \overline{k}[V_s]$.

Similarly, from identity \eqref{eq:
Kronecker trick chow form}, we deduce the existence of polynomials
$\widehat{Q}_1\klk \widehat{Q}_n\in k[Y_1\klk Y_{n-s+1}]$, each of degree
at most $\delta_s-1$, such that
\begin{equation}
 \label{eq: Kronecker trick chow form with partial derivative specialized}
\frac{\partial P_{V_s}}{\partial Y_{n-s+1}}(\bfs\lambda,Y_1\klk
Y_{n-s+1})X_k\equiv \widehat{Q}_k (Y_1\klk Y_{n-s+1})\end{equation}
in $V_s$ for $1\le k\le n$.

From now on, we work with linear forms $Y_1,\ldots,
Y_{n-s+1}$ that satisfy the conditions of Proposition
\ref{prop: preproc: Noether pos + prim elem}. To this end,
we consider the localization defined by the condition
$A_s\not=0$, that is, instead of considering the
morphism $\Phi$ of the proof of Lemma \ref{lemma: preproc: integral extension},
we consider the following morphism:
\begin{align}
\Phi_s: (\A^{(n-s+1)n}\times V_s)\cap\{A_s\not=0\} &\to
(\A^{(n-s+1)n}\times\A^{n-s})\cap\{A_s\not=0\}\label{eq: definition Phi_s}
\\(\bfs\lambda,\bfs x)&\mapsto\big(\bfs\lambda,
Y_1(\bfs x)\klk Y_{n-s}(\bfs x)\big).\notag
\end{align}
We have the following result.
\begin{lemma}\label{lemma: Phi_s is finite}
$\Phi_s$ is a finite morphism.
\end{lemma}
\begin{proof}
First we show that $\Phi_s$ is surjective. Let $(\bfs\lambda,\bfs
y)\in(\A^{(n-s+1)n}\times\A^{n-s})\cap\{A_s\not=0\}$. Since
$A_s(\bfs\lambda)\not=0$, Proposition \ref{prop: preproc: Noether
pos + prim elem} shows that the map $\pi_s:V_s\to\A^{n-s}$
defined by $\bfs Y^*:=\bfs \lambda^*\cdot \bfs X$ is a finite
morphism. In particular, there exists $\bfs x\in V_s$ with
$\pi_s(\bfs x)=\bfs\lambda^*\cdot \bfs x=\bfs y$. This shows that
$\Phi_s(\bfs\lambda,\bfs x)=(\bfs\lambda,\bfs\lambda^*\cdot\bfs x)=
(\bfs\lambda, \bfs y)$.

Now let $\bfs w_1\klk \bfs w_n\in\A^n$ be $\ck$--linearly
independent elements and let $\ell_1:=\bfs w_1\cdot \bfs X\klk
\ell_n:=\bfs w_n\cdot \bfs X$. By definition, the polynomial
$P_{V_s}\in \ck[\bfs\Lambda,\bfs\Lambda\cdot\bfs X]$ of the proof
of Lemma \ref{lemma: preproc: integral extension} vanishes on
$\A^{(n-s+1)n}\times V_s$. Considering $P_{V_s}$ as an element of
$\ck[\bfs\Lambda][\bfs\Lambda\cdot \bfs X]=\ck[\bfs\Lambda][\widetilde{\bfs
Y}]$, it is of degree $\delta_s$. Further, by construction, the
(nonzero) coefficient $A_s^1 \in\ck[\bfs\Lambda]$ of the monomial
$\widetilde{Y}_{n-s+1}^{\delta_s}$ in $P_{V_s}$ is a divisor of
$A_s$.

As $A_s$ is a unit of the localization
$\ck[\bfs\Lambda,\bfs\Lambda^*\cdot\bfs X]_{A_s}$, the equality
$P_{V_s}(\bfs \Lambda^*,\bfs w_j,\bfs\Lambda^*\cdot\bfs X,\ell_j)=0$
is an integral dependence equation for the coordinate function
induced by $\ell_j$ in the ring extension
$\ck[\bfs\Lambda,\bfs\Lambda^*\cdot\bfs
X]_{A_s}\hookrightarrow\ck[\bfs\Lambda][V_s]_{A_s}$ for $1\le j\le
s$. Since $\ck[\ell_1\klk\ell_n]=\ck[\xon]$, we deduce the assertion
of the lemma.
\end{proof}
%
%------------------------------------------------------------------------
%------------------------------------------------------------------------
%------------------------------------------------------------------------
%
\subsection{Lifting fibers not meeting a discriminant}
\label{subsec: lifting fibers}
Assume we are given linearly independent linear forms $Y_1,\dots,
Y_{n}\in k[\bfs X]$ defining variables in Noether position with
respect to $V_s$, as in  Proposition \ref{prop: preproc: Noether pos
+ prim elem}. Let $\pi_s: V_s \rightarrow \mathbb{A}^{n-s}$ be the
finite morphism defined by $Y_1,\dots, Y_{n-s}$, and let $J_s\subset
k[\bfs X]$ denote the ideal
$J_s:=((F_1,\ldots,F_s):G^\infty,Y_1-p_1,\dots, Y_{n-s}-p_{n-s})$.
We say that a point $\bfs p^s:=(p_1,\ldots,p_{n-s})\in k^{n-s}$ is a
\emph{lifting point} of $\pi_s$ with respect to the system
$F_1=0,\dots, F_s=0, G\neq 0$ if $J_s$ is radical. We call the
zero--dimensional variety $V_{\bfs p^s}:=\pi_s^{-1}(\bfs p^s)$ the
\emph{lifting fiber} of $\bfs p^s$
\footnote{Let $\pi_s^{-1}(\bfs p^s)\subset\{G\not=0\}$. A straightforward computation 
shows that $\pi_s^{-1}(\bfs p^s)$ is unramified if and only if 
$\det(\partial F_i^{\bfs Y} (\bfs x)/\partial Y_{n-s+j})_{1\le i\le s,1\le j\le s}\not=0$ 
for all $\bfs x \in\pi_s^{-1}(\bfs p^s)$. By \cite[Theorem 18.15]{Eisenbud95},
this condition is equivalent to the ideal $(F_1,\ldots,F_s,Y_1-p_1,\dots, Y_{n-s}-p_{n-s})$ being
a radical ideal.}.

The notion of lifting fiber in this framework was first introduced
in \cite{GiHaHeMoMoPa97}. It was formalized in \cite{HeKrPuSaWa00},
where it was shown how a Kronecker representation of a lifting fiber
of an equidimensional variety can be used to tackle certain
fundamental algorithmic problems (see also \cite{GiLeSa01},
\cite{Schost03}, \cite{BoMaWaWa04}, \cite{PaSa04}, and
\cite{JeMaSoWa09} for extensions, refinements and algorithmic
developments related to lifting fibers). The notion is also
prominent in numerical algebraic geometry, where it is known as a
{\em witness set} (see, e.g., \cite{SoWa05}; see also
\cite{SoVeWa08} for a dictionary between lifting fibers and witness
sets).

In this section, we show that through a generic choice of
coordinates, one obtains a lifting point $\bfs p^s\in k^{n-s}$ such
that the corresponding lifting fiber $V_{\bfs p^s}$ does not
intersect the hypersurface $\{G=0\}$ for $1\le s \le r$.
Additionally, we ensure that the lifting fiber $V_{\bfs
p^{s+1}}:=\pi_{s+1}^{-1}(\bfs p^{s+1})$, for $0\le s \le r-1$, has
the following property: for any point $\bfs q\in V_{\bfs p^{s+1}}$,
the morphism $\pi_s$ is unramified at $\pi_s(\bfs q)$. This ensures
that no point of $V_{\bfs p^{s+1}}$ lies in the locus of $V_s$ which
cannot be parametrized by the Kronecker representation of $V_s$
under consideration, thereby leading to significant algorithmic
simplifications. To this end, we must control those points $\bfs
p\in\A^{n-s}$ whose fibers $V_{\bfs p}$, under any admissible change
of coordinates, either intersect the hypersurface $\{G=0\}$ or fail
to be lifting points. Both conditions can be captured by the
vanishing of a certain hypersurface $\{H=0\}\subset
(\A^{(n-s+1)n}\times V_s)\cap\{A_s\not=0\}$.

The approach developed in this section builds upon and extends the
methods introduced in \cite{HeMaWa01} for the case $k=\Q$ and in
\cite{CaMa06a} for finite fields. In contrast with these earlier
works, which deal with closed varieties, we consider here the more
general setting of locally closed sets over a perfect field $k$.
This broader framework requires a refined analysis of the geometric
conditions under consideration, leading to improved degree bounds
for the hypersurfaces defining the exceptional parameter choices.
The organization of this subsection follows the structure of
\cite{HeMaWa01,CaMa06a}, but the arguments are adapted and
strengthened to accommodate these generalizations.

In the following technical result, we show that the ``defective''
points $\bfs p\in\A^{n-s}$ lie in a hypersurface of ``low'' degree.
This improvement is critical for tightening the estimates in
\cite[Lemma 1$(iii)$]{HeMaWa01} and \cite[Lemma 3.2]{CaMa06a}.
\begin{lemma}\label{lemma: preproc: degree estimate} With notations
and assumptions as above, fix $s$ with $1\le s\le r$. Let $A_s$ be
the polynomial from Proposition \ref{prop: preproc:
Noether pos + prim elem} and let $H\in k[\bfs\Lambda,\bfs X]\setminus
k[\bfs \Lambda]$ be a polynomial with $\deg_{\bfs X} H\le D_{\bfs
X}$ and $\deg_{\bfs \Lambda} H\le D_{\bfs\Lambda}$. Suppose that the
Zariski closure ${V}_s^H$ of $(\A^{(n-s+1)n}\times
V_s)\cap\{H=0,A_s\not=0\}$ is empty or of pure dimension
$(n-s+1)(n+1)-2$. Then the Zariski closure of the image of ${V}_s^H$
under the morphism $\Phi_s$ of \eqref{eq: definition Phi_s} is
contained in a hypersurface of $\A^{(n-s+1)n}\times\A^{n-s}$ defined
by a polynomial $B_s^H\in k[\bfs\Lambda,\widetilde{Y}_1\klk
\widetilde{Y}_{n-s}]$ with
$$\deg_{\widetilde{\bfs Y}^*}B_s^H\le D_{\bfs X}\delta_s,\quad
\deg_{\bfs \Lambda}B_s^H\le (D_{\bfs \Lambda}+(n-s)D_{\bfs
X})\delta_s.$$
\end{lemma}
\begin{proof}
Without loss of generality we may assume that $V_s$ is irreducible
and ${V}_s^H$ has pure
dimension $(n-s+1)(n+1)-2$. %and $H$ is an irreducible element of
%$k[\bfs\Lambda,\bfs X]\setminus k[\bfs \Lambda]$.

By Lemma \ref{lemma: Phi_s is finite}, the  morphism $\Phi_s$ is
finite. It follows that $\Phi_s(V_s^H)$ is a hypersurface of
$(\mathbb{A}^{(n-s+1)n}\times \mathbb{A}^{n-s})\cap\{A_s\not=0\}$,
and hence its Zariski closure is a hypersurface of
$\mathbb{A}^{(n-s+1)n}\times \mathbb{A}^{n-s}$. Since both $V_s^H$
and $\Phi_s$ are defined over $k$, there exists a polynomial $P\in
k[\bfs\Lambda, \widetilde{Y}_1\klk \widetilde{Y}_{n-s}]$ of minimal
degree whose vanishing defines the Zariski closure of
$\Phi_s(V_s^H)$, where $\widetilde{Y}_1\klk \widetilde{Y}_{n-s}$ are
new indeterminates. In particular,
$P(\bfs\Lambda,\bfs\Lambda^*\cdot\bfs X)$ vanishes on $V_s^H$, and
therefore, it belongs to the radical of the ideal
$(F_1,\ldots,F_s,H)$ in $k[\bfs\Lambda,\bfs X]_{A_s}$, and {\em a
fortiori}, to the radical of $(F_1,\ldots,F_s,H):A_s^\infty$ in
$k[\bfs\Lambda,\bfs X]$ (see, e.g., \cite[Chapter IV, \S
8]{ZaSa58}). The minimality of $P$ implies that $A_s$ and
$P$ are coprime in $k[\bfs\Lambda,\bfs X]$.

Next, let $\mathcal{K}$ be an algebraic closure of $k(\bfs\Lambda)$,
and let $W_s\subset \A^n_{\mathcal{K}}$ be the affine variety obtained from
$V_s$ by extension of coefficients, namely, $W_s$ is defined by the
ideal $I(V_s)\otimes_k \mathcal{K}$. Let
$\Phi_{s,\mathcal{K}}:W_s\to\A^{n-s}_\mathcal{K}$ be the corresponding
morphism, which remains finite.

Under our assumptions,
${W_s}\cap \{H=0\}$ has pure dimension $n-s-1$. Hence,
the image $\Phi_{s,\mathcal{K}}({W_s}\cap \{H=0\})$
is a hypersurface of $\A^{n-s}(\mathcal{K})$.
Observe that $W_s$, $H$ and $\Phi_{s,\mathcal{K}}$ are
defined over $k(\bfs\Lambda)$, which implies that
$\Phi_{s,\mathcal{K}}({W_s}\cap \{H=0\})$ is also defined over
$k(\bfs \Lambda)$, therefore by a polynomial
$B_s^H\in k(\bfs\Lambda)[\widetilde{Y}_1\klk \widetilde{Y}_{n-s}]$
of minimal degree. We may assume that
$B_s^H$ belongs to $k[\bfs\Lambda][\widetilde{Y}_1\klk
\widetilde{Y}_{n-s}]$ and it is primitive as a polynomial of
$k[\bfs\Lambda][\widetilde{Y}_1\klk\widetilde{Y}_{n-s}]$.

By the Hilbert Nullstellensatz, $B_s^H(\bfs\Lambda, \bfs
\Lambda^*\cdot\bfs{X})$ lies in the radical of $(F_1,\dots, F_s,
H)$ in $k(\bfs\Lambda) [\bfs{X}]$, and $B_s^H$ is of minimal degree
in $k(\bfs\Lambda)[\widetilde{Y}_1\klk \widetilde{Y}_{n-s}]$
with this property. %Cleaning denominators we conclude that  there
%exists $D(\bfs\Lambda)\in k[\bfs\Lambda]\setminus \{0\}$ of minimal
%degree such that $D(\bfs\Lambda)B_s^H(\bfs\Lambda,
%\bfs\Lambda^*\cdot{\bfs X})$
Observe that the ideal defined by $(F_1,\dots, F_s, H)$ in
$k[\bfs\Lambda]_{A_s}[\bfs{X}]$ is the contraction of the ideal
defined by $(F_1,\dots, F_s, H)$ in $k(\bfs\Lambda) [\bfs{X}]$. By
\cite[Chapter IV, \S 8]{ZaSa58}, the radical of the former is the
contraction of the radical of the latter. We conclude that
$B_s^H(\bfs\Lambda, \bfs\Lambda^*\cdot{\bfs X})$ lies in the radical
of the ideal $(F_1,\ldots,F_s,H)$ in $k[\bfs\Lambda,\bfs X]_{A_s}$,
and thus in the radical of the ideal $(F_1,\dots, F_s,
H):A_s^\infty$ in $k[\bfs\Lambda, \bfs{X}]$.

On the other hand, considering the polynomial $P$ above as a
polynomial in $k(\bfs\Lambda)[\bfs X]$, we have that
$P(\bfs\Lambda,\bfs\Lambda^*\cdot\bfs X)$ vanishes on ${W_s}\cap
\{H=0\}$, and thus $P(\bfs\Lambda,\widetilde{Y}_1\klk
\widetilde{Y}_{n-s})$ vanishes on $\Phi_{s,\mathcal{K}}({W_s}\cap
\{H=0\})$. As a consequence, the definition of $B_s^H(\bfs\Lambda,
\widetilde{Y}_1\klk \widetilde{Y}_{n-s})$ implies that
$B_s^H(\bfs\Lambda, \widetilde{Y}_1\klk \widetilde{Y}_{n-s})$
divides $P(\bfs\Lambda,\widetilde{Y}_1\klk \widetilde{Y}_{n-s})$ in
$k(\bfs\Lambda)[\widetilde{Y}_1\klk \widetilde{Y}_{n-s}]$. In
addition, the fact that $B_s^H(\bfs
\Lambda,\bfs\Lambda^*\cdot\bfs{X})$ lies in the radical ideal of
$(F_1,\dots, F_s, H):A_s^\infty$ in $k[\bfs\Lambda,\bfs X]$,
together with the definition of $P$, imply that
$P(\bfs\Lambda,\widetilde{Y}_1\klk \widetilde{Y}_{n-s})$ divides
$A_s(\bfs \Lambda)B_s^H(\bfs\Lambda, \widetilde{Y}_1\klk
\widetilde{Y}_{n-s})$, and thus $B_s^H(\bfs\Lambda,
\widetilde{Y}_1\klk \widetilde{Y}_{n-s})$, in
$k[\bfs\Lambda][\widetilde{Y}_1\klk \widetilde{Y}_{n-s}]$.
Furthermore, since $B_s^H$ is primitive as a polynomial of
$k[\bfs\Lambda][\widetilde{Y}_1\klk \widetilde{Y}_{n-s}]$, we
conclude that $P$ is primitive as well. It follows that both $P$ and
$B_s^H$ have the same degree in $\widetilde{Y}_1\klk
\widetilde{Y}_{n-s}$, and being both primitive as polynomials in
$k[\bfs\Lambda][\widetilde{Y}_1\klk \widetilde{Y}_{n-s}]$, they must
differ by a nonzero constant in $k$.

To estimate
the degrees of this polynomial, we apply \cite[Theorem
3.1]{DaKrSo13}. According to this result, there
exist a polynomial $B_s^H$ with the bounds:
\begin{align*}
\deg_{\widetilde{\bfs Y}^*}B_s^H&\le \deg (W_s\cap \{H=0\})\le D_{\bfs X}\delta_s,\\
\deg_{\bfs \Lambda}B_s^H&\le h(W_s\cap \{H=0\})+
\deg (W_s\cap \{H=0\})(n-s),\end{align*}
where $h$ is the height as defined in \cite[Definition
2.2]{DaKrSo13}. Since $h(W_s)=0$ by definition,
by \cite[Corollary 2.20]{DaKrSo13} we have
$$h(W_s\cap \{H=0\})\le D_{\bfs X}\left(h(W_s) +\deg W_s\frac{D_{\bfs \Lambda}}{D_{\bfs X}}\right)=
D_{\bfs \Lambda}\deg W_s.$$
It follows that
$$\deg_{\bfs \Lambda}B_s^H\le (D_{\bfs \Lambda}+(n-s)D_{\bfs
X})\delta_s,$$
which yields the desired estimate.
\end{proof}

Now we apply Lemma \ref{lemma: preproc: degree estimate}
to control the set of points which fail to satisfy the desired conditions.
To ensure that $\bfs p\in\A^{n-s}$
is a lifting point of $V_s$ and that $\bfs p^*\in\A^{n-s-1}$
is a lifting point of $V_{s+1}$, we introduce the following
polynomials $D_s,D_{s+1}\in k[\bfs\Lambda,\bfs X]$:
$$D_s:=\det\begin{pmatrix}
  \Lambda_{1,1} & \dots & \Lambda_{1,n} \\
  \vdots &  & \vdots \\
  \Lambda_{n-s,1} & \dots & \Lambda_{n-s,n} \\
  \frac{\partial F_1}{\partial X_1} & \dots &
  \frac{\partial F_1}{\partial X_n} \\
  \vdots &  & \vdots \\
    \frac{\partial F_s}{\partial X_1} & \dots &
    \frac{\partial F_s}{\partial X_n}
\end{pmatrix},\quad
D_{s+1}:=\det\begin{pmatrix}
  \Lambda_{1,1} & \dots & \Lambda_{1,n} \\
  \vdots &  & \vdots \\
  \Lambda_{n-s-1,1} & \dots & \Lambda_{n-s-1,n} \\
  \frac{\partial F_1}{\partial X_1} & \dots &
  \frac{\partial F_1}{\partial X_n} \\
  \vdots &  & \vdots \\
    \frac{\partial F_{s+1}}{\partial X_1} & \dots &
    \frac{\partial F_{s+1}}{\partial X_n}
\end{pmatrix}.$$
We show that the polynomials $G$, $D_s$ and $D_{s+1}$ meet the requirements
of Lemma \ref{lemma: preproc: degree estimate}.
\begin{lemma}\label{lemma: preproc: D_s cuts well}
The following assertions hold:
\begin{enumerate}
  \item The Zariski closure of $(\A^{(n-s+1)n} \times
V_s)\cap\{G=0,A_s\not=0\}$ is empty or of pure dimension $(n-s+1)(n+1)-2$.
  \item The Zariski closure of $(\A^{(n-s+1)n} \times
V_s)\cap\{D_s=0,A_s\not=0\}$ is empty or of pure dimension
$(n-s+1)(n+1)-2$.
  \item The Zariski closure of
$(\A^{(n-s)n}\times V_{s+1})\cap\{D_{s+1}=0,A_{s+1}\not=0\}$ is
empty or of pure dimension $(n-s)(n+1)-2$.
\end{enumerate}
\end{lemma}
\begin{proof}
We begin with the first assertion. Let
$V_s=\mathcal{C}_1\cup\cdots\cup\mathcal{C}_N$ be the decomposition
of $V_s$ into irreducible components. By definition, each component
$\mathcal{C}_j$ has dimension $n-s$ and it is not contained in
$\{G=0\}$. Therefore, $\mathcal{C}_j\cap\{G=0\}$ is either empty or has pure
dimension $n-s-1$. It follows that
$$(\A^{(n-s+1)n} \times V_s)\cap\{G=0\}=
\bigcup\limits_{i=1}^{\,\,N}
\A^{(n-s+1)n}\times(\mathcal{C}_i\cap\{G=0\})$$
is either empty or has pure dimension $(n-s+1)n+n-s-1=(n-s+1)(n+1)-2$.
Since this variety has a cylindrical structure, the subset $\{A_s\not=0\}$
is Zariski dense open within it, which proves the first assertion.

For the second assertion, observe that
$$\A^{(n-s+1)n}\times V_s=\bigcup\limits_{i=1}^{\,\,N}
\A^{(n-s+1)n}\times\mathcal{C}_i$$
is a decomposition into irreducible components. Take any component
$\A^{(n-s+1)n}\times\mathcal{C}$, and let
$\bfs x\in\mathcal{C}$ be a nonsingular point of $V_s$ with $G(\bfs x)\not=0$.
Since $F_1,\ldots,F_s$ define a radical ideal in $k[X_1,\ldots,X_n]_G$,
the gradients $\nabla F_1(\bfs x),\ldots,\nabla F_s(\bfs x)$ are
linearly independent. Thus,
$D_s(\bfs\Lambda,\bfs x)\not=0$, and there exists
$\bfs\lambda\in\A^{(n-s+1)n}$ with $D_s(\bfs\lambda,\bfs x)\not=0$.
This shows that there exists a point $(\bfs\lambda,\bfs
x)\in\A^{(n-s+1)(n+1)}\times \mathcal{C}$ not belonging to the
hypersurface $\{D_s=0\}$. On the other hand, $D_s(\bfs 0,\bfs x)=0$,
where $\bfs 0$ represents the zero matrix of $\A^{(n-s+1)n}$. This
implies that $(\A^{(n-s+1)n}\times V_s)\cap\{D_s=0\}$ is of pure
dimension $(n-s+1)(n+1)-2$, and hence the Zariski closure of
$(\A^{(n-s+1)n}\times V_s) \cap\{D_s=0,A_s\not=0\}$ is either empty
or has pure dimension $(n-s+1)(n+1)-2$.

The third assertion follows similarly by applying the same
reasoning to $V_{s+1}$, $A_{s+1}$ and $D_{s+1}$.
\end{proof}

Now consider the following morphisms:
\begin{align*}
\Phi_s: (\A^{(n-s+1)n}\times V_s)\cap\{A_s\not=0\} &\to
(\A^{(n-s+1)n}\times\A^{n-s})\cap\{A_s\not=0\}
\\(\bfs\lambda,\bfs x)&\mapsto\big(\bfs\lambda,
Y_1(\bfs x)\klk Y_{n-s}(\bfs x)\big),\\[2ex]
\Phi_{s+1}:(\A^{(n-s)n} \times V_{s+1})\cap\{A_{s+1}\not=0\}&\to
(\A^{(n-s)n}\times\A^{n-s-1})\cap\{A_{s+1}\not=0\}\\
(\bfs\lambda^*,\bfs x)&\mapsto\big(\bfs\lambda^*, Y_1(\bfs
x),\dots,Y_{n-s-1}(\bfs x)\big).
\end{align*}
As shown in Lemma \ref{lemma: Phi_s is finite}, the morphism $\Phi_s$
is finite. The same argument applies to show that
$\Phi_{s+1}$ is also finite.

In view of Lemma \ref{lemma: preproc: D_s cuts well},
we may assume without loss of generality that
$(\A^{(n-s+1)n}\times V_s)\cap\{D_s=0,A_s \not=0\}$ has pure
dimension $(n-s+1)(n+1)-2$ (or pure codimension 1
in $(\A^{(n-s+1)n}\times V_s)\cap\{A_s\not=0\}$).
It follows that $\Phi_s(\{D_s=0\})$ is a
hypersurface of $(\A^{(n-s+1)n}\times\A^{n-s})\cap\{A_s\not=0\}$,
defined by the polynomial $B_s^D$, that is,
$$\Phi_s(\{D_s=0,A_s\not=0\})=\{B_s^D=0, A_s\not=0\}.$$
This hypersurface consists of all the pairs
$(\bfs\lambda,\bfs p)\in(\A^{(n-s+1)n}\times\A^{n-s})\cap\{A_s\not=0\}$
such that $\bfs p$ fails to be a lifting point of $V_s$ under the
change of coordinates $\bfs Y:=\bfs\lambda\cdot \bfs X$.

Let $\rho_s\in k[\bfs\Lambda,\widetilde{Y}_1\klk \widetilde{Y}_{n-s}]$ be the
(nonzero) discriminant of $V_s$, as
defined in (\ref{eq:def_disc_chow}). Now we address the more delicate
condition: for any
point $\bfs p\in V_{\bfs p^*}$, the morphism $\pi_s$ must be
unramified at $\pi_s(\bfs p)$. Considering both ${B}_s^D$ and
$\rho_s$ as elements of $k[\bfs\Lambda,\bfs\Lambda^*\cdot\bfs X]$,
we have the following result.
\begin{lemma}\label{lemma: preproc: zariski closure rho_s B_s}
The Zariski closure of $(\A^{(n-s+1)n} \times
V_{s+1})\cap\{(\rho_s{B}_s^D)(\bfs\Lambda,\bfs\Lambda^*\cdot\bfs
X)=0,A_{s+1}\not=0\}$ is empty or has pure dimension
$(n-s+1)(n+1)-3$.
\end{lemma}
\begin{proof}
Since $\A^{(n-s+1)n}\times V_{s+1}$ has a cylindrical structure,
no irreducible component is contained
in $\{A_s=0\}$, so the open set $\{A_s\not=0\}$ intersects
each component densely. Suppose by contradiction that there
exists an irreducible component $\mathcal{D}$ of
$\A^{(n-s+1)n}\times V_{s+1}$ such that %contained in
%$\Phi_s^{-1}(\{\rho_sB_s^D=0\})$. Then
%
$$\mathcal{D}\cap\{A_s\not=0\}\subset\Phi_s^{-1}(\{\rho_sB_s^D=0\}).
$$
Then
$$\Phi_s(\mathcal{D}\cap\{A_s\not=0\})\subset \Phi_s\circ
\Phi_s^{-1}(\{\rho_sB_s^D=0\})\subset\{\rho_s
B_s^D=0\} \cap\{A_s\not=0\}.$$
%
%We claim that the condition $\Phi_s(\mathcal{D}\cap\{A_s\not=0\}) \subset
%\{\rho_sB_s^D=0\} \cap\{A_s\not=0\}$ leads to a contradiction.
Observe that
$\mathcal{D}$ can be expressed as $\mathcal{D}=\A^{(n-s+1)n}\times
\mathcal{D}_0$, where $\mathcal{D}_0$ is an irreducible component of
$V_{s+1}$. Let $\bfs x\in\mathcal{D}_0$ be a nonsingular point of
$V_{s+1}$ with $G(\bfs x)\not=0$, which is also a nonsingular point of $V_s$. Hence, for a
generic choice of $\bfs\lambda\in\A^{(n-s+1)n}$, we have $A_s(\bfs\lambda)\not=0$, the fiber
$W_s:=V_s\cap \{\bfs \lambda^*\cdot\bfs X=\bfs \lambda^*\cdot\bfs
x\}$ is unramified (see, e.g., \cite[\S 5A]{Mumford95}), and the
linear form $\bfs \lambda_{n-s+1}\cdot\bfs X$ separates the points
of $W_s$. This shows that any point $\bfs y\in V_s\cap
\{\bfs\lambda^*\cdot\bfs X=\bfs \lambda^*\cdot\bfs x\}$ satisfies the
conditions $D_s(\bfs\lambda,\bfs y)\not=0$ and $\rho_s(\bfs
\lambda,\bfs y)\not=0$. We conclude that the point $(\bfs
\lambda,\bfs \lambda^*\cdot\bfs x)$ belongs to the set
$$\Phi_s(\mathcal{D} \cap\{A_s\not=0\})\setminus(\{\rho_sB_s^D=0\} \cap\{A_s\not=0\}),$$
contradicting the inclusion
$\Phi_s(\mathcal{D} \cap\{A_s\not=0\})\subset \{\rho_s B_s^D=0\} \cap\{A_s\not=0\}$.

Hence, $\mathcal{D}\cap\{\rho_sB_s^D=0\}$ is empty or has pure
dimension $(n-s+1)(n+1)-3$ for any irreducible component
$\mathcal{D}$ of $\A^{(n-s+1)n}\times V_{s+1}$, which readily proves
the lemma.
\end{proof}

Now we are ready to prove the main result of this section, namely,
we show that there exists a polynomial condition of
``low'' degree which guarantees that, for any choice of the coefficients
defining the linear forms $Y_1,\ldots,Y_{n-s+1}$, and the points
$\bfs p^s:=(p_1,\ldots,p_{n-s})$ and $\bfs p^{s+1}:=(p_1,\ldots,p_{n-s-1})$,
satisfying this condition, all the required properties involving
$Y_1,\ldots,Y_{n-s+1}$, $\bfs p^s$ and $\bfs p^{s+1}$ are fulfilled.
\begin{theorem}
\label{th: preproc: all conditions} With notations as in Proposition
\ref{prop: preproc: Noether pos + prim elem}, fix $s$ with $1\le
s<r$. There exists a nonzero polynomial $B_s\in\ck[\bfs \Lambda,
\widetilde{Y}_1\klk\widetilde{Y}_{n-s}]$, of degree at most
$2(n-s+2)nd\delta_s\delta_{s+1}\max\{\delta_s,\delta_{s+1}\}$, such
that for any $(\bfs \lambda,\bfs p^s)\in\A^{(n-s+1)n}\times \A^{n-s}$
with $B_s(\bfs \lambda,\bfs p^s)\not=0$, the following conditions are
satisfied: if $\bfs Y:=(Y_1\klk Y_{n-s+1}):=\bfs \lambda\cdot \bfs
X$, then:
\begin{enumerate}
\item \label{item: preproc: conditions pi_s}The morphism $\pi_s:V_s\to\A^{n-s}$ defined by
$Y_1\klk Y_{n-s}$ is finite, the point $\bfs p^s:=(p_1\klk
p_{n-s})\in\A^{n-s}$ is a lifting point of $\pi_s$, $\pi_s^{-1}(\bfs
p^s)\subset\{G\not=0\}$, and $Y_{n-s+1}$ is a primitive element of
$\pi_s^{-1}(\bfs p^s)$.
\item \label{item: preproc: conditions pi_s+1}
Let $\bfs p^{s+1}:=(p_1\klk p_{n-s-1})\in\A^{n-s-1}$. Then the morphism
$\pi_{s+1}:V_{s+1} \to \A^{n-s-1}$ defined by $Y_1\klk Y_{n-s-1}$ is
finite, the point $\bfs p^{s+1}$ is a lifting point of $\pi_{s+1}$,
$\pi_{s+1}^{-1}(\bfs p^{s+1})\subset\{G\not=0\}$, and $Y_{n-s}$ is a
primitive element of $\pi_{s+1}^{-1}(\bfs p^{s+1})$.
\item \label{item: preproc: conditions fiber of fiber}
Every point $\bfs q\in\pi_s\big(\pi_{s+1}^{-1}(\bfs p^{s+1})\big)$ is a
lifting point of $\pi_s$ and $Y_{n-s+1}$ is a primitive element of
$\pi_s^{-1}(\bfs q)$ for any $\bfs q\in\pi_s\big(\pi_{s+1}^{-1}(\bfs
p^{s+1})\big)$.
\item \label{item: preproc: condition fiber with G=0} No point $\bfs
q\in \pi_s\big(\pi_s^{-1}(\{\bfs p^{s+1}\}\times\A^1)\cap\{G=0\}\big)$
belongs to $\pi_s\big(\pi_{s+1}^{-1}(\bfs p^{s+1})\big)$.
\end{enumerate}
\end{theorem}
\begin{proof}
In view of Lemma \ref{lemma: preproc: D_s cuts well},
we may assume without loss of generality that
all the sets appearing in its statement have
pure codimension 1 in $(\A^{(n-s+1)n}\times V_s)\cap\{A_s \not=0\}$ and
$(\A^{(n-s)n}\times V_{s+1})\cap\{A_{s+1} \not=0\}$, respectively.
It follows that
there exist polynomials $B_s^D,B_s^G\in
k[\bfs\Lambda,\widetilde{Y}_1\klk \widetilde{Y}_{n-s}]$ and
$B_{s+1}^D,B_{s+1}^G\in k[\bfs\Lambda^*,
\widetilde{Y}_1\klk\widetilde{Y}_{n-s-1}]$ of minimal degree
such that the following identities hold:
\begin{align}
\Phi_s(\{D_s=0\}) & =\{B_s^D=0,A_s \not=0\},\label{eq: identity BsD}\\
\Phi_s(\{G=0\}) & =\{G=0,A_s \not=0\},\label{eq: identityBsG}\\
\Phi_{s+1}(\{D_{s+1}=0\}) & =\{B_{s+1}^D=0,A_{s+1}\not=0\},\notag\\
\Phi_{s+1}(\{G=0\}) & =\{G=0,A_{s+1}\not=0\}.\notag
\end{align}
By Lemma \ref{lemma: preproc: degree estimate}, the following degree bounds
are satisfied:
\begin{align*}%\deg_{\widetilde{\bfs Y}^*}B_s^H&\le D_{\bfs X}\delta_s,\quad
%\deg_{\bfs \Lambda}B_s^H\le (D_{\bfs \Lambda}+(n-s)D_{\bfs
%X})\delta_s\\
 \deg_{\widetilde{\bfs Y}^*}B_s^D&\le s(d-1)\delta_s,\hskip2cm \deg_{\bfs
\Lambda} B_s^D\le (n-s)(1+s(d-1))\delta_s,\\
 \deg_{\widetilde{\bfs Y}^*} B_s^G&\le d\delta_s,\hskip3.15cm  \deg_{\bfs \Lambda} B_s^G\le (n-s) d\delta_s,\\
 \deg_{\widetilde{\bfs Y}^*} B_{s+1}^D&\le (s+1)(d-1)\delta_{s+1},\quad\ \ \deg_{\bfs \Lambda}
 B_{s+1}^D\le (n-s-1)(1+(s+1)(d-1))\delta_{s+1},\\
  \deg_{\widetilde{\bfs Y}^*} B_{s+1}^G&\le d\delta_{s+1},\hskip2.75cm  \deg_{\bfs \Lambda} B_{s+1}^G\le (n-s-1) d\delta_{s+1}.
\end{align*}
We have also the following degree bounds for the discriminants:
\begin{align*}
\deg_{\widetilde{\bfs Y}^*}\rho_s\le
\delta_s(2\delta_s-1),\qquad&\qquad
\deg_{\bfs \Lambda}\rho_s\le(n-s+1)\delta_s(2\delta_s-1),\\
\deg_{\widetilde{\bfs Y}^*}\deg\rho_{s+1}\le
\delta_{s+1}(2\delta_{s+1}-1),&\quad \deg_{\bfs
\Lambda}\rho_{s+1}\le(n-s)\delta_{s+1}(2\delta_{s+1}-1).
\end{align*}
We will show that the nonvanishing of the polynomials $A_s$ and
$A_{s+1}$ of Proposition \ref{prop: preproc: Noether pos + prim
elem} and $B_s^D$, $B_{s+1}^D$, $B_s^G$, $B_{s+1}^G$, $\rho_s$,
$\rho_{s+1}$ guarantees that conditions \eqref{item: preproc:
conditions pi_s} and \eqref{item: preproc: conditions pi_s+1} are
fulfilled.

To address condition \eqref{item: preproc: conditions fiber of
fiber}, in view of Lemma \ref{lemma: preproc: zariski closure rho_s B_s},
we may assume without loss of generality
that the set
$$(\A^{(n-s+1)n}\times
  V_{s+1})\cap\{(\rho_sB_s^D)(\bfs\Lambda,\bfs\Lambda\cdot\bfs X)=0,A_{s+1}\not=0\}$$
has pure codimension 1 in $(\A^{(n-s+1)n}\times
V_{s+1})\cap\{A_{s+1}\not=0\}$. Define the map
\begin{align*}
\Psi_s:(\A^{(n-s+1)n}\times
  V_{s+1})\cap\{A_{s+1}\not=0\}&\to
(\A^{(n-s+1)n}\times\A^{n-s-1})\cap\{A_{s+1}\not=0\},
\\
(\bfs\lambda,\bfs x)&\mapsto\big(\bfs \lambda,Y_1(\bfs x),\ldots,
Y_{n-s-1}(\bfs x)\big).
\end{align*}
Arguing as in Lemma \ref{lemma: Phi_s is finite}, one sees that
$\Psi_s$ is a finite morphism. Hence, the image
$\Psi_s(\{(\rho_sB_s^D)(\bfs\Lambda,\bfs\Lambda\cdot\bfs X)=0\})$ is
a hypersurface of $(\A^{(n-s+1)n}\times\A^{n-s-1})\cap\{A_{s+1}\not=0\}$,
defined by a minimal-degree polynomial
$\widetilde{B}_s\in k[\bfs\Lambda,\widetilde{Y}_1\klk
\widetilde{Y}_{n-s-1}]$, namely:
\begin{equation}\label{eq: identity BsTilde}
\Psi_s(\{(\rho_sB_s^D)(\bfs\Lambda,\bfs\Lambda\cdot\bfs X)=0\})=
\{\widetilde{B}_s=0,A_{s+1}=0\}.
\end{equation}
By Lemma \ref{lemma: preproc: degree estimate}, we have
$$\deg_{\widetilde{\bfs Y}^*}\widetilde{B}_s\le (s(d-1)+2\delta_s)\delta_s\delta_{s+1},
\quad \deg_{\bfs \Lambda}\widetilde{B}_s\le
2(n-s+1)(s(d-1)+2\delta_s)\delta_s\delta_{s+1}.$$
The nonvanishing of $\widetilde{B}_s$ will account for the
fulfillment of condition \eqref{item: preproc: conditions fiber of
fiber}.

Now consider condition \eqref{item: preproc: condition
fiber with G=0}. Let $(\bfs \lambda,\bfs
p^s)\in\A^{(n-s+1)n}\times\A^{n-s}$ be such that conditions \eqref{item:
preproc: conditions pi_s} and \eqref{item: preproc: conditions
pi_s+1} are satisfied. Since the morphism $\pi_s:V_s\to\A^{n-s}$ is finite, by Fact
\ref{fact: preimage finite mapping} it follows that
$\pi_s^{-1}(\{\bfs p^{s+1}\}\times \A^1)$ is of pure dimension 1.
Furthermore, it has degree at most $\delta_s$, because
$$\pi_s^{-1}(\{\bfs p^{s+1}\}\times
\A^1)=V_s\cap\{Y_1=p_1,\ldots,Y_{n-s-1}=p_{n-s-1}\}$$ is a linear
section of $V_s$. Observe that $$\pi_s^{-1}(\bfs
p^s)=\pi_s^{-1}(\{\bfs p^{s+1}\}\times \A^1)\cap\{Y_{n-s}=p_{n-s}\}$$ is a
zero--dimensional linear section of $\pi_s^{-1}(\{\bfs p^{s+1}\}\times
\A^1)$. The fact that $\bfs p^s\in k^{n-s}$ is a lifting point of
$\pi_s$ implies $\#\pi_s^{-1}(\bfs p^s)=\delta_s$. In particular,
$\pi_s^{-1}(\bfs p^s)$ intersects each irreducible component of
$\pi_s^{-1}(\{\bfs p^{s+1}\}\times \A^1)$. Indeed, let
$\mathcal{C}_1,\ldots,\mathcal{C}_t$ be the irreducible components
of $\pi_s^{-1}(\{\bfs p^{s+1}\}\times \A^1)$. We have
\begin{align*}
\delta_s=\#\pi_s^{-1}(\bfs p^s)\le \sum_{i=1}^t\#(\mathcal{C}_i\cap\pi_s^{-1}(\bfs p^s))& =\sum_{i=1}^t\#(\mathcal{C}_i\cap\{Y_{n-s}=p_{n-s}\})\\
& \le\sum_{i=1}^t\deg \mathcal{C}_i=\deg \pi_s^{-1}(\{\bfs p^{s+1}\}\times \A^1)\le \delta_s. \end{align*}
This proves that all the inequalities above are equalities, which shows our claim.

As $\pi_s^{-1}(\bfs
p^s)\subset\{G\not=0\}$, we see that $\pi_s^{-1}(\{\bfs p^{s+1}\}\times
\A^1)\cap\{G=0\}$ is of dimension at most 0. In particular, a
generic linear form $Y_{n-s}$ separates the points of
$\pi_s^{-1}(\{\bfs p^{s+1}\}\times \A^1)\cap\{G=0\}$ from those of
$\pi_{s+1}^{-1}(\bfs p^{s+1})\subset\{G\not=0\}$. It follows that the polynomials
$B_s^G\in k[\bfs\Lambda,\widetilde{Y}_1\klk \widetilde{Y}_{n-s}]$
and $P_{V_{s+1}}\in k[\bfs\Lambda,\widetilde{Y}_1\klk
\widetilde{Y}_{n-s}]$ (of the proof of Lemma \ref{lemma: preproc: integral extension})
share no common factors in
$k[\bfs\Lambda,\widetilde{Y}_1\klk
\widetilde{Y}_{n-s-1}][\widetilde{Y}_{n-s}]$, and the
resultant
$$\widehat{B}_s:=\mathrm{Res}_{\widetilde{Y}_{n-s}}(B_s^G,P_{V_{s+1}})\in
k[\bfs\Lambda][\widetilde{Y}_1\klk \widetilde{Y}_{n-s-1}]$$
is nonzero, with the following degree bounds:
$$\deg_{\widetilde{\bfs Y}}\widehat{B}_s\le  2d\delta_s\delta_{s+1},\qquad\deg_{\bfs\Lambda}
\widehat{B}_s\le 2(n-s)d\delta_s\delta_{s+1}.$$

Let $$B_s:=A_sA_{s+1}\rho_s\rho_{s+1}
B_s^DB_{s+1}^DB_s^GB_{s+1}^G\widetilde{B}_s\widehat{B}_s.$$
From the bounds above we deduce that $$\deg B_s\le
2(n-s+2)nd\delta_s\delta_{s+1}\max\{\delta_s,\delta_{s+1}\}.$$
Now suppose $(\bfs\lambda,\bfs p^s)\in\A^{(n-s+1)(n+1)} \times\A^{n-s}$ is
such that $B_s(\bfs\lambda,\bfs p^s)\not=0$. We claim that
$(\bfs\lambda,\bfs p^s)$ satisfies all the conditions of the theorem.

Let $\bfs \lambda^*$ be the submatrix of the first $n-s$ rows of $\bfs\lambda$,
and let $\bfs p^{s+1}$ denote the first $n-s-1$
coordinates of $\bfs p^s$. Since $A_s(\bfs\lambda)A_{s+1}
(\bfs\lambda^*) \not=0$, from Proposition \ref{prop: preproc:
Noether pos + prim elem} we conclude that the morphisms
$\pi_s:V_s\to\A^{n-s}$ and $\pi_{s+1}:V_{s+1}\to\A^{n-s-1}$, defined
by the linear forms $(Y_1\klk Y_{n-s}):=\bfs\lambda\cdot\bfs X$ and
$(Y_1\klk Y_{n-s-1}):=\bfs\lambda^*\cdot\bfs X$, are
finite. Since $A_s(\bfs \lambda)\not=0$, from
\eqref{eq: identity BsD} it follows that the condition
$B_s^D(\bfs \lambda,\bfs p^s)\not=0$ implies $D_s(\bfs \lambda,\bfs
x)\not=0$ for any $\bfs x\in\pi_s^{-1}(\bfs p^s)$. Therefore, $\bfs p^s$
is a lifting point of $\pi_s$. Similarly, the point
$\bfs p^{s+1}$ is a lifting point of $\pi_{s+1}$. The conditions
$\rho_s(\bfs\lambda,\bfs p^s)\not=0$ and $\rho_{s+1}(\bfs
\lambda^*,\bfs p^{s+1})\not=0$ ensure that $Y_{n-s+1}$ and $Y_{n-s}$ are
primitive elements of $\pi_s^{-1}(\bfs p^s)$ and $\pi_{s+1}^{-1}(\bfs
p^{s+1})$, respectively. Finally, as $B_s^G(\bfs\lambda,\bfs
p^s)\not=0$, by \eqref{eq: identityBsG} it follows that
$\pi_s^{-1}(\bfs p^s)\subset\{G\not=0\}$. Similarly, the
condition $B_{s+1}^G(\bfs\lambda^*,\bfs
p^{s+1})\not=0$ implies $\pi_{s+1}^{-1}(\bfs p^{s+1})\subset\{G\not=0\}$.
As a consequence,
conditions \eqref{item: preproc: conditions pi_s} and \eqref{item:
preproc: conditions pi_s+1} of the statement of the theorem are
satisfied.

On the other hand, from \eqref{eq: identity BsTilde} it follows that the conditions
$\widetilde{B}_s(\bfs\lambda,\bfs p^{s+1})\not=0$ and
$A_{s+1}(\bfs\lambda^*)\not=0$ imply
$(\rho_sB_s^D)\big(\bfs\lambda,\bfs p^{s+1},Y_{n-s}(\bfs x)\big)\not=0$
for any $\bfs x\in\pi_{s+1}^{-1}(\bfs p^{s+1})$. Therefore, as
$A_s(\bfs \lambda)\not=0$, from \eqref{eq: identity BsD}
we deduce that $D_s(\bfs\lambda,\bfs q)\not=0$,
and $\rho_s(\bfs\lambda,\pi_s(\bfs q))\not=0$, for any $\bfs
q\in\pi_s^{-1}(\bfs p^{s+1},Y_{n-s}(\bfs x))$ with $\bfs
x\in\pi_{s+1}^{-1}(\bfs p^{s+1})$. This shows that condition
\eqref{item: preproc: conditions fiber of fiber} holds.

Finally, as
conditions \eqref{item: preproc: conditions pi_s} and \eqref{item:
preproc: conditions pi_s+1} hold, the remarks above show that
$\pi_s^{-1}(\{\bfs p^{s+1}\}\times \A^1)\cap\{G=0\}$ has dimension at
most 0. Furthermore, we have $(\bfs\lambda,\bfs q)\in
\Phi_s(\{G=0\})$ for any $\bfs q\in\pi_s\big(\pi_s^{-1}(\{\bfs
p^{s+1}\}\times \A^1)\cap\{G=0\}\big)$, which implies
$B_s^G(\bfs\lambda,\bfs q)=0$ by \eqref{eq: identityBsG}. Each such point $\bfs q$ can be
expressed as $\bfs q=(\bfs p^{s+1},\xi)$ with $\xi\in\ck$ and
$B_s^G(\bfs\lambda,\bfs p^{s+1},\xi)=0$. On the other hand,
$P_{V_{s+1}}\big(\bfs\lambda,\pi_s(\bfs
q)\big)=P_{V_{s+1}}\big(\bfs\lambda,\bfs p^{s+1},Y_{n-s}(\bfs q)\big)=0$
for each point $\bfs q\in\pi_s(\pi_{s+1}^{-1}(\bfs p^{s+1}))$. Since
$\widehat{B}_s(\bfs\lambda,\bfs p^{s+1})\not=0$, the polynomials
$B_s^G(\bfs\lambda,\bfs p^{s+1},Y_{n-s})$ and
$P_{V_{s+1}}(\bfs\lambda,\bfs p^{s+1},Y_{n-s})$ have no common roots in
$\ck$, which shows that condition \eqref{item: preproc: condition
fiber with G=0} is satisfied.
\end{proof}

We conclude this section with a result that summarizes the conclusions
of Theorem \ref{th: preproc: all conditions} for
$1\le s<r$, namely, it asserts the existence
of a low-degree condition that must be satisfied by the
coefficients of linear forms $Y_1,\ldots,Y_n$, and the
coordinates of a point $\bfs p:=(p_1,\ldots,p_{n-1})$, in order
for all the requirements of Theorem \ref{th: preproc: all conditions}
to be satisfied for $1\le s< r$.
\begin{corollary}\label{coro: preproc: all conditions for all s}
Let $\bfs \Lambda:=(\Lambda_{ij})_{1\le
i,j\le n}$ be a matrix of indeterminates over
$\ck$, and let $\widetilde{\bfs Y}:=\bfs \Lambda \cdot \bfs X$. There
exists a nonzero polynomial $B\in\ck[\bfs \Lambda,
\widetilde{Y}_1\klk\widetilde{Y}_{n-1}]$ of degree at most
$2n^2rd\delta^3$ such
that for any $(\bfs \lambda,\bfs p)\in\A^{n^2}\times \A^{n-1}$
with $B(\bfs \lambda,\bfs p)\not=0$, if $\bfs Y:=(Y_1\klk Y_{n-s+1}):=\bfs \lambda\cdot \bfs
X$, then conditions
\eqref{item: preproc: conditions pi_s}, \eqref{item: preproc: conditions pi_s+1},
\eqref{item: preproc: conditions fiber of fiber},
\eqref{item: preproc: condition fiber with G=0} of Theorem \ref{th: preproc: all conditions} are satisfied
for $1\le s<r$.
\end{corollary}
\begin{proof}
Let
$$
B:=\prod_{s=1}^{r-1}B_s,$$
where $B_s$ is the polynomial from
Theorem \ref{th: preproc: all conditions} for $1\le s<r$. We have
$$\deg B=\sum_{s=1}^{r-1}\deg B_s\le
\sum_{s=1}^{r-1}2(n-s+2)nd\delta_s\delta_{s+1}\max\{\delta_s
,\delta_{s+1}\}\le 2n^2rd\delta^3.$$
By Theorem \ref{th: preproc: all conditions} we immediately
deduce the corollary.
\end{proof}
%
%----------------------------------------------------------------------
%----------------------------------------------------------------------
%----------------------------------------------------------------------
%----------------------------------------------------------------------
%----------------------------------------------------------------------
%----------------------------------------------------------------------
%----------------------------------------------------------------------
%----------------------------------------------------------------------
%
\section{Lifting curves}\label{sec: lifting fibers}
Let the notations and assumptions be as in Section \ref{section:
locally closed sets}. The algorithm we describe proceeds in $r$
steps. For each $1\le s< r$, the $s$th step takes as input a Kronecker
representation of a lifting fiber of the variety $V_s$, and outputs
a Kronecker representation of a lifting fiber of $V_{s+1}$.

Fix $s$ with $1\le s< r$. From this point on, we assume that the variables
$X_1\klk X_n$ satisfy all the conditions stated in Theorem \ref{th:
preproc: all conditions}. In particular, this implies that the
morphisms
$$\begin{array}{rclrcl}
\pi_s:V_s&\to&\A^{n-s}, & \quad\pi_{s+1}:V_{s+1}&\to&\A^{n-s-1},\\
\bfs x&\mapsto&(x_1\klk x_{n-s}) & \bfs x&\mapsto&(x_1\klk
x_{n-s-1})
\end{array}$$
are finite and generically unramified, and satisfy
$$
\deg\pi_s=\delta_s:=\deg V_s,\qquad \deg\pi_{s+1}=\delta_{s+1}:=\deg
V_{s+1}.
$$

Denote
\begin{align*}
R_s:=k[X_1\klk X_{n-s}],\qquad&A_s:=k[V_s],\qquad\qquad\\
K_s:=k(X_1\klk X_{n-s}),\qquad&B_s:=K_s\otimes A_s.
\end{align*}
According to Proposition \ref{prop: preproc: Noether pos + prim
elem} and its proof, we have:
\begin{itemize}
  \item $A_s$ is an integral ring extension of $R_s$,
  \item $B_s$ is a
$K_s$--vector space of dimension $\deg\pi_s=\delta_s$.
\end{itemize}
Moreover:
\begin{itemize}
  \item $k[V_s]$ is a Cohen-Macaulay ring, and
a free $R_s$-module of finite rank (see \cite{GiHeSa93}).
\end{itemize}
For any $F\in k[X_1\klk X_n]$, let $\eta_F:A_s\to A_s$ denote the
$R_s$-linear endomorphism of multiplication by $F$. Denote by $m_F\in K_s[T]$ the
minimal polynomial of $\eta_F$. Our assumptions imply that $m_F$ is a
monic polynomial in $R_s[T]$.

Let $\bfs p^s:=(p_1\klk p_{n-s})\in k^{n-s}$ be a point satisfying the
hypotheses of Theorem \ref{th: preproc: all conditions}. In
particular, the following conditions hold:
\begin{itemize}
  \item[$({\sf A}_1)$] The fiber $\pi_s^{-1}(\bfs p^s)$ is unramified,
  \item[$({\sf A}_2)$] $\pi_s^{-1}(\bfs p^s)\subset\{G\not=0\}$,
  \item[$({\sf A}_3)$] $X_{n-s+1}$ separates the points of $\pi_s^{-1}(\bfs
  p^s)$.
  \end{itemize}

Suppose we are given univariate polynomials $m_s,v_{n-s+2}^s\klk
v_n^s\in k[T]$, with $\deg m_s=\delta_s$ and $\deg v_i^s<\delta_s$
for $n-s+2\le i\le n$, defining a Kronecker representation of
$\pi_s^{-1}(\bfs p^s)$ with $X_{n-s+1}$ as primitive element, namely,
  \begin{align*}
  \pi_s^{-1}(\bfs p^s)=\{\bfs p^s\}\times\{(x_{n-s+1}\klk
  x_n)\in\A^s:\ &m_s(x_{n-s+1})=0,\\
  &x_{n-s+i}=v_{n-s+i}^s(x_{n-s+1})\ \,(2\le i\le s)\}.
  \end{align*}
By Theorem \ref{th: preproc: all conditions}, the point
$\bfs p^{s+1}:=(p_1\klk p_{n-s-1})$ satisfies the following properties:
\begin{itemize}
  \item[$({\sf B}_1)$] The fiber $\pi_{s+1}^{-1}(\bfs p^{s+1})$ is unramified,
  \item[$({\sf B}_2)$] $\pi_{s+1}^{-1}(\bfs p^{s+1})\subset\{G\not=0\}$,
  \item[$({\sf B}_3)$] $X_{n-s}$ separates the points of $\pi_{s+1}^{-1}(\bfs
  p^*)$,
  \item[$({\sf B}_4)$] For every $\bfs q\in\pi_s\big(\pi_{s+1}^{-1}(\bfs p^{s+1})\big)$,
  the fiber $\pi_s^{-1}(\bfs q)$ is unramified and $X_{n-s+1}$
  separates its points.
   \end{itemize}

The $s$th recursive step of the algorithm consists of
two main phases: lifting and elimination. In this section, we focus
on the lifting phase, which produces a Kronecker representation of the
variety
\begin{equation}\label{eq: def lifting curve}
C_s:=\pi_s^{-1}\left(\{X_1=p_1\klk X_{n-s-1}=p_{n-s-1}\}\right).
\end{equation}
As will become clear, $C_s$ is of pure dimension 1
---that is, it is a curve. We refer to $C_s$ as the
{\em lifting curve} associated
with $\pi_s$ and the point $\bfs p^s$.
%
%----------------------------------------------------------------------
%----------------------------------------------------------------------
%----------------------------------------------------------------------
%----------------------------------------------------------------------
%
\subsection{Properties of the $s$th lifting curve}\label{section: lifting curve properties}
In this section, we establish several properties of the variety
$C_s$ defined in \eqref{eq: def lifting curve}. We begin by showing
that $C_s$ is a curve of degree $\delta_s$, and that the projection
onto its $(n-s)$th coordinate defines a suitable direction of
projection. %Proposition~\ref{prop: properties M_s} below extends
%Lemma~4.1 of~\cite{CaMa06a} from the case of closed sets over finite
%fields to the more general setting of locally closed sets over
%arbitrary perfect fields. While both results establish the existence
%and degree of the corresponding lifting curve,
%Proposition~\ref{prop: properties M_s} further provides explicit
%properties of the associated minimal polynomial, which are essential
%for the subsequent algorithmic developments.
%
\begin{proposition}\label{prop: properties M_s}
The variety $C_s$ from \eqref{eq: def lifting curve} has pure
dimension 1 and degree $\delta_s$, and the morphism
$\pi_s^*:C_s\to\A^1$, given by $\pi_s^*(\bfs x):=x_{n-s}$, is finite
of degree $\delta_s$. Furthermore, let $M_s\in k[X_{n-s},T]$ denote
the minimal polynomial of $X_{n-s+1}$ in the $k(X_{n-s})$-algebra
extension $k(X_{n-s})\hookrightarrow
k(X_{n-s})\otimes_{k[X_{n-s}]}k[C_s]$. Then $M_s$ is a monic
polynomial in $k[X_{n-s}][T]$ of degree $\deg_TM_s=\deg
M_s=\delta_s$. Moreover, $M_s$ is a separable element of
$k(X_{n-s})[T]$, and $M_s$ and its derivative $\frac{\partial
M_s}{\partial T}$ are relatively prime in $k(X_{n-s})[T]$. In
particular, $\rho_s:=\mathrm{disc}_TM_s\in
k[X_{n-s}]\setminus\{0\}$.
\end{proposition}
\begin{proof}
Since $\pi_s$ is a finite morphism and $\{X_1=p_1\klk
X_{n-s-1}=p_{n-s-1}\}\subset\A^{n-s}$ is an irreducible variety of
dimension 1, the first claim about $C_s$ follows from Fact \ref{fact: preimage
finite mapping}. As $C_s=V_s\cap\{X_1=p_1\klk
X_{n-s-1}=p_{n-s-1}\}$, the B\'ezout inequality \eqref{eq:
Bezout} implies $\deg C_s\le \deg V_s=\delta_s$. Moreover, since
the fiber $\pi_s^{-1}(\bfs p^s)$ is unramified and $\deg\pi_s=\delta_s$, we have
$\#\pi_s^{-1}(\bfs p^s)=\delta_s$. Hence,
\begin{align*}
\delta_s=\deg \pi_s^{-1}(\bfs p^s)&=\deg
C_s\cap\{X_{n-s}=p_{n-s}\}\le\deg C_s,
\end{align*}
where the last inequality also follows from the B\'ezout
inequality \eqref{eq: Bezout}. This proves the second assertion on $C_s$.

Since $\pi_s:V_s\to\A^{n-s}$ is a finite morphism, it follows that
$\pi_s^*$ is finite as well. In particular, we obtain an
integral ring extension $k[X_{n-s}]\hookrightarrow k[C_s]$ and
a finite-dimensional $k(X_{n-s})$--algebra extension $k(X_{n-s})\hookrightarrow
k(X_{n-s})\otimes_{k[X_{n-s}]}k[C_s]$. This extension arises from
specializing the $K_s$--algebra extension $K_s\hookrightarrow
K_s\otimes_{R_s}k[V_s]$, and since the latter is a
$K_s$--vector space of dimension $\delta_s$, it follows that
$k(X_{n-s})\otimes_{k[X_{n-s}]}k[C_s]$ must be a $k(X_{n-s})$-vector space of
dimension at most $\delta_s$.

Furthermore, as the fiber
$(\pi_s^*)^{-1}(p_{n-s})=\pi_s^{-1}(\bfs p^s)$ has cardinality
$\delta_s$, and $X_{n-s+1}$ separates the
points in $\pi_s^{-1}(\bfs p^s)$, the coordinate $X_{n-s+1}$ takes $\delta_s$
distinct values on this fiber.
Since $M_s$ vanishes on $\pi_s^{-1}(\bfs p^s)\subset C_s$, it follows that
the polynomial $M_s(p_{n-s},T)$ has $\delta_s$ distinct
roots in $\overline{k}$. Hence, the inequalities $$\delta_s\le \deg
M_s(p_{n-s},T)\le \deg_TM_s\le \delta_s$$ must all be equalities. This shows that
$k(X_{n-s})\otimes_{k[X_{n-s}]}k[C_s]$ is a
$k(X_{n-s})$--vector space of dimension $\delta_s$.

The equality $\deg M_s=\delta_s$ is a consequence of, e.g., \cite[Proposition
1]{SaSo96}. In particular, we have $\deg M_s(p_{n-s},T)=\deg_TM_s=\deg M_s$.
Since $M_s(p_{n-s},T)$ is square-free, and thus separable,
we conclude that $M_s$ is square-free as well, and separable as a polynomial
in $k(X_{n-s})[T]$. Moreover, since $M_s(p_{n-s},T)$
and $\frac{\partial M_s}{\partial T}(p_{n-s},T)$ are
coprime in $k[T]$, it follows that $M_s(X_{n-s},T)$ and
$\frac{\partial M_s}{\partial T}(X_{n-s},T)$ are coprime in
$k(X_{n-s})[T]$. This concludes the proof.
\end{proof}

Next, we discuss the remaining polynomials forming the
Kronecker representation of $C_s$ that we aim to compute.
By specializing \eqref{eq: Kronecker trick
chow form with partial derivative specialized} we deduce
the existence of polynomials $W_{n-s+2}^s\klk W_n^s\in k[X_{n-s},T]$, of
degree at most $\delta_s-1$, such that
\begin{equation}\label{eq: parametriz lifting curve with partial deriv}
\frac{\partial M_s}{\partial T}(X_{n-s},X_{n-s+1})\cdot
  X_{n-s+i}\equiv W_{n-s+i}^s(X_{n-s},X_{n-s+1})
\end{equation}
in the coordinate ring of $C_s$, for $2\le i\le s$. The polynomials $M_s,W_{n-s+2}^s\klk
W_n^s\in k[X_{n-s},T]$ thus provide a Kronecker representation of $C_s$, with
$X_{n-s+1}$ as primitive element. In the following, we describe
an algorithm for computing such a representation.

It is also useful to note that, by specializing \eqref{eq: Kronecker trick chow form with discrim
specialized}, we obtain polynomials $V_{n-s+2}^s\klk
V_n^s\in k[X_{n-s},T]$, of degree at most $\delta_s-1$, such that
\begin{equation}\label{eq: parametriz lifting curve with discrim}
\rho_s(X_{n-s})\cdot
  X_{n-s+i}\equiv V_{n-s+i}^s(X_{n-s},X_{n-s+1})
\end{equation}
in $C_s$ for $2\le i\le s$, where
$\rho_s:=\mathrm{Res}_T(M_s,\partial M_s/\partial T)$ with a slight
abuse of notation.

We now explicitly describe the birational morphism between the
lifting curve $C_s$ and the hypersurface $\{M_s=0\}\subset\A^2$
associated to the Kronecker representation above (compare with
\cite[Lemma 3.5]{CaMa06a} for closed sets over finite fields).
\begin{lemma}\label{lemma: lifting curve birational to M_s=0}
Let
$$\pi_{s,s+1}:C_s\to\{M_s(X_{n-s},X_{n-s+1})=0\}\subset\A^2,\quad\pi_{s,s+1}(\bfs
x):=(x_{n-s},x_{n-s+1}).$$
Then $\pi_{s,s+1}$ is a birational
mapping, which restricts to an isomorphism between the Zariski open dense
subsets $C_s\setminus\{\rho_s(X_{n-s})=0\}$ of $C_s$ and
$\{M_s=0\}\setminus \{\rho_s(X_{n-s})=0\}$ of $\{M_s=0\}$. Moreover,
the inverse mapping $$\psi:\{M_s=0\}\setminus
\{\rho_s(X_{n-s})=0\}\to C_s\setminus\{\rho_s(X_{n-s})=0\}$$ is
given by:
$$
\psi(x_{n-s},x_{n-s+1}):=\left(\bfs
p^{s+1},x_{n-s},x_{n-s+1},\frac{W_{n-s+2}^s(x_{n-s},x_{n-s+1})}{\frac{\partial
M_s}{\partial T}(x_{n-s},x_{n-s+1})} \klk
\frac{W_n^s(x_{n-s},x_{n-s+1})}{\frac{\partial M_s}{\partial
T}(x_{n-s},x_{n-s+1})}\right).
$$
\end{lemma}
\begin{proof}
The morphism $\pi_{s,s+1}$ is well-defined since
$M_s(X_{n-s},X_{n-s+1})$ vanishes identically on $C_s$, and it maps
$C_s\setminus\{\rho_s(X_{n-s})=0\}$ to $\{M_s=0\}\setminus
\{\rho_s(X_{n-s})=0\}$. The identity \eqref{eq: parametriz lifting curve
with discrim} ensures that $\pi_{s,s+1}$ is injective on
$C_s\setminus\{\rho_s(X_{n-s})=0\}$.

To show surjectivity, it suffices to verify that $\psi$ is
well-defined.
Let $(x_{n-s},x_{n-s+1})\in \A^2$ be such that
$M_s(x_{n-s},x_{n-s+1})=0$ and $\rho_s(x_{n-s})\not=0$.
Then $\frac{\partial M_s}{\partial
T}(x_{n-s},x_{n-s+1})\not=0$. Let $H\in I(C_s)$,
and set
$H_1:=(\frac{\partial M_s}{\partial T})^{\deg H}H$. Then we may
write $$H_1=H_2\Big(X_1\klk X_{n-s+1},\frac{\partial
M_s}{\partial T}X_{n-s+2}\klk \frac{\partial M_s}{\partial
T}X_n,\frac{\partial M_s}{\partial T}\Big)$$ for some
$H_2\in k[X_1\klk X_n,X_{n+1}]$. Since $H_1\in I(C_s)$,
identity \eqref{eq: parametriz lifting curve with partial deriv}
implies that $$H_2\Big(\bfs
p^{s+1},X_{n-s},X_{n-s+1},W_{n-s+2}^s(X_{n-s},X_{n-s+1})\klk
W_n^s(X_{n-s},X_{n-s+1}),\frac{\partial M_s}{\partial T}(X_{n-s},X_{n-s+1})\Big)$$ vanishes on $C_s$, and hence is divisible
by $M_s(X_{n-s},X_{n-s+1})$. It follows that $H_1$ vanishes
on the image of $\psi$, and therefore so does $H$. Thus, $\psi$
maps $\{M_s=0\}\setminus \{\rho_s(X_{n-s})=0\}$ into
$C_s\setminus\{\rho_s(X_{n-s})=0\}$, and from \eqref{eq:
parametriz lifting curve with partial deriv} we see that $\psi$
is the inverse map of $\pi_{s,s+1}$ restricted to $C_s\setminus
\{\rho(X_{n-s}=0\}$. This completes the proof.
\end{proof}

We now include a simple observation that will be used
later.
\begin{remark}\label{rem: fiber meets every irred comp curve}
The fiber $\pi_s^{-1}(\bfs p^s)=C_s\cap\{X_{n-s}=p_{n-s}\}$ intersects
every irreducible component of $C_s$.
\end{remark}
\begin{proof}
Let $C_s=\cup_{i\in\mathcal{I}}\mathcal{C}_i$ be the decomposition
of $C_s$ into irreducible components. Since
$C_s\cap\{X_{n-s}=p_{n-s}\}=\pi_s^{-1}(\bfs p^s)$, the intersection
$\mathcal{C}_i\cap\{X_{n-s}=p_{n-s}\}$ has dimension at most zero.
Moreover,
$$\delta_s=\#\pi_s^{-1}(\bfs p^s)\le \sum_{i\in\mathcal{I}}
\#(\mathcal{C}_i\cap\{X_{n-s}=p_{n-s}\})\le \sum_{i\in\mathcal{I}}
\deg\mathcal{C}_i=\deg C_s=\delta_s.$$
Hence, all inequalities must be equalities, which implies that
$\mathcal{C}_i\cap\{X_{n-s}=p_{n-s}\}$ has cardinality equal to
$\deg\mathcal{C}_i$, which is positive for each $i\in\mathcal{I}$.
\end{proof}

We recall that, in addition to the conditions $({\sf B}_1)$--$({\sf B}_4)$
listed above,
the point $\bfs p^{s+1}$ is assumed to satisfy the following condition related to
the lifting curve $C_s$:
\begin{itemize}
  \item[$({\sf B}_5)$] No point $\bfs
q\in \pi_s\big(C_s\cap\{G=0\}\big)$ belongs to
$\pi_s\big(\pi_{s+1}^{-1}(\bfs p^{s+1})\big)$.
\end{itemize}

The following result, which adapts and extends \cite[Lemma
4.1]{CaMa06a} to our setting, will be essential for proving the
correctness of the procedure for computing a Kronecker
representation of the lifting curve $C_s$. Observe that the image of
the map $C_s\to\A^{s+1}$, $\bfs x\mapsto (x_{n-s}\klk x_n)$, is
isomorphic to $C_s$. For simplicity of notations, we will henceforth
use $C_s$ to denote either the curve of $\A^n$ as defined in
\eqref{eq: def lifting curve}, or its image under this projection.
The next result shows that we are in a position to apply the
Newton-Hensel process described in Section \ref{subsec: NH lifting}.
\begin{proposition}\label{prop: F_i(p*,X) generate radical ideal}
With the above notations and assumptions, the following
statements hold:
\begin{enumerate}
  \item The polynomials $F_j(\bfs p^{s+1},X_{n-s}\klk X_n)$ $(1\le j\le
s)$ form a regular sequence in the localized ring $k[X_{n-s}\klk X_n]_{G(\bfs
p^{s+1},X_{n-s}\klk X_n)}$, and they generate the localization of the ideal of
$C_s$ in this ring.
\item The morphism $\pi_s^*:C_s\to\A^1$, defined by $\pi_s^*(x_{n-s}\klk x_n):=x_{n-s}$,
is finite and generically unramified, and its fiber over $p_{n-s}$,
namely $(\pi_s^*)^{-1}(p_{n-s})$, is unramified and has cardinality
$\delta_s=\deg C_s$.
\end{enumerate}
\end{proposition}
\begin{proof}
We first prove that the polynomials $F_j(\bfs p^{s+1},X_{n-s}\ldots X_n )$ $(1\le j\le
s)$ form a regular sequence in $k[X_{n-s}\klk X_n]_{G(\bfs
p^{s+1},X_{n-s}\klk X_n)}$. To this end, by
\cite[Theorem 17.6]{Matsumura86}, it suffices to show that, for each
$1\le i\le s$,
$$Z_i:=\overline{\{F_j(\bfs p^{s+1},X_{n-s}\ldots X_n)=0:1\le j\le
i\}\setminus \{G(\bfs p^{s+1},X_{n-s}\klk X_n)=0\}}$$ has pure
dimension $s+1-i$.

Let $L_{s+1}\subset\A^n$ be the linear subvariety defined by the
equations $X_1=p_1\klk X_{n-s-1}=p_{n-s-1}$. Then
$Z_i=V_i\cap L_{s+1}=\pi_i^{-1}(\pi_i(L_{s+1}))$ for $1\le i\le s$.
Since $\pi_i:V_i\to\A^{n-i}$ is a finite morphism and
$\pi_i(L_{s+1})\subset\A^{n-i}$ is an irreducible variety of
dimension $\dim_{\A^{n-i}}\pi_i(L_{s+1})=n-i-(n-s-1)=s+1-i$, it
follows from  Fact
\ref{fact: preimage finite mapping} that $Z_i$ has
pure dimension $s+1-i$, as desired.

Now we show that the polynomials $F_j(\bfs p^{s+1},X_{n-s},\dots, X_n)$ $(1\le j\le
s)$ generate the ideal of $C_s$ in the localized ring $k[X_{n-s}\klk X_n]_{G(\bfs
p^*,X_{n-s}\klk X_n)}$. By definition, we have
$$C_s\setminus\{G=0\}=(V_s\setminus \{G=0\})\cap \{X_1=p_1\klk
X_{n-s-1}=p_{n-s-1}\}.$$
Hence, the common set of zeros in $\A^n\setminus\{G=0\}$ of the polynomials
$F_j(\bfs p^{s+1},X_{n-s},\dots, X_n)$ $(1\le j\le s)$ is precisely
$C_s\setminus\{G=0\}$. Since $\bfs p^s$ is a lifting point of $\pi_s$,
the Jacobian determinant
$$J_F(\bfs p^{s+1},X_{n-s}\klk X_n):=\det\big(\partial
F_i(\bfs p^{s+1},X_{n-s}\klk X_n)/\partial X_{n-s+j}\big)_{1\le i,j\le
s}$$
does not vanish at any point of
$C_s\setminus\{G=0\}\cap\{X_{n-s}=p_{n-s}\}=\pi_s^{-1}(\bfs p^s)$.
By Remark \ref{rem: fiber meets every irred comp
curve}, this implies that the coordinate function of
$C_s\setminus\{G=0\}$ defined by $J_F(\bfs p^{s+1},X_{n-s}\klk X_n)$
does not vanish identically on any irreducible component of
$C_s\setminus\{G=0\}$, and is therefore not a zero divisor of
$k[C_s]_G$. By \cite[Theorem 18.15]{Eisenbud95},
the ideal generated by the polynomials $F_j(\bfs p^{s+1}, X_{n-s}\klk X_n)$ $(1\le j\le
s)$ is radical in $k[X_{n-s}\klk X_n]_{G(\bfs p^{s+1},X_{n-s}\klk
X_n)}$, and thus it coincides with the ideal of $C_s$ in that ring.

We finally prove the claims about the morphism $\pi^*_s$. Since $\pi_s$
is a finite morphism of degree $\delta_s$, its restriction to $C_s$
induces a finite morphism $\pi^*_s:C_s\to\A^1$ of degree at most $\delta_s$.
Moreover, because the fiber $\pi_s^{-1}(\bfs p^s)$ is unramified and has
cardinality $\delta_s$, the fiber $(\pi_s^*)^{-1}(p_{n-s})$
is also unramified of cardinality $\delta_s$, and it follows that
$\pi_s^*$ is generically unramified and has degree exactly $\delta_s$.
\end{proof}
%
%----------------------------------------------------------------------
%----------------------------------------------------------------------
%----------------------------------------------------------------------
%----------------------------------------------------------------------
%
\subsection{A Kronecker representation of the $s$th lifting curve}
\label{subsec: NH lifting}
According to Lemma \ref{lemma: lifting curve birational to M_s=0},
there is an isomorphism between the Zariski open dense subsets
$C_s\setminus\{\rho_s(X_{n-s})=0\}$ of $C_s$ and
$\{M_s(X_{n-s},T)=0\}\setminus \{\rho_s(X_{n-s})=0\}$ of
$\{M_s(X_{n-s},T)=0\}\subset\A^2$. This allows us to
represent $C_s$ in terms of the zeros of $M_s(X_{n-s},T)$.

The zeros of $M_s$ can be represented using the well--known Hensel
lemma (see, e.g., \cite[Lecture 12]{Abhyankar90}). More precisely,
the facts that $M_s$ is a monic element of $k[X_{n-s}][T]$ of degree
$\delta_s$, and that
$$\rho_s(X_{n-s}):=\mathrm{Res}_T
\left(M_s(X_{n-s},T),\frac{\partial M_s}{\partial
T}(X_{n-s},T)\right)\not=0,$$
imply that $M_s$ admits a factorization of the form
$$M_s=\prod_{j=1}^{\delta_s}(T-\sigma_j)$$ in $\overline{k}[\![X_{n-s}
- p_{n-s}]\!][T],$
where $\overline{k}[\![X_{n-s} - p_{n-s}]\!]$ denotes the ring of
formal power series in $X_{n-s}-p_{n-s}$. We will see that
approximations of order $\delta_s$ of
$\sigma_1\klk\sigma_{\delta_s}$ are sufficient for our purposes, and
discuss how such approximations can be computed.

Let $\bfs F_s:=(F_1(\bfs p^{s+1},X_{n-s}\klk X_n) \klk F_s(\bfs
p^{s+1},X_{n-s}\klk X_n))$, and let
$$
J_s(X_{n-s}, \ldots, X_n) := \left(\begin{array}{ccc}
     \tfrac{\partial F_1}{\partial X_{n-s+1}} & \ldots & \tfrac{\partial
     F_1}{\partial X_n}\\
     \vdots &  & \vdots\\
     \tfrac{\partial F_s}{\partial X_{n-s+ 1}} &\ldots  & \tfrac{\partial F_s}{\partial X_n}
   \end{array}\right) (\bfs p^{s+1},X_{n-s}, \ldots, X_n)$$
be the Jacobian matrix of $\bfs F_s$ with respect to $X_{n-s+1},
\ldots, X_n$. Considering $\bfs F_s$ and $J_s$ as elements of
$k[X_{n-s}] [X_{n-s+1}, \ldots, X_n]$, the Newton--Hensel operator
associated to $\bfs F_s$ is the $s$--tuple $N_{\bfs F_s}$ of
rational functions of $k(X_{n-s}) (X_{n-s+1}, \ldots, X_n)$ defined
by
\begin{equation}
 N_{\bfs F_s} (X_{n-s+1}, \ldots, X_n)^t : = \left(\begin{array}{c}
     X_{n-s+1}\\
     \vdots\\
     X_n
   \end{array}\right) - \big(J_s^{- 1} \bfs F_s^t\big)(X_{n-s+1}, \ldots, X_n),
\end{equation}
where $^t$ denotes transposition.

As a consequence of Proposition
\ref{prop: F_i(p*,X) generate radical ideal}, we have the following
result, showing that, starting from an arbitrary point of
$\pi_s^{-1}(\bfs p^s)$, the $j$--th iteration of $N_{\bfs F_s}$ is
well-defined. In the following, $(X_{n-s} -
p_{n-s})$ denotes the maximal ideal of $\overline{k}[\![X_{n-s} -
p_{n-s}]\!]$ generated by $X_{n-s}-p_{n-s}$. The proof of this
result follows {\em mutatis mutandis} the proof of
Assertions (I) and (II) in \cite{HeKrPuSaWa00}.
\begin{proposition}
\label{prop: Iteration NH} Let $\bfs\xi := (\bfs p^s,\xi_{n-s+1}\klk
\xi_n) \in \pi_s^{-1}(\bfs p^s)$. For each $j\ge 0$, consider the
$s$--tuple $\bfs R^{(j,\bfs\xi)}:= (R_{n-s+1}^{(j,\bfs\xi)}\klk
R_n^{(j,\bfs\xi)})\in\overline{k}[\![X_{n-s} - p_{n-s} ]\!]^s$
defined recursively by
$$\bfs R^{(0,\bfs\xi)} := (\xi_{n-s+1}, \ldots, \xi_n), \qquad
\bfs R^{(j+1,\bfs\xi)} := N_{\bfs F_s} (\bfs R^{(j,\bfs\xi)}) .$$
Then the sequence $(\bfs R^{(j,\bfs\xi)})_{j \ge 0}$ is
well--defined, and the following properties hold for each $j\ge 0$:
\begin{enumerate}
    \item $F_i (\bfs p^{s+1},X_{n-s},\bfs R^{(j,\bfs\xi)}) \in (X_{n-s} -
    p_{n-s})^{2^j}$ for $1 \le i \le s$;
    \item $\det  J_s (X_{n-s},\bfs R^{(j,\bfs\xi)}) \notin (X_{n-s}
    - p_{n-s})$;
    \item $R_i^{(j+1,\bfs\xi)}\equiv R_i^{(j,\bfs\xi)}\mod (X_{n-s}
    - p_{n-s})^{2^j}$ for $n-s+2\le i\le n$.
  \end{enumerate}
\end{proposition}

From Proposition \ref{prop: Iteration NH}, the
Newton--Hensel theorem follows, ensuring the existence of the
parametrizations by power series for the branches of $C_s$ near
each point of a lifting fiber $\pi_s^{-1}(\bfs p^s)$.
More precisely, we have the following result, whose proof
follows {\em mutatis mutandis} that of \cite[Lemma
3]{HeKrPuSaWa00}.
\begin{theorem}
\label{th: NewtonHensel} For any point $\bfs\xi:=(\bfs
p,\xi_{n-s+1}\klk \xi_n) \in \pi_s^{-1}(\bfs p^s)$, there exists a
unique $s$--tuple of formal power series $\bfs R^{\bfs\xi}:=(R_{n-s+
1}^{\bfs\xi}\klk R_n^{\bfs\xi}) \in \overline{k}[\![X_{n-s} -
p_{n-s}]\!]^s$ satisfying:
\begin{itemize}
    \item $F_1(\bfs p^{s+1}, X_{n-s},\bfs R^{\bfs\xi})=0\klk F_s(\bfs p^{s+1}, X_{n-s},\bfs R^{\bfs\xi})=0$;
    \item $\bfs R^{\bfs\xi}(p_{n-s})=(\xi_{n-s+1}\klk \xi_n)$.
  \end{itemize}
\end{theorem}

Condition $({\sf A}_2)$ asserts that $G(\bfs\xi)\not=0$ for any
$\bfs\xi:=(\bfs p^s,\xi_{n-s+1}\klk \xi_n) \in \pi_s^{-1}(\bfs p^s)$.
Thus, by the second claim of Theorem \ref{th: NewtonHensel}, it
follows that, for all $\bfs\xi\in \pi_s^{-1}(\bfs p^s)$,
$$G(\bfs p^{s+1},X_{n-s},\bfs R^{\bfs\xi})\not=0.$$

Collecting the $\delta_s$ vectors of power series $\bfs
R^{\bfs\xi}$ yields a description of $C_s$ to compute a Kronecker
representation as in \eqref{eq:
parametriz lifting curve with partial deriv}. More precisely:
\begin{lemma}
\label{lm:TFI:FactorizMUPowerSeries} With the notations of Theorem
\ref{th: NewtonHensel}, we have the factorization
\begin{equation}
M_s(X_{n-s}, T) = \prod_{\xi \in \pi_s^{-1}(\bfs
p)}(T-R_{n-s+1}^{\bfs\xi}) \label{eq:TFI:FactorizMU}
\end{equation}
in $\overline{k}[\![X_{n-s}-p_{n-s}]\!][T]$.
\end{lemma}
\begin{proof}
Since $M_s(X_{n-s},X_{n-s+1})$  vanishes on $C_s$, by
Proposition \ref{prop: F_i(p*,X) generate radical ideal} we conclude
that it belongs to the (radical) ideal generated by the
polynomials $F_i(\bfs p^{s+1},X_{n-s}\klk X_n)$ $(1 \le i \le s)$ in
$k[X_{n-s}\klk X_n]_{G(\bfs p^{s+1},X_{n-s}\klk X_n)}$. Let $\bfs\xi \in
\pi_s^{-1}(\bfs p^s)$. Since $F_i(\bfs p^{s+1},X_{n-s}, \bfs R^{\bfs\xi})=
0$ for $1 \le i \le s$ and $G(\bfs p^{s+1},X_{n-s}, \bfs
R^{\bfs\xi})\not=0$, we obtain
$$M_s(X_{n-s},R_{n-s+1}^{\bfs\xi}) = 0.$$
Moreover, as $\bfs R^{\bfs\xi}(p_{n-s}) = (\xi_{n-s+1}, \ldots, \xi_n)$ for
all $\bfs\xi:=(\bfs p^s,\xi_{n-s+1}\klk\xi_n) \in \pi_s^{-1}(\bfs p^s)$,
and $X_{n-s+1}$ separates the points of $\pi_s^{-1}(\bfs p^s)$, it
follows that $R_{n-s+1}^{\bfs\xi}\neq R_{n-s+1}^{\bfs\xi'}$ for
$\bfs\xi \neq
\bfs\xi'$. Thus, the terms $T -R_{n-s+1}^{\bfs\xi}$ are coprime,
implying that their product
divides $M_s(X_{n-s}, T)$. Since both sides are monic polynomials of
the same degree, the lemma follows.
\end{proof}

Proposition \ref{prop: Iteration NH} shows that, starting at
$\bfs\xi \in \pi_s^{-1}(\bfs p^s)$, the $j$--th iteration of the
Newton-Hensel operator $N_{\bfs F_s}$ yields a vector of power
series $\bfs R^{(j,\bfs\xi)}\in\overline{k}[\![X_{n-s} - p_{n-s}
]\!]^s$ satisfying $\bfs R^{(j,\bfs\xi)}\equiv\bfs R^{\bfs\xi}$
modulo $(X_{n-s} - p_{n-s})^{2^j}$ for any $j\ge 0$. It follows that
$$
M_s(X_{n-s}, T) \equiv \prod_{\bfs\xi \in \pi_s^{-1}(\bfs p^s)} (T -
R_{n-s+1}^{(j,\bfs\xi)})\mod (X_{n-s} -p_{n-s})^{2^j} .$$
Now, as $\deg_{X_{n-s}} M_s \le\delta_s$, for $j_0:=\left\lceil \log
(\delta_s+1)\right\rceil$, the Taylor expansion of order $2^{j_0}$ of
$M_s$ in powers of $X_{n-s}-p_{n-s}$ yields $M_s$, which proves that
$$M_s(X_{n-s}, T) = \left( \prod_{\bfs\xi \in
\pi_s^{-1}(\bfs p^s)} (T-R_{n-s+1}^{(j_0,\bfs\xi)})\mod (X_{n-s} -
p_{n-s})^{2^{j_0}} \right).$$
This suggests a procedure for computing $M_s(X_{n-s},T)$. The
coordinates of the elements of the fiber $\pi_s^{-1}(\bfs p^s)$ can
be expressed in the splitting field of the minimal polynomial
$m_s\in k[T]$ of $X_{n-s+1}$ in $\pi_s^{-1}(\bfs p^s)$. To avoid
working in such a field extension, we start the iterations of the
Newton--Hensel operator $N_{\bfs F_s}$ with the class of $T$ in the
quotient ring $k[T]/(m_s(T))$, which represents all the roots of
$m_s$. Following an idea introduced in \cite[Algorithm 1]{GiLeSa01},
the critical point of the algorithm below is that it allows one to
incrementally compute the polynomials $M_s,W_{n-s+1} ,\ldots,W_n\in
k[X_{n-s},T]$ defining the Kronecker representation of $C_s$ under
consideration.

In the description and analysis of the algorithm, we use the
following terminology: for $t\ge 0$ and $\sigma, \tau \in
\overline{k}[\![X_{n-s} -p_{n-s}]\!]$, we say that $\tau = \sigma$
{\em with precision} $t$ if $\tau \equiv \sigma\mod(X_{n-s} -
p_{n-s})^t$. Moreover, if $Q_1, Q_2 \in \overline{k}[\![X_{n-s}
-p_{n-s}]\!][T]$ are polynomials in a single variable $T$ (of the
same degree), we say that $Q_2 = Q_1$ with precision $t$ if each
coefficient of $Q_2$ equals the corresponding coefficient of $Q_1$
with precision $t$.

\begin{algorithm}[Newton lifting]\label{algo: Newton lifting} ${}$
\begin{itemize}
\item[]{\bf Input:} The dense representation of polynomials
$m_s, v^{s}_{n-s+2}\klk v^{s}_n\in k[T]$ with $\deg m_s=\delta_s$ and
$\deg v_{n-s+i}^s<\delta_s$ for $2\le i\le s$, such that $m_s(X_{n-s+1})\equiv 0$
and $X_{n-s+i}\equiv v_{n-s+i}^s(X_{n-s+1})$ over $\pi_s^{-1}(\bfs p^s)$
for $2\le i\le s$.
\item[]{\bf Output:} The dense representation of polynomials $M_s,
W^s_{n-s+2}\klk W^s_n\in k[X_{n-s},T]$ forming the Kronecker
representation of $C_s$, with $X_{n-s+1}$ as primitive element, that
\emph{lifts} the input Kronecker representation, that is, with the
following properties:
\begin{enumerate}
  \item $\deg M_s=\deg_TM_s=\delta_s$ and $\deg W^s_{n-s+i}<\delta_s$
  for $2\le i\le s$;
  \item $M_s(X_{n-s},X_{n-s+1})\equiv 0$ on $C_s$;
  \item $\frac{\partial M_s}{\partial T}(X_{n-s},X_{n-s+1})\cdot
  X_{n-s+i}\equiv W^s_{n-s+i}(X_{n-s},X_{n-s+1})$ on $C_s$ for $2\le
  i\le s$;
  \item $M_s(p_{n-s},T)=m_s$ and
  $\big((\frac{\partial M_s}{\partial T})^{-1}\cdot W^s_{n-s+i}\big)(p_{n-s},T)=v^s_{n-s+i}
  \mod m_s$ for $2\le i \le s$.
\end{enumerate}\smallskip

\item[]{\bf 1} Set $M^0:=m_s(T)$, $V_{n-s+1}^0:=T$.
\item[]{\bf 2} Set $V_{n-s+i}^0:=v^{s}_{n-s+i}(T)$ for $2 \le i \le s$
as elements of $k[\![X_{n-s}-p_{n-s}]\!][T]$.
\item[]{\bf 3} Set $j:=0$.
\item[]{\bf 4} While $2^j < \delta_s+1$, do
  \begin{enumerate}
    \item[]{\bf 4.1} Compute $v_{n-s+1}^{j+1}\klk v_n^{j+1}$ modulo
    $M^j$ with precision $2^{j+1}$  as follows:
    \[ (v_{n-s+1}^{j+1}, \ldots, v_n^{j+1}):=N_{\bfs F_s} (V_{n-s+1}^j\klk V_n^j) . \]
    \item[]{\bf 4.2} Compute $\Delta^{j+1}:= v_{n-s+1}^{j+1}-T$.

    \item[]{\bf 4.3} For $i=n-s+1\klk n$ do
    \begin{enumerate}
      \item[]{\bf 4.3.1} Compute $\Delta_i^{j+1}:=\Delta^{j+1} \frac{\partial v_i^{j+1}}{\partial
      T}$ mod $M^j$ with precision $2^{j+1}$.

      \item[]{\bf 4.3.2} Set $V_i^{j+1}:=v_i^{j+1}- \Delta_i^{j+1}.$
    \end{enumerate}
    \item[]{\bf 4.4} Compute $\Delta_{M}^{j+1}:=\Delta^{j+1} \frac{\partial M^j}{\partial
    T}$ mod $M^j$ with precision $2^{j+1}$.

    \item[]{\bf 4.5} Set $M^{j+1}:=M^j-\Delta_{M}^{j+1}$.

    \item[]{\bf 4.6} Set $j:=j+1$
  \end{enumerate}
  \item[]{\bf 5} Return $M^{j}$ with precision $\delta_s+1$.
  %\item[]{\bf 6} For $i=n-s+1\klk n$, set $V_i:=V_i^{j+1}$ with precision $\delta_s$.
  \item[]{\bf 6} For $i=n-s+1\klk n$, return $\frac{\partial M^j}{\partial
    T}V_i^{j} \mod M^{j}$ with precision $\delta_s+1$.
\end{itemize}
\end{algorithm}

The algorithm requires performing arithmetic operations in the ring $k[\![X_{n-s} -p_{n-s}]\!]$
with a prescribed precision $t$. Concretely, this means that we are given the dense representations
of two polynomials in $k[X_{n-s}]$, and the desired arithmetic
operation is carried out modulo $(X_{n-s}-p_{n-s})^t$. In particular,
under the assumption that the input
polynomials have degree at most $t$---as is the case in Algorithm
\ref{algo: Newton lifting}---each such operation can be performed using
$\mathcal{O}({\sf M}(t))$ arithmetic operations over $k$.

We first analyze the behavior of the main loop of Algorithm
\ref{algo: Newton lifting}.
\begin{proposition}\label{prop: main loop Newton lifting}
Fix $j\ge 0$, and assume that we are given polynomials $M^j,V_{n-s+1}^j, \ldots V_n^j\in
k[\![X_{n-s} -p_{n-s}]\!][T]$, with $\deg_TM^j=\delta_s$ and $\deg_T
V_{n-s+i}^j<\delta_s$ for $1\le i\le s$, satisfying the following
congruences with precision $2^j$:
\begin{align}
F_i (\bfs p^{s+1}, X_{n-s}, \bfs V^j) & \equiv 0  \mod M^j\quad (1 \le i
\le s),
\label{eq: CongruenceProofHL1}\\
V_{n-s+1}^j& = T, \label{eq: CongruenceProofHL2}
\end{align}
where $\bfs V^j:=(V_{n-s+1}^j, \ldots, V_n^j)$. Assume further that
$\det J_s(X_{n-s}, \bfs V^j)$ is invertible modulo $M^j$. Then the
$(j+1)$th iteration of the loop of step {\bf 4} can be correctly
executed and outputs polynomials $M^{j+1},V_{n-s+1}^{j+1}, \ldots
V_n^{j+1} \in k[\![X_{n-s} -p_{n-s}]\!][T]$, with
$\deg_TM^{j+1}=\delta_s$ and $\deg_T V_{n-s+i}^{j+1}<\delta_s$ for
$1\le i\le s$, satisfying the following congruences with precision
$2^{j+1}$:
\begin{align*}
F_i (\bfs p^{s+1}, X_{n-s},\bfs V^{j+1})  & \equiv 0  \mod M^{j+1} \quad
(1 \le i \le s),
\\
V_{n-s+1}^{j+1}& = T,
\end{align*}
where $\bfs V^{j+1}\!:=(V_{n-s+1}^{j+1}, \ldots, V_n^{j+1})$.
Moreover, $\det J_s(X_{n-s},\bfs V^{j+1} )$ is invertible modulo
$M^{j+1}$. In addition, if $G(\bfs p^{s+1}, X_{n-s}, \bfs V^j)\in
k[\![X_{n-s} -p_{n-s}]\!][T]$ is invertible modulo $M^j$, then
$G(\bfs p^{s+1}, X_{n-s}, \bfs V^{j+1})$ is invertible modulo
$M^{j+1}$.
\end{proposition}
\begin{proof}
Regarding the execution of the $(j+1)$th iteration, the only point
to check is whether the Newton-Hensel operator $N_{\bfs F_s}
(X_{n-s}, \bfs V^j)$ is well-defined. By hypothesis, $\det J_s(
X_{n-s},\bfs V^j)\mod M^j$ is invertible, and thus the
matrix $J_s(X_{n-s},\bfs V^j)$ admits an inverse modulo
$M^j$. Therefore, $N_{\bfs F_s} (X_{n-s}, \bfs V^j)$
is well--defined modulo $M^j$.

Now we analyze the output produced by the $(j+1)$th iteration. Set $\bfs
v^{j+1}:=(v_{n-s+1}^{j+1}\klk v_n^{j+1})$. We first claim that
\begin{equation}\label{eq: F_v(j+1)_=_0_mod_M_j}
\bfs F_s(X_{n-s}, \bfs v^{j+1}) \mod M^j = 0
\end{equation}
with precision $2^{j+1}$. Indeed, from (\ref{eq: CongruenceProofHL1})
we have $\bfs F_s (X_{n-s},\bfs V^j) = 0\mod M^j$ with precision $2^j$.
By the definition of $\bfs v^{j+1}$, it follows that
$$\bfs v^{j+1}-\bfs V^j = -
(J_s^{-1} \bfs F_s) (X_{n-s},\bfs V^j)
=0\mod M^j,$$
with precision
$2^j$. Since
$\deg_T(v_{n-s+i}^{j+1}-V_{n-s+i}^j)<\delta_s$ for $1\le i\le s$, we
conclude that
\begin{equation}\label{eq: Congruence w(j+1) W(j)}
\bfs v^{j+1}-\bfs V^j = 0
\end{equation}
with precision $2^j$, and thus, $(\bfs v^{j+1}-\bfs V^j)^2 = 0$ with
precision $2^{j+1}$. Applying the second-order Taylor expansion of
$\bfs F_s$ in powers of $X_{n-s+1}- V_{n-s+1}^j\klk X_n - V_n^j$, we
obtain
\begin{align*}
\bfs F_s(X_{n-s}, \bfs v^{j+1}) & \equiv  \bfs F_s(X_{n-s},\bfs V^j)+J_s(X_{n-s}, \bfs V^j)
(\bfs v^{j+1}-\bfs V^j) \mod M^j \\
\nonumber  & \equiv \bfs F_s(X_{n-s, }\bfs V^j) - J_s(X_{n-s},\bfs
V^j)(J_s^{-1} \bfs F_s) (X_{n-s},\bfs V^j)\mod M^j=\bfs 0,
\end{align*}
with precision $2^{j+1}$. This proves the claim. Note that, as
$\bfs v^{j+1}$ agrees with $\bfs V^j$ with precision $2^j$, step
{\bf 4.1} in the main loop only involves modifications of the
coefficients of $\bfs V^j$ of order $2^j$ or higher.

It follows that $\bfs F_s$ vanishes at
$\bfs v^{j+1}$ with precision $2^{j+1}$, but modulo $M^j$,
which represents the output with
precision only $2^j$. The next steps in the main loop
correct this lack of precision. To analyze these steps, recall that
we define
$$\Delta^{j+1}:=v_{n-s+1}^{j+1}- T.$$
We may consider $\Delta^{j+1}$ as the ``error'' or ``failure'' of
$\bfs v^{j+1}$ to correctly parametrize the coordinates of
$X_{n-s+1}\klk X_n$ with precision $2^{j+1}$. Observe that
$$\Delta^{j+1}=0,\quad \Delta^{j+1}_{M}=0,\quad \mathrm{and}\quad
\Delta_i^{j+1}=0\quad (n-s+1 \le i \le n),$$
all with precision ${2^j}$. In fact, $\Delta^{j+1}=0$ with precision
${2^j}$ because $v_{n-s+i}^{j+1}=V_{n-s+i}^j$ with precision $2^j$
and (\ref{eq: CongruenceProofHL2}) holds, while the remaining
assertions follow immediately from this one. The fact that
$\Delta^{j+1}=0$ with precision $2^j$ indicates that $\bfs v^{j+1}$
correctly parametrizes $X_{n-s+1}\klk X_n$ with precision $2^j$.

\begin{claim} Fix $j\geq 0$. The following congruences hold with precision $2^{j+1}$:
 \begin{align}
% \Delta^{j+1}&=0  \ \ \textrm{with precision} \ \ 2^j,\label{eq: lemma: main loop Newton lifting_1}\\
% \Delta^{j+1}(T\pm\Delta^{j+1})&=\Delta^{j+1}, \label{eq: lemma: main loop Newton lifting_2}\\
 M^j(T-\Delta^{j+1})&\equiv 0 \mod M^{j+1},\label{eq: lemma: main loop Newton lifting_3}\\
 v^{j+1}_i(T-\Delta^{j+1})&\equiv V^{j+1}_i \mod M^{j+1} \textrm{ for }n-s+1\le i \le n.
 \label{eq: lemma: main loop Newton lifting_4}
 \end{align}
\end{claim}
\begin{proof}[Proof of Claim]
%The Taylor expansion of $\Delta^{j+1}(T\pm\Delta^{j+1})$ in powers
%of $\Delta^{j+1}$ is
%$$
%\Delta^{j+1}(T\pm\Delta^{j+1})=\Delta^{j+1}\pm \frac{\partial
%\Delta^{j+1}}{\partial T} \Delta^{j+1} +
%\mathcal{O}\big((\Delta^{j+1})^2\big).
%$$
%Since $\Delta^{j+1}=0$ with precision $2^{j}$, we have
%$\frac{\partial \Delta^{j+1}}{\partial T}\Delta^{j+1}=0$ and
%$(\Delta^{j+1})^2=0$ with precision $2^{j+1}$. This proves
%\eqref{eq: lemma: main loop Newton lifting_2}.
%
We start by expanding $M^j(T-\Delta^{j+1})$ using Taylor expansion
around $T$, in
powers of $\Delta^{j+1}$:
$$
M^j(T-\Delta^{j+1})=M^j(T)-\frac{\partial M^j}{\partial T}\Delta^{j+1}
+\mathcal{O}\big((\Delta^{j+1})^2\big).
$$
Since $\Delta^{j+1}=0$ with precision $2^{j}$, it follows that
$(\Delta^{j+1})^2=0$ with precision $2^{j+1}$. Hence,
\begin{equation}\label{eq: Taylor expansion M}
M^j(T-\Delta^{j+1})=M^j(T)-\frac{\partial M^j}{\partial T}\Delta^{j+1}
+\mathcal{O}\big((\Delta^{j+1})^2\big),
\end{equation}
with precision $2^{j+1}$. By the definition of $M^{j+1}$,
we have $M^j=M^{j+1}+\Delta^{j+1}_{M}$, where $\Delta^{j+1}_{M}$ is the
remainder of the division of $\Delta^{j+1}\frac{\partial
M^j}{\partial T}$ by $M^j$, taken with precision $2^{j+1}$. Let
$h\in k[\![X_{n-s}-p_{n-s}]\!][T]$ denote the quotient of this
division. By \eqref{eq: Taylor expansion M}, and the definition
of $M^{j+1}$ and $\Delta^{j+1}_{M}$, we have
\begin{align*}
M^j(T-\Delta^{j+1})=
M^j-(h(T)M^j+\Delta^{j+1}_{M})
&=
(M^{j+1}+\Delta^{j+1}_{M})\big(1-h(T)\big)-\Delta^{j+1}_{M}\\
&=M^{j+1}\big(1-h(T)\big)- h(T)\Delta^{j+1}_{M}(T),
\end{align*}
with precision $2^{j+1}$. Since $\Delta^{j+1}\frac{\partial
M^j}{\partial T}=0$ with precision $2^{j}$, and $M^j$ is monic,
the quotient $h(T)$ vanishes with precision $2^{j}$, and thus
$\Delta^{j+1}_{M}=0$ with precision $2^{j}$. We conclude that
$h(T)\Delta^{j+1}_{M}(T)=0$ with precision $2^{j+1}$, which implies
\eqref{eq: lemma: main loop Newton lifting_3}.

Now we turn to \eqref{eq: lemma: main loop Newton
lifting_4}. For $n-s+1\le i\le n$, consider the Taylor expansion of
$v^{j+1}_i(T-\Delta^{j+1})$ around $T$, in
powers of $\Delta^{j+1}$. We have
$$
v^{j+1}_i(T-\Delta^{j+1})=v^{j+1}_i(T)-\Delta^{j+1}\frac{\partial v^{j+1}_i}{\partial T}(T),
$$
with precision $2^{j+1}$. By construction,
$v^{j+1}_i=V^{j+1}_i+\Delta^{j+1}_i$ with precision $2^{j+1}$, where
$\Delta^{j+1}_{i}$ is the remainder of the division of
$\Delta^{j+1}\frac{\partial v^{j+1}_i}{\partial T}$ by $M^j$, taken
with precision $2^{j+1}$. Let $h_i\in k[\![X_{n-s}-p_{n-s}]\!][T]$
denote the quotient of this division. Then
\begin{equation}\label{eq: proof: main loop Newton lifting_2}v^{j+1}_i(T-\Delta^{j+1})=V_i^{j+1}+\Delta^{j+1}_i-(h_i(T)M^j(T)+\Delta_i^{j+1})
=V_i^{j+1}-h_i(T)M^j(T),
\end{equation}
with precision $2^{j+1}$. We observe that $h_i(T)M^j(T)\equiv 0\mod
M^{j+1}$ with precision $2^{j+1}$. Indeed, similarly as before we
see that $h_i=0$, and then $\frac{\partial h_i}{\partial T}=0$, with
precision $2^{j}$. It follows that $
h_i(T-\Delta^{j+1})=h_i(T)-\frac{\partial h_i}{\partial
T}\Delta^{j+1} + \mathcal{O}\big((\Delta^{j+1})^2\big)=h_i(T) $ with
precision $2^{j+1}$. Then, by \eqref{eq: lemma: main loop Newton
lifting_3},
\begin{align*}
h_i(T)M^j(T)=h_i(T)\Big(M^j(T)-\frac{\partial M^j}{\partial T}\Delta^{j+1}\Big)&=
h_i(T-\Delta^{j+1})M^j(T-\Delta^{j+1})\\&\equiv 0\mod M^{j+1},
\end{align*}
with precision $2^{j+1}$. This, combined with \eqref{eq: proof: main
loop Newton lifting_2}, yields \eqref{eq: lemma: main loop Newton
lifting_4}, which completes the proof of the claim.
\end{proof}

Substituting $T-\Delta^{j+1}$ for $T$ in
\eqref{eq: F_v(j+1)_=_0_mod_M_j}, we obtain
\begin{equation}\label{eq: F_v(j+1)_=_0_mod_M_j_1}
\bfs F_s\big(X_{n-s}, \bfs v^{j+1}(T-\Delta^{j+1})\big)\equiv 0 \mod M^{j}(T-\Delta^{j+1}),
\end{equation}
with precision $2^{j+1}$. Taking into account \eqref{eq: lemma: main
loop Newton lifting_3} and \eqref{eq: lemma: main loop Newton
lifting_4}, we conclude that $$\bfs F_s\big(X_{n-s}, \bfs V^{j+1}
\big) \equiv 0 \mod M^{j+1}$$ with precision $2^{j+1}$, as desired.
Furthermore, since by definition $v^{j+1}_{n-s+1}=T+\Delta^{j+1}$,
substituting $T-\Delta^{j+1}$ for $T$ in this expression we obtain
$V^{j+1}_{n-s+1}\equiv T \mod M^{j+1}$, with precision $2^{j+1}$.

It remains to prove the invertibility of $\det J_s( X_{n-s}, \bfs
V^{j+1})$ and $G(\bfs p^{s+1}, X_{n-s}, \bfs V^{j+1})$, both modulo
$M^{j+1}$. Observe that the invertibility of the former holds if and
only if the resultant $\mathrm{Res}_T\big(\det J_s( X_{n-s}, \bfs
V^{j+1}),M^{j+1}\big)$ is an invertible element of
$\overline{k}[\![X_{n-s} -p_{n-s}]\!]$. By \eqref{eq: Congruence
w(j+1) W(j)}, we have $\bfs v^{j+1}=\bfs V^j$ with precision $2^j$.
Furthermore, as $V_i^{j+1}:=v_i^{j+1}- \Delta_i^{j+1}$, and
$\Delta_i^{j+1}=0$ with precision $2^j$, for $1\le i\le s$, it
follows that $\bfs V^{j+1}=\bfs V^j$ with precision $2^j$, which in
turn implies $\det J_s( X_{n-s}, \bfs V^{j+1})=\det J_s(X_{n-s},
\bfs V^j)$ with precision $2^j$. As $M^{j+1}=M^j$ with precision
$2^j$, we conclude that
$$\mathrm{Res}_T\big(\det J_s(X_{n-s}, \bfs
V^{j+1}),M^{j+1}\big)=\mathrm{Res}_T\big(\det J_s(
X_{n-s}, \bfs V^j),M^j\big)$$
with precision $2^j$. Since the right-hand side is invertible by
hypothesis, the assertion follows. The same argument proves that the
invertibility of $G(\bfs p^{s+1}, X_{n-s}, \bfs V^j)$ modulo $M^j$
implies the invertibility of $G(\bfs p^{s+1}, X_{n-s}, \bfs V^{j+1})$
modulo $M^{j+1}$.
\end{proof}

Let $M_s,W_{n-s+i}^s\ldots,W_n^s\in k[X_{n-s}, T]$ be the
polynomials that form the Kronecker representation of $C_s$, with
primitive element $X_{n-s+1}$, whose existence is established in Section
\ref{section: lifting curve properties}. Define
$$V^{(s)}_{n-s+1}:=T\ \textrm{ and }\ V^{(s)}_{n-s+i}:=\Big(\frac{\partial M_s}{\partial
T}\Big)^{-1}W^s_{n-s+i} \mod M_s\ \textrm{ for }\ 2 \le i \le s.$$
\begin{remark}\label{rem: V_i_well_def}
The Kronecker representation $M_s,W_{n-s+i}^s\ldots,W_n^s$ of $C_s$
``lifts'' the Kronecker representation
$m_s,w_{n-s+1}^s,\ldots,w_n^s$ of the lifting fiber $\pi_s^{-1}(\bfs
p^s)$, in the sense that $M_s(p_{n-s},T)=m_s$ and
$W^s_{n-s+i}(p_{n-s},T)=w^s_{n-s+i}$ for $2\le i \le s$. Moreover,
since the discriminant
$\rho_s:=\mathrm{Disc}_{T}\big(M_s(X_{n-s},T)\big)$ satisfies
$\rho_s(p_{n-s})\neq 0$ by the choice of the lifting point $\bfs p^s$,
it follows that each $V^{(s)}_{n-s+i}$ is a well--defined element of
$k[\![X_{n-s}-p_{n-s}]\!][T]$ for $2 \le i \le s$, with 
$V^{(s)}_{n-s+i}(p_{n-s},T)=v^{(s)}_{n-s+i}$.
\end{remark}
Now we are in position to establish the correctness of
Algorithm \ref{algo: Newton lifting}.
\begin{theorem}\label{theo: correctedness: Newton lifting}
Algorithm \ref{algo: Newton lifting} correctly computes the
coefficients in $k[X_{n-s}]$ of the polynomials $M_s(X_{n-s}, T)$,
$W_{n-s+i}^{s}(X_{n-s}, T)$ $(2 \le i \le s)$ defining a
Kronecker representation of $C_s$ with $X_{n-s+1}$ as primitive element.
\end{theorem}
\begin{proof}
%First we observe that the second step of the algorithm is
%well--defined, namely, the derivative $\partial M^0/\partial T=m_s'$
%is invertible modulo $M^0:=m_s$. Indeed, since the lifting fiber
%$\pi_s^{-1}(\bfs p)$ has exactly $\delta_s$ elements, and
%$X_{n-s+1}$ is a primitive element of $\pi_s^{-1}(\bfs p)$, its
%minimal polynomial $m_s$ has degree $\delta_s$ and possesses
%$\delta_s$ distinct roots in $\overline{k}$. Thus, $m_s$ is
%square--free, which implies that $m_s'$ is invertible modulo $m_s$.
%
%Now we analyze the main loop. %Denote by $M_s^j$, $W_i^j$ $(n-s+1 \le
%i\le n)$ the output of the $j$--th step of the main loop.
In view of Proposition \ref{prop: main loop Newton lifting}, we
claim that
\begin{align*}
F_i (\bfs p^{s+1}, X_{n-s}, \bfs V^0) & \equiv 0 \mod M^0\quad (1 \le i
\le s),
\\
V_{n-s+1}^0& = T,
\end{align*}
with precision 1, and that $\det J_s(X_{n-s}, \bfs V^0)$ is invertible
modulo $M^0$, where \linebreak $\bfs V^0:=(V_{n-s+1}^0,\ldots,V_n^0)$.

Since $M^0:=m_s$ and $V_{n-s+i}^0:=v_{n-s+i}^s$ for
$2\le i\le s$, it suffices to prove that $F_i(\bfs p^s,\bfs V^0
(p_{n-s},T))$ $(1 \le i \le s)$ and $V_{n-s+1}^0 (p_{n-s},T)- T$
vanish at all roots of $m_s$, and $\det J_s(\bfs p^s,\bfs V^0
(p_{n-s},T))$ does not vanish at any root of $m_s$.

By assumption, $X_{n-s+i}\equiv
v^s_{n-s+i}(X_{n-s+1})$ over $\pi_s^{-1}(\bfs p^s)$ for $2\le i\le s$.
Therefore, for each point $(\bfs
p^s,x_{n-s+1}\klk x_n) \in \pi_s^{-1}(\bfs p^s)$, we have
\begin{equation}\label{eq: 0th parametriz NH}
x_{n-s+i} = V_{n-s+i}^0(p_{n-s},x_{n-s+1}).
\end{equation}
Since $F_i(\bfs p^s,x_{n-s+1}\klk x_n)=0$ for $1 \le i \le s$, and
$x_{n-s+i}= V_{n-s+i}^0(p_{n-s},x_{n-s+1})$ for $2\le i\le s$, it
follows that $F_i(\bfs p^s,\bfs V^0 (p_{n-s},T))$ vanishes at all
roots $x_{n-s+1}$ of $M^0(p_{n-s}, T)$, for $1 \le i \le s$. The
second assertion is clear by definition.

Finally, as $\pi_s^{-1}(\bfs p^s)$ is a lifting fiber, Proposition
\ref{prop: F_i(p*,X) generate radical ideal} proves that it is
unramified. Thus, $\det J_s(\bfs p^s, x_{n-s+1}\klk x_n)\not=0$ for
each point $(\bfs p^s,x_{n-s+1}\klk x_n) \in \pi_s^{-1}(\bfs p^s)$, and
it follows that $\det J_s(\bfs p^s,\bfs V^0(p_{n-s},T))$ does not
vanish at any root $x_{n-s+1}$ of $m_s$. This completes the proof of
the claim.

By Proposition \ref{prop: main loop Newton lifting}, we deduce that
Algorithm \ref{algo: Newton lifting} executes correctly. It remains
to prove that its output consists of the polynomials $M_s,
W^s_{n-s+1}\klk W^s_n\in{k}[X_{n-s},T]$ which define a Kronecker
representation of the lifting curve $C_s$, with $X_{n-s+1}$ as
primitive element.

By construction, $M_s\in k[X_{n-s},T]$ is monic in
$T$ and satisfies $\deg_TM_s=\deg M_s=\delta_s$. Moreover,
$m_s=M_s(p_{n-s},T)$. Since $m_s$ is square-free of degree $\delta_s$,
Hensel's lemma implies that $M_s$ factors in
$\overline{k}[\![X_{n-s}-p_{n-s}]\!][T]$ as
$$M_s=\prod_{x \in \{m_s=0\}}(T-R_{n-s+1}^{x}),$$
where $R_{n-s+1}^x(p_{n-s})=x$ for each root $x$ of $m_s$. Denote
$R_{n-s+i}^x:=V_{n-s+i}^{(s)}(R_{n-s+1}^x)$ for $2\le i\le s$ and
each $x \in \{m_s=0\}$, and let $\bfs R^x:=(R_{n-s+1}^x\klk
R_n^x)$. We claim that, for each $x \in
\{m_s=0\}$,
\begin{itemize}
    \item $F_1(\bfs p^{s+1}, X_{n-s},\bfs R^x)=0\klk F_s(\bfs p^{s+1}, X_{n-s},\bfs R^x)=0$;
    \item $(p_{n-s},\bfs R^x(p_{n-s}))\in\pi_s^{-1}(\bfs p^s)$.
  \end{itemize}
Indeed, fix such an $x$ and set $\bfs V^{(s)}:=(V^{(s)}_{n-s+1} \klk
V^{(s)}_n)$. Since $M_s(R_{n-s+1}^x)=0$
and $\bfs R^x=\bfs V^{(s)}(R_{n-s+1}^x)$, together with the relation
$$F_i (\bfs p^{s+1}, X_{n-s}, \bfs V^{(s)}) \mod M_s  = 0 \quad (1 \le i
\le s),$$
it follows immediately that $F_i (\bfs p^{s+1}, X_{n-s}, \bfs R^x)=0$
for $1\le i\le s$. For the second assertion, since $\bfs R^x=\bfs V^{(s)}(R_{n-s+1}^x)$
and the identities $V_{n-s+i}(p_{n-s},T)=v_{n-s+1}$ and $R^x_{n-s+1}(p_{n-s})=x$ hold,
by substituting $p_{n-s}$ for $X_{n-s}$ it follows that 
$$(\bfs p^s,\bfs R^x(p_{n-s}))=(\bfs p^s,x,v_{n-s+2}(x),\ldots,v_n(x))\in\pi_s^{-1}(\bfs p^s).$$

Finally, we verify that the output of Algorithm \ref{algo: Newton
lifting} is the desired one. Suppose the main loop is executed for
all $j\ge 0$. The fact that $M^{j+1}=M^j$ and
$V_{n-s+i}^{j+1}=V_{n-s+i}^j$ $(1\le i\le s)$ with precision $2^{j}$
for all $j\geq 0$ shows that there exist unique polynomials
$M^\infty,V_{n-s+1}^{\infty}, \ldots, V_n^{\infty}\in{k}[\![X_{n-s}
-p_{n-s}]\!][T]$ satisfying $M^\infty= M^j$ and $V_{n-s+i}^{\infty}=
V^{j}_{n-s+i}$ $(1\le i \le s)$ with precision $2^{j}$ for all
$j\geq 0$.

In particular, we have $\deg_TM^\infty=\delta_s$ and
$\deg_TV_{n-s+i}^{\infty}<\delta_s$ for $1\le i\le s$. Furthermore,
since the congruences \eqref{eq: CongruenceProofHL1} and \eqref{eq:
CongruenceProofHL2} hold for all $j \geq 0$, it follows that
\begin{align*}
F_i (\bfs p^{s+1}, X_{n-s}, \bfs V^\infty) \mod M^\infty & = 0 \quad (1
\le i \le s), \\ V_{n-s+1}^\infty & = T.
\end{align*}

As $m_s=M^\infty(p_{n-s},T)$, applying Hensel's lemma once more,
$M^\infty$ admits a factorization %in
%$\overline{k}[\![X_{n-s}-p_{n-s}]\!][T]$ of the form
%
$$M^\infty=\prod_{x \in \{m_s=0\}}(T-S_{n-s+1}^{x}),$$
where $S_{n-s+1}^x(p_{n-s})=x$ for each $x\in\{m_s=0\}$. Set
$S_{n-s+i}^x:=V_{n-s+i}^\infty(S_{n-s+1}^x)$ for $2\le i\le s$, and
$\bfs S^x:=(S_{n-s+1}^x, \ldots, S_n^x)$. Then
\begin{equation}\label{eq: NH revised 1}
F_1(\bfs p^{s+1}, X_{n-s},\bfs S^x)=0\klk F_s(\bfs p^{s+1}, X_{n-s},\bfs S^x)=0.
\end{equation}
We claim that
\begin{equation}\label{eq: NH revised 2}
(p_{n-s},\bfs S^x(p_{n-s}))\in\pi_s^{-1}(\bfs p^s).
\end{equation}
First we show that $G(\bfs p^{s+1},X_{n-s},\bfs V^\infty)$
is invertible modulo $M^\infty(X_{n-s},T)$ in $k[\![X_{n-s}-p_{n-s}]\!][T]$.
Indeed, let 
$$D^\infty:=\mathrm{Res}_T(G(\bfs p^{s+1},X_{n-s},\bfs V^\infty),M^\infty(X_{n-s},T)).$$
Since $\bfs V^\infty(p_{n-s},T)=\bfs v^s(T):=(v^s_{n-s+1}(T),\ldots,v^s_n(T))$ and
$M^\infty(p_{n-s},T)=m_s(T)$, we obtain  
$D^\infty(p_{n-s})=\mathrm{Res}_T(G(\bfs p^{s},\bfs v^s),m_s)$, which is nonzero
because $G$ does not vanish on any point of $\pi_s^{-1}(\bfs p^s)$. Hence
$D^\infty$ is a unit of $k[\![X_{n-s}-p_{n-s}]\!]$, which implies the 
claimed invertibility. Since $M^\infty(S_{n-s+1}^x)=0$, the fact that
$D^\infty$ is a unit of $k[\![X_{n-s}-p_{n-s}]\!]$ shows that 
$G(\bfs p^{s+1},X_{n-s},\bfs V^\infty(S_{n-s+1}^x))=
G(\bfs p^{s+1},X_{n-s},\bfs S^x)$ is also a unit of $k[\![X_{n-s}-p_{n-s}]\!]$.
Consequently, $G(\bfs p^s,\bfs S^x(p_{n-s}))\not=0$. Together
with \eqref{eq: NH revised 1}, this implies \eqref{eq: NH revised 2}.

Finally, we claim that 
\begin{equation}\label{eq: NH revised 3}
\bfs R^x(p_{n-s})=\bfs S^x(p_{n-s}). 
\end{equation}
By definition, 
$S_{n-s+1}^x(p_{n-s})=x=R_{n-s+1}^x(p_{n-s})$. Furthermore,
$$S_{n-s+i}^x(p_{n-s})=V_{n-s+i}^\infty(p_{n-s},S_{n-s+1}^x(p_{n-s}))
=V_{n-s+i}^\infty(p_{n-s},x)=v_{n-s+i}^s(x)$$
for $2\le i\le s$. On the other hand,  
$$R_{n-s+i}^x(p_{n-s})=V_{n-s+i}^s(p_{n-s},R_{n-s+1}^x(p_{n-s}))
=V_{n-s+i}^s(p_{n-s},x)=v_{n-s+i}^s(x).$$
Thus $R_{n-s+i}^x(p_{n-s})=v_{n-s+i}^s(x)=S_{n-s+i}^x(p_{n-s})$ for 
$2\le i\le s$. 

Combining \eqref{eq: NH revised 1}, \eqref{eq: NH revised 2} and
\eqref{eq: NH revised 3} with Theorem \ref{th: NewtonHensel} yields $\bfs R^x=\bfs S^x$
for each $x\in \{m_s=0\}$. In particular,
$$M^\infty=\prod_{x \in \{m_s=0\}}(T-S_{n-s+1}^{x})=
\prod_{x \in \{m_s=0\}}(T-R_{n-s+1}^{x})=M_s.$$
Moreover, for each $i$ with $1\le i\le s$, the equalities
$$V^{(s)}_{n-s+i}(R_{n-s+1}^x)=R_{n-s+i}^x=S_{n-s+i}^x=V_{n-s+i}^\infty(R_{n-s+1}^x)$$
for each of the $\delta_s$ distinct roots $x \in\overline{k}$ of
$m_s$, together with
$\deg_TV_{n-s+i}^{\infty},\deg_TV_{n-s+i}^{(s)}<\delta_s$,
imply that $V_{n-s+i}^{\infty}=V^{(s)}_{n-s+i}$.

Thus  $M_{s}=M^j$
and $V^{(s)}_{n-s+i}=V^{j}_{n-s+i} (1\le i \le s)$ with precision
$2^{j}$ for all $j\geq 0$. Setting $j_0:=\lfloor
\log_2(\delta_s+1) \rfloor$, we obtain with precision $2^{j_0}$:
\begin{align*}
%W^{j_0}_{n-s+i} & :=
\frac{\partial M^{j_0}}{\partial T}V^{j_0}_{n-s+i} \mod M^{j_0}
            & = \frac{\partial M_s}{\partial T}V^{(s)}_{n-s+i} \mod M_s
            = W^s_{n-s+i} \quad (1\le i \le s).
\end{align*}
Since $\frac{\partial M^{j_0}}{\partial T}V^{j_0}_{n-s+i} \mod
M^{j_0}$ and $W^s_{n-s+i}$ have degree at most $\delta_s< 2^{j_0}$
in $X_{n-s}$, it follows that they are equal. Similarly,
$M^{j_0}=M_s$. Hence the output of Algorithm \ref{algo: Newton
lifting} consists of the desired polynomials $M_s,W^s_{n-s+1}\klk
W^s_n$.
\end{proof}
\subsection{The cost of Newton lifting}
We now proceed to analyze the complexity of Algorithm \ref{algo:
Newton lifting}. Its output consists of polynomials of
$k[X_{n-s},T]$, which are given by their dense representation.
Algorithm \ref{algo: Newton lifting} involves the evaluation of the
Newton-Hensel operator
$$N_{\bfs F_s}(X_{n-s+1}\klk X_n)^t:=( X_{n-s+1}\klk X_n
)^t-(J_s^{-1}{\bfs F_s}^t)(X_{n-s+1}\klk X_n),$$
where $\bfs F_s:=(F_1(\bfs p^{s+1},X_{n-s}\klk X_n) \klk F_s(\bfs
p^{s+1},X_{n-s}\klk X_n))$, and $J_s$ denotes the Jacobian matrix of
${\bfs F_s}$ with respect to $X_{n-s+1}\klk X_n$. Viewing $N_{\bfs
F_s}$ as a vector of elements of $k[X_{n-s}][X_{n-s+1}\klk X_n]$, we
are interested in the cost of its evaluation over $k[X_{n-s}]$. We
begin with the following result.
\begin{lemma}\label{lemma: complexity det J and Adj J}
If $F_1\klk F_s\in k[X_1\klk X_n]$ are given by a straight-line
program $\beta$ in $k[X_1\klk X_n]$ of length $L$, then the
determinant $\det J_s$ and the adjoint matrix $\mathrm{Adj}(J_s)$
can be evaluated by a straight-line program $\beta'$ in $k[X_{n-s}]$
of length $\mathcal{O}(Ls+s^4)$.\end{lemma}
\begin{proof}
From the straight-line program $\beta$, we
immediately obtain a straight-line program to evaluate the
specialized polynomials $F_1(\bfs
p^{s+1},X_{n-s}\klk X_n)\klk F_s(\bfs p^{s+1},X_{n-s}\klk X_n)$, using at most
$L$ arithmetic operations in $k[X_{n-s}\klk X_n]$. We retain
the notation $\beta$ for this new straight-line program.

Applying the Baur--Strassen theorem to $\beta$ yields a
straight-line program which evaluates the partial derivatives of
$F_1\klk F_s$ with respect to $X_{n-s+1}\klk X_n$, using
$\mathcal{O}(Ls)$ arithmetic operations in $k[X_{n-s}]$ (see
\cite{BaSt83}).

To this straight-line program we append, for instance, the
Samuelson--Berkowitz algorithm (see \cite{Berkowitz84}) to
compute both the determinant and the adjoint matrix of $J_s$, with
an additional cost of $\mathcal{O}(s^4)$ operations.
\end{proof}

Let $\bfs V^j:=(V_{n-s+1}^j\klk V_n^j)$ for each $j\ge 0$. The
$j$th iteration of step {\bf 4} of Algorithm \ref{algo: Newton
lifting} requires computing $N_{\bfs F_s}(\bfs V^j)$ modulo $M^j$
with precision $2^{j+1}$. %This in particular requires the computation of $\det J_s(X_{n-s}, \bfs V^{j})^{-1}\mod M^j$
%with precision $2^{j+1}$ for any $j\geq 0$.
For this purpose, we establish the following result.
%
%\begin{lemma}\label{remark: efficient_Jacobian inversion}
%Let $h_0:=\det J_s(X_{n-s}, \bfs V^{j-1})^{-1}\mod M^{j-1}$ with
%precision $2^{j}$ and
%
%$$h_1:=2h_0-\det J_s(X_{n-s}, \bfs V^{j-1})\, h_0^2 \mod M^{j-1}$$
%
%with precision $2^{j+1}$. Then $h_1\equiv \det J_s(X_{n-s}, \bfs
%V^j)^{-1}\mod M^j$ with precision $2^{j+1}$.
%\end{lemma}
%
%\begin{proof}
%We have the following relations modulo $M^j$ with precision
%$2^{j+1}$:
%
%\begin{align}\label{eq: remark: efficient_Jacobian inversion}
%1-\det J_s(X_{n-s}, \bfs V^j)\, h_1 &\equiv 1-2\,\det J_s(X_{n-s},
%\bfs V^j)\,h_0+\big(\det J_s(X_{n-s}, \bfs
%V^j)\, h_0\big)^{2}\nonumber\\
%&= \big(1-\det J_s(X_{n-s}, \bfs V^j)\, h_0\big)^{2}.
%\end{align}
%
%By hypothesis, $1-\det J_s(X_{n-s}, \bfs V^{j})\, h_0\equiv 0 \mod
%M^{j}$ with precision $2^j$.
%Since $\bfs V^{j-1}=\bfs V^{j}$ and
%$M^{j-1}=M^{j}$  with precision $2^j$, we deduce that
%It follows that
%$I-J_s(X_{n-s}, \bfs V^j)\, h_0\equiv 0 \mod M^j$ with precision
%$2^j$. Thus
%the right-hand side of \eqref{eq: remark:
%efficient_Jacobian inversion} vanishes modulo $M^{j}$ with
%precision $2^{j+1}$, which proves the lemma.
%\end{proof}
%
%
%Now we can describe and analyze the computation of $N_{\bfs
%F_s}(\bfs V^j)$ modulo $M^j$ with precision $2^{j+1}$.
%
\begin{lemma}\label{lemma: complexity NH operator}
The evaluation of $N_{\bfs F_s}(X_{n-s},\bfs V^j)$ modulo $M^j$ can
be performed using $\mathcal{O}((Ls+s^4+\log_2\delta_s){\sf M}(\delta_s))$
arithmetic operations in $k[\![X_{n-s}-p_{n-s}]\!]$.
\end{lemma}
\begin{proof}
By Lemma \ref{lemma: complexity det J and Adj J}, we can compute
$\det J_s$ and $\mathrm{Adj}(J_s)$ via a straight--line program
using $\mathcal{O}(Ls+s^4)$ arithmetic operations in $k[X_{n-s}]$. We
substitute the values $\bfs V^j$ for the variables $X_{n-s+1}\klk
X_n$, and perform all these arithmetic operations in
$k[\![X_{n-s}-p_{n-s}]\!][T]$ modulo $M^j$. Since $M^j$ is a monic
element of $k[\![X_{n-s}-p_{n-s}]\!][T]$, each such modular
operation involves $\mathcal{O}({\sf M}(\delta_s))$ arithmetic
operations in $k[\![X_{n-s}-p_{n-s}]\!]$. Hence, computing both
$\det J_s(X_{n-s}, \bfs V^j)$ and $\mathrm{Adj}(J_s(X_{n-s}, \bfs
V^j))$ modulo $M^j$ requires $\mathcal{O}((Ls+s^4){\sf M}(\delta_s))$
arithmetic operations in $k[\![X_{n-s}-p_{n-s}]\!]$.

According to Proposition \ref{prop: main loop Newton lifting}, $\det
J_s(X_{n-s}, \bfs V^j)$ is invertible modulo $M^j$. Therefore, we
can compute its inverse $\det J_s(X_{n-s}, \bfs V^j)^{-1}$ modulo
$M^j$ with $\mathcal{O}({\sf M}(\delta_s)\log_2\delta_s)$ arithmetic
operations in $k[\![X_{n-s}-p_{n-s}]\!]$.

Finally, we compute modulo $M^j$ the multiplication
$$(J_s^{-1}{\bfs F_s}^t)(X_{n-s}, \bfs V^j)=
\big(\det J_s^{-1}\mathrm{Adj}(J_s){\bfs F_s}^t\big)(X_{n-s}, \bfs
V^j)$$
and the subtraction
$$N_{\bfs F_s}(X_{n-s},\bfs V^j)^t=(\bfs V^j)^t-(J_s^{-1}{\bfs F_s}^t)(X_{n-s}, \bfs V^j)$$
with $\mathcal{O}(s^2{\sf M}(\delta_s))$ arithmetic operations in
$k[\![X_{n-s}-p_{n-s}]\!]$.

In total, the procedure requires
$\mathcal{O}((Ls+s^4+\log_2\delta_s){\sf M}(\delta_s))$ arithmetic
operations in $k[\![X_{n-s}-p_{n-s}]\!]$, as claimed. %Now, since each
%arithmetic operation in $k[\![X_{n-s}-p_{n-s}]\!]$ is performed with
%precision $2^{j+1}$, and therefore consists of ${\sf M}(2^{j+1})$
%arithmetic operations in $k$, we readily deduce the lemma.
\end{proof}

The matrix inversion of the Jacobian matrix $J_s(X_{n-s}, \bfs V^j)$
required for evaluating the Newton--Hensel operator $N_{\bfs
F_s}(X_{n-s},\bfs V^j)$ in step {\bf 4.1} of Algorithm~\ref{algo:
Newton lifting} can be performed incrementally, as described in
\cite[Section~4.3]{GiLeSa01}. Following this approach, the inversion
can be implemented using Newton iteration, thereby avoiding the
additional $\log_2(\delta_s)$ factor in the first term of the
complexity estimate of Lemma~\ref{lemma: complexity NH operator}.
Nevertheless, since this improvement does not affect the overall
asymptotic cost of the main algorithm---which is dominated by other
routines of higher complexity---we have chosen to retain the present
formulation for the sake of clarity and uniformity of exposition.

Now we analyze the overall cost of Algorithm \ref{algo: Newton
lifting}.
\begin{theorem}\label{th: cost Newton lifting}
Assume that $F_1\klk F_s\in k[X_1\klk X_n]$ are given by a
straight--line program $\beta$ in $k[X_1\klk X_n]$ of length $L$.
Then Algorithm \ref{algo: Newton lifting} can be implemented
with $\mathcal{O}((Ls+s^4+\log_2\delta_s){\sf M}(\delta_s)^2)$ arithmetic
operations in $k$.
\end{theorem}
\begin{proof}
%Step {\bf 2} involves computing the inverse of $m_s'$ modulo $m_s$,
%and then multiplying it by $w_{n-s+i}^s$ modulo $m_s$, for $2\le
%i\le s$. The inversion takes $\mathcal{O}({\sf
%M}(\delta_s)\log\delta_s)$ arithmetic operations in $k$, and each
%multiplication costs $\mathcal{O}({\sf M}(\delta_s))$ arithmetic
%operations in $k$. Hence, the first step can performed using
%$\mathcal{O}({\sf M}(\delta_s)(s+\log\delta_s))$ arithmetic
%operations in $k$.
%
We analyze the loop of step {\bf 4}. Let $j$ be fixed such that
$2^j<\delta_s$. The $j$th iteration operates with elements of
$k[\![X_{n-s}-p_{n-s}]\!][T]$, and all the arithmetic operations
are performed at precision $2^{j+1}$. Therefore, we analyze the
number of arithmetic operations in $k[\![X_{n-s}-p_{n-s}]\!]$
performed in this iteration. Step {\bf 4.1} requires the evaluation
$N_{\bfs F_s}(\bfs V_j)$, which, by Lemma \ref{lemma: complexity NH
operator}, requires $\mathcal{O}((Ls+s^4+\log_2\delta_s){\sf
M}(\delta_s))$ arithmetic operations in $k[\![X_{n-s}-p_{n-s}]\!]$.
Steps {\bf 4.2}, {\bf 4.3}, {\bf 4.4} and {\bf 4.5} consist of
$\mathcal{O}(s)$ arithmetic operations, modulo $M^j$, of polynomials
of $k[\![X_{n-s}-p_{n-s}]\!][T]$ of degree at most $\delta_s$. As a
consequence, these steps require $\mathcal{O}(s\,{\sf M}(\delta_s))$
arithmetic operations in $k[\![X_{n-s}-p_{n-s}]\!]$. It follows that
the complete $j$th iteration can be performed using
$\mathcal{O}((Ls+s^4+\log_2\delta_s){\sf M}(\delta_s))$ arithmetic
operations in $k[\![X_{n-s}-p_{n-s}]\!]$. Each such operation is
performed with precision $2^{j+1}$, and then consists of ${\sf
M}(2^{j+1})$ arithmetic operations in $k$, so the total cost of the
$j$th iteration is
$$\mathcal{O}((Ls+s^4+\log_2\delta_s){\sf M}(\delta_s){\sf M}(2^{j+1}))$$
arithmetic operations in $k$. Summing over $j=0$ to $\lfloor\log\delta_s\rfloor$,
we obtain the total cost of the main loop:
$$\sum^{\lfloor \log\delta_s\rfloor}_{j=0}\mathcal{O}((Ls+s^4+\log_2\delta_s){\sf
M}(\delta_s){\sf M}(2^{j+1}))=\mathcal{O}((Ls+s^4+\log_2\delta_s){\sf
M}(\delta_s)^2).$$
The remaining steps have a negligible cost compared to this bound.
\end{proof}
%
%----------------------------------------------------------------------
%----------------------------------------------------------------------
%----------------------------------------------------------------------
%----------------------------------------------------------------------
%----------------------------------------------------------------------
%----------------------------------------------------------------------
%----------------------------------------------------------------------
%----------------------------------------------------------------------
%
\section{From the $s$th lifting curve to the $(s+1)$th lifting fiber}
\label{sec: intersection step}
In this section, we discuss the elimination phase of the $s$th
recursive step in the main algorithm for computing a Kronecker
representation of the Zariski closure $V:=V_r$ of $V(F_1\klk
F_r)\setminus V(G)$. Recall that, for $1\le j\le r$, the variety
$V_j\subset\A^n$ denotes the Zariski closure of
$\{F_1=0,\ldots,F_j=0,G\not=0\}$, and that $\pi_j:V_j\to\A^{n-j}$
denotes the projection onto the first $n-j$ coordinates. Throughout
this section, we assume that the variables $X_1,\ldots,X_n$ and the
point $\bfs p^s\in\A^{n-s}$ satisfy all the conditions of Theorem
\ref{th: preproc: all conditions}. The goal of this step is to
compute a Kronecker representation of the lifting fiber
$\pi_{s+1}^{-1}(\bfs p^{s+1})$, given a Kronecker representation of the
lifting curve $C_s:=\pi_s^{-1}(\{\bfs p^{s+1}\}\times\A^1)$.
%
%----------------------------------------------------------------------
%----------------------------------------------------------------------
%----------------------------------------------------------------------
%----------------------------------------------------------------------
%
\subsection{Computing projections of the $(s+1)$th lifting fiber}
The computation of a Kronecker representation of the lifting fiber
$\pi_{s+1}^{-1}(\bfs p^{s+1})$ is based on the following result.
\begin{proposition}\label{prop: C_s cap F_(s+1) and G has dim zero}
The following assertions hold:
\begin{enumerate}
\item $C_s\cap V(F_{s+1})\setminus\{G=0\}=\pi_{s+1}^{-1}(\bfs p^{s+1})$.
  \item $C_s\cap
  V(G)$ has dimension at most zero.
    \item $C_s\cap
  V(F_{s+1})$ has dimension zero.
\end{enumerate}
\end{proposition}
\begin{proof}
We start with the first assertion. Let $(\bfs p^{s+1},\bfs x)\in
\pi_{s+1}^{-1}(\bfs p^{s+1})$. Then $F_j(\bfs p^{s+1},\bfs x)=0$ for $1\le
j\le s+1$, which implies that $(\bfs p^{s+1},\bfs x)\in V(F_{s+1})$.
Furthermore, since $\pi_{s+1}^{-1}(\bfs p^{s+1})\subset\{G\not=0\}$,
Proposition \ref{prop: F_i(p*,X) generate radical ideal}$(1)$
implies that $(\bfs p^{s+1},\bfs x)\in C_s\setminus\{G=0\}$. It follows
that $\pi_{s+1}^{-1}(\bfs p^{s+1})\subset C_s\cap
V(F_{s+1})\setminus\{G=0\}$.

Conversely, let $(\bfs p^{s+1},\bfs x)\in (C_s\setminus\{G=0\})\cap
V(F_{s+1})$. Then $F_j(\bfs p^{s+1},\bfs x)=0$ for $1\le j\le s+1$ and
$G(\bfs p^{s+1},\bfs x)\not=0$. Since $\bfs p^{s+1}$ is a lifting point of
$\pi_{s+1}$, the polynomials $F_j(\bfs p^{s+1},X_{n-s}\klk X_n)$ for
$1\le j\le s+1$, together with $X_i-p_i$ for $1\le i\le n-s-1$,
generate the radical ideal of $\pi_{s+1}^{-1}(\bfs p^{s+1})$. Hence,
$(\bfs p^{s+1},\bfs x)\in \pi_{s+1}^{-1}(\bfs p^{s+1})$, completing the
proof of the first assertion.

To prove the second assertion, observe from Proposition
\ref{prop: F_i(p*,X) generate radical ideal}$(1)$ that
$C_s\setminus\{G=0\}$ has pure dimension 1. Moreover, since
$(C_s\setminus\{G=0\})\cap\{X_{n-s}=p_{n-s}\}=\pi_s^{-1}(\bfs p^s)$
has cardinality $\delta_s$, we deduce that
$$\delta_s\le \deg(C_s\setminus\{G=0\})\le \deg(C_s) =\delta_s,$$
which implies that $\deg(C_s\setminus\{G=0\})=\deg(C_s)$. Hence,
$C_s\setminus\{G=0\}$ is Zariski dense in $C_s$, and so
$C_s\cap\{G=0\}$ has dimension at most zero, proving the second
assertion.

For the third assertion, by the first assertion we have
\begin{align*}
C_s\cap V(F_{s+1})=&\big(C_s\cap V(F_{s+1})\cap \{G=0\}\big)\cup \big(C_s\cap
V(F_{s+1})\setminus \{G=0\}\big)\\ \subset& C_s\cap \{G=0\}\cup
\pi_{s+1}^{-1}(\bfs p^{s+1}).
\end{align*}
Since both $C_s\cap \{G=0\}$ and
$\pi_{s+1}^{-1}(\bfs p^{s+1})$ are finite, this proves that $C_s\cap V(F_{s+1})$
is finite, hence of dimension zero.
\end{proof}

Proposition \ref{prop: C_s cap F_(s+1) and G has dim zero} suggests
a strategy to compute a Kronecker representation of
$\pi_{s+1}^{-1}(\bfs p^{s+1})$: determine the intersection $C_s\cap
V(F_{s+1})$ and discard the points lying in $V(G)$. To achieve this,
we compute the image of this intersection under the map
$\pi_s^*:C_s\to\A^1$, given by $\pi_s^*(x_{n-s}\klk x_n):=x_{n-s}$.

Condition $({\sf B}_3)$ of Section \ref{sec: lifting fibers} ensures
that $\pi_s^*$ is injective on $\pi_{s+1}^{-1}(\bfs p^{s+1})$, while
condition $({\sf B}_5)$ guarantees that the points in $C_s\cap V(G)$
can be identified and removed via their images under $\pi_s^*$.
Additionally, the Kronecker representation of $C_s$ under
consideration provides a ``good'' parametrization of the points
$(\bfs p^{s+1},x_{n-s}\klk x_n)$ of $C_s$ for which $\rho_s(\bfs
p^{s+1},x_{n-s})\not=0$. Condition $({\sf B}_4)$ ensures that the points
of $\pi_{s+1}^{-1}(\bfs p^{s+1})\subset C_s$ admit such a good
representation.

This procedure requires computing the projection $\pi_s^*$ of the
points in an intersection $C_s\cap V(F)$, where $F\in k[X_{n-s}\klk
X_n]$ satisfies $\#\big((C_s\setminus\{\rho_s(\bfs p^{s+1},X_{n-s})=0\})\cap
V(F)\big)<\infty$. Since $\pi_s^*:C_s\to\A^1$ is a finite morphism, the
coordinate ring $k[C_s]$ becomes a finite $k[X_{n-s}]$--module. Moreover,
$k[C_s]$ is a Cohen-Macaulay ring and a free $k[X_{n-s}]$-module
(see \cite{GiHeSa93}). It follows that
$k(X_{n-s})\otimes_{k[X_{n-s}]} k[C_s]$ is a (finite-dimensional)
$k(X_{n-s})$--vector space. We consider the homothety
$\eta_F:k(X_{n-s})\otimes_{k[X_{n-s}]} k[C_s]\to
k(X_{n-s})\otimes_{k[X_{n-s}]} k[C_s]$ by $F$, and its
characteristic polynomial $\chi_F\in k(X_{n-s})[T]$.

The following lemma shows that
the square--free part of the constant term of $\chi_F$
describes the image of the projection of the intersection.
\begin{lemma}\label{lemma: properties homothety by F}
Let $F\in k[X_{n-s}\klk X_n]$ be such that $\#(C_s\cap
V(F))<\infty$. Let $\chi_F\in k[X_{n-s}][T]$ be the characteristic
polynomial of the homothety $\eta_F:k(X_{n-s})\otimes_{k[X_{n-s}]}
k[C_s]\to k(X_{n-s})\otimes_{k[X_{n-s}]} k[C_s]$ by $F$, and let
$\widetilde{a}\in k(X_{n-s})$ be the constant term of $\chi_F$. Then
$\widetilde{a}\in k[X_{n-s}]$. Let $a_F\in k[X_{n-s}]$ be the monic
polynomial which consists of the product of the irreducible factors
on $k[X_{n-s}]$ of $\widetilde{a}$. Then:
\begin{enumerate}
  \item $a_F\in k[X_{n-s}]\setminus \{0\}$,
  \item $a_F\in I(C_s\cap V(F))$.
\end{enumerate}
\end{lemma}
\begin{proof}
According to \cite[Corollary 2]{GiLeSa01}, $\chi_F\in k[X_{n-s},T]$.
Thus, its constant term $\widetilde{a}$ belongs to $k[X_{n-s}]$, and
so does its square-free part $a_F$.

By the Cayley-Hamilton theorem, $\chi_{F}(F)=0$ in
$k[C_s]$, hence the constant term $\widetilde{a}$ of
$\chi_{F}$ lies in the ideal generated by $I(C_s)$ and $F$,
proving the second assertion.

To prove the first assertion, let $m_F\in k(X_{n-s})[T]$ be the
minimal polynomial of $\eta_F$. As $\chi_{F}\in k[X_{n-s},T]$, we
have $m_F\in k[X_{n-s},T]$, and its constant term $b\in k[X_{n-s}]$
is equal to zero if and only if $\widetilde{a}=0$. Suppose $b=0$.
Then $m_F=T\cdot \widetilde{m}$ for some $\widetilde{m}\in
k[X_{n-s},T]$, so $0=m_F(F)=F\cdot \widetilde{m}(F)$ in $k[C_s]$.
Minimality of $m_F$ implies $\widetilde{m}(F)\not=0$ in $k[C_s]$, so
$F$ is a zero divisor in $k[C_s]$, and hence $F$ vanishes on an
irreducible component of $C_s$, contradicting the
condition $\#(C_s\cap V(F))<\infty$. Therefore, $b\not=0$, so
$\widetilde{a}\not=0$. This readily implies the first assertion.
\end{proof}

Now we describe the algorithm for computing the constant term of the
characteristic polynomial of the homothety $\eta_F$ by a polynomial
$F$ as above. This algorithm isolates a routine that was only
implicitly contained in the algorithm underlying
\cite[Proposition~4.4]{CaMa06}, and extends it by introducing an
interpolation stage that allows the method to work over arbitrary
perfect fields. This reformulation not only broadens the
applicability of the procedure but also clarifies the internal
structure of the computation, making explicit the intermediate
objects involved in the process.
\begin{algorithm}[Computation of a projection]\label{algo: computation projection} ${}$
\begin{itemize}
\item[]{\bf Input:}
\begin{enumerate}
\item
A Kronecker representation of $C_s$, as
in the output of Algorithm \ref{algo: Newton
lifting};
\item
$F\in k[X_{n-s}\klk X_n]$ such that
$\#(C_s\setminus\{\rho_s(\bfs p^{s+1},X_{n-s})=0\})\cap V(F)<\infty$;
\item
Access to uniformly random elements of a finite set
$\mathcal{S}\subset k$ of size at least
$\varepsilon^{-1}(D_s+1)(\delta_s^2+D_s)$, %(or of a finite field
%extension of $k$ if $k$ does not have enough points),
where $0<\varepsilon<1$ and $D_s:=d\cdot\delta_s$. \end{enumerate}

\item[]{\bf Output:} The dense representation of the
square--free part of the constant term $a_F\in k[X_{n-s}]$ of the
characteristic polynomial $\chi_F\in k[X_{n-s}][T]$ of the homothety
$\eta_F:k[C_s]\to k[C_s]$ of multiplication by $F$.
%\begin{itemize}
%\item[{\bf 1}] Use Newton lifting to obtain a geometric
%solution of $C_s$ with $X_{n-s+1}$ as primitive element, that is,
%polynomials $M_s\in k[X_{n-s},T]$ and $W_{n-s+2}^s\klk W_n^s\in
%k[X_{n-s},T]$ as in the output of Algorithm \ref{algo: Newton
%lifting}.
\item[]{\bf 1} For $j=0,\ldots,D_s$, do

\begin{itemize}
\item[]{\bf 1.1} Choose $\alpha_j\in\mathcal{S}$ at random;
\item[]{\bf 1.2} Compute $V_{n-s+i}^s(\alpha_j,T)= \frac{\partial M_s}{\partial
T}(\alpha_j,T)^{-1} W_{n-s+i}^s(\alpha_j,T)\mod M_s(\alpha_j,T)$ for
$2\le i\le s$;

\item[]{\bf 1.3} Compute $F^*(\alpha_j,T)=
F(\alpha_j,T,V_{n-s+2}^s(\alpha_j,T)\klk V_n^s(\alpha_j,T))\mod
M_s(\alpha_j,T)$;
\item[]{\bf 1.4} Compute
$\widetilde{a}(\alpha_j):=\mathrm{Res}_T\left(M_s(\alpha_j,T),F^*(\alpha_j,T)\right)$.
\end{itemize}
\item[]{\bf 2} Interpolate the polynomial $\widetilde{a}$ from its values.
\newline {\em ***[Observe that $\widetilde{a}\in
k[X_{n-s}]\setminus k$]***}
\item[]{\bf 3} Compute the square--free part $a_F$
of $\widetilde{a}$.
%\end{itemize}
\end{itemize}
\end{algorithm}

\begin{remark}\label{rem: k definability algo projection}
We assume that the base field $k$ has sufficiently large cardinality
to allow the choice of the evaluation points $\alpha_j$ required in
step {\bf 1.1} inside \(k\). If this is not the case, one may
perform the computations over a finite extension of $k$ of adequate
size; this increases the cost of Algorithm \ref{algo: computation
projection} only by an additional logarithmic factor. To avoid
further complicating the complexity estimates we henceforth state
the costs under the convention that \(k\) is large enough, with the
understanding that a passage to a small extension is always possible
at the indicated (logarithmic) overhead.

We stress that the polynomial $a_F$ produced in step {\bf 1.3} is
defined over $k$ (i.e. $a_F\in k[X_{n-s}]$) regardless of whether
the intermediate evaluation points $\alpha_j$ lie in $k$ or in a
finite extension, since both the lifting curve $C_s$ and the
polynomial $F$ are defined over $k$.
\end{remark}

The following result shows that Algorithm \ref{algo: computation
projection} functions correctly as specified.
\begin{lemma}\label{lemma: correctness algo computation projection}
Assume that $k$ has cardinality at least
$\varepsilon^{-1}(D_s+1)(\delta_s^2+D_s)$. For a lucky choice of
$\alpha_0,\ldots,\alpha_{D_s}$ in step {\bf 1.1},  Algorithm
\ref{algo: computation projection} correctly computes $a_F$. Such a
choice is lucky with probability $1-\varepsilon$.
\end{lemma}
\begin{proof}
We first prove the correctness of a simplified version of the
algorithm, where steps {\bf 1} and {\bf 2} are replaced by the
following steps:
\begin{itemize}
\item[{\bf 1'}] Compute $V_{n-s+i}^s(X_{n-s},T):=
\frac{\partial M_s}{\partial T}(X_{n-s},T)^{-1}
W_{n-s+i}^s(X_{n-s},T)\mod M_s(X_{n-s},T)$ in $k(X_{n-s})[T]$ for
$2\le i\le s$.
\item[{\bf 2'}] Compute $F^*(X_{n-s},T):=
F(X_{n-s},T,V_{n-s+2}^s(X_{n-s},T)\klk V_n^s(X_{n-s},T))$.
\item[{\bf 3'}] Compute $\widetilde{a}:=\mathrm{Res}_T\left(M_s(X_{n-s},T),F^*(X_{n-s},T)\right)\in
k[X_{n-s}]$.
\end{itemize}

For simplicity, we denote by $X_i$ ($n-s\le i\le n$) both the
variables in $k[X_1\klk X_n]$ and the corresponding coordinate
functions in $k[C_s]$. By Proposition \ref{prop: F_i(p*,X) generate
radical ideal}, the $k(X_{n-s})$--vector space
$k(X_{n-s})\otimes_{k[X_{n-s}]} k[C_s]$ has dimension $\delta_s$.
Since $X_{n-s+1}$ is a primitive element and its minimal polynomial
$M_s(X_{n-s},T)$ has degree $\delta_s$, the set $\{1,X_{n-s+1}\klk
X_{n-s+1}^{\delta_s-1}\}$ forms a basis of this vector space.

To perform step {\bf 1'} we require that
$\frac{\partial M_s}{\partial T}(X_{n-s},T)$ is
invertible modulo $M_s(X_{n-s},T)$, which holds because these
polynomials are coprime in
$k(X_{n-s})[T]$ by Proposition \ref{prop:
properties M_s}.

Taking into account that
$$\frac{\partial M_s}{\partial T}(X_{n-s},T)^{-1}\cdot
\frac{\partial M_s}{\partial T}(X_{n-s},T)\equiv 1\mod
M_s(X_{n-s},T),$$
and $M_s(X_{n-s},T)$ is the minimal polynomial of $X_{n-s+1}$ in
$k[C_s]$, we conclude that
$$\frac{\partial M_s}{\partial T}(X_{n-s},X_{n-s+1})^{-1}\cdot
\frac{\partial M_s}{\partial T}(X_{n-s},X_{n-s+1})\equiv 1$$
in $C_s$. This implies that
\begin{equation}\label{eq: aux proof correctness algo computing proj}
X_{n-s+i}\equiv \frac{\partial M_s}{\partial
T}(X_{n-s},X_{n-s+1})^{-1} \cdot W_{n-s+i}^s(X_{n-s},X_{n-s+1})
\end{equation}
in $C_s$ for $2\le i\le s$. The
right--hand side of \eqref{eq: aux proof correctness algo computing
proj}, reduced by $M_s(X_{n-s},X_{n-s+1})$, provides the representation of
$X_{n-s+i}$ in terms of the basis $\{1,X_{n-s+1}\klk
X_{n-s+1}^{\delta_s-1}\}$ for $2\le i\le s$. Hence,
%$F\equiv F^*(X_{n-s},T)\mod M_s(X_{n-s},T)$, where
%$F^*(X_{n-s},T)$ is the element of $k(X_{n-s})[T]$ computed in step
%2 of the algorithm, and thus
%
$$F\equiv F^*(X_{n-s},X_{n-s+1})$$
in $k(X_{n-s})[C_s]$, where $F^*(X_{n-s},T)$ is the polynomial computed in step
{\bf 2'}. %The expression in the right--hand side of the
%previous expression, reduced by $M_s(X_{n-s},X_{n-s+1})$,
Let $M_{X_{n-s+1}}\in k(X_{n-s})^{\delta_s\times \delta_s}$ be the
matrix of the homothety $\eta_{X_{n-s+1}}$ by $X_{n-s+1}$ in the
basis \linebreak $\{1,X_{n-s+1}\klk X_{n-s+1}^{\delta_s-1}\}$. Then
$M_{X_{n-s+1}}$ is the companion matrix of $M_s(X_{n-s},X_{n-s+1})$.
As $M_s(X_{n-s},T)$ is square--free, the matrix $M_{X_{n-s+1}}$ is
diagonalizable, its eigenvalues $\xi_1\klk\xi_{\delta_s}$ being the
roots of $M_s(X_{n-s},T)$ in an algebraic closure of ${k(X_{n-s})}$.
It follows that the matrix of $\eta_F=F^*(X_{n-s},\eta_{X_{n-s+1}})$
in such a basis is $M_F=F^*(X_{n-s},M_{X_{n-s+1}})$, and its
characteristic polynomial is
$\chi_F=\prod_{i=1}^{\delta_s}(T-F^*(X_{n-s},\xi_i))$. In
particular,
$$\chi_F(0)=(-1)^{\delta_s}\prod_{i=1}^{\delta_s}F^*(X_{n-s},\xi_i)=
\pm\mathrm{Res}_T\left(F^*(X_{n-s},T),M_s(X_{n-s},T)\right),$$
due to the expression for the resultant in roots.
%
%Let $(\bfs p^{s+1},\bfs x)=(\bfs p^{s+1},x_{n-s}\klk x_n)$ be a point of
%$(C_s\setminus\{\rho_s(\bfs p^{s+1},X_{n-s})=0\})\cap V(F)$. Then
%$F(x_{n-s},x_{n-s+1})=F(\bfs p^{s+1},\bfs x)=0$ and
%$M(x_{n-s},x_{n-s+1})=0$. It follows that $F(x_{n-s},T)$ and
%$M_s(x_{n-s},T)$ have a common root in $\overline{k}$ and therefore
%$\widetilde{a}(x_{n-s})=0$, which proves the second
%assertion. If $\widetilde{a}=0$, then $M_s(X_{n-s},T)$ and
%$F^*(X_{n-s},T)$ have a nontrivial common divisor, which implies
%that $(C_s\setminus\{\rho_s(\bfs p^{s+1},X_{n-s})=0\})\cap V(F)$ has
%positive dimension.
%
Thus, the constant term $\widetilde{a}$ of
$\eta_F$ is, up to a sign, as computed in step {\bf 3'}.
%, which implies that $a_F$ is the square--free part of
%$\widetilde{a}$.

To justify Algorithm \ref{algo: computation projection}, we note
that in steps {\bf 1}--{\bf 3} are performed substituting $\alpha_j$
for $X_{n-s}$
for $0\le j\le D_s$. %By Lemma \ref{lemma: properties homothety by F},
The preceding arguments show that the polynomial $\widetilde{a}\in
k[X_{n-s}]$ computed in step {\bf 3'} is the constant term of the
characteristic polynomial of the homothety
$\eta_F:k(X_{n-s})\otimes_{k[X_{n-s}]} k[C_s]\to
k(X_{n-s})\otimes_{k[X_{n-s}]} k[C_s]$ by $F$. According to
\cite[Corollary 2]{GiLeSa01}, we have $\deg \widetilde{a}\le
D_s:=\delta_s\cdot\deg F$. Consequently, $\widetilde{a}$ can be
computed via interpolation once $D_s+1$ pairwise-distinct
evaluations of $\widetilde{a}$ are obtained.

Step {\bf 1.2}, however,
requires that $\frac{\partial M_s}{\partial T}(\alpha_j,T)$ and
$M_s(\alpha_j,T)$ be relatively prime for $0\le j\le D_s$; this
condition holds if and only if $\rho_s(\alpha_j)\not=0$ for $0\le
j\le D_s$. Since $\deg\rho_s\le \delta_s^2$, the  probability that one
of $D_s$ random, non-repeated choices
of $\alpha_j$ fails is bounded above by
\begin{align*}
\sum_{j=0}^{D_s}\frac{\delta_s^2+j}{\varepsilon^{-1}(D_s+1)(\delta_s^2+D_s)-j}&\le
\sum_{j=0}^{D_s}\frac{\delta_s^2+j}{\varepsilon^{-1}((D_s+1)(\delta_s^2+D_s)-j)}
\\&\le
\varepsilon(D_s+1)\frac{\delta_s^2+D_s}{(D_s+1)(\delta_s^2+D_s)}=
\varepsilon.\end{align*}
\end{proof}

Next, we analyze the cost of Algorithm \ref{algo: computation
projection}.
\begin{lemma}\label{lemma: cost computation projection}
Suppose that $F$ is given by a straight--line program of length $L$.
Then Algorithm \ref{algo: computation projection} requires
$\mathcal{O}((L+s\,\log_2\delta_s)\,d\,\delta_s\,{\sf M}(\delta_s))$
arithmetic operations and identity tests in $k$, and
$\mathcal{O}(d\,\delta_s)$ extractions of $p$--th roots in $k$ (if
$k$ has characteristic $p>0$).
\end{lemma}
\begin{proof}
First, assume that a lucky choice of $\alpha_0\klk \alpha_{D_s}$ in
step {\bf 1.1} is made. We then perform steps {\bf 1.2}, {\bf 1.3},
and {\bf 1.4} of Algorithm \ref{algo: computation projection} for
$0\le j\le D_s$.

For each $j$, step $\mathbf{1.2}$  requires the multipoint
evaluations $M_s(\alpha_j,T)$, $({\partial M_s}/{\partial
T})(\alpha_j,T)$, and $W_{n-s+i}^s(\alpha_j,T)$ for $2\le i\le s$.
More precisely, considering $M_s$, ${\partial M_s}/{\partial T}$,
and $W_{n-s+i}^s$ as polynomials of $k[X_{n-s}][T]$, each
coefficient in $k[X_{n-s}]$ of these polynomials must be evaluated
at $\alpha_0\klk \alpha_{D_s}$. Since there are
$\mathcal{O}(s\,\delta_s)$ such polynomials, all of degree at most
$\delta_s$, these evaluations require $\mathcal{O}(s\,\delta_s\,{\sf
M}(D_s)\log_2D_s)$ arithmetic operations in $k$. Since we have
assumed that $\delta > d$, we may, without loss of generality,
suppose that $\delta_s > d$ for all $1<s \le r$. As $D_s :=
d\,\delta_s$, this assumption justifies the inclusion
$$\mathcal{O}(s\,\delta_s\,{\sf M}(D_s)\log_2D_s)\subset
\mathcal{O}(s\,d\,\delta_s\,{\sf M}(\delta_s)\log_2\delta_s).$$ This
simplification does not affect the overall asymptotic complexity of
the algorithm, since the case $\delta_s \le d$ would only lead to
smaller cost estimates.

Next, the inversion $({\partial M_s}/{\partial
T})(\alpha_j,T)^{-1}$, and the multiplication $({\partial
M_s}/{\partial T})(\alpha_j,T)^{-1} W_{n-s+i}^s(\alpha_j,T)$ for
$2\le i\le s$, modulo the monic polynomial $M_s(\alpha_j,T)$ of
degree $\delta_s$, are carried out for $0\le j\le D_s$. This step
requires an additional $\mathcal{O}(s\,d\,\delta_s\,{\sf
M}(\delta_s)\log_2\delta_s)$ arithmetic operations and identity
tests in $k$.

In step $\mathbf{1.3}$, we replace $\alpha_j$ for $X_{n-s}$, and
$V_{n-s+i}(\alpha_j,T)$ for $X_{n-s+i}$ for $2\le i\le s$, in the
straight--line program representing $F$, and perform all arithmetic
operations modulo $M_s(\alpha_j,T)$. This procedure can be performed
with $\mathcal{O}(L\,{\sf M}(\delta_s))$ arithmetic operations in
$k$ for each $j$, for a total of $\mathcal{O}(L\,d\,\delta_s\,{\sf
M}(\delta_s))$ arithmetic operations in $k$.

Step $\mathbf{1.4}$ involves computing $D_s+1$ resultants of
polynomials of degree at most $\delta_s$, requiring
$\mathcal{O}(d\,\delta_s\,{\sf M}(\delta_s)\log_2\delta_s)$
arithmetic operations and identity tests in $k$. The interpolation
step {\bf 2} recovers the dense representation of $\widetilde{a}$ at
cost
$$\mathcal{O}({\sf M}(D_s)\log_2D_s)\subset
\mathcal{O}(d\,{\sf M}(\delta_s)\log_2\delta_s).$$

Finally, step {\bf 3} computes the square--free part $a_F$ of
$\widetilde{a}$, which requires $\mathcal{O}(d\,{\sf
M}(\delta_s)\log_2\delta_s)$ arithmetic operations and identity
tests in $k$, and $\mathcal{O}(d\,\delta_s)$ extractions of $p$--th
roots in $k$ (if $k$ has characteristic $p>0$).

On the other hand, if any choice $\alpha_j$ is not lucky,
then the inversion $(\partial M_s/\partial T)(\alpha_j,T)^{-1}$
modulo $M(\alpha_j,T)$ fails, and the computation halts,
incurring similar costs.
\end{proof}

For later use, we state a few remarks
on the polynomials $F^*$ and $a_F$.
\begin{remark}\label{rem: properties F* and a_F}
Under the assumptions of Algorithm \ref{algo: computation
projection}, let $\rho_s:=\mathrm{disc}_{T}M_s(X_{n-s},T)$, and let
$(\xi,\varphi)\in\A^2$ satisfy $M_s(\xi,\varphi)=0$ and
$\rho_s(\xi)\not=0$. Then:
\begin{enumerate}
  \item There exists a
  unique point $\bfs x\in C_s$ with $F^*(\xi,\varphi)=F(\bfs x)$, $x_{n-s}=\xi$
  and $x_{n-s+1}=\varphi$.
  \item If $a_F(\xi)\not=0$, then $F(\bfs x)\not=0$.
\end{enumerate}
\end{remark}
\begin{proof}
By Lemma \ref{lemma: lifting curve birational to M_s=0},
there is a unique point $\bfs x\in C_s$ with $x_{n-s}=\xi$ and
$x_{n-s+1}=\varphi$, explicitly given by
$$
\bfs x:=\psi(\xi,\varphi):=\left(\bfs
p^{s+1},\xi,\varphi,V_{n-s+2}^s(\xi,\varphi)\klk
V_n^s(\xi,\varphi)\right).
$$
It follows that $F(\bfs x)=F(\psi(\xi,\varphi))=F^*(\xi,\varphi)$.
Moreover, if $a_F(\xi)\not=0$, then $M_s(\xi,T)$ and $F^*(\xi,T)$
have no common zeros in $\overline{k}$, implying
$F^*(\xi,\varphi)=F(\bfs x)\not=0$.
\end{proof}

We will apply Algorithm \ref{algo: computation projection} to
$F_{s+1}(\bfs p^{s+1},X_{n-s}\klk X_n)$ and $G(\bfs
p^{s+1},X_{n-s}\klk X_n)$. Prop\-o\-sition \ref{prop: C_s cap
F_(s+1) and G has dim zero} ensures that these polynomials satisfy
the input requirements of Algorithm \ref{algo: computation
projection}. Let $\chi_{F_{s+1}},\chi_{G}\in k[X_{n-s},T]$ be the
characteristic polynomials of the homotheties
$\eta_{F_{s+1}},\eta_{G}:k(X_{n-s})\otimes_{k[X_{n-s}]}k[C_s]\to
k(X_{n-s})\otimes_{k[X_{n-s}]}k[C_s]$ by $F_{s+1}(\bfs
p^{s+1},X_{n-s}\klk X_n)$ and $G(\bfs p^{s+1},X_{n-s}\klk X_n)$,
respectively. Denote by $a_{F_{s+1}},a_{G}\in k[X_{n-s}]$ the
square--free parts of the constant terms of
$\chi_{F_{s+1}},\chi_{G}$, respectively. The following lemma
summarizes their key properties. This result extends the approach of
\cite[Lemma 8]{HeMaWa01}, which addressed the case of a closed set
defined over $\Q$, to the more general setting of locally closed
sets over an arbitrary perfect field~$k$.
\begin{lemma}\label{lemma: properties a_(F_(s+1))}
The following properties hold:
\begin{enumerate}
%  \item $a_{F_{s+1}}\in k[X_{n-s}]\cap (F_1(\bfs p^{s+1},X_{n-s}\klk X_n)\klk F_{s+1}(\bfs
%p^*,X_{n-s}\klk X_n))_{G(\bfs p^{s+1},X_{n-s}\klk X_n)}$,
\item For any $\xi\in \ck$ with $(\rho_s\cdot a_{G})(\xi)\not=0$, where
$\rho_s:=\mathrm{disc}_{T}M_s(X_{n-s},T)$, we have
$a_{F_{s+1}}(\xi)=0$ if and only if there exists $\bfs x\in
\pi_{s+1}^{-1}(\bfs p^{s+1})$ with $\pi_s^*(\bfs x)=\xi$.
\item $\mathrm{gcd}(m_{s+1},\rho_s\cdot a_G)=1$.
\end{enumerate}
\end{lemma}
\begin{proof}
%According to Lemma \ref{lemma: properties homothety by F}, the
%constant term $\widetilde{a}_{F_{s+1}}\in k[X_{n-s}]$ of
%$\chi_{F_{s+1}}$ belongs to the ideal generated by $I(C_s)$ and
%$F_{s+1}$. It follows that
%%
%$$\widetilde{a}_{F_{s+1}}\in (F_1(\bfs p^{s+1},X_{n-s}\klk X_n)\klk
%F_{s+1}(\bfs p^{s+1},X_{n-s}\klk X_n))_{G(\bfs p^{s+1},X_{n-s}\klk X_n)}.$$
%%
%Taking into account that this ideal is radical, the first assertion
%follows.

%Now, observe that,
For the first assertion, let $\xi\in \ck$ satisfy $(\rho_s\cdot
a_{G})(\xi)\not=0$. If $a_{F_{s+1}}(\xi)=0$, then, since
$a_{F_{s+1}}$ is the square-free part of the resultant of
$M_s(X_{n-s},T)$ and $F_{s+1}^*(X_{n-s},T)$ with respect to $T$,
there exists $\varphi\in\ck$ with
$M_s(\xi,\varphi)=F_{s+1}^*(\xi,\varphi)=0$. By Remark \ref{rem:
properties F* and a_F}, there is a unique $\bfs x\in C_s$ with
$x_{n-s}=\xi$ and $x_{n-s+1}=\varphi$, such that $F_{s+1}(\bfs
x)=F_{s+1}^*(\xi,\varphi)=0$. That same remark also ensures that
$G(\bfs x)\not=0$. It follows that $\bfs x\in C_s\cap
V(F_{s+1})\setminus V(G)$, that is, $\bfs x\in \pi_{s+1}^{-1}(\bfs
p^{s+1})$.

Conversely, if $\bfs x\in
\pi_{s+1}^{-1}(\bfs p^{s+1})$ satisfies $\pi_s^*(\bfs x)=\xi$, then, since
$a_{F_{s+1}}\in k[X_{n-s}]$ vanishes on $C_s\cap V(F_{s+1})\supset
\pi_{s+1}^{-1}(\bfs p^{s+1})$ by Lemma \ref{lemma: properties homothety
by F}, we deduce that $a_{F_{s+1}}(\xi)=0$. This proves the first
assertion.

For the second assertion, suppose that $m_{s+1}$ and
$\rho_s$ have a nontrivial common factor in $k[X_{n-s}]$. Then
there exists $\xi\in\ck$ such that $\bfs q:=(\bfs
p^{s+1},\xi)\in\pi_s\big(\pi_{s+1}^{-1}(\bfs p^{s+1})\big)$
and $\rho_s(\xi)=0$. %Let $\bfs q\in \pi_{s+1}^{-1}(\bfs p^{s+1})$ be such that
%$\pi_s(\bfs q)=(\bfs p^{s+1},\xi)$.
By condition $({\sf B}_4)$, $\bfs q$ is a lifting point of $\pi_s$ and
$X_{n-s+1}$ separates the points of $\pi_s^{-1}(\bfs q)$. This
implies $\rho_s(\xi)=\mathrm{Disc}_T(M_s(\xi,T))\not=0$,
contradicting our assumption. We conclude that $m_{s+1}$ and
$\rho_s$ cannot share nontrivial common factors in $k[X_{n-s}]$.

On the other hand, if there exists $\xi\in\ck$ with
$m_{s+1}(\xi)=a_{G}(\xi)=0$, then $\bfs q:=(\bfs
p^{s+1},\xi)\in\pi_s\big(\pi_{s+1}^{-1}(\bfs p^{s+1})\big)$,
and since $\gcd(m_{s+1},\rho_s)=1$, we have $\rho_s(\xi)\not=0$, while
$a_G(\xi)=0$. Denote $G^*(X_{n-s},T):=
G(X_{n-s},T,V_{n-s+2}^s(X_{n-s},T)\klk V_n^s(X_{n-s},T))$. Since
$a_G(\xi)=0$, there exists $\varphi\in\overline{k}$ with
$M_s(\xi,\varphi)=G^*(\xi,\varphi)=0$. By Remark \ref{rem:
properties F* and a_F}, there is $\bfs x\in\pi_s^{-1}(\bfs q)$
with $G^*(\xi,\varphi)=G(\bfs x)=0$. Hence, $\bfs x\in
C_s\cap\{G=0\}$, and thus $\bfs q\in
\pi_s\big(C_s\cap\{G=0\}\big)$ lies in
$\pi_s\big(\pi_{s+1}^{-1}(\bfs p^{s+1})\big)$, which contradicts
condition $({\sf B}_5)$.
\end{proof}
%
%----------------------------------------------------------------------
%----------------------------------------------------------------------
%----------------------------------------------------------------------
%----------------------------------------------------------------------
%
\subsection{A Kronecker representation of the $(s+1)$th lifting fiber}
Algorithm \ref{algo: computation projection} serves as the foundation
for computing a Kronecker representation of the lifting fiber
$\pi_{s+1}^{-1}(\bfs p^{s+1})$.
%
%Next we describe the procedure which computes a geometric solution
%of the zero--dimensional fiber $\pi_{s+1}^{-1}(\bfs p^{s+1})$.

This procedure relies on the following subroutine, which describes
how to compute a Kronecker representation of a zero--dimensional
$k$--variety $\mathcal{V}\subset\A^2$ from the minimal polynomials
$m,m_1,m_2\in k[T]$ of three separating linear forms
$\mathcal{L},\mathcal{L}_1,\mathcal{L}_2\in k[X,Y]$ for
$\mathcal{V}$.
\begin{algorithm}[Bidimensional shape lemma]
\label{algo: bidimensional shape lemma} ${}$
\begin{itemize}
\item[]{\bf Input:}
\begin{enumerate}
\item Linear forms $\mathcal{L}:=X,\mathcal{L}_1:=X+\lambda_1Y,
\mathcal{L}_2:=X+\lambda_2 Y\in k[X,Y]$ that separate the points of
a zero--dimensional $k$-variety $\mathcal{V}\subset\A^2$, with
$\mathcal{L}_2$ separating the points of
$\{m(X)=0,m_1(\mathcal{L}_1)=0\}$, \item The minimal polynomials
$m,m_1,m_2\in k[T]$ of $\mathcal{L},\mathcal{L}_1,\mathcal{L}_2$ in
$\mathcal{V}$.
\end{enumerate}

\item[]{\bf Output:} A polynomial $v\in k[T]$ with $\deg v<\deg\mathcal{V}$ and
  $Y-v(X)\equiv 0$ in $\mathcal{V}$.

\item[]{\bf 1}\ Compute the dense representation of
$m_i(X+\lambda_i Y)\in k[X,Y]$ for $i=1,2$.
\item[]{\bf 2} Compute $\mathrm{gcd}\big(m_1(X+\lambda_1
Y),m_2(X+\lambda_2
Y)\big)=Y-v(X)$ in $\big(k[X]/(m(X))\big)[Y]$\footnote{
In this context, by the \emph{(monic) gcd} of two polynomials $f_1,f_2$ in $(k[X]/(m(X)))[Y]$,
with $m\in k[X]$ monic, reduced, but not necessarily irreducible, we mean any monic generator of the ideal
$(f_1,f_2)\subset (k[X]/(m(X)))[Y]$ (see \cite{MoRi95}).
%Although $(k[X]/(m(X)))[T]$ is not, in general, a principal ideal ring when $m$ is reducible,
%the gcd is well defined componentwise through the Chinese Remainder Theorem applied to the
%factorization $m=\prod_i m_i$ into irreducible factors.
%In this setting, the computation can be performed deterministically by applying the Euclidean algorithm
%coefficientwise modulo $m$, since each component ring $k[X]/(m_i^{e_i})$ is a local principal ideal ring.
%This approach is standard in computer algebra and is discussed, for example, in
%\cite[Sections~6.4--6.5]{GaGe99} and \cite[Sections~2.3 and~6.4]{GaGe99}.
%
%The existence and computation of a (monic) gcd in the ring
%\((k[x]/(m(x)))[T]\), where \(m(x)\in k[x]\) need not be irreducible,
%can be supported by the work of Moreno Maza and Rioboo
%\cite{MoRi95}. Their paper studies polynomial gcd computations
%over towers of algebraic extensions and, in particular, treats the
%case where the coefficient ring is (non-canonically) isomorphic to a
%product of fields. The approach reduces the gcd computation to
%componentwise gcds over the factor fields, followed by an assembly
%step. We note that this setting differs from the more general theory
%of gcds over arbitrary regular rings; the result of Moreno Maza and
%Rioboo is precisely tailored to coefficient rings arising from
%factorizations of the modulus polynomial \(m(x)\) and hence provides
%the adequate justification for treating gcds in
%\((k[x]/(m(x)))[T]\) as computable by componentwise methods.
%
}.
\end{itemize}
\end{algorithm}

We remark that Step~{\bf 2} will be implemented using \emph{dynamic
evaluation}. This approach allows one to compute the greatest common
divisor in $\big(k[X]/(m(X))\big)[Y]$ without performing explicit
factorizations, by exploiting the decomposition of $k[X]/(m(X))$ as
a product of fields. Adopting dynamic evaluation notably simplifies
the analysis of the success probability of Algorithm~6
in~\cite{HeMaWa01}, and avoids the factorization stage required in
the algorithm underlying Lemma~4.5 of~\cite{CaMa06a}, which was
restricted to the case of finite base fields.

The following lemma justifies the correctness of Algorithm
\ref{algo: bidimensional shape lemma} (compare with \cite[Lemma
6]{HeMaWa01}).
\begin{lemma}
With the notations of Algorithm \ref{algo: bidimensional shape
lemma}, we have the following assertions:
\begin{itemize}
  \item There exists a unique $v\in k[T]$ with $\deg
  v<\#\mathcal{V}$ such that $Y\equiv v(X)$ in $\mathcal{V}$;
  \item Let $\alpha\in\ck$ be a root of $m$, $r\in k[T]$ the
  irreducible factor of $m$ vanishing on $\alpha$, and ${\sf
  K}:=k[T]/(r(T))\cong k(\alpha)$. Then
$$\mathrm{gcd}\big(m_1(\alpha+\lambda_1 Y),m_2(\alpha+\lambda_2 Y)\big)
=Y-v(\alpha)\in {\sf
  K}[T].$$
\end{itemize}
\end{lemma}
\begin{proof}
The first assertion follows from the classical shape lemma.

Let $\alpha\in\ck$ be a root of $m$. Then there exist a unique
$\beta\in\ck$ such that $(\alpha,\beta)\in\mathcal{V}$. Observe that
$m_1(\alpha+\lambda_1\beta)=m_2(\alpha+\lambda_2\beta)=0$, implying
that $Y-\beta=Y-v(\alpha)$ divides the nonzero polynomial
$m_i(\alpha+\lambda_iY)$ for $i=1,2$. We claim that
$$\mathcal{V}=\{(x,y)\in\A^2:m(x)=0,m_1(x+\lambda_1 y)=0,m_2(x+\lambda_2 y)=0\}.$$
Define $\mathcal{W}$ as the set in right--hand side. The
inclusion $\mathcal{V}\subset\mathcal{W}$ is clear since
$m(X)$, $m_1(X+\lambda_1 Y)$ and $m_2(X+\lambda_2 Y)$
vanish on $\mathcal{V}$. Conversely, since $\mathcal{L}_2$
separates the points of
$\mathcal{W}_1:=\{m(X)=0,m_1(\mathcal{L}_1)=0\}$, the linear from
$\mathcal{L}_2$ maps injectively $\mathcal{W}_1$ to $\A^1$.
Restricting $\mathcal{L}_2$ to $\mathcal{V}$, the image is
$\{\xi\in\ck:m_2(\xi)=0\}$, showing that
$\mathcal{V}=\mathcal{W}$.
%for
%$(x,y)\in\mathcal{W}$, the fact that $m(x)=0$ implies that there
%exists $y'\in\ck$ with $(x,y')\in\mathcal{V}$. In particular, the
%points $(x,y)$ and $(x,y')$ belong to the variety
%$\{m(X)=0,m_1(\mathcal{L}_1)=0\}$. equalities $m(x)=0$ and
%$m_1(x+\lambda_1y)=m_1(x+\lambda_1y')=0$ hold.

From this, we deduce that $\beta$ is the only common root in
$\ck$ of $m_1(\alpha+\lambda_1Y)$ and $m_2(\alpha+\lambda_2Y)$,
completing the proof.
\end{proof}

Next, we analyze the computational cost of Algorithm \ref{algo:
bidimensional shape lemma}.
\begin{lemma}\label{lemma: cost bidimensional shape lemma}
Algorithm \ref{algo: bidimensional shape lemma} can be performed using
$\mathcal{O}({\sf M}(\#\mathcal{V})^2\log_2(\#\mathcal{V}))$ arithmetic
operations and equality tests in $k$.
\end{lemma}
\begin{proof}
The dense representation of $m_i(X+\lambda_i Y)\in k[X,Y]$ for
$i=1,2$ can be obtained by interpolation. More precisely, consider
the ring homomorphism $\sigma:k[X,Y]\to k[X]$ defined by the
Kronecker substitution $\sigma(X):=X$,
$\sigma(Y):=X^{\#\mathcal{V}+1}$. Any $F\in k[X,Y]$ of degree at
most $\#\mathcal{V}$ in each variable $X$, $Y$ is uniquely
determined by its image $\sigma(F)$ (without computational cost). To
compute the dense representation of $\sigma(m_i(X+\lambda_i
Y))=m_i(X+\lambda_i X^{\#\mathcal{V}+1})$, perform multipoint
evaluations $m_i(\alpha_j+\lambda_i\, \alpha_j^{\#\mathcal{V}+1})$
at $N:=\#\mathcal{V}(\#\mathcal{V}+1)+1$ distinct points
$\alpha_1\klk\alpha_N\in k$ (or in a finite field extension of $k$
if $k$ lacks enough elements). Then interpolation reconstructs
$m_i(X+\lambda_i X^{\#\mathcal{V}+1})$, and thus $m_i(X+\lambda_i
Y)$.

The multipoint evaluations can be performed with
$$\mathcal{O}({\sf M}(N)\log_2N)=\mathcal{O}({\sf
M}((\#\mathcal{V})^2)\log_2(\#\mathcal{V}))$$ arithmetic operations in
$k$, and within a similar cost the interpolation step is
completed.

To compute $$\mathrm{gcd}\big(m_1(X+\lambda_1 Y),m_2(X+\lambda_2
Y)\big)=Y-v(X)$$ in $\big(k[X]/(m(X))\big)[Y]$, we use dynamical
evaluation (see \cite{DeDiDu85}): we perform the Euclidean algorithm
in $k[X][Y]$, computing each arithmetic
operation in $k[X]$ modulo $m(X)$. Each test for equality to zero
eventually induces a splitting of the modulo, leading to a
splitting of the computation. Ultimately, we obtain polynomials
$v_1\klk v_t,m_1\klk m_t\in k[X]$ such that $m=m_1\cdots m_t$, $\deg
v_i<\deg m_i$ for $1\le i\le t$ and
$\mathrm{gcd}\big(m_1(X+\lambda_1 Y),m_2(X+\lambda_2
Y)\big)=Y-v_i(X)$ in $\big(k[X]/(m_i(X))\big)[Y]$  for $1\le i\le
t$. The final polynomial $v\in k[X]$ is obtained
via Chinese remaindering. By \cite[Proposition
4.1]{DaMoScXi06}, this step requires $\mathcal{O}({\sf
M}(\#\mathcal{V})^2\log_2(\#\mathcal{V}))$ arithmetic operations and
equality tests in $k$.

Noting that ${\sf M}((\#\mathcal{V})^2)\le {\sf
M}(\#\mathcal{V})^2$, the lemma follows.
\end{proof}

Now we are in position to describe the procedure for computing a
Kronecker representation of the fiber $\pi_{s+1}^{-1}(\bfs p^{s+1})$.
According to Lemma \ref{lemma: properties a_(F_(s+1))}, the roots of
$a_{F_{s+1}}$ are the $(n-s)$th coordinates of the points of
$\pi_{s+1}^{-1}(\bfs p^{s+1})$, except for those corresponding to roots
of $\rho_s$ or $a_G$. Therefore, by removing from $a_{F_{s+1}}$ the
roots of these polynomials, we obtain the minimal polynomial
$m_{s+1}$ of $X_{n-s}$ in $\pi_{s+1}^{-1}(\bfs p^{s+1})$.

Through a deformation process, we compute the minimal polynomials of
two linear forms $\mathcal{L}_1:=X_{n-s}+\lambda_1X_{n-s+1},
\mathcal{L}_2:=X_{n-s}+\lambda_2X_{n-s+1}\in k[X_{n-s},X_{n-s+1}]$.
Using these polynomials, Algorithm \ref{algo: bidimensional shape
lemma} computes the polynomial $v_{n-s+1}^{s+1}\in k[T]$
parametrizing $X_{n-s+1}$ in terms of $X_{n-s}$.

\begin{algorithm}[Computation of $m_{s+1}$ and $v_{n-s+1}^{s+1}$] ${}$
\label{algo: comput m_(s+1) and v_(n-s+1)}
\begin{itemize}
\item[]{\bf Input:}
\begin{enumerate}
\item
A Kronecker representation of $C_s$, as in
the output of Algorithm \ref{algo: Newton lifting};
\item Access to
uniformly random elements of a finite set $\mathcal{S}\subset k$ of
size at least $2\varepsilon^{-1}\delta^4$, %(or of a finite field
%extension of $k$ if $k$ does not have enough points),
where $0<\varepsilon<1$.\end{enumerate}

\item[]{\bf Output:} Polynomials $m_{s+1},v_{n-s+1}^{s+1}\in k[T]$ with the
following properties:
\begin{enumerate}
  \item $\deg m_{s+1}=\delta_{s+1}$ and $\deg
  v_{n-s+1}^{s+1}<\delta_{s+1}$;
  \item $m_{s+1}(X_{n-s})\equiv 0$ and
  $X_{n-s+1}-v_{n-s+1}^{s+1}(X_{n-s})\equiv 0$ in $\pi_{s+1}^{-1}(\bfs
  p^{s+1})$.
\end{enumerate}

\item[]{\bf 1} Apply Algorithm \ref{algo: computation projection} to
$F_{s+1}(\bfs p^{s+1},X_{n-s}\klk X_n)$ to compute $a_{F_{s+1}}\in
k[X_{n-s}]$.
\item[]{\bf 2} Apply Algorithm \ref{algo: computation projection} to
$G(\bfs p^{s+1},X_{n-s}\klk X_n)$ to compute $a_{G}\in k[X_{n-s}]$.
\item[]{\bf 3} Compute $\rho_s:=\mathrm{disc}_TM_s(X_{n-s},T)$
modulo $a_{F_{s+1}}$. {\em ***[$\rho_s\in k[X_{n-s}]\setminus\{0\}$]***}
\item[]{\bf 4} Compute $m_{s+1}(X_{n-s}):=a_{F_{s+1}}/
\mathrm{gcd}(a_{F_{s+1}},a_{G}\cdot\rho_s)$.
\item[]{\bf 5} For $i=1,2$, do
\begin{itemize}
  \item[]{\bf 5.1} Choose $\lambda_i\in\mathcal{S}$ at random;
  \item[]{\bf 5.2} Set $\mathcal{L}_i:=X_{n-s}+\lambda_i X_{n-s+1}$;
  \item[]{\bf 5.3} Perform the change of variables
  $X_{n-s}\to \mathcal{L}_i-\lambda_i X_{n-s+1}$.
%  $M_{\lambda_i}:=
%M_s(X_{n-s},\mathcal{L}_i-\lambda_iX_{n-s})\in
%k[X_{n-s},\mathcal{L}_i]$;
  \item[]{\bf 5.4} Apply steps {\bf 1}--{\bf 4} to compute
the minimal polynomial $m_{\lambda_i}\in k[\mathcal{L}_i]$ of
$\mathcal{L}_i$ in $\pi_{s+1}^{-1}(\bfs p^{s+1})$.
%:=\mathrm{Res}_{X_{n-s}}(M_{\lambda_i}(X_{n-s},\mathcal{L}_i),m_{s+1}(X_{n-s}))
%\in k[\mathcal{L}_i]$ and clean squares.
\end{itemize}
\item[]{\bf 6} Apply Algorithm \ref{algo: bidimensional shape lemma}
to compute the polynomial $v_{n-s+1}^{s+1}$.
\end{itemize}

\end{algorithm}

\begin{remark}\label{rem: k definability comput m_(s+1) and v_(n-s+1)}
We note that even if the choices in step {\bf 5.2} are made in a
proper finite field extension of $k$, the output remains defined
over $k$. Indeed, $m_{s+1}$ is defined over $k$ (see Remark
\ref{rem: k definability algo projection}), and $v_{n-s+1}^{s+1}$ is
defined over $k$, because it parametrizes $X_{n-s+1}$ in terms of
$X_{n-s}$ in $\pi_{s+1}^{-1}(\bfs p^{s+1})$, which is itself defined
over $k$.
\end{remark}

\begin{lemma}\label{lemma: correctness comput m_(s+1) and v_(n-s+1)}
Assume that $k$ has cardinality at least at least
$2\varepsilon^{-1}\delta^4$. For a lucky choice of
$\lambda_1,\lambda_2$ in step {\bf 5.1}, Algorithm \ref{algo: comput
m_(s+1) and v_(n-s+1)} works as specified. Such a choice is
successful with probability at least $1-\varepsilon$.
\end{lemma}
\begin{proof}
Steps {\bf 1} and {\bf 2} can be carried out by Proposition
\ref{prop: C_s cap F_(s+1) and G has dim zero}, and their outputs
are the polynomials $a_{F_{s+1}}$ and $a_G$, respectively, due to
Lemma \ref{lemma: correctness algo computation projection}. Step
{\bf 4} yields the polynomial $m_{s+1}(X_{n-s})$ by Lemma
\ref{lemma: properties a_(F_(s+1))}. A random choice of
$\lambda_1,\lambda_2\in\mathcal{S}$ produces linear forms
$\mathcal{L}_i:=X_{n-s}+\lambda_i X_{n-s+1}$ ($i=1,2$) satisfying
the input requirements of Algorithm \ref{algo: bidimensional shape
lemma}. Thus, %the input specification of Algorithm \ref{algo:
%bidimensional shape lemma} is met, and therefore
the polynomial $v_{n-s+1}^{s+1}$ is obtained in step {\bf 6}.

For the probability estimate, for $i\in\{1,2\}$ and
$\lambda_i\in\mathcal{S}$, the linear form
$\mathcal{L}_i:=X_{n-s}+\lambda_iX_{n-s+1}$ must separate the points
of $(s+1)$th lifting fiber $\pi_{s+1}^{-1}(\bfs p^{s+1})$. Moreover,
$\mathcal{L}_2$ must separate the points of the zero-dimensional
variety
$$\mathcal{V}:=\{m_{s+1}(X_{n-s})=0,m_1(\mathcal{L}_1)=0\}\subset\A^2.$$
We observe that $\mathcal{L}_i$ separates the points of
$\pi_{s+1}^{-1}(\bfs p^{s+1})$ if and only if
$$\prod_{\bfs \xi\not=\bfs\psi}\big(\mathcal{L}_i(\bfs\xi)
-\mathcal{L}_i(\bfs\psi)\big)\not=0,$$
where the product ranges over all distinct pairs $\bfs\xi,\bfs\psi$
in $\pi_{s+1}^{-1}(\bfs p^{s+1})$. This expression is a nonzero
polynomial in $\lambda_i$ of degree ${\delta_{s+1}\choose 2}$.
Furthermore, $\mathcal{L}_2$ must separate the points of
$\mathcal{V}$, which has degree at most $\delta_{s+1}^2$. Therefore,
this condition can be expressed as the nonvanishing of a polynomial
in $\lambda_2$ of degree at most ${\delta_{s+1}^2\choose 2}$. Since
$\lambda_1,\lambda_2$ are randomly chosen in a set of cardinality
$\varepsilon^{-1}\delta^4$, the probability that any choice fails is
bounded by
$$\frac{{\delta_{s+1}\choose 2}}{
\varepsilon^{-1}\delta^4}+ \frac{{\delta_{s+1}\choose 2}+
{\delta_{s+1}^2\choose 2}}{
\varepsilon^{-1}\delta^4}\le\varepsilon,$$
completing the proof.
\end{proof}

It remains to analyze the cost and probability of success of
Algorithm \ref{algo: comput m_(s+1) and v_(n-s+1)}.
\begin{lemma}
 \label{lemma: cost comput m_(s+1) and v_(n-s+1)}
Algorithm \ref{algo: comput m_(s+1) and v_(n-s+1)} can be
implemented to run using $ \mathcal{O}\big((L+s\,\log_2\delta)\,{\sf
M}(d\delta)\,{\sf M}(\delta) \big)$
arithmetic operations in $k$,
$\mathcal{O}\big((L+s\,\log_2\delta)\,{\sf M}(\delta)\,{\sf
M}(d\delta)\big)$ identity tests in $k$ and $\mathcal{O}(\delta)$
extractions of $p$--th roots in $k$ (if $k$ has characteristic
$p>0$). It succeeds with probability $1-7\varepsilon$.
\end{lemma}
\begin{proof}
From the straight--line program computing $F_{s+1}$ and $G$, we
readily obtain a straight--line program of length $L$ that computes
$F_{s+1}(\bfs p^{s+1},X_{n-s}\klk X_n)$ and $G(\bfs p^{s+1},X_{n-s}\klk
X_n)$. By Lemma \ref{lemma: cost computation projection}, steps {\bf
1} and {\bf 2} can be performed with
$\mathcal{O}((L+s\,\log_2\delta_s)d\,\delta_s\,{\sf M}(\delta_s))$
arithmetic operations and identity tests in $k$, and
$\mathcal{O}(\delta_s)$ extractions of $p$--th roots in $k$ (if $k$
has characteristic $p>0$).

To compute the dense representation of $\rho_s$ modulo $a_{F_{s+1}}$
we calculate the discriminant of $M_s(X_{n-s},T)$ over the ring
$k[X_{n-s}]/\big(a_{F_{s+1}}(X_{n-s})\big)$. This is done using
dynamical evaluation, as in step {\bf 2} of Algorithm \ref{algo:
bidimensional shape lemma}. Specifically, we compute the
discriminant of $M_s(X_{n-s},T)$ with respect to $T$, carrying out
arithmetic operations in $k[X_{n-s}]$ modulo $a_{F_{s+1}}$. Each
equality test eventually leads to a splitting of the computation,
followed by a step of Chinese remaindering to reconstruct the dense
representation of $\rho_s$ modulo $a_{F_{s+1}}$, as in the proof of
Lemma \ref{lemma: cost bidimensional shape lemma}.

Computing $\mathrm{disc}_TM_s(X_{n-s},T)$ modulo $a_{F_{s+1}}$
requires $\mathcal{O}\big({\sf M}(\delta_s)\log_2\delta_s\big)$
arithmetic operations and equality tests in
$k[X_{n-s}]/\big(a_{F_{s+1}}(X_{n-s})\big)$. Since $a_{F_{s+1}}$ has
degree at most $d\delta_s$, each arithmetic operation modulo
$a_{F_{s+1}}$ costs $\mathcal{O}\big({\sf M}(d\delta_s)\big)$
arithmetic operations in $k$. Therefore, this step can be completed
using $\mathcal{O}\big({\sf M}(\delta_s){\sf
M}(d\delta_s)\log_2\delta_s\big)$ arithmetic operations and equality
tests in $k$.

Step {\bf 4} requires $\mathcal{O}\big({\sf
M}(d\delta_s)\log_2\delta_s\big)$ arithmetic operations and identity
tests in $k$. Thus, the total cost for steps {\bf 1}--{\bf 4} is
$$\mathcal{O}\big((L+s\,\log_2\delta_s)d\,\delta_s\,{\sf M}(\delta_s)
+{\sf M}(\delta_s){\sf M}(d\delta_s)\log_2\delta_s\big)\subset
\mathcal{O}\big((L+s\,\log_2\delta_s)\,{\sf M}(d\delta_s)\,{\sf M}(\delta_s)
\big)$$
arithmetic operations,
$\mathcal{O}\big((L+s\,\log_2\delta_s)d\,\delta_s\,{\sf
M}(\delta_s)+{\sf M}(\delta_s)\,{\sf
M}(d\delta_s)\log_2\delta_s\big)$ identity tests in $k$, and
$\mathcal{O}(\delta_s)$ extractions of $p$--th roots in $k$ (if $k$
has characteristic $p>0$).

In step {\bf 5}, we compute the dense representations of
$M_s\big(\mathcal{L}_i-\lambda_iX_{n-s+1},X_{n-s+1}\big)$ for
$i=1,2$. We apply a similar process as for the computation of the
dense representation of $m_i(X+\lambda_i Y)$ in step {\bf 1} of
Algorithm \ref{algo: bidimensional shape lemma}, which requires
$\mathcal{O}({\sf M}(\delta_s)^2\log_2\delta_s)$ arithmetic
operations and equality tests in $k$. We then apply steps {\bf
1}--{\bf 4} again for $i=1,2$, with the same asymptotic cost.

Finally, the cost of step {\bf 6} is bounded by Lemma \ref{lemma:
cost bidimensional shape lemma}, with $\#\mathcal{V}=\delta_{s+1}$.
Combining the costs of all step yields the claimed complexity.

For the probability of success, failure can occur during calls to
Algorithms \ref{algo: computation projection} and \ref{algo:
bidimensional shape lemma}. Algorithm \ref{algo: computation
projection} is called in steps {\bf 1}, {\bf 2}, and twice in step
{\bf 5.4} for each $i\in\{1,2\}$. With the input conditions of
Algorithm \ref{algo: computation projection}, each invocation fails
with probability at most $\varepsilon$, so the overall failure
probability is at most $6\varepsilon$. Additionally, Algorithm
\ref{algo: bidimensional shape lemma} is called in step {\bf 6} and
may fail with probability at most $\varepsilon$. This proves the
lemma.
\end{proof}

Finally, it remains to parametrize the variables $X_{n-s+2}\klk X_n$
in terms of $X_{n-s}$ in $\pi_{s+1}^{-1}(\bfs p^{s+1})$. Recall that
$$\frac{\partial M_s}{\partial T}(X_{n-s},X_{n-s+1})\cdot
  X_{n-s+i}\equiv W_{n-s+i}^s(X_{n-s},X_{n-s+1})$$
holds on $C_s$, and thus on $\pi_{s+1}^{-1}(\bfs p^{s+1})$, for $2\le
i\le s$. We have the following result.
\begin{lemma}\label{lemma: h_s coprime with m_(s+1)}
 With the notations above, the polynomial
$$h_s:=\frac{\partial M_s}{\partial T}(X_{n-s},v_{n-s+1}^{s+1}(X_{n-s}))$$
is coprime with $m_{s+1}(X_{n-s})$.
\end{lemma}
\begin{proof}
Since $M_s(X_{n-s},X_{n-s+1})$ vanishes on $C_s$, it also vanishes
on $\pi_{s+1}^{-1}(\bfs p^{s+1})$. Therefore,
$M_s(X_{n-s},v_{n-s+1}^{s+1}(X_{n-s}))$ vanishes on
$\pi_{s+1}^{-1}(\bfs p^{s+1})$, and thus it is divisible by
$m_{s+1}(X_{n-s})$. Moreover, since $\rho_s$ belongs to the ideal of
$k[X_{n-s},T]$ generated by $M_s(X_{n-s},T)$ and $(\partial
M_s/\partial T)(X_{n-s},T)$, we have
\begin{align*}
\rho_s(X_{n-s})&=b(X_{n-s})\frac{\partial M_s}{\partial
T}(X_{n-s},v_{n-s+1}^{s+1}(X_{n-s}))+ c(X_{n-s})
M_s(X_{n-s},v_{n-s+1}^{s+1}(X_{n-s}))\\
&\equiv b(X_{n-s})\frac{\partial M_s}{\partial
T}(X_{n-s},v_{n-s+1}^{s+1}(X_{n-s}))\mod m_{s+1}(X_{n-s}),
\end{align*}
for some $b,c\in k[X_{n-s}]$. If $m_{s+1}(X_{n-s})$ and $h_s$ shared
a nontrivial common factor in $k[X_{n-s}]$, then $m_{s+1}(X_{n-s})$
and $\rho_s(X_{n-s})$ would also share a nontrivial common factor,
contradicting Lemma \ref{lemma: properties a_(F_(s+1))}.
\end{proof}

As a consequence, there exists a well--defined inverse $g_s\in k[T]$
of $h_s$ modulo $m_{s+1}(T)$. Therefore, for $2\le i\le s$, the
following holds on $\pi_{s+1}^{-1}(\bfs p^{s+1})$:
\begin{equation}\label{eq: equality param X_(n-s+i)}
X_{n-s+i}\equiv g_s(X_{n-s})\cdot
W_{n-s+i}^s(X_{n-s},v_{n-s+1}^{s+1}(X_{n-s})). \end{equation}
Reducing the right--hand side modulo $m_{s+1}(X_{n-s})$ for $2\le
j\le s$ yields the remaining polynomials $v_{n-s+2}^{s+1}\klk
v_n^{s+1}\in k[T]$, completing the Kronecker representation of
$\pi_{s+1}^{-1}(\bfs p^{s+1})$.
\begin{algorithm}[Conclusion of the recursive step] ${}$
\label{algo: conclusion recursive step}
\begin{itemize}
\item[]{\bf Input:}
\begin{enumerate}
\item A Kronecker representation of $C_s$, as provided by the output of
Algorithm \ref{algo: Newton lifting}; \item Polynomials
$m_{s+1},v_{n-s+1}^{s+1}\in k[T]$, as provided by the output of
Algorithm \ref{algo: comput m_(s+1) and v_(n-s+1)}.\end{enumerate}

\item[]{\bf Output:} Polynomials $v_{n-s+2}^{s+1}\klk v_n^{s+1}\in k[T]$ satisfying the
following properties:
\begin{enumerate}
  \item $\deg v_{n-s+j}^{s+1}<\delta_{s+1}$ for $2\le j\le s$;
  \item
  $X_{n-s+i}-v_{n-s+i}^{s+1}(X_{n-s})\equiv 0$ in $\pi_{s+1}^{-1}(\bfs
  p^*)$  for $2\le i\le s$.
\end{enumerate}

\item[]{\bf 1} Compute $h_s:=
\frac{\partial M_s}{\partial T}(X_{n-s},v_{n-s+1}^{s+1}(X_{n-s}))
\mod m_{s+1}(X_{n-s})$.
\item[]{\bf 2} Compute $g_s:=h_s^{-1}\mod m_{s+1}(X_{n-s})$.
\item[]{\bf 3} For $i=2\klk s$, do
\begin{itemize}
  \item[]{\bf 3.1} Compute $w_{n-s+i}^{s+1}:=W_{n-s+i}^s(X_{n-s},
  v_{n-s+1}^{s+1}(X_{n-s}))\mod m_{s+1}(X_{n-s})$;
  \item[]{\bf 3.2} Compute
$v_{n-s+i}^{s+1}:=g_s(X_{n-s})\cdot w_{n-s+i}^{s+1}(X_{n-s})\mod
m_{s+1}(X_{n-s})$.
\end{itemize}
\end{itemize}
\end{algorithm}

\begin{remark}
 \label{rem: correctness algo conclusion recursive step}
The correctness of Algorithm \ref{algo: conclusion recursive step}
follows directly from Lemma \ref{lemma: h_s coprime with m_(s+1)}
and \eqref{eq: equality param X_(n-s+i)}.
\end{remark}

We analyze the cost of Algorithm \ref{algo: conclusion recursive
step} in the next lemma.
\begin{lemma}\label{lemma: cost conclusion recursive step}
Algorithm \ref{algo: conclusion recursive step} can be executed
using $\mathcal{O}\big(s\,\delta\,{\sf M}(\delta)\big)$ arithmetic
operations in $k$, and $\mathcal{O}({\sf M}(\delta))$ identity tests
in $k$.
\end{lemma}
\begin{proof}
Considering $\frac{\partial M_s}{\partial T}(X_{n-s},T)$ as an
element of $k[X_{n-s}][T]$, its evaluation at
$T=v_{n-s+1}^{s+1}(X_{n-s})$ can be carried out using the well-known
Horner scheme with $\mathcal{O}(\delta_s)$ arithmetic operations in
$k[X_{n-s}]$. Thus, computing $h_s$ in step {\bf 1} requires
$\mathcal{O}(\delta_s)$ arithmetic operations in $k[X_{n-s}]$ modulo
$m_{s+1}(X_{n-s})$, amounting to $\mathcal{O}\big(\delta_s{\sf
M}(\delta_{s+1})\big)$ arithmetic operations in $k$. Similarly, each
polynomial $w_{n-s+i}^{s+1}$ in step {\bf 3.1} can be computed
within the same cost.

The modular inversion in step {\bf 2} requires $\mathcal{O}\big({\sf
M}(\delta_{s+1})\big)$ arithmetic operations in $k$ and identity
tests in $k$. Finally, step {\bf 3.2} costs $\mathcal{O}\big(s\,{\sf
M}(\delta_{s+1})\big)$ arithmetic operations in $k$. The lemma then
follows immediately.
\end{proof}
%
%----------------------------------------------------------------------
%----------------------------------------------------------------------
%----------------------------------------------------------------------
%----------------------------------------------------------------------
%----------------------------------------------------------------------
%----------------------------------------------------------------------
%----------------------------------------------------------------------
%----------------------------------------------------------------------
%
\section{The whole algorithm and its complexity analysis}
\label{sec: the whole algorithm}
Let $F_1\klk F_r,G$ $(r\le n)$ be polynomials in $k[X_1\klk X_n]$ of
degree at most $d$, such that $F_1,\ldots,F_r$ form a reduced
regular sequence on the open set $\{G\not=0\}$, that is, satisfying
hypotheses (${\sf H}_1$)--(${\sf H}_2$) of Section \ref{section:
locally closed sets}. Denote by $V_s\subset\A^n$ the Zariski closure
of $\{F_1=0\klk F_s=0,G\not=0\}$ for $1\le s\le r$, and let
$V:=V_r$. Recall that $V_s$ has pure dimension $n-s$ for $1\le s\le
r$, and denote by $\delta_s$ the degree of $V_s$ for $1\le s\le r$.
We denote $\delta:=\max\{\delta_1,\ldots,\delta_r\}$ and assume that
$\delta>d$.

In this section, we describe the whole algorithm for computing a
Kronecker representation of a lifting fiber of $V$ and analyze its
cost over a general perfect field $k$. This analysis is expressed in
terms of the number of arithmetic operations, equality tests,
and extractions of $p$--th roots in $k$ (if $k$ has characteristic
$p>0$). We then explain how this algorithm applies in the cases
where $k$ is a finite field or the field of rational numbers $\Q$.
%
%----------------------------------------------------------------------
%----------------------------------------------------------------------
%----------------------------------------------------------------------
%----------------------------------------------------------------------
%
\subsection{The algorithm for a general field $k$}
With the notations and assumptions as above, the following algorithm computes a
Kronecker representation of a lifting fiber of the variety $V$.
\begin{algorithm}[Main algorithm] ${}$
\label{algo: main algorithm for k}
\begin{itemize}
\item[]{\bf Input:}
\begin{enumerate}
  \item
Polynomials $F_1,\ldots,F_r,G\in k[X_1,\ldots,X_n]$ satisfying
(${\sf H}_1$)--(${\sf H}_2$);
\item Access to uniformly random
elements of a finite set $\mathcal{S}\subset k$ of size at least
$\varepsilon^{-1}2n^2r\delta^4$, %(or from a finite field extension
%of $k$ if $k$ does not have enough points),
where $0<\varepsilon<1$.
\end{enumerate}
\item[]{\bf Output:}
\begin{enumerate}
\item
Linear forms $Y_1,\ldots,Y_n\in k[X_1,\ldots,X_n]$,
\item
A point $\bfs p^r:=(p_1,\ldots,p_{n-r})\in k^{n-r}$,
\item
Polynomials $m, w_{n-r+2}\klk w_n\in k[T]$,
\end{enumerate}
with the
following properties:
\begin{enumerate}
  \item $Y_1,\ldots,Y_{n-r+1}$ and $\bfs p^r:=(p_1,\ldots,p_{n-r})$ satisfy the conditions
  of Theorem \ref{th: preproc: all conditions};
  \item
  The linear map $\pi_r:\A^n\to\A^{n-r}$ defined by $Y_1,\ldots,Y_{n-r}$, the linear form $Y_{n-r+1}$, and
  the polynomials $m, w_{n-r+2}\klk w_n$ yield a Kronecker
  representation of the lifting fiber $\pi_r^{-1}(\bfs p^r)$ of $V$.
\end{enumerate}

\item[]{\bf 1} Randomly choose the coefficients of linear forms $Y_1,\ldots,Y_n\in
k[X_1,\ldots,X_n]$ from $\mathcal{S}^{n\times n}$. Set $\bfs Y:=(Y_1,\ldots,Y_n)$.
\item[]{\bf 2} Randomly choose the coefficients of a point $\bfs p\in k^{n-1}$
from $\mathcal{S}^{n-1}$. Set $\bfs p^s:=(p_1,\ldots,p_{n-s})$ for $1\le s\le r$.
\item[]{\bf 3} Compute the polynomials
$F_i^{\bfs Y}:=F_i(\bfs Y)$ $(1\le i\le r)$ and $G^{\bfs Y}:=G(\bfs
Y)$.
\item[]{\bf 4} Compute the dense representation $m_1\in k[T]$ of
$$m_1=F_1^{\bfs Y}(\bfs p^1,T)/
\gcd\big(F_1^{\bfs Y}(\bfs p^1,T),G^{\bfs Y}(\bfs p^1,T)\big).$$
\item[]{\bf 5} For $s=1\klk r-1$, do
\begin{itemize}
  \item[]{\bf 5.1} Apply Algorithm \ref{algo: Newton lifting} to compute
  the Kronecker representation of
$$C_s:=\pi_s^{-1}(\{Y_1=p_1,\ldots,Y_{n-s-1}=p_{n-s-1}\}),$$ with
$Y_{n-s+1}$ as primitive element;
  \item[]{\bf 5.2} Apply Algorithm \ref{algo: comput m_(s+1) and v_(n-s+1)} to compute polynomials
$m_{s+1},v_{n-s+1}^{s+1}\in k[T]$ such that  $m_{s+1}(Y_{n-s})\equiv 0$ and
  $Y_{n-s+1}-v_{n-s+1}^{s+1}(Y_{n-s})\equiv 0$ in $\pi_{s+1}^{-1}(\bfs
  p^{s+1})$, with $\deg
m_{s+1}=\delta_{s+1}$ and $\deg
  v_{n-s+1}^{s+1}<\delta_{s+1}$.
  \item[]{\bf 5.3} Apply
Algorithm \ref{algo: conclusion recursive
  step} to compute polynomials $v_{n-s+2}^{s+1}\klk v_n^{s+1}\in k[T]$
  such that $Y_{n-s+j}-v_{n-s+j}^{s+1}(Y_{n-s})\equiv 0$ in $\pi_{s+1}^{-1}(\bfs
  p^{s+1})$ for $2\le j\le s$,
  with $\deg v_{n-s+j}^{s+1}<\delta_{s+1}$
  for $2\le j\le s$.
\end{itemize}
\item[]{\bf 6} Output $m:=m_r$ and $w_{n-r+j}:=\frac{\partial m}{\partial T}\cdot v_{n-r+j}^r\mod m$
for $2\le j\le r$.
\end{itemize}
\end{algorithm}

First we analyze the probability that Algorithm \ref{algo: main
algorithm for k} correctly computes a Kronecker representation of a
lifting fiber of the input variety $V$.
\begin{theorem}\label{th: main algo - prob success}
Assume that $k$ has cardinality at least
$\varepsilon^{-1}2n^2r\delta^4$. Then Algorithm \ref{algo: main algorithm
for k} computes a Kronecker representation of a lifting fiber of $V$
with probability at least $1-8r\varepsilon$.
\end{theorem}
\begin{proof}
First, we analyze the probability that the linear forms $Y_1\klk
Y_n$ chosen in step {\bf 1} and the point $\bfs p$ chosen in step
{\bf 2} satisfy the conditions of Theorem \ref{th: preproc: all
conditions} for $1\le s< r$. According to Corollary \ref{coro:
preproc: all conditions for all s}, this holds if the coefficients
of $\bfs Y$ and $\bfs p$ do not annihilate the polynomial $B$ from
Corollary \ref{coro: preproc: all conditions for all s}. Recall that
$\deg B\le 2n^2rd\delta^3$.
By Lemma \ref{lemma: Zippel_Schwartz}, we conclude that a random
choice of the coefficients of $\bfs Y$ and $\bfs p$ is successful
with probability at least $1-\varepsilon$.

Assume that the linear forms $Y_1\klk Y_n$ chosen in step {\bf 1}
and the point $\bfs p$ chosen in step {\bf 2} satisfy the conditions
of Theorem \ref{th: preproc: all conditions} for $1\le s\le r$.

By definition, $V_1$ is the Zariski closure of $V(F_1)\setminus
V(G)$, i.e., $V_1$ is the union of the irreducible components of
$V(F_1)$ not contained in $V(G)$. Since the linear forms
$Y_1,\ldots,Y_n$ and the point $\bfs p^1$ satisfy the hypotheses of
Theorem \ref{th: preproc: all conditions}, the map
$\pi_1:V_1\to\A^{n-1}$ defined by $Y_1,\ldots,Y_{n-1}$ is a finite
morphism, $\bfs p^1$ is a lifting point of $\pi_1$, and $Y_n$ is a
primitive element of $\pi_1^{-1}(\bfs p^1)$. In particular, if
$\mathcal{C}$ is an irreducible component of $V_1$ and
$F_\mathcal{C}$ is the irreducible factor of $F_1$ corresponding to
$\mathcal{C}$,
$$\#\big(\mathcal{C}\cap\pi_1^{-1}(\bfs p^1)\big)
=\deg\mathcal{C}=\#\{F_\mathcal{C}(\bfs p^1,Y_n)=0\}.$$
For any irreducible component $\mathcal{C}$ of
$V(F_1)$ not included in $V_1$, if $F_\mathcal{C}$ is the
irreducible factor of $F_1$ corresponding to $\mathcal{C}$, then
$\mathcal{C}\subset V(G)$, and thus $F_\mathcal{C}$
divides $G$. It follows that $ \gcd(F_1^{\bfs Y},G^{\bfs Y})$ is the
defining polynomial for the union $V_1^c$ of the irreducible
components of $V(F_1)$ contained in $V(G)$, and $ \gcd\big(F_1^{\bfs
Y}(\bfs p^1,Y_n),G^{\bfs Y}(\bfs p^1,Y_n)\big)$ defines
the zero-dimensional variety $V_1^c\cap
\{Y_1=p_1,\ldots,Y_{n-1}=p_{n-1}\}$. As a consequence, the
quotient $F_1^{\bfs
Y}(\bfs p^1,T)/ \gcd\big(F_1^{\bfs Y}(\bfs p^1,T),G^{\bfs Y}(\bfs
p^1,T)\big)$ is the defining polynomial for the zero-dimensional
variety $V_1\cap
\{Y_1=p_1,\ldots,Y_{n-1}=p_{n-1}\}=V_1\cap\pi_1^{-1} (\bfs p^1)$.
This justifies the correctness of the expression of $m_1$ of step {\bf
4}. Since no parametrized variables appear in
the Kronecker representation of the lifting fiber of $V_1$ under consideration,
the polynomials $m_1$ itself provides the complete description of
the Kronecker representation of $V_1$ we seek.

It follows that Algorithm \ref{algo: main algorithm for k} proceeds
into the loop of step {\bf 5} as expected. Then,
Theorem \ref{theo: correctedness: Newton lifting} establishes the
correctness of the application of Algorithm \ref{algo: Newton lifting}
in step {\bf 5.1}. According to Lemmas
\ref{lemma: correctness comput m_(s+1) and v_(n-s+1)} and
\ref{lemma: cost comput m_(s+1) and v_(n-s+1)}, the application of Algorithm
\ref{algo: comput m_(s+1) and v_(n-s+1)} in step {\bf 5.2}
is correct with probability at least $1-7\varepsilon$ for $1\le s<r$,
and thus with probability at least $1-7(r-1)\varepsilon$ for the
entire loop.
Finally, the correctness of the application of Algorithm \ref{algo:
conclusion recursive step} in step {\bf 5.3} is justified in Remark
\ref{rem: correctness algo conclusion recursive step}. The
conclusion of the theorem then readily follows.
\end{proof}

Finally, we analyze the cost of Algorithm \ref{algo: main algorithm
for k}.
\begin{theorem}\label{th: main algo for k}
Suppose that $F_1\klk F_r,G\in k[X_1\klk X_n]$ are given by a
straight--line program $\beta$ in $k[X_1\klk X_n]$ of length $L$.
Then Algorithm \ref{algo: main algorithm for k} can be executed using
$$\mathcal{O}\big(n^4+r^2(L+n^2+r^3+\log_2\delta)\,{\sf M}(d\delta)\,{\sf
M}(\delta)\big)
$$
arithmetic operations in $k$,
$\mathcal{O}\big(r(L+n^2+r\log_2\delta)\,{\sf M}(d\delta)\,{\sf
M}(\delta)\big)$
identity tests in $k$, and $\mathcal{O}(r\delta)$
extractions of $p$--th roots in $k$  (if $k$ has characteristic
$p>0$).
\end{theorem}
\begin{proof}
Given the coefficients of the linear forms $\bfs
Y:=(Y_1,\ldots,Y_n)$ from step {\bf 1}, the computation in step {\bf
3} consists of performing the change of variables $\bfs X\to\bfs Y$
in the polynomials $F_1,\ldots,F_r,G$. Recall that these polynomials
are given by a straight-line program $\beta$ of length $L$. If
$\mathcal{M}\in k^{n\times n}$ is the matrix defining this change of
variables, namely $\bfs Y=\mathcal{M}\bfs X$, then the transformed polynomials
are $F_i^{\bfs Y}:=F_i(\mathcal{M}^{-1}\bfs Y)$
$(1\le i\le r)$ and $G^{\bfs Y}:=G(\mathcal{M}^{-1}\bfs Y)$.

Thus, we compute the inverse $\mathcal{M}^{-1}$. Using the
Samuelson-Berkowitz algorithm, this can be done with
$\mathcal{O}(n^4)$ arithmetic operations in $k$. Next, we construct a
straight-line program $\beta'$ to evaluate
$F_1^{\bfs Y},\ldots,F_r^{\bfs Y},G^{\bfs Y}$. For this, we
simply prepend the multiplication $\mathcal{M}^{-1}\bfs Y$ to the
straight-line program $\beta$. Since this
multiplication requires $\mathcal{O}(n^2)$ arithmetic operations in $k$,
the straight-line program $\beta'$ has length $L':=\mathcal{O}(L+n^2)$.

Step {\bf 4} consist of computing the dense representation
of the polynomial $m_1$. This involves computing the dense
representations of $F_1^{\bfs Y}(\bfs p^1,T)$ and $G^{\bfs Y}(\bfs
p^1,T)$. As these polynomials have degree at most $d$, their dense
representation can be obtained by executing the
straight-line program $\beta'$ in the ring $k[T]/(T^{d+1})$.
Specifically, we initialize the variables $(X_1,\ldots,X_n)$ in $\beta'$
with $(\bfs p^1,T)$ and perform all the arithmetic
operations modulo $(T^{d+1})$. Then, given the dense representation
of $F_1^{\bfs Y}(\bfs p^1,T)$ and $G^{\bfs Y}(\bfs p^1,T)$, we
compute their gcd and finally obtain $m_1$ by division. All these
computations can be performed with $\mathcal{O}\big(L'{\sf M}(d)+{\sf
M}(d)\log_2 d\big)$ arithmetic operations in $k$.

Now we consider the cost of the $s$th step of the main loop {\bf 5}.
According to Theorem \ref{th: cost Newton lifting}, the step {\bf
5.1} can be implemented to run with
$$\mathcal{O}\big((L's+s^4+\log_2\delta_s){\sf M}(\delta_s)^2\big)$$
operations in $k$. On the other hand, Lemma \ref{lemma: cost comput
m_(s+1) and v_(n-s+1)} asserts that Algorithm \ref{algo: comput
m_(s+1) and v_(n-s+1)} requires
$$ \mathcal{O}\big((L'+s\,\log_2\delta)\,{\sf M}(d\delta)\,{\sf
M}(\delta) \big)$$
arithmetic operations in $k$,
$\mathcal{O}\big((L'+s\,\log_2\delta)\,{\sf M}(d\delta)\,{\sf
M}(\delta)\big)$ identity tests in $k$, and $\mathcal{O}(\delta)$
extractions of $p$--th roots in $k$. Finally, Lemma \ref{lemma: cost
conclusion recursive step} shows that Algorithm \ref{algo:
conclusion recursive step} can be executed with
$$\mathcal{O}\big(s\,\delta\,{\sf M}(\delta)\big)$$
arithmetic
operations in $k$, and $\mathcal{O}({\sf M}(\delta))$ identity tests
in $k$. We conclude that the $s$th iteration of the main loop
requires
$$
\mathcal{O}\big((L's+s^4+s\log_2\delta)\,{\sf M}(d\delta)\,{\sf
M}(\delta)\big)
$$
arithmetic operations in $k$,
$\mathcal{O}\big((L'+s\,\log_2\delta)\,{\sf M}(d\delta)\,{\sf
M}(\delta)\big)$
identity tests in $k$, and $\mathcal{O}(\delta)$
extractions of $p$--th roots in $k$.

Step {\bf 6} has a negligible cost in comparison.

By summing the costs, the number of
identity tests, and the number of extractions of $p$--th roots over
$1\le s\le r$, we readily deduce the theorem.
\end{proof}
%
%----------------------------------------------------------------------
%----------------------------------------------------------------------
%----------------------------------------------------------------------
%----------------------------------------------------------------------
%
\subsection{The algorithm for a finite field $\fq$}
Next, we discuss how Algorithm \ref{algo: main algorithm for k} is
applied when $k$ is the finite field $\fq$ of $q$ elements, where $q$
is a power of a prime $p$. Given $r,n$ with $r\le n$, we assume that we
have polynomials
$F_1\klk F_r,G\in\fq[X_1\klk X_n]$ of
degree at most $d$ satisfying hypotheses (${\sf H}_1$)--(${\sf H}_2$) of
Section \ref{section: locally closed sets}. Our goal is
to compute a Kronecker representation of a lifting fiber
of the Zariski closure $V$ of $\{F_1=0\klk
F_r=0,G\not=0\}$.

If the cardinality $q$ is ``small'', it may happen that no
Kronecker representation of a lifting fiber is defined over
$\fq$. In that case, we work over a finite extension of $\fq$
with sufficiently large cardinality, applying
Algorithm \ref{algo: main algorithm for k} {\em mutatis mutandis}
over this extension.

On the other hand, for the case where the field
$\fq$ has sufficiently large cardinality so that there
exists a Kronecker representation of a lifting
fiber of $V$ defined over $\fq$, a more detailed discussion is
needed. Specifically, given $0<\varepsilon<1/2$, we
assume from now on that
$q>2\varepsilon^{-1}n^2rd\delta^3$.  The algorithm is as
follows.
\begin{algorithm}[Main algorithm for $k=\fq$] ${}$
\label{algo: main algorithm for fq}
\begin{itemize}
\item[]{\bf Input:}
\begin{enumerate}
\item Polynomials $F_1,\ldots,F_r,G\in \fq[X_1,\ldots,X_n]$ satisfying
(${\sf H}_1$)--(${\sf H}_2$);
\item Access to uniformly random elements of a finite set
$\mathcal{S}\subset \fq$ of size at least
$\varepsilon^{-1}2n^2rd\delta^3$, where $0<\varepsilon<1/2$.
\end{enumerate}
\item[]{\bf Output:}
\begin{enumerate}
\item
Linear forms $Y_1,\ldots,Y_n\in\fq[X_1,\ldots,X_n]$,
\item
A point $\bfs p^r:=(p_1,\ldots,p_{n-r})\in\fq^{n-r}$,
\item
Polynomials $m, w_{n-r+2}\klk w_n\in\fq[T]$,
\end{enumerate}
with the following properties:
\begin{enumerate}
  \item The linear map $\pi_r:V\to\A^{n-r}$ defined by $Y_1,\ldots,Y_{n-r}$
  is a finite morphism;
  \item $\bfs p^r$ is a lifting point of $\pi_r$;
  \item
  The map $\pi_r:V\to\A^{n-r}$, together with the linear form $Y_{n-r+1}$ and
  the polynomials $m, w_{n-r+2}\klk w_n$, forms a Kronecker
  representation of $\pi_r^{-1}(\bfs p^r)$.
\end{enumerate}

\item[]{\bf 1} Randomly choose the coefficients of linear forms $Y_1,\ldots,Y_n\in
\fq[X_1,\ldots,X_n]$ from $\mathcal{S}^{n\times n}$. Set $\bfs Y:=(Y_1,\ldots,Y_n)$.
\item[]{\bf 2} Randomly choose the coefficients
of a point $\bfs p:=(p_1,\ldots,p_n)$ from $\mathcal{S}^{n-1}$. Set
$\bfs p^s:=(p_1,\ldots,p_{n-s})$ for $1\le s\le r$.
\item[]{\bf 3} Compute the polynomials
$F_i^{\bfs Y}:=F_i(\bfs Y)$ $(1\le i\le r)$ and $G^{\bfs Y}:=G(\bfs
Y)$.
\item[]{\bf 4} Compute the dense representation $m_1\in\fq[T]$ of
$$m_1=F_1^{\bfs Y}(\bfs p^1,T)/
\gcd\big(F_1^{\bfs Y}(\bfs p^1,T),G^{\bfs Y}(\bfs p^1,T)\big).$$
\item[]{\bf 5} For $e:=\max\{1,\log_q(12r\varepsilon^{-1}\delta^4)\}$, set
$k:=\fqe$.
\item[]{\bf 6} For $s=1\klk r-1$, do
\begin{itemize}
  \item[]{\bf 6.1} Apply Algorithm \ref{algo: Newton lifting}
  to compute a Kronecker representation of
$$C_s:=\pi_s^{-1}(\{Y_1=p_1,\ldots,Y_{n-s-1}=p_{n-s-1}\}),$$ with
$Y_{n-s+1}$ as primitive element;
  \item[]{\bf 6.2} Apply Algorithm \ref{algo: comput m_(s+1) and v_(n-s+1)} over $k$
  to compute polynomials $m_{s+1},v_{n-s+1}^{s+1}\in\fq[T]$ such that
  $m_{s+1}(Y_{n-s})\equiv 0$ and
  $Y_{n-s+1}-v_{n-s+1}^{s+1}(Y_{n-s})\equiv 0$ in $\pi_{s+1}^{-1}(\bfs
  p^{s+1})$, with $\deg
m_{s+1}=\delta_{s+1}$ and $\deg
  v_{n-s+1}^{s+1}<\delta_{s+1}$.
  \item[]{\bf 6.3} Apply
Algorithm \ref{algo: conclusion recursive step} to compute
polynomials $v_{n-s+2}^{s+1}\klk v_n^{s+1}\in\fq[T]$
such that   $Y_{n-s+i}-v_{n-s+i}^{s+1}(Y_{n-s})\equiv 0$ in $\pi_{s+1}^{-1}(\bfs
  p^{s+1})$  for $2\le i\le s$,
with $\deg
v_{n-s+i}^{s+1}<\delta_{s+1}$ for $2\le i\le s$.
\end{itemize}
\item[]{\bf 7} Output $m:=m_r$ and $w_{n-r+j}:=\frac{\partial m}{\partial T}\cdot v_{n-r+j}^r\mod m$
for $2\le j\le r$.
\end{itemize}
\end{algorithm}

Algorithm \ref{algo: main algorithm for fq} is similar to Algorithm
\ref{algo: main algorithm for k}, with one important difference: in
the main loop of step {\bf 6} we work over an extension $\fqe$ of
$\fq$ to ensure that random choices in $\fqe$ are successful with
high probability. It is worth noting that, in the finite field
setting, the required condition on the cardinality of the base
field~$\fq$ is significantly less restrictive than in the case of a
general perfect field~$k$. The key point is that, although the
computations in step {\bf 6} are carried out over the extension
$\fqe$, the resulting output remains defined over $\fq$. The bit
complexity and probability of success are analyzed in the following
result.
\begin{theorem}\label{th: main algorithm for fq}
Let $F_1\klk F_r,G\in\fq[X_1\klk X_n]$ be polynomials of maximum
degree $d>0$ satisfying hypotheses $({\sf H}_1)$--$({\sf H}_2)$. For
$1\le s\le r$, denote by $V_s$ the Zariski closure of
$\{F_1=\cdots=F_s=0,G\not=0\}$ and by $\delta_s$ its degree. Let
$\delta:=\max\{\delta_1,\ldots,\delta_r\}$, and assume that
$\delta>d$. Suppose that $F_1\klk F_r,G\in\fq[X_1\klk X_n]$ are
given by a straight--line program $\beta$ in $\fq[X_1\klk X_n]$ of
length $L$. Given $0<\varepsilon<1/2$, assume that
$q>4\varepsilon^{-1}n^2rd\delta^3$. Then there exists a finite
morphism $\pi:V_r\to\A^{n-r}$ defined over $\fq$, and a lifting
fiber $\pi^{-1}(\bfs p^r)$ of $V_r$ defined over $\fq$, with a
Kronecker representation defined over $\fq$, which can be computed
with bit complexity
$$\mathcal{O}\Big(\big((n^4+r^2(L+n^2+r^3+\log_2\delta)\,{\sf M}(d\delta)\,{\sf
M}(\delta)\,{\sf M}(\log_2(r\delta))\big)\,{\sf M}(\log_2q)\Big)$$
and probability of success at least $1-2\varepsilon$.
\end{theorem}
\begin{proof}
Let $B$ be the (nonzero) polynomial from Corollary
\ref{coro: preproc: all conditions for all s}. Since
$B\in\cfq[\bfs\Lambda,\widetilde{\bfs Y}]$ has degree at most
$2n^2rd\delta^3$, we conclude (by, e.g.,  \cite[Lemma 2.1]{CaMa06})
that $B$ has at most $$\deg B\cdot q^{n^2+n-2}\le \varepsilon
q^{n^2+n-1}<q^{n^2+n-1}$$ zeros in $\fq^{n^2+n-1}$. It follows that
there exists $(\bfs\lambda,\bfs p)\in\fq^{n^2+n-1}$ such that
$B(\bfs\lambda,\bfs p)\not=0$. Setting $\bfs Y=(Y_1,\ldots,Y_n):=
\bfs\lambda \cdot \bfs X$ and $\bfs p:=(p_1,\ldots,p_{n-1})$,
Theorem \ref{th: preproc: all conditions} guarantees that the linear
map $\pi:V\to\A^{n-r}$ defined by $Y_1,\ldots,Y_{n-r}$ is a finite
morphism, the fiber $\pi^{-1}(\bfs p^r)$ defined by $\bfs
p^r:=(p_1,\ldots,p_{n-r})$ is a lifting fiber of $\pi$, and
a Kronecker representation of $\pi^{-1}(\bfs p^r)$ exists over $\fq$.

Now we explain why Algorithm \ref{algo: main algorithm for fq}
computes a Kronecker representation of such a lifting fiber $\pi^{-1}(\bfs
p^r)$ with the stated complexity and probability of success.

From the proof of Theorem \ref{th: main algo for k}, we know that
steps {\bf 3} and {\bf 4} can be executed with
$\mathcal{O}\big(L'{\sf M}(d)+{\sf M}(d)\log_2 d\big)$ arithmetic
operations in $\fq$, where $L':=\mathcal{O}(L+n^2)$. We next analyze
the cost of the $s$th iteration of the main loop {\bf 6}. According to
Theorem \ref{th: cost Newton lifting}, the step {\bf 6.1} can be carried out with
$$\mathcal{O}\big((L's+s^4+\log_2\delta_s){\sf M}(\delta_s)^2\big)$$
operations in $\fq$. For step {\bf 6.2}, Lemma \ref{lemma: cost
comput m_(s+1) and v_(n-s+1)} shows that it requires
$\mathcal{O}\big((L'+s\,\log_2\delta)\,{\sf M}(d\delta)\,{\sf
M}(\delta) \big)$
arithmetic operations in $k:=\fqe$,
$\mathcal{O}\big((L'+s\,\log_2\delta)\,{\sf M}(d\delta)\,{\sf
M}(\delta)\big)$ identity tests in $k$, and $\mathcal{O}(\delta)$
extractions of $p$--th roots in $k$. Since each arithmetic operation or
identity test in $k$ requires $\mathcal{O}({\sf M}(e))$ arithmetic
operations in $\fq$, and each extraction of a $p$--th root in $k$
costs $\mathcal{O}(\log (q^e/p))$ arithmetic operations
in $\fq$, we conclude that the overall cost of step {\bf 6.2}
is $$\mathcal{O}\big((L'+s\,\log_2\delta)\,{\sf
M}(d\delta)\,{\sf M}(\delta)\,{\sf M}(e) \big)$$
arithmetic operations in $\fq$.

Finally, by Lemma \ref{lemma: cost conclusion recursive step},
step {\bf 6.3} can be executed with
$$\mathcal{O}\big(s\,\delta\,{\sf M}(\delta)\big)$$
arithmetic operations in $\fq$. Therefore, the total cost
of the $s$th step of the main loop is
$$\mathcal{O}\big(s(L'+s^3+\log_2\delta)\,{\sf M}(d\delta)\,{\sf
M}(\delta)\,{\sf M}(e)\big)
$$
arithmetic operations in $\fq$. Summing over $1\le s\le r$
and noting that each arithmetic operation can be
performed with $\mathcal{O}({\sf M}(\log_2q))$ bit operations, we
directly obtain the stated overall complexity bound.

We now analyze the probability of success. The first step is a selection
of $(\bfs\lambda,\bfs p)\in\fq^{n^2+n-1}$.  By our assumption
$q>2\varepsilon^{-1}n^2rd\delta^3$, we know that
$B(\bfs\lambda,\bfs p)\not=0$ with probability at least $1-\varepsilon$.

Let $\bfs Y=(Y_1,\ldots,Y_n):=
\bfs \lambda\cdot\bfs X$ and $\bfs p:=(p_1,\ldots,p_{n-1})$.
Recall
that $(Y_1,\ldots,Y_{n-s})$ and $\bfs p^s:=(p_1,\ldots,p_{n-s})$
satisfy all the conditions of Theorem \ref{th: preproc: all
conditions} for $1\le s< r$. Steps {\bf 3}, {\bf 4}, {\bf 6.1}
and {\bf 6.3} are deterministic.
Step {\bf 6.2}, however, involves random choices through Algorithms \ref{algo:
computation projection} and \ref{algo: bidimensional shape lemma}.
Algorithm \ref{algo: computation projection} is invoked six times
per iteration, each call requiring a random choice in $\fqe$.
Using the bound
$$q^e\ge 24r\varepsilon^{-1}\delta^4\ge 12r\varepsilon^{-1}
(D_s+1)(\delta_s^2+D_s),
$$
we deduce that each such random
choice fails with probability at most $\frac{\varepsilon}{12r}$,
yielding an overall failure
probability of at most $\frac{\varepsilon}{2r}$.
Additionally,  Algorithm \ref{algo: bidimensional shape lemma} makes
two random choices per iteration. By
$$q^e\ge 2r\varepsilon^{-1}\delta^4\ge  2r\varepsilon^{-1}\delta_s^4,$$
these choices fail with probability at most
$\frac{\varepsilon}{2r}$.
Consequently, the total probability of failure in step {\bf 6.2} per
iteration is bounded by $\frac{\varepsilon}{r}$, and the probability
of success over all iterations is at least
$1-\varepsilon$.

Combining this with the probability of
success for selecting $(\bfs\lambda,\bfs p)$ yields an overall
probability of success of at least $1-2\varepsilon$ for Algorithm \ref{algo:
main algorithm for fq}. This completes the proof.
\end{proof}
%
%----------------------------------------------------------------------
%----------------------------------------------------------------------
%----------------------------------------------------------------------
%----------------------------------------------------------------------
%----------------------------------------------------------------------
%----------------------------------------------------------------------
%----------------------------------------------------------------------
%----------------------------------------------------------------------
%
\section{Solving systems defined over $\mathbb{Q}$}
\label{sec: systems over Q}
In this section we consider polynomials $F_1\klk F_r,G$ $(r\le n)$
of $\mathbb{Z}[X_1\klk X_n]$ of degree at most $d$, with coefficients
of bit length most $h$, satisfying conditions (${\sf H}_1$)--(${\sf H}_2$) of Section
\ref{section: locally closed sets}, and describe an algorithm for
computing a Kronecker representation of a lifting fiber
of the Zariski closure $V$ of $\{F_1=0 \klk F_r=0, G \neq 0\}$.

A direct application of Algorithm \ref{algo: main algorithm for k}
for $k=\Q$ may lead to exponential growth in the bit length of the
integers involved in intermediate computations. To avoid this issue,
a modular approach is proposed in \cite{GiMa19}, which consists of
performing the computations modulo a prime number $p$, followed by a
step of $p$--adic lifting and rational reconstruction to recover the
coefficients of the polynomials defining the output Kronecker
representation over $\Q$.

This approach requires, on the one hand, the choice of a ``lucky''
prime $p$---that is, one that yields a good modular reduction
with relatively small bit length---and, on the other hand, prior
knowledge of the bit length of the integers appearing in the output
Kronecker representation. In the following, we describe the main
algorithmic aspects of these two steps, referring the reader to
\cite{GiMa19} for details and complete proofs.

In the following, the height ${\sf h}(\alpha)$ of $\alpha:=p/q\in\Q$,
with $p,q\in\Z\setminus\{0\}$ coprime, is ${\sf
h}(\alpha):=\max\{\log_2 |p|, \log_2 |q|\}$; the height of $0$ is
defined as $0$. By the bit length of an integer we mean its height.
We define the height ${\sf h}(F)$ of a polynomial $F$ with
coefficients in $\Q$ as the  maximum of the heights of its
coefficients. For a prime $p$, we denote by $\Z_{(p)}$ the subring
of $\Q$ defined by $\Z_{(p)}=\{\alpha \in \Q: \alpha = a/b, \text{
with } a, b \in \Z \text{ and } p \nmid b \}$. For a polynomial $M$
with coefficients in $\Z_{(p)}$, $M_p$ denotes the polynomial with
coefficients in $\fp:=\Z/p\Z$ obtained by reduction modulo $p$. We
denote by $\Z_p$ the ring of $p$-adic integers. Given a $p$-adic
integer $\alpha$ and $k\in \mathbb{N}$, by the $p$-adic
approximation of order $k$ of $\alpha$ we mean the element $m\in \Z$
with $0 \le m < p^k$ such that $\alpha \equiv m \mod p^k$. Furthermore,
for a univariate polynomial  $f:=\sum_{i=0}^ na_iT^i\in \Z_p[T]$ of
formal degree $n$, by its $p$-adic approximation of order $k$ we
mean the polynomial $g:=\sum_{i=0}^ nb_iT^i \in \Z[T]$ with $0\le b_i
< p^k$ and $a_i\equiv b_i \mod p^k$ for $0 \le i \le n$.
%
%----------------------------------------------------------------------
%----------------------------------------------------------------------
%----------------------------------------------------------------------
%----------------------------------------------------------------------
%
\subsection{The choice of a lucky prime $p$ for a good modular reduction}
The first task is to make a ``lucky'' choice of an integer matrix
$\bfs\lambda\in \Z^{n^2}$ and an integer point $\bfs p\in \Z^{n-1}$,
which will serve as the
matrix of coefficients for the linear forms defining our change of
variables, and the coordinates of our lifting points. We seek such
integers to have relatively small bit length.

Let $\bfs \Lambda:=(\Lambda_{ij})_{1\le i, j\le n}$ denote a set of $n^2$
indeterminates over $\overline{\Q}$. According to \cite{GiMa19},
there exists a nonzero polynomial $\textsf{R}\in \overline{\Q}[\bfs
\Lambda]$ such that
\begin{equation}\label{eq: definition D}
\deg \textsf{R}\le
D:=r(n+1)((n+1)d\delta^2+2\delta^3+n^22^{n-1}d\delta^{2}),
\end{equation}
which satisfies the following property: for every $\bfs \lambda \in
\Z^{n^2}$ such that $\textsf{R}(\bfs \lambda)\neq 0$, there exists a
nonzero polynomial $\textsf{N}_{\scriptscriptstyle \bfs \lambda}\in \Z[Z_1
\klk Z_{n-1}]$ with $\deg
\textsf{N}_{\scriptscriptstyle \bfs \lambda} \le D$ such that, if
$$\mathfrak{N}:=\det(\bfs \lambda)\textsf{N}_{\scriptscriptstyle \bfs
\lambda}(\bfs p)\neq 0$$ for some $\bfs p\in \Z^{n-1}$, then
$\mathfrak{N}$ is divisible by every ``unlucky'' prime. We will call
such a pair $(\bfs \lambda, \bfs p)$ and any prime $p$ not dividing $\mathfrak{N}$
``lucky''.
%In the sequel we discuss the properties of lucky points $(\bfs
%\lambda, \bfs p)$ and lucky primes $p$ and the probability of
%success of their choice.

More precisely, the polynomials $\textsf{R}$ and
$\textsf{N}_{\scriptscriptstyle \bfs \lambda}$ are constructed
so that, for a lucky pair $(\bfs\lambda, \bfs p)$ and a
lucky prime $p$, the hypotheses (${\sf H}_1$)--(${\sf H}_2$) of Section
\ref{section: locally closed sets}, as well as conditions
$(1)$--$(4)$ of Theorem \ref{th: preproc: all conditions}, remain
satisfied after reducing $F_1\klk F_r, G$ modulo $p$. To state
this precisely, we introduce the following notations.

For
$(\bfs \lambda, \bfs p)\in \Z^{n^2}\times\Z^{n-1}$, define
$Y_i:=\lambda_{i1}X_1 + \cdots + \lambda_{in}X_n$ for $1\le i \le
n$, $\bfs \lambda^s:=(\lambda_{ij})_{\scriptscriptstyle 1 \le i \le
n-s+1, 1\le j \le n}$, and $\bfs p^s:=(p_1 \klk p_{n-s})$ for $1\le s
\le r$. Let $\pi_s:V_s\to\A^{n-s}$ denote the linear map
defined by $Y_1,\ldots,Y_{n-s}$. Given a prime $p$, denote by
$V_{s,p}$ the Zariski closure in $\mathbb{A}_{\scriptscriptstyle \cfp}^n
:=\A^n(\cfp)$ of the set
$\{F_{1,p}=0 \klk F_{s,p}=0 , G_p \neq 0\}$ for $1\le s \le r$.

With this notation, we can state the first result on the properties of lucky
pairs $(\bfs \lambda, \bfs p)$ and lucky primes $p$ (see
\cite[Theorem 5.10]{GiMa19}).
\begin{theorem} \label{th:_lucky prime_a}
Let $(\bfs\lambda, \bfs p)\in \Z^{n^2}\times
\Z^{n-1}$ be a lucky pair, and let
$\mathfrak{N}:=\det(\bfs \lambda)\textsf{N}_{\bfs \lambda}(\bfs p)$.
Then conditions $(1)$--$(4)$ of Theorem \ref{th: preproc: all
conditions} hold for $1\le s < r$. Furthermore,
if $p$ is a prime such that $p\nmid \mathfrak{N}$, the linear
forms $Y_{1,p}\klk Y_{n,p}$ define a new set of variables in
$\cfp[X_1 \klk X_n]$, and the following hold for
$1 \leq s < r$, with $\bfs q:=\bfs p^s$ and $\bfs q^*:=\bfs
p^{s+1}$:
\begin{enumerate}
  \item  The polynomials $F_{1,p}\klk F_{s,p}$ and $G_p$ satisfy
  hypotheses $({\sf H}_1)$--$({\sf H}_2)$ of Section \ref{section: locally closed sets}
  for $k:=\cfp$. Moreover, $V_{s,p}\subset
  \mathbb{A}_{\scriptscriptstyle \cfp}^n$
  has degree $\delta_s$ for $1\le s \le r$.
  \item The map $\pi_{s,p}:V_{s,p}\rightarrow
  \mathbb{A}_{\scriptscriptstyle \cfp}^{n-s}$
  defined by $Y_{1,p},\dots,Y_{n-s,p}$ is a finite morphism;
  $\bfs q_p\in \fp^{n-s}$ is a lifting point of $\pi_{s,p}$
  with $\pi_{s,p}^{-1}(\bfs q_p)\subset \{G_p\neq 0\}$;
  and $Y_{n-s+1,p}$ is a primitive element of $\pi_{s,p}^{-1}(\bfs
  q_p)$.
  \item The map $\pi_{s+1,p}:V_{s+1,p}\rightarrow
  \mathbb{A}_{\scriptscriptstyle \cfp}^{n-s-1}$
  defined by $Y_{1,p},\dots,Y_{n-s-1,p}$ is a finite morphism;
  $\bfs q^*_p$ is a lifting point of $\pi_{s+1,p}$
  with $\pi_{s+1,p}^{-1}(\bfs q_p^*)\subset \{G_p\neq 0\}$; and
  $Y_{n-s,p}$ is a primitive element of $\pi_{s+1,p}^{-1}(\bfs
  q^*_p)$.
  \item Any $\bfs \xi\in \pi_{s,p}\bigl(\pi_{s+1,p}^{-1}(\bfs q^*_p)\bigr)$
 % satisfies $\rho_{s,p}(\bfs\lambda_p, \bfs \xi)\neq 0$. In particular, any such $\bfs \xi$
  is a lifting point of $\pi_{s,p}$, and $Y_{n-s+1,p}$ is a primitive element
  of $\pi_{s,p}^{-1}(\bfs \xi)$.
  \item No point of $\pi_{s,p}(W_{\bfs q_p^*}\cap \{G_p=0\})$
  belongs to $\pi_{s,p}\bigl(\pi_{s+1,p}^{-1}(\bfs q^*_p)\bigr)$.
\end{enumerate}
\end{theorem}

Let $(\bfs\lambda, \bfs p)\in \Z^{n^2}\times \Z^{n-1}$ be a
lucky pair. Then the linear forms $Y_1 \klk Y_n$ define a new set
of variables for $\Q[X_1 \klk X_n]$, the map
$\pi_s:V_s\to\A^{n-s}$ defined by $Y_1,\ldots,Y_{n-s}$ is a finite
morphism, and $\bfs p^s$ is a lifting point of $\pi_s$, for $1\le s
\le r$. Let $m_s$ $w^s_{n-s+2} \klk w^s_n$ be the polynomials of
$\Q[T]$ which, together with the linear forms $Y_1 \klk Y_n$, the map
$\pi_s: V_s\rightarrow \A^{n-s}$, and the point $\bfs p^s$, define a
Kronecker representation of the fiber $\pi_s^{-1}(\bfs p^s)$, with
$Y_{n-s+1}$ as primitive element, for $1\le s \le r$. Thus, $m_s$ is
the minimal polynomial of $Y_{n-s+1}$ in $\pi_s^{-1}(\bfs p^s)$, and we have
$$m'_s(Y_{n-s+1})Y_{n-s+i}\equiv w^s_{n-s+i}(Y_{n-s+1})\textrm{ in }
\pi_s^{-1}(\bfs p^s)$$ for $2\le i \le s$, with $\deg m_s=\delta_s$ and
$\deg w^s_{n-r+i} < \delta_s$ for $2\le i \le s$.

%As it is shown in \cite[Section 4.3]{GiMa19}, the polynomials $m_s$,
%$w^s_{n-s+2} \klk w^s_n$ are obtained by specializing a Chow form of
%the variety $V_s$ in $\bfs \lambda^s$ and $\bfs p^s$. Indeed, with
%the notations of Section 3, let $P_{V_s}\in \Z[\bfs \Lambda^s, Z_1
%\klk Z_{n-s+1}]$ be a Chow form of $V_s$. If $T$ is a new
%indeterminate, then we have
%$$
%m_s=\frac{P_{V_s}(\bfs \lambda^s, \bfs p^s, T)}{A_s^1(\bfs
%\lambda^{s+1})}, \quad
%w^s_{j}=\sum_{k=1}^{n}\frac{\lambda_{jk}}{A_s^1(\bfs
%\lambda^{s+1})}\frac{\partial P_{V_s}}{\partial
%\lambda_{n-s+1,k}}(\bfs \lambda^s, \bfs p^s, T) \quad  (n-s+2\le j
%\le n).
%$$
%
The following result asserts that the output obtained
modulo a lucky prime $p$ in Algorithm \ref{algo: kronecker for Q}
corresponds to the first-order $p$-adic approximations of the polynomials
in the Kronecker representation over $\Q$ (see \cite[Proposition
5.12]{GiMa19}).
\begin{theorem} \label{th:_lucky prime_b}
Let $(\bfs\lambda, \bfs p)\in \Z^{n^2}\times \Z^{n-1}$ be a lucky
pair, and let $\mathfrak{N}$ be the integer from Theorem \ref{th:_lucky
prime_a}. For $1\le s \le r$, let:
\begin{itemize}
\item  $m_s, w^s_{n-s+2} \klk w^s_n\in \mathbb{Q}[T]$ denote a
Kronecker representation
of the lifting fiber $\pi_s^{-1}(\bfs p^s)$, with $Y_{n-s+1}$ as
primitive element;
\item  $M_s, W^s_{n-s+2} \klk W^s_n\in \mathbb{Q}[Y_{n-s}, T]$
denote a Kronecker representation
of the lifting curve $C_s=\pi_s^{-1}(Y_1=p_1 \klk
Y_{n-s-1}=p_{n-s-1})$, with $Y_{n-s+1}$ as primitive element.
\end{itemize}
If $p$ is a prime such that $p \nmid \mathfrak{N}$, then all these
polynomials have coefficients in $\Z_{(p)}$, and their reductions
modulo $p$ yield:
\begin{itemize}
\item a Kronecker representation $m_{s,p}, w^s_{{n-s+2},p} \klk w^s_{n,p}\in \fp[T]$
of the lifting fiber $\pi_{s,p}^{-1}(\bfs p^s_p)$, with
$Y_{{n-s+1},p}$ as a primitive element; and
\item a Kronecker representation $M_{s,p}, W^s_{{n-s+2},p} \klk W^s_{n}\in \fp[Y_{n-s}, T]$
of the lifting curve $C_{s,p}:=\pi_{s,p}^{-1}(\{Y_{1,p}=p_{1,p} \klk
Y_{n-s-1,p}=p_{n-s-1,p}\})$, with $Y_{{n-s+1},p}$ as a primitive
element.
\end{itemize}
\end{theorem}

Finally, an additional property of a lucky prime $p$ is required for the
$p$-adic lifting step in Algorithm \ref{algo: kronecker for Q} to
succeed (see Proposition \ref{prop: correctness algo p
adic lift}): namely, $p$ must not divide either the integer
$\alpha_s$ defined in \cite[(5.4)]{GiMa19}, or the integer $A_s(\bfs
\lambda^{s+1})$. Consequently, we have the following result.
\begin{proposition} \label{prop:_lucky prime_c}
Under the hypotheses and notations of Theorem \ref{th:_lucky prime_b},
for $1\le s \le r$ the polynomial $m_s$ belongs to $\Z_{(p)}[T]$, and
the polynomial $G(\bfs p^s, Y_{n-s+1} \klk Y_n)m_s(Y_{n-s+1})$ belongs to the
radical of the ideal generated in  $\Z_{(p)}[Y_{n-s+1} \klk Y_n]$ by the
polynomials
$F_1(\bfs p^s, Y_{n-s+1} \klk Y_n) \klk F_s(\bfs p^s, Y_{n-s+1} \klk
Y_n)$.
\end{proposition}

Next we analyze the probability of success in selecting a lucky
pair $(\bfs \lambda,\bfs p)$ and a lucky prime $p$.
Let $0<\varepsilon < 1$, and define the set $\mathcal{S}:=\{0,
\dots, \lfloor 2\varepsilon^{-1}D\rfloor\}$, where $D$ is defined as
in \eqref{eq: definition D}. The following result shows that a
pair $(\bfs\lambda, \bfs p)$ chosen uniformly at random in
$\mathcal{S}^{n^2}\times \mathcal{S}^{n-1}$ is lucky with high
probability (see \cite[Lemma 6.1]{GiMa19}).
\begin{lemma}\label{lemma: lambda_and_p-prob} Let $(\bfs \lambda, \bfs p)$
be a pair chosen uniformly at random in $\mathcal{S}^{n^2}\times
\mathcal{S}^{n-1}$. Then the probability that $\textsf{R}(\bfs
\lambda )\cdot\textsf{N}_{\scriptscriptstyle \bfs \lambda}(\bfs
p)\neq 0$ is greater than $1-\varepsilon$.
\end{lemma}
\begin{proof}
Since $\deg \textsf{R}\le D$, it follows from Lemma \ref{lemma: Zippel_Schwartz}
that for a random choice of $\bfs \lambda$ in
$\mathcal{S}^{n^2}$, the probability that $\textsf{R}(\bfs \lambda
)\neq 0$ is greater that $1-\frac{\varepsilon}{2}$. Similarly, as
$\deg \textsf{N}_{\scriptscriptstyle \bfs \lambda}\le D$, for a
point $\bfs p$ chosen uniformly at random in $\mathcal{S}^{n-1}$,
the conditional probability that $\textsf{N}_{\scriptscriptstyle
\bfs \lambda}(\bfs p)\neq 0$, given that $\textsf{R}(\bfs \lambda
)\neq 0$, is greater than $1-\frac{\varepsilon}{2}$. Therefore, the
probability that $\textsf{R}(\bfs \lambda )\cdot
\textsf{N}_{\scriptscriptstyle \bfs \lambda}(\bfs p)\neq 0$ for a
random $(\bfs \lambda, \bfs p)$ in $\mathcal{S}^{n^2}\times \mathcal{S}^{n-1}$
is at least
$(1-\frac{\varepsilon}{2})^2\ge 1-\varepsilon$.
\end{proof}

Let $(\bfs \lambda, \bfs p)\in \mathcal{S}^{n^2}\times
\mathcal{S}^{n-1}$ be a lucky pair. Using height estimates from
\cite{DaKrSo13}, we obtain upper bounds for the heights of the
polynomials defining a Kronecker representation of the lifting fiber
$\pi_s^{-1}(\bfs p^s)$, with $Y_{n-s+1}$ as primitive element (see
\cite[Theorem A.7]{GiMa19}), and the integer $\mathfrak{N}$ (see
\cite[Theorem A.20]{GiMa19}).

\begin{proposition}\label{prop: eta_s_height_estimate}
With the notation of Lemma \ref{lemma: lambda_and_p-prob} and
Theorem \ref{th:_lucky prime_a}, let $(\bfs \lambda, \bfs p)\in
\mathcal{S}^{n^2}\times \mathcal{S}^{n-1}$ be a lucky pair. Then
the maximum height $\eta_s$ of the polynomials $m_s, w^s_{n-s+2}
\klk w^s_n$ satisfies $\eta_s\in\SO \bigl(nd^{s-1}(h+nd+d\log_2(\varepsilon^{-1}))\bigr)$.
Moreover, there exists an integer $\mathfrak{H}$ such that
  \begin{equation}\label{eq: log_frak_N_height}
{\sf h}(\mathfrak{N})\leq\mathfrak{H} \quad  \textrm{and} \quad
\log_2 \mathfrak{H} \in \mathcal{O}^{\sim}\bigr(\log_2
(d^{r}2^nh\log_2\varepsilon^{-1})\bigl).
\end{equation}
\end{proposition}

For the computation of a lucky prime, we rely on the following
well--known result for finding a small prime not dividing
a given integer.
 %(see e. g. \cite[Theorem 18.8]{GaGe99})
%
\begin{lemma}\label{lemma: finding_primes}
Let $B$, $\ell$ be positive integers, and let $M$ be a nonzero
integer such that $\log_2 |M|\leq \frac{B}{\ell}$. Then there exists
a randomized algorithm which, given $B$ and any positive integer
$k$, returns a prime $p$ with $B< p\le 2B$ such that $p\nmid M$. The
algorithm uses an expected number of $\mathcal{O}^\sim(k\,\log_2^2
B)$ bit operations and returns a correct result with probability at
least
\[
\Bigl(1-\frac{\log_2 B}{2^{k-1}}\Bigr)\Bigl(1-\frac{2}{\ell}\Bigr).
\]
\end{lemma}
\begin{proof}
According to, e.g.,  \cite[Theorem 18.8]{GaGe99}, there exists a
randomized algorithm that computes a random prime $p$ with $B<p\le
2B$ using an expected number of $\mathcal{O}^\sim(k\,\log_2^2 B)$
bit operations and with probability at least $1-{\log_2
B}/{2^{k-1}}$. Furthermore, a randomly chosen prime $p$ in this
interval satisfies $p\nmid M$ with probability at least
$1-{2}/{\ell}$. Combining these two probabilities yields the result.
\end{proof}
We conclude with the following lemma, which guarantees the computation of
a lucky prime of small bit length.
\begin{lemma} \label{lemma: computation_of_a_lucky_prime}
Under the assumptions of Theorem \ref{th:_lucky prime_a} and
Proposition \ref{prop: eta_s_height_estimate}, there exists a
randomized algorithm that, given $\mathfrak{H}$, computes a prime
$p$ satisfying $4\varepsilon^{-1}\mathfrak{H} < p \leq
10\varepsilon^{-1}\mathfrak{H}$ such that $p \nmid \mathfrak{N}$.
The algorithm uses an expected number of $\SO \bigr(\log_2^2 (d^r
2^nh\varepsilon^{-1})\log_2(\varepsilon^{-1})\bigl)$ bit operations
and succeeds with probability at least $1-\varepsilon$.
\end{lemma}
\begin{proof}
The lemma follows applying Lemma \ref{lemma: finding_primes} with
$\ell:=\lceil 4\varepsilon^{-1}\rceil$, $B:=\ell\mathfrak{H}$ and
$k:=\lceil \log_2\ell + \log_2(\log_2\ell\mathfrak{H})\rceil$, and
by invoking inequality $\eqref{eq: log_frak_N_height}$.
\end{proof}
%
%----------------------------------------------------------------------
%----------------------------------------------------------------------
%----------------------------------------------------------------------
%----------------------------------------------------------------------
%
\subsection{Reconstruction of a Kronecker representation over $\Q$}
\label{subsec: lifting_the_integers}
%
%Combining the algorithm of \cite{CaMa06a} with $p$--adic lifting,
%a randomized algorithm for computing a Kronecker representation of a
%zero--dimensional fiber $V_{\bfs p^r}$ of the Zariski closure
%$\mathcal{V}_r$ of $\mathcal{V}(F_1, \dots,
%F_r)\setminus\mathcal{V}(G)$ is obtained.

Let $m_r, w^r_{n-r+2} \klk w^r_n\in \mathbb{Q}[T]$ be a Kronecker
representation of the lifting fiber $\pi_r^{-1}(\bfs p^r)$, with
$Y_{n-r+1}$ as primitive element. Let $p$ be a lucky prime as in
Theorems \ref{th:_lucky prime_a} and \ref{th:_lucky prime_b}. By
Theorem \ref{th:_lucky prime_b}, the reduction modulo $p$ of these
polynomials yields a Kronecker representation $m_{r, p}, w^r_{n-r+2,
p} \klk w^r_{n, p}\in \fp[T]$ of the lifting fiber
$\pi_{r,p}^{-1}(\bfs p_p^r)$, with $Y_{n-r+1,p}$ as primitive
element.

In this section, we describe a $p$-adic lifting procedure to
reconstruct the polynomials $m_r, w^r_{n-r+2} \klk w^r_n$ from
$m_{r, p}, w^r_{n-r+2, p} \klk w^r_{n, p}$. Briefly, this procedure
computes $p$-adic approximations of these polynomials to sufficient
precision, followed by rational reconstruction of their coefficients
(see Lemma \ref{lemma: rational_reconstruction_algorithm}).

The key tool enabling this $p$-adic lifting is the following
algorithm (see \cite[Theorem 2]{GiLeSa01}). This algorithm is a
version of Algorithm \ref{algo: Newton lifting}, but defined over
$\Z$ instead of $k[\![X_{n-s} - p_{n-s}]\!]$. The correctness of
Algorithm \ref{algo: Global Newton Iterator} follows as that of
Proposition \ref{prop: main loop Newton lifting} by analogous
arguments.
\begin{algorithm}[Global Newton Iterator]\label{algo: Global Newton Iterator} ${}$
\begin{itemize}
\item[]{\bf Input:}
\begin{enumerate}
\item
Polynomials $G_1 \klk G_r$ in $\Z[X_1 \klk X_r]$;
\item A monic polynomial
$m\in \Z[T]$ and polynomials $\bfs v:=(v_1\klk v_r) \in \Z[T]^r$ with
$\deg v_i< \deg m$ for $1\le i \le r$;
\item
An ideal ${\sf I}\subset
\Z$; \end{enumerate} such that the following conditions are satisfied in $(\Z/{\sf I})[T]$:
    \begin{itemize}
      \item $v_1=T$;
      \item $G_i(\bfs v)\equiv 0 \mod m$ for $1\le i \le r$;
      \item If $J\in \Z[X_1 \klk X_r]^{r\times r}$ is the Jacobian
      matrix of $G_1 \klk G_r$ with respect to $X_1 \klk X_r$, then  $J(\bfs v)$ is invertible modulo $m$.
    \end{itemize}
\item[]{\bf Output:}
A monic polynomial $M\in \Z[T]$ and $\bfs V:=(V_{1}\klk V_r)\in
\Z[T]^r$ with $\deg M=\deg m$ and $\deg V_{i}<\deg M$ for $1\le i\le
r$, such that $M\equiv m$ and $V_{i}\equiv v_{i} \mod {\sf I}$ for
$1\le i \le r$, and satisfying in $(\Z/{\sf I}^{\, 2})[T]$:
\begin{itemize}
  \item $V_{1}=T$;
  \item  $G_{i}(\bfs V)\equiv 0 \mod (M)$ for $1\le i \le r$.
 % \item $Q'$, $J_s(\bfs p, \bfs V)$ and $G(\bfs p, \bfs V)$ are invertible modulo $Q$.
\end{itemize}

\smallskip

    \item[]{\bf 1.} Compute $J(\bfs v)^{-1}\in (\Z/{\sf I})[T]^{r\times
    r}$, the inverse of $J(\bfs v)$ in $(\Z/{\sf I})[T]/(m)$.
    \item[] {\bf 2.} In $(\Z/{\sf I}^{\,2})[T]/(m)$, perform:
    \begin{enumerate}
    \item[]{\bf 2.1.}
     Compute
    \[ \bfs w^t:=(w_{1}\klk w_r):=\bfs v^t - J(\bfs v)^{-1} \cdot  (F_1(\bfs v), \dots, F_r(\bfs v))^t. \]
    \item[]{\bf 2.2.} Compute $\Delta:= w_{1}-T$.
    \item[]{\bf 2.3.} For $i=1\klk r$, do
    \begin{enumerate}
      \item[]{\bf 2.3.1.} Compute $\Delta_i:=\Delta \frac{\partial w_i}{\partial T}$.
      \item[]{\bf 2.3.2.} Set $V_i:=w_i- \Delta_i$.
    \end{enumerate}
    \item[]{\bf 2.4} Compute $\Delta_{m}:=\Delta \frac{\partial m}{\partial T}$.
    \item[]{\bf 2.5} Set $M:=m-\Delta_{m}$.
    \end{enumerate}
  \item[]{\bf 3.} Return $M$ and $\bfs V:=(V_{1}, \dots, V_r)$.
\end{itemize}
\end{algorithm}

\begin{remark}\label{remark: Global Newton Iterator}
The polynomial matrix $J(\bfs V)$ is invertible modulo $M$ in
$(\Z/{\sf I^2})[T]^{r\times r}$. Indeed, if $g\in (\Z/{\sf
I})[T]^{r\times r}$ denotes the inverse of $J(\bfs v)$ in $((\Z/{\sf
I})[T]/(m))^{r\times r}$, then $2g-J(\bfs V)g^2$ is the inverse of
$J(\bfs V)$ in $((\Z/{\sf I}^2)[T]/(M))^{r\times r}$ (see, e.g.,
\cite[Theorem 9.2]{GaGe99}). This fact ensures that Algorithm
\ref{algo: Global Newton Iterator} can be applied iteratively.
Furthermore, if $m'$ is invertible modulo $m$ in $(\Z/{\sf I})[T]$,
then, by a similar argument, it follows that $M'$ is invertible
modulo $M$ in $(\Z/{\sf I^2})[T]$.
\end{remark}

Now assume that a lucky choice of $(\bfs\lambda, \bfs p)\in
\mathcal{S}^{n^2}\times \mathcal{S}^{n-1}$ has been made as in Lemma
\ref{lemma: lambda_and_p-prob}. Let $m_r, w^r_{n-r+2} \klk w^r_n\in
\mathbb{Q}[T]$ denote the corresponding Kronecker representation of
the lifting fiber $\pi_r^{-1}(\bfs p^r)$, with $Y_{n-r+1}$ as
primitive element. Assume further that a lucky prime $p$ has been
computed as in Lemma \ref{lemma: computation_of_a_lucky_prime}, and
let $m_{r,p}, w^r_{{n-r+2},p} \klk w^r_{n,p}\in \fp[T]$ be the
Kronecker representation of the lifting fiber $\pi_{r,p}^{-1}(\bfs
p_p^r)$, with $Y_{{n-r+1},p}$ as primitive element.

Define $v^r_{n-r+1,p}:=T$, and
$v^r_{n-r+i,p}:=(m'_{r,p})^{-1}w^r_{n-r+i,p} \mod m_{r,p}$ for $2\le
i \le r$. Moreover, denote $F_i^{\bfs Y}:=F_i(Y_1,\dots,Y_{n})$ as
the polynomial obtained by rewriting $F_i(X_1, \dots, X_n)$ in terms
of the new variables $Y_1,\dots, Y_n$, and let $J_r:=({\partial
F_i}/{\partial Y_{n-r+j}})_{1\le i, j \le r}$. Since $\bfs p_p^r$ is
a lifting point of $\pi_{r,p}: V_{r,p} \rightarrow \A_{\cfp}^{n-r}$,
the Jacobian determinant $\det J_r(\bfs p_p^r, Y_{n-r+1,p}\klk
Y_{n,p})$ of the polynomials $F_{1,p}(\bfs p_p^r, Y_{n-r+1,p}\klk
Y_{n,p}), \dots, F_{r,p}(\bfs p_p^r, Y_{n-r+1,p}\klk Y_{n,p})$ with
respect to the variables $Y_{n-r+1,p} \klk Y_{n,p}$ induces an
invertible element in the coordinate ring $\fp[\pi_{r,p}^{-1}(\bfs
p_p)]$ of the lifting fiber.
%$$\overline{\mathcal{J}}_{r,p}= (F_{1,p}(\bfs
%p_p^r, Y_{n-r+1,p}\klk Y_{n,p}), \dots, F_{r,p}(\bfs p_p^r,
%Y_{n-r+1,p}\klk Y_{n,p})): G_p(\bfs
%p_p^r, Y_{n-r+1,p}\klk Y_{n,p})^{\infty}$$

Thus, letting $\bfs v_p:=(v^r_{n-r+1,p} \klk v^r_{n,p})$, we have in
$(\Z/p\Z)[T]$:
 \begin{itemize}
      \item $v^r_{n-r+1, p}=T$;
      \item $F_{i}(\bfs p_p^r, \bfs v_p)\equiv 0 \mod m_{r,p}$ for $1\le i \le s$;
      \item  $J_r(\bfs p_p^r, \bfs v_p)$ is invertible modulo $m_{r,p}$.
    \end{itemize}
We can therefore apply Algorithm \ref{algo: Global Newton Iterator}
with input $G_i:=F_i(\bfs p^r, Y_{n-r+1} \klk Y_n)$ for $1 \le i \le
r$, $I=p\Z$, $m:=m_{r,p}$, and $\bfs v:=\bfs v_p:= (T,
v^r_{{n-r+2},p} \klk v^r_{n,p})$. Iterating this procedure yields
the following algorithm to compute $p$-adic approximations of
arbitrary order of the polynomials $m_r, w^r_{n-r+2} \klk w^r_n$.
\begin{algorithm}[$p$-adic lifting]\label{algo: p_adic_lifting_algorithm} ${}$
\begin{itemize}
\item[]{\bf Input:}
\begin{enumerate}
\item The polynomials $F_1(\bfs p^r, Y_{n-r+1} \klk Y_n) \klk F_r(\bfs p^r, Y_{n-r+1} \klk Y_n)$;
\item A lucky prime $p$ as in Theorem \ref{th:_lucky prime_a};
\item A Kronecker representation  $m_{r,p}, w^r_{{n-r+2},p} \klk w^r_{n,p}\in \fp[T]$
of the lifting fiber $\pi_{r,p}^{-1}(\bfs p^r_p)$, with $Y_{{n-r+1},p}$ as primitive element;
\item An integer $k:=2^{j_0}$.
\end{enumerate}

\item[]{\bf Output:} The $p$--adic approximations of order $k$ of the polynomials
$m_r, w^r_{n-r+2} \klk w^r_n\in \mathbb{Q}[T]$ which form a
Kronecker representation of the lifting fiber $\pi_r^{-1}(\bfs p^r)$, with
$Y_{n-r+1}$ as primitive element.

\item[]{\bf 1} Compute $v_{n-r+i,p}:=(m'_{r,p})^{-1}w^r_{n-r+i,p} \mod m_{r,p}$ for $2\le i \le r$.

\item[]{\bf 2} Set $M^0:=m_{r,p}$, $V^0:=T$ and $V^0_{n-r+i}:=v^r_{n-r+i,p}$ for $2\le i \le r$.

 \item[]{\bf 3.} While $0\le j < j_0$, do
 \begin{enumerate}
 \item[]{\bf 3.1} Apply Algorithm \ref{algo: Global Newton Iterator} on input
 $G_i:=F_i(\bfs p^r, Y_{n-r+1 \klk Y-n})$ for $1\le i \le r$, $m:=M^j$ and
 $\bfs v:=(V^j_{n-r+1} \klk V^j_n)$, with ${\sf I}:=p^{2^j}\Z$.
 \item[]{\bf 3.2} Set  $M^{j+1}$, $V^{j+1}_{n-r+1} \klk V^{j+1}_n$ as the output of the previous step.
 \item[]{\bf 3.3} Set $j:=j+1$.
 \end{enumerate}
 \item[]{\bf 4.} Return $M^{j_0}$ and  $W^{j_0}_{n-r+i}:=(M^{j_0})'V^{j_0}_{n-r+i} \mod M^{j_0}$ for $2\le i \le r$.
  \end{itemize}
\end{algorithm}

\begin{proposition}\label{prop: correctness algo p adic lift}
Algorithm \ref{algo: p_adic_lifting_algorithm} works correctly as
stated.
\end{proposition}
\begin{proof}
By the correctness of Algorithm \ref{algo: Global Newton Iterator},
we know that steps {\bf 1}--{\bf 3} of Algorithm \ref{algo:
p_adic_lifting_algorithm} are well-defined. It remains to prove that
the algorithm indeed outputs the claimed $p$-adic approximations.

For each $j \geq 0$, let $M^j$, $V^j_{n-r+1} \klk V^j_n$ be the
polynomials computed in steps {\bf 3.1} and {\bf 3.2} of the
algorithm. Define $$W^j_{n-r+i}:=(M^j)'V^j_{n-r+i} \mod M^j$$ in
$(\Z/p^{2^j}\Z)[T]$ for $2\le i \le r$ and $j\geq 0$. Denote by
$m^j_r$ and $w^{r, \, j}_{n-r+i}$ the $p$-adic approximations of
order $2^j$ of $m_r$ and $w^r_{n-r+i}$, respectively. We aim to
prove that for every $j\ge 0$,
$$M^j\equiv m^j_r\textrm{ and }W^j_{n-r+i}\equiv w^{r,\, j
}_{n-r+i} \mod p^{2^j}$$
for $2\le i\le r$. This will establish the correctness of the
algorithm.

Define $\bfs V^j:=(V^j_{n-r+1} \klk V_n^j)$ for $j\ge 0$. Then, for
each $j\ge 0$, the following hold in $(\Z/p^{2^j}\Z)[T]$:
\begin{align}
F_i(\bfs p^r, \bfs V^j)& \equiv 0 \mod M^j \textrm{ for }1\le i \le r; \label{eq: correctness_lifting_algo_1}\\
\textrm{Both }(M^j)' \textrm{ and }&J_r(\bfs p^r, \bfs V^j)\textrm{
are invertible modulo }M^j.\label{eq: correctness_lifting_algo_2}
\end{align}
%\begin{itemize}
% \item[(1)]   $F_i(\bfs p^r, \bfs V^j) \equiv 0 \mod M^j$ for $1\le i \le r$; \label{eq: correctness_lifting_algo_1}
% \item[(2)]  $(M^j)'$ and $J_r(\bfs p^r, \bfs V^j)$  are invertible modulo $M^j$.\label{eq: correctness_lifting_algo_2}
%\end{itemize}
Since $M^{j+1}\equiv M^j$ and $V^{j+1}_{n-r+i}\equiv V^j_{n-r+i}
\mod p^{2^j}$ for $1\le i \le r$ and all $j\geq 0$, the sequences
$(M^j)_{j\ge 0}$ and $(V^j_{n-r+i})_{j\ge 0}$ converge in $\Z_p[T]$
under the $p$-adic topology, where $\Z_p$ is the ring of $p$-adic
integers. Let $M^{\infty}$ and $V^{\infty}_{n-r+i}$ denote their
limits for $1\le i \le r$, and write $\bfs
V^{\infty}:=(V^{\infty}_{n-r+1} \klk V^{\infty}_{n})$.

Passing to the limit in $\eqref{eq: correctness_lifting_algo_1}$, we
have in $\Z_p[T]$:
 \begin{equation}\label{eq: correctness_lifting_algo_3}
  F_i(\bfs p^r, \bfs V^{\infty}) \equiv 0 \mod M^{\infty} \textrm{ for } 1\le i \le r.
\end{equation}
Moreover, since $\Z_{(p)}$ is a subring of $\Z_p$, by Proposition
\ref{prop:_lucky prime_c} there exists an integer $\mu>0$ such that
$$
(G(\bfs p^r, Y_{n-r+1} \klk Y_n)m_r(Y_{n-r+1}))^{\mu}\in (F_i(\bfs
p^r, Y_{n-r+1} \klk Y_n): 1\le i \le r)\Z_p[Y_{n-r+1},\ldots,Y_n].
$$
Substituting $\bfs V^{\infty}$ for $(Y_{n-r+1} \klk Y_n)$, and using
\eqref{eq: correctness_lifting_algo_3} we obtain in $\Z_p[T]$:
\begin{equation}\label{eq: correctness_lifting_algo_4}
 (G(\bfs p^r, \bfs V^{\infty})m_r)^{\mu} \equiv 0 \mod M^{\infty}.
\end{equation}
We claim that $G(\bfs p^r, \bfs V^{\infty})$ is  invertible modulo
$M^{\infty}$ in $\Z_p[T]$. Indeed, since $M^{\infty}\equiv M^0$ and
$\bfs V^{\infty}\equiv \bfs V^0 \mod p$,
$$
\mathrm{Res}\big(G(\bfs p^r, \bfs V^{\infty}), M^{\infty}\big)\equiv
\mathrm{Res}\big(G(\bfs p^r, \bfs V^{0}), M^{0}\big) \mod p.
$$
We know that $M^0=m_{r,p}$ is the minimal polynomial of
$Y_{n-r+1,p}$ in $\pi^{-1}_{r,p}(\bfs p^r_p)$ and $\bfs V^{0}=(T,
v^r_{n-r+2, p} \klk v^r_{n,p})$ parametrizes this fiber. Therefore,
$$\pi^{-1}_{r,p}(\bfs p^r_p)=\{(\bfs p_p^r, \bfs V^{0}(\alpha)):
m_{r,p}(\alpha)=0\}.$$ Since  $\pi^{-1}_{r,p}(\bfs p^r_p) \subset
\{G_p\neq 0\}$, the polynomials $m_{r,p}$ and $G(\bfs p^r, \bfs
V^{0})$ have no common roots in $\cfp$, and their resultant is
nonzero. It follows that $\mathrm{Res}_T\big(G(\bfs p^r, \bfs
V^{\infty}), M^{\infty}\big)$ is a unit of $\Z_p$, which implies the
claim.

From \eqref{eq: correctness_lifting_algo_4}, we deduce
\begin{equation}\label{eq: correctness_lifting_algo_5}
 (m_r)^{\mu} \equiv 0 \mod M^{\infty}.
\end{equation}
Since $M^0=m_{r,p}\in(\Z/p\Z)[T]$ is monic and square-free, we see
that $M^{\infty}$ is monic and square-free in $\Z_p[T]$. Therefore,
by \eqref{eq: correctness_lifting_algo_5} it follows that $m_r
\equiv 0 \mod M^{\infty}$ in $\Z_p[T]$. But as $\deg M^{\infty}=\deg
m_r$ and both polynomials are monic, we conclude
\begin{equation}\label{eq: correctness_lifting_algo_6}
 M^{\infty}= m_r.
\end{equation}
As a consequence,
$$
M^j\equiv M^{\infty}=m_r\equiv m^j_r \mod p^{2^j}\textrm{ for all
}j\geq 0,
$$
as asserted.

Next, for each $j\geq 0$, let $v^{r, \, j}_{n-r+2} \klk v^{r, \,
j}_{n}$ be the $p$-adic approximations of order $2^j$ of
$v^r_{n-r+2} \klk v^r_n$. We will prove that $v^{r, \,
j}_{n-r+i}\equiv V^j_{n-r+i} \mod p^{2^j}$ for every $j\geq 0$. We
proceed by induction on $k$ for $0\le k\le j$, proving
$$
v^{r, \, j}_{n-r+i}\equiv V^j_{n-r+i} \mod p^{2^k}.
$$
The base case $k=0$ holds since by construction
$$
v^{r, \, j}_{n-r+i}\equiv v^r_{n-r+i,p}=V^0_{n-r+i}\equiv
V^j_{n-r+i} \mod p.
$$
Suppose that the claim holds for $k-1$. Define $\bfs v:=(T,
v^r_{n-r+2} \klk v^r_n)$. Consider the Taylor expansion of $\bfs
F_r(\bfs p^r, Y_{n-r+1} \klk Y_n)$ around $\bfs V^{\infty}$:
\begin{equation}\label{eq: taylor_expansion_Z_p}
\bfs F_r(\bfs p^r, \bfs v)-\bfs F_r(\bfs p^r, \bfs
V^{\infty})=J_r(\bfs p^r, \bfs V^{\infty})\cdot ( \bfs v-\bfs
V^{\infty})+\mathcal{O}\big( (\bfs v-\bfs V^{\infty})^2\big).
\end{equation}
Since $m_r, w^r_{n-r+2} \klk w^r_n$ form a Kronecker representation
of $\pi_r^{-1}(\bfs p^r)$, with $Y_{n-r+1}$ as primitive element,
$F_i(\bfs p^r, \bfs v)\equiv 0 \mod m_r$ in $\Q[T]$ for $1\le i \le
r$. Moreover, as  $m_r$ and $v^r_{n-r+2} \klk v^r_n$ belong to
$\Z_{(p)}[T]$, and $m_r$ is monic, we deduce that  $F_i(\bfs p^r,
\bfs v)\equiv 0 \mod m_r$ in $\Z_{(p)}[T]$, and hence in
$\Z_{p}[T]$, for $1\le i \le r$. Together with \eqref{eq:
correctness_lifting_algo_3}, \eqref{eq: correctness_lifting_algo_6}
and  \eqref{eq: taylor_expansion_Z_p}, this yields
$$
0\equiv J_r(\bfs p^r, \bfs V^{\infty})\cdot ( \bfs v-\bfs
V^{\infty})+\mathcal{O}\big( (\bfs v-\bfs V^{\infty})^2\big) \mod
M^{\infty}
$$
in $\Z_{p}[T]$. Let $\bfs v^j:=(T, v^{r, \, j}_{n-r+2} \klk v^{r, \,
j}_n)$. We deduce that
$$
0\equiv J_r(\bfs p^r, \bfs V^j)\cdot ( \bfs v^j-\bfs
V^j)+\mathcal{O}\big( (\bfs v^j-\bfs V^j)^2\big) \mod M^{j}
$$
in $(\Z/p^{2^j}\Z)[T]$ and, since $k\le j$, \textit{a fortiori} in
$(\Z/p^{2^k}\Z)[T]$. By the induction hypothesis, $\bfs v^j\equiv
\bfs V^j \mod p^{2^{k-1}}$, and thus $(\bfs v^j-\bfs V^j)^2\equiv 0
\mod p^{2^k}$, leading to
$$
0\equiv J_r(\bfs p^r, \bfs V^j)\cdot ( \bfs v^j-\bfs V^j) \mod M^{j}
$$
in $(\Z/p^{2^k}\Z)[T]$. Since $J_r(\bfs p^r, \bfs V^j)$ is
invertible modulo $M^j$ in $(\Z/p^{2^j}\Z)[T]$, it is invertible
modulo $M^j$ in $(\Z/p^{2^k}\Z)[T]$, and $v^{r, \, j}_{n-r+i}\equiv
V^j_{n-r+i} \mod M^j$ in $(\Z/p^{2^k}\Z)[T]$ for $1\le i \le r$.
Since $\deg v^{r, \, j}_{n-r+i}, \deg V^j_{n-r+i} < \deg M^j$, it
follows that $v^{r, \, j}_{n-r+i}= V^j_{n-r+i}$ in
$(\Z/p^{2^k}\Z)[T]$ for $1\le i \le r$, proving the claim.

Finally, since $M^j\equiv m^j_r$ and $V^j_{n-r+i}\equiv v^{r, \,
j}_{n-r+i} \mod p^{2^j}$, we obtain $$ W^j\equiv (m^j_r)'v^{r, \,
j}_{n-r+i}\equiv w^{r, \, j}_{n-r+i} \mod m^j_r $$ in
$(\Z/p^{2^j}\Z)[T]$. As $\deg W^j_{n-r+i}, \deg w^{r, \, j}_{n-r+i}
< \deg m^j_r$, this implies $$W^j_{n-r+i}\equiv w^{r, \, j}_{n-r+i}
\mod p^{2^j}$$ for all $j\geq 0$. This completes the proof.
\end{proof}

We have the following result regarding the computational complexity
of these computations.%; its proof follows a similar reasoning as
%that of Theorem \ref{th: cost Newton lifting}.
%
\begin{proposition}\label{prop: p_adic_lifting_cost}
Suppose that $F_1(\bfs p^r, Y_{n-r+1} \klk Y_n)
\klk F_r(\bfs p^r, Y_{n-r+1} \klk Y_n)$ are given by
a straight--line program of length $L$. Then Algorithm \ref{algo:
p_adic_lifting_algorithm} can be performed using
$$\mathcal{O}\Bigl(\bigl((Lr+r^4+\log_2(\delta_r)\log_2\log_2 p){\sf M}(\delta_r){\sf M}(\log_2p)+
(Lr+r^3){\sf M}(\delta_r){\sf M}(k\log_2p\bigr)\Bigr)$$ bit
operations.
\end{proposition}
\begin{proof}
Step {\bf 1} consists of one inversion and $r-1$ multiplications in
the ring $\fp[T]/(m_{r,p})$. The inversion can be performed using
$\mathcal{O}({\sf M}(\delta_r)\log_2 \delta_r)$ additions and
multiplications, plus $\mathcal{O}(\delta_r)$ inversions in $\fp$,
amounting to a total cost of
$$\mathcal{O}\bigl({\sf
M}(\delta_r)\log_2(\delta_r){\sf M}(\log_2p) + \delta_r{\sf
M}(\log_2p)\log_2\log_2p\bigr)= \mathcal{O}\bigl({\sf
M}(\delta_r)\log_2(\delta_r){\sf M}(\log_2p)\log_2\log_2p\bigr)$$
bit operations. The $r-1$ multiplications each cost
$\mathcal{O}\bigl({\sf M}(\delta_r){\sf M}(\log_2 p)\bigr)$, so step
{\bf 1} requires in total $$\mathcal{O}\bigl(
(r+\log_2(\delta_r)\log_2(\log_2p)){\sf M}(\delta_r){\sf M}(\log_2p)
\bigr)$$ bit operations.

We now analyze step {\bf 3}. We begin with $j=0$ and apply Algorithm
\ref{algo: Global Newton Iterator} with input $G_i:=F_i({\bfs p^r},
Y_{n-r+1}, \dots, Y_n)$ for $1\le i \le r$, $m:=M^0$ and ${\bfs
v}:={\bfs V^0}:=(V^0_{n-r+1}, \dots, V^0_n)$ with ${\sf I}:=p\Z$.

Step {\bf 1} in Algorithm \ref{algo: Global Newton Iterator}
computes the inverse $(J_r(\bfs V^0))^{-1}$ in $\bigl(\fp[T]/(
M^0)\bigr)^{r\times r}$. According to Lemma \ref{lemma: complexity
det J and Adj J}, the determinant and the adjoint matrix of
$J_r(\bfs V^0)$ can be computed with $\mathcal{O}(Lr+r^4)$ arithmetic operations
in $\fp[T]/( M^0)$, and thus with $\mathcal{O}((Lr+r^4){\sf
M}(\delta_r){\sf M}(\log_2p)\bigr)$ bit operations. The inversion of
$\det (J_r(\bfs V^0))$ in $\fp[T]/( M^0)$ requires an additional
$$\mathcal{O}\bigl({\sf M}(\delta_r)\log_2(\delta_r){\sf
M}(\log_2p)\log_2\log_2p\bigr)$$ bit operations. Then, $(J_r(\bfs
V^0))^{-1}=(\det (J_r(\bfs V^0)))^{-1}\mathrm{Adj} (J_r(\bfs V^0))$
can be computed with  $r^2$ additional multiplications in $\fp[T]/(
M^0)$, which amounts to $\mathcal{O}\bigl(r^2{\sf M}(\delta_r){\sf
M}(\log_2 p)\bigr)$ bit operations. Thus the computation of
$(J_r(\bfs V^0))^{-1}$ requires a total of
$$
\mathcal{O}\bigl((Lr+r^4 +\log_2(\delta_r)\log_2(\log_2p)){\sf
M}(\delta_r){\sf M}(\log_2p)\bigr)
$$
bit operations.

In step {\bf 2.1}, we compute
$$
{\bfs w^t}:=({\bfs V^0})^t-(J_r(\bfs V^0))^{-1}\cdot (F_1({\bfs
p^r}, \bfs V^0), \cdots, F_1({\bfs p^r}, \bfs V^0))
$$
in $(\Z/p^2\Z)[T]/( M^0)$. Evaluating the polynomials $F_1({\bfs
p^r}, \bfs V^0), \ldots, F_r({\bfs p^r}, \bfs V^0)$ requires
$\mathcal{O}\bigl(L\bigr)$ arithmetic operations in $(\Z/p^2\Z)[T]/(
M^0)$, and computing the matrix-vector product and subtraction adds
another $\mathcal{O}(r^2)$ arithmetic operations in $(\Z/p^2\Z)[T]/(
M^0)$. Hence, step {\bf 2.1} costs
$$\mathcal{O}\bigl ((Lr+r^2){\sf
M}(\delta_r){\sf M}(\log_2(p^2))\bigr)$$
bit operations.

Steps {\bf 2.2} through {\bf 2.5} involve $\mathcal{O}(r)$ additions
and multiplications in $(\Z/p^2\Z)[T]/( M^0)$, with cost
$\mathcal{O}\bigl( r{\sf M}(\delta_r){\sf M}(\log_2(p^2))\bigr)$ bit
operations. Therefore, the total cost of Algorithm \ref{algo: Global
Newton Iterator} for $j=0$ is dominated by the computation of
$(J_r(\bfs V^0))^{-1}$ in step {\bf 1}.

Now, let $0 <  j < j_0$ and suppose, by induction, that step {\bf 3}
of Algorithm \ref{algo: p_adic_lifting_algorithm} has been executed
for all $0 \le j' < j$.  In particular, assume we have already
computed the polynomials $M^{j}$, $\bfs V^{j}:=(V^{j}_{n_r+1},
\cdots, V^{j}_n)$ of $(\Z/p^{2^{j}}\Z)[T]$, and a matrix $C\in
(\Z/p^{2^{j-1}})[T]^{r\times r}$ such that $C$ is the inverse of
$J_r(\bfs V^{j-1})$ in
$\bigl((\Z/p^{2^{j-1}}\Z)[T]/(M^{j-1})\bigr)^{r\times r}$.

To compute the polynomials $M^{j+1}$, $\bfs
V^{j+1}:=(V^{j+1}_{n-r+1}, \cdots, V^{j+1}_n)$ of
$(\Z/p^{2^{j+1}}\Z)[T]$, we apply Algorithm \ref{algo: Global Newton
Iterator} with input $m:=M^j$, ${\bfs v}:={\bfs V^{j}}$, and ${\sf
I}=p^{2^j}\Z$. According to Remark \ref{remark: Global Newton
Iterator}, the inverse $(J_r(\bfs V^j))^{-1}$ can be obtained via
the identity
$$(J_r(\bfs V^j))^{-1}=2C-J_r(\bfs V^j)C^2\in (\Z/p^{2^j})[T]^{r\times r},$$
which requires $\mathcal{O}\bigl(r^3{\sf M}(\delta_r){\sf
M}(\log_2(p^{2^j}))\bigr)$ bit operations. Step {\bf 2} of Algorithm
\ref{algo: Global Newton Iterator} involves evaluations of the
$F_i$'s and various updates, requiring $\mathcal{O}\bigl((Lr+r^2){\sf
M}(\delta_r){\sf M}(\log_2(p^{2^{j+1}}))\bigr)$ bit operations.

Therefore, the total cost of step {\bf 3} of Algorithm \ref{algo:
p_adic_lifting_algorithm} for fixed $j$ is
$$\mathcal{O}\bigl((Lr+r^3){\sf M}(\delta_r){\sf
M}(\log_2(p^{2^{j+1}}))\bigr)$$ bit operations. Summing over $1\le j
< j_0-1$, we obtain a total cost of
$$
\mathcal{O}\bigl((Lr+r^3){\sf M}(\delta_r)\sum_{j=1}^{j_0-1}{\sf
M}(\log_2(p^{2^{j+1}}))\bigr)=\mathcal{O}\bigl((Lr+r^3){\sf
M}(\delta_r){\sf M}(k\log_2(p)\bigr)
$$
bit operations.

Finally, step {\bf 4} of Algorithm \ref{algo:
p_adic_lifting_algorithm} consists of $r-1$ multiplications in
$(\Z/p^{k}\Z)[T]/M^{j_0}$, requiring $\mathcal{O}\bigl(r{\sf
M}(\delta_r){\sf M}(k\log_2 p)\bigr)$ bit operations.
\end{proof}

For the final step of reconstructing the output Kronecker
representation over $\Q$, we employ the well-known {\sf Rational
reconstruction} algorithm (see, for example, \cite[\S
5.10]{GaGe99}). Its complexity is summarized in the following lemma.
\begin{lemma}\label{lemma: rational_reconstruction_algorithm}
Let $\alpha:=p/q\in\Q$, with $p,q\in\Z\setminus\{0\}$ coprime, and
let $f,g\in \mathbb{Z}$ with $0\leq g< f$, satisfying the following
conditions:
\begin{itemize}
  \item The height of $\alpha$ satisfies ${\sf h}(\alpha)\le \frac{1}{2}\log_2(f/2)$.
 % $m \geq 2\max \{|p|,q\}^{2}$.
  \item $\gcd (q,f)=1$, and $pq^{-1}\equiv g \mod (f)$, where $q^{-1}$ denotes the inverse of $q$ modulo $f$.
\end{itemize}
Given $g$ and $f$, the integers $p$ and $q$ can be deterministically
computed with $\mathcal{O}({\sf M}(\log_2f))$ bit operations.
\end{lemma}

By combining Algorithm \ref{algo: p_adic_lifting_algorithm} with the
rational reconstruction algorithm from the lemma above, we obtain
the following procedure for reconstructing the output Kronecker
representation over $\Q$ from its modular reduction.
\begin{algorithm}[Reconstruction of a Kronecker representation]
\label{algo: geometric_solution_reconstruction} ${}$
\begin{itemize}
\item[]{\bf Input:}
\begin{enumerate}
\item The polynomials $F_1(\bfs p^r, Y_{n-r+1} \klk Y_n) \klk F_r(\bfs p^r, Y_{n-r+1} \klk
Y_n)$.
\item A lucky prime $p$ as in Theorem \ref{th:_lucky prime_a}.
\item The polynomials  $m_{r,p}, w^r_{{n-r+2},p} \klk w^r_{n,p}\in \fp[T]$
forming a Kronecker representation of  $\pi_{r,p}^{-1}(\bfs p^r_p)$,
with $Y_{{n-r+1},p}$ as primitive element,
\item An upper bound $\eta \in \mathbb{N}$ for the heights of the coefficients
of $m_r, w^r_{n-r+2} \klk w^r_n$ in the Kronecker representation of
$\pi_{r}^{-1}(\bfs p^r)$ over $\Q$, with $Y_{n-r+1}$ as primitive
element.
\end{enumerate}

\item[]{\bf Output:} The polynomials $m_r, w^r_{n-r+2} \klk w^r_n$ forming
a Kronecker representation of  $\pi_r^{-1}(\bfs p^r)$, with
$Y_{n-r+1}$ as primitive element.

\item[]{\bf 1} Apply Algorithm \ref{algo: p_adic_lifting_algorithm} with $k$
equal to the smallest power of $2$ such that $\eta\le
(1/2)\log_2(p^{k}/2)$, to compute $p$-adic approximations of order
$k$ of $m_r, w^r_{n-r+2} \klk w^r_n$.

\item[]{\bf 2} Apply rational reconstruction to
determine the coefficients of $m_r, w^r_{n-r+2} \klk w^r_n$ from the
approximations obtained in step {\bf 1}.
  \end{itemize}
\end{algorithm}

We then have the following result regarding the overall complexity:
\begin{lemma}\label{lemma: lifting_the_integers}
Suppose that $F_1(\bfs p^r, Y_{n-r+1} \klk Y_n) \klk F_r(\bfs p^r, Y_{n-r+1} \klk Y_n)$ are given by a
straight--line program in $\fp[Y_{n-r+1} \klk Y_n]$ of length at most $L$.
Then Algorithm \ref{algo: geometric_solution_reconstruction} can be
performed using $$\SO\bigl((Lr+r^4)\delta_r\eta +
r\delta_r\log_2p\bigr)$$ bit operations.
\end{lemma}
\begin{proof}
If $k=1$, then $\eta\le \log_2 p$, and step {\bf 1} does not involve
any computations. By Lemma \ref{lemma:
rational_reconstruction_algorithm}, step {\bf 2} requires
$\SO(r\delta_r\log_2p\bigr)$ bit operations.

On the other hand, if $k>1$, then we have $k\log_2p \in
\mathcal{O}(\eta)$. According to Proposition \ref{prop:
p_adic_lifting_cost}, step {\bf 1} requires
$\SO((Lr+r^4)\delta_r\eta)$ bit operations, and step {\bf 2} adds
another $\SO(r\delta_r\eta)$ bit operations. This completes the
proof.
\end{proof}
%
%----------------------------------------------------------------------
%----------------------------------------------------------------------
%----------------------------------------------------------------------
%----------------------------------------------------------------------
%
\subsection{The whole algorithm for $k=\Q$}
By combining the algorithms for computing a Kronecker representation
over a finite field, $p$-adic lifting, and reconstruction over $\Q$,
we obtain an algorithm to compute a Kronecker representation over
$\Q$ of a lifting fiber of the input variety $V$.

We proceed as follows. Assume that the input polynomials $F_1\klk
F_r,G$ are provided by a straight--line program $\beta$ of length at
most $L$ with integer parameters. First, we randomly choose a lucky
pair $(\bfs\lambda, \bfs p)\in \mathcal{S}^{n^2}\times
\mathcal{S}^{n-1}$, where $\mathcal{S}$ is defined as earlier. Then,
we compute a ``lucky'' prime $p$ according to Lemma \ref{lemma:
computation_of_a_lucky_prime}. Reducing the parameters of $\beta$
modulo $p$ we obtain a straight--line program $\beta_p$ of length at
most $L$, representing the polynomials $F_{1,p}, \dots,F_{r,p},G_p$.
We can then apply Algorithm \ref{algo: main algorithm for k} with
$k:=\fp$ to compute a Kronecker representation $m_{r,p},
w^r_{{n-r+2},p} \klk w^r_{n,p}$ of the lifting fiber
$\pi_{r,p}^{-1}(\bfs p^r_p)$ of
$V_{r,p}:=\overline{V(F_{1,p},\ldots,F_{r,p})\setminus V(G_p)}$,
with $Y_{n-r+1,p}$ as primitive element. Finally, using the rational
reconstruction of Algorithm \ref{algo:
geometric_solution_reconstruction}, we lift these polynomials to
obtain the Kronecker representation $m_r, w^r_{n-r+2} \klk w^r_n$ of
the lifting fiber $\pi_r^{-1}(\bfs p^r)$ of $V_r$ over $\Q$, with
$Y_{n-r+1}$ as primitive element. This yields an efficient algorithm
to compute the desired Kronecker representation, provided that the
random choices of $(\bfs\lambda, \bfs p)$ and $p$ are lucky.

We summarize
the procedure in the following algorithm.
\begin{algorithm}[Kronecker representation over $\Q$] ${}$
\label{algo: kronecker for Q}
\begin{itemize}
\item[]{\bf Input:}
\begin{enumerate}
\item Polynomials $F_1,\ldots,F_r,G\in \Z[X_1,\ldots,X_n]$ satisfying
(${\sf H}_1$)--(${\sf H}_2$);
\item A parameter $\varepsilon$ with
$0<\varepsilon<1/2$;
\item Access to uniformly random elements of the
finite set $\mathcal{S}:=\{0,\ldots,\lceil
2\varepsilon^{-1}D\rceil\}$, where $D$ is defined as in \eqref{eq:
definition D}.
\end{enumerate}

\item[]{\bf Output:}
\begin{enumerate}
\item
Linear forms $Y_1,\ldots,Y_n\in\Z[X_1,\ldots,X_n]$;
\item
A point $\bfs p^r:=(p_1,\ldots,p_{n-r})\in\Z^{n-r}$;
\item
Polynomials $m, w_{n-r+2}\klk w_n\in\Q[T]$;
\end{enumerate}
such that:
\begin{enumerate}
  \item The linear map $\pi_r:V\to\A^{n-r}$ defined by $Y_1,\ldots,Y_{n-r}$
  is a finite morphism;
  \item $\bfs p^r$ is a lifting point of $\pi_r$;
  \item
  $\pi_r:V\to\A^{n-r}$, the linear form $Y_{n-r+1}$, and
  the polynomials $m, w_{n-r+2}\klk w_n$, form a Kronecker
  representation of $\pi_r^{-1}(\bfs p^r)$.
\end{enumerate}

\item[]{\bf 1} Randomly select the coefficients $\bfs \lambda\in\mathcal{S}^{n\times n}$
of linear forms $Y_1,\ldots,Y_n\in \Z[X_1,\ldots,X_n]$. Set $\bfs
Y:=(Y_1,\ldots,Y_n)$.
\item[]{\bf 2} Randomly select the coordinates of a point
$\bfs p:=(p_1,\ldots,p_{n-1})\in\mathcal{S}^{n-1}$. Set $\bfs
p^s:=(p_1,\ldots,p_{n-s})$ for $1\le s\le r$.
\item[]{\bf 3} Compute a lucky prime $p$ via
Lemma \ref{lemma: computation_of_a_lucky_prime}.
\item[]{\bf 4} Compute the reduced linear forms $Y_{1,p} \klk Y_{n,p}
\in \fp[X_1 \klk X_n]$ and the point $\bfs p_p:=(p_{1,p} \klk
p_{n-1,p})\in\fp^{n-1}$. Set $\bfs Y_p:=(Y_{1, p},\ldots,Y_{n, p})$
and $\bfs p^s_p:=(p_{1,p} \klk p_{n-s, p})$ for $1\le s \le r$.
\item[]{\bf 5} Obtain a straight-line program for evaluating
the polynomials $F_{i, p}^{\bfs Y}:=F_{i, p}(\bfs Y_p)$ $(1\le i\le r)$ and $G_p^{\bfs Y}:=G_p(\bfs
Y_p)$.
\item[]{\bf 6} Compute the dense representation $m_{1,p}\in\fp[T]$ of
$$m_{1,p}=F_{1, p}^{\bfs Y}(\bfs p_p^1,T)/
\gcd\big(F_{1,p}^{\bfs Y}(\bfs p_p^1,T),G_p^{\bfs Y}(\bfs
p_p^1,T)\big).$$
\item[]{\bf 7} Set
$e:=\max\{1,\log_p(12r\varepsilon^{-1}d^{4r})\}$; define $k:=\fpe$.
\item[]{\bf 8} For $s=1\klk r-1$, do
\begin{itemize}
  \item[]{\bf 8.1} Apply Algorithm \ref{algo: Newton lifting}
  to compute a Kronecker representation of
$$C_s:=\pi_{s,p}^{-1}(\{Y_{1,p}=p_{1,p},\ldots,Y_{n-s-1,p}=p_{n-s-1,p}\}),$$ with
$Y_{n-s+1,p}$ as primitive element;
  \item[]{\bf 8.2} Apply Algorithm \ref{algo: comput m_(s+1) and v_(n-s+1)} over $k$
  to compute polynomials $m_{s+1,p},v_{n-s+1,p}^{s+1}\in\fp[T]$ satisfying $m_{s+1,p}(Y_{n-s,p})\equiv 0$ and
  $Y_{n-s+1,p}-v_{n-s+1,p}^{s+1}(Y_{n-s,p})\equiv 0$ in $\pi_{s+1,p}^{-1}(\bfs
  p_p^{s+1})$, with $\deg
m_{s+1,p}=\delta_{s+1}$, $\deg
  v_{n-s+1,p}^{s+1}<\delta_{s+1}$.
  \item[]{\bf 8.3} Apply
Algorithm \ref{algo: conclusion recursive
  step} to compute polynomials $v_{n-s+2,p}^{s+1}\klk v_{n, p}^{s+1}\in\fp[T]$
  satisfying
  $Y_{n-s+i,p}-v_{n-s+i,p}^{s+1}(Y_{n-s,p})\equiv 0$ in $\pi_{s+1,p}^{-1}(\bfs
  p_p^{s+1})$  for $2\le i\le s$, with $\deg v_{n-s+i,p}^{s+1}<\delta_{s+1}$ for $2\le i\le
  s$.
\end{itemize}
\item[]{\bf 9} Compute $w_{n-r+i,p}^r:=(m'_{r,p})^{-1}v_{n-r+i,p}^r$ for $2\le i \le r$.
\item[]{\bf 10} Compute an upper bound $\eta \in \SO(nd^{r-1}(h+nd))$ for the heights of $m_r,
 w^r_{n-r+2} \klk w^r_n$.
 \item[]{\bf 11} Set $m_p:=m_{r,p}$ and $w_{n-r+2,p}:=w_{n-r+2,p}^r\klk
w_{n,p}:=w_{n,p}^r$.
 \item[]{\bf 12}
 Apply Algorithm \ref{algo: geometric_solution_reconstruction} using
 $m_{r,p}$,
 $w_{n-r+2,p}^r\klk w_{n,p}^r$  and $\eta$ to obtain $m_r,w_{n-r+2}^r\klk w_{n}^r$.
\item[]{\bf 13} Output $Y_1 \klk Y_n$, $\bfs p^r$, $m:=m_r$ and $w_{n-r+2}:=w^r_{n-r+2} \klk w_n:=v^r_n$.
\end{itemize}
\end{algorithm}

We establish the correctness of this algorithm.
\begin{theorem}
For a lucky choice of $\bfs \lambda$,  $\bfs p$ and  $p$, Algorithm
\ref{algo: kronecker for Q} correctly computes a Kronecker representation
of a lifting fiber of $V$.
\end{theorem}
\begin{proof}
Suppose the randomly chosen $\bfs \lambda$ and $\bfs p$ in steps
{\bf 1} and {\bf 2} are lucky. Then, by Theorem \ref{th:_lucky
prime_b}, the linear forms $Y_1 \klk Y_n$ and the point $\bfs
p^r:=(p_1 \klk p_{n-r})$ satisfy the output properties (1) and (2),
with $Y_{n-r+1}$ is a primitive element of $\pi_r^{-1}(\bfs p^r)$.
Let $m_r\in \Q[T]$ denote the minimal polynomial of $Y_{n-r+1}$ over
$\pi_r^{-1}(\bfs p^r)$, and let $w_{n-r+2}^r\klk w_{n}^r\in \Q[T]$
be the polynomials satisfying $m'_r(Y_{n-r+1})Y_{n-r+i}\equiv
w_{n-r+i}(Y_{n-r+1})$ over $\pi_r^{-1}(\bfs p^r)$, with $\deg
w_{n-r+i}< \deg m_r$, for $2\le i \le r$. Then $Y_1 \klk Y_n$, $\bfs
p^r$, $m_r, w_{n-r+2}^r\klk w_{n}^r$ satisfy output condition
(3).

Assuming that the prime $p$ computed in step {\bf 3} is lucky,
Theorem \ref{th:_lucky prime_a} ensures that conditions (1)--(5)
hold for $1\le s < r$, and Theorem \ref{th:_lucky prime_b}
guarantees that the linear forms $Y_{1,p}\klk Y_{n,p}$, the point
$\bfs p_p^s$, and the polynomials $m_{s,p}, w_{n-s+2,p}^s\klk
w_{n,p}^s$, form a Kronecker representation of $\pi_{s,p}^{-1}(\bfs
p_p^s)$ for $1\le s \le r$. Therefore, step {\bf 6} is well-defined
and, for $e$ as in step {\bf 7}, step {\bf 8} is also well-defined.
By Proposition \ref{prop: eta_s_height_estimate}, the height bound
$\eta$ in step {\bf 10} can be computed. Consequently, the
polynomials $m_{r,p}, w_{n-r+2,p}^r\klk w_{n,p}^r$, the prime $p$
and the integer $\eta$ satisfy the input assumptions of Algorithm
\ref{algo: geometric_solution_reconstruction}, and the linear forms
$Y_1,\ldots,Y_n$, the point $\bfs p^r$ and the polynomials $m,
w_{n-r+2} \klk w_n$ of step {\bf 13} form the correct Kronecker
representation over $\Q$.
\end{proof}

Finally, we analyze the computational cost.
\begin{theorem}\label{th: cost algorithm over Q}
Let $F_1\klk F_r,G\in\Z[X_1\klk X_n]$ be polynomials of degree at
most $d>0$ with integer coefficients of bit length at most $h$,
satisfying hypotheses $({\sf H}_1)$--$({\sf H}_2)$. For $1\le s\le
r$, let $V_s$ denote the Zariski closure of
$\{F_1=\cdots=F_s=0,G\not=0\}$, and $\delta_s$ its degree. Define
$\delta:=\max\{\delta_1,\ldots,\delta_s\}$, and assume that
$\delta>d$. Suppose $F_1\klk F_r,G\in\Z[X_1\klk X_n]$ are given by a
straight--line program in $\Z[X_1\klk X_n]$ of length $L$ with
parameters of bit length at most $h$. Let $0<\varepsilon<1/18r$ and
$m\in \mathbb{N}$ be given. Then Algorithm \ref{algo: kronecker for
Q} computes a Kronecker representation of a lifting fiber of $V_r$
of degree $\delta_r$ with bit complexity
$$\SO\Big((n^5+n(L+n^2)d\delta^2)\log_2(\varepsilon^{-1})+ n^2Ld^r\delta_r(h+\log(\varepsilon^{-1}))
+\log_2^3(\varepsilon^{-1})\Big),$$
and probability of success at least
$(1-2\varepsilon)(1-8r\varepsilon)\ge (1-9r\varepsilon)$.
\end{theorem}
\begin{proof}
By Lemma \ref{lemma: computation_of_a_lucky_prime}, step {\bf 3}
requires
$\SO\big(\log_2^2(d^r2^nh\varepsilon^{-1})\log(\varepsilon^{-1})\big)$
bit operations. By Theorem \ref{th: main algorithm for fq}, steps
{\bf 6}--{\bf 8} require
$$\SO\Big(\big(n^4+r^2(L+n^2+r^3)d\delta^2\big)\log_2p\Big)$$
bit operations, and since $\log_2p\in
\SO(\log_2(d^r2^nh\varepsilon^{-1}))$, these steps need
$$\SO\Big(\big(n^5+r^2n(L+n^2+r^3)d\delta^2\big)\log_2(h\varepsilon^{-1})\Big)$$
bit operations. By Lemma \ref{lemma: lifting_the_integers}, step
{\bf 12} adds $\SO\big((Lr+r^4)\delta_r\eta+n\delta_r\log_2p\big)$
bit operations. Taking into account the height of $p$ and the
estimate for $\eta$ of Proposition \ref{prop:
eta_s_height_estimate}, this is
$$\SO\big(n^2d^rL\delta_r(h+\log_2(\varepsilon^{-1}))\big)$$
bit operations, which yields the total cost as claimed. The
probability of success follows from Lemmas \ref{lemma:
lambda_and_p-prob} and \ref{lemma: computation_of_a_lucky_prime} and
the probability of success of the main loop of Algorithm \ref{algo:
main algorithm for k}.
\end{proof}

%
%----------------------------------------------------------------------
%----------------------------------------------------------------------
%----------------------------------------------------------------------
%----------------------------------------------------------------------
%----------------------------------------------------------------------
%----------------------------------------------------------------------
%----------------------------------------------------------------------
%----------------------------------------------------------------------
%

\bmhead{Acknowledgements}

The authors would like to thank the anonymous referees for their thorough reading of the manuscript and for their insightful comments and suggestions, which helped to improve both the clarity and the accuracy of the paper. Their observations led to several corrections and to a clearer exposition of the algorithmic and probabilistic arguments.

The authors were partially supported by the grants
PIP CONICET 11220210100361CO, PIO CONICET-UNGS 14420140100027 and UNGS
30/3084

%\bibliography{refs1,finite_fields,numeric,polyhedral,Puiseux_Pham}



\end{document}